\documentclass{amsart}

\usepackage{amsfonts,amssymb,mathtools,mathabx,mathrsfs}

\usepackage[usenames,dvipsnames]{xcolor}
\usepackage{enumerate}
\usepackage{palatino}
\usepackage{todonotes}
\usepackage{verbatim}

\usepackage{tikz}
\usetikzlibrary{arrows, positioning, calc, intersections}
\usetikzlibrary{decorations.pathreplacing, decorations.markings}

\usepackage{hyperref}
\hypersetup{
    unicode=false,          		% non-Latin characters in Acrobat’s bookmarks
    pdftoolbar=true,        		% show Acrobat’s toolbar?
    pdfmenubar=true,        	% show Acrobat’s menu?
    pdffitwindow=false,     		% window fit to page when opened
    pdfstartview={FitH},    		% fits the width of the page to the window
    linktocpage=true,			% link to page not section title in TOC
    pdfnewwindow=true,      	% links in new PDF window
    colorlinks=true,       			% false: boxed links; true: colored links
    linkcolor=blue,          		% color of internal links (change box color with linkbordercolor)
    citecolor=PineGreen,    	% color of links to bibliography
    filecolor=magenta,      		% color of file links
    urlcolor=cyan           		% color of external links
}

\theoremstyle{plain}
\newtheorem{theorem}{Theorem}[section]
\newtheorem{corollary}[theorem]{Corollary}
\newtheorem{lemma}[theorem]{Lemma}
\newtheorem{proposition}[theorem]{Proposition}

\theoremstyle{definition}
\newtheorem{definition}[theorem]{Definition}

\theoremstyle{remark}
\newtheorem{remark}[theorem]{Remark}

\numberwithin{figure}{section}
\numberwithin{equation}{section}

\allowdisplaybreaks

\usepackage{float}

\usepackage[font=small, format=plain, labelfont=bf, up, textfont=it, up]{caption}
\usepackage{subcaption}

\DeclareMathOperator{\dist}{dist}
\DeclareMathOperator{\real}{Re}
\DeclareMathOperator{\imag}{Im}
\DeclareMathOperator{\sgn}{sgn}

\DeclareFontFamily{U}{mathx}{\hyphenchar\font45}
\DeclareFontShape{U}{mathx}{m}{n}{
      <5> <6> <7> <8> <9> <10>
      <10.95> <12> <14.4> <17.28> <20.74> <24.88>
      mathx10
      }{}
\DeclareSymbolFont{mathx}{U}{mathx}{m}{n}
\DeclareFontSubstitution{U}{mathx}{m}{n}
\DeclareMathAccent{\widecheck}{0}{mathx}{"71}
\DeclareMathAccent{\wideparen}{0}{mathx}{"75}

\begin{document}
\title{The Direct Scattering Transform for the Intermediate Long Wave Equation}
%\author{Joel Klipfel}
%\address[Klipfel]{ Matrix Research, Inc., Dayton, Ohio 43450}
%\email{ joel.klipfel@matrixresearch.com}
\author{Peter Perry}
\address[Perry]{Department of Mathematics, University of Kentucky, Lexington, Kentucky, 40506--0027}
%% PAP corrected his email
\email{pperr0@uky.edu}
\author{Yilun Wu}
\address[Wu]{Department of Mathematics, University of Oklahoma, Norman, Oklahoma, 73019-3103}
\email{allenwu@ou.edu}
\date{\today}
%\thanks{This is the current  version of the paper.}

\begin{abstract}
    In this paper, we provide a rigorous study of the direct scattering problem proposed by Kodama, Ablowitz and Satsuma for the complete integrability theory of the Intermediate Long Wave (ILW) equation. We derive the Lax operators for ILW, and study their spectral theory. We show finiteness and simplicity of the discrete spectrum, existence and uniqueness of the Jost solutions, and construct the scattering data for the Lax operator. We also show that the Jost solutions satisfy a nonlocal Riemann-Hilbert problem whose singularity data are given by the scattering data, and that the scattering potential can be reconstructed from the Jost solutions. This work paves the way for a  solution to ILW via the inverse scattering transform.
\end{abstract}

\maketitle

\tableofcontents

%%%%%%%%%%%%%%%%%%%
%		
%		TikZ Stuff
%
%%%%%%%%%%%%%%%%%%%

%% arrow in middle of line ->- 
\tikzset{->-/.style={decoration={
  markings,
  mark=at position .55 with {\arrow[scale=0.6]{triangle 45}} },postaction={decorate}}
}

%%%%%%%%%%%%%%%%%%%

%%

\newcommand{\dee}{\partial}
\newcommand{\dotarg}{\, \cdot \,}
\newcommand{\eps}{\varepsilon}
\newcommand{\lam}{\lambda}

\newcommand{\C}{\mathbb{C}}
\newcommand{\N}{\mathbb{N}}
\newcommand{\R}{\mathbb{R}}
\newcommand{\Z}{\mathbb{Z}}

\newcommand{\calC}{\mathcal{C}}
\newcommand{\calD}{\mathcal{D}}
\newcommand{\calF}{\mathcal{F}}
\newcommand{\calS}{\mathcal{S}}

\newcommand{\scrS}{\mathscr{S}}

\newcommand{\calH}{\mathcal{H}}

\newcommand{\loc}{{\mathrm{loc}}}
\newcommand{\unif}{{\mathrm{unif}}}

\newcommand{\norm}[2][ ]{\left\Vert #2 \right\Vert_{#1}}
\newcommand{\bigO}[2][ ]{\mathcal{O}_{#1} \left( #2 \right)}

\newcommand{\dint}{\displaystyle{\int}}

\newcommand{\mathitem}[1]{\item[${#1}$]}

\newcommand{\ba}{\breve{a}}
\newcommand{\bb}{\breve{b}}

\newcommand{\sidenote}[1]{\marginpar{\scriptsize{\color{purple} #1}}}

\newcommand{\cl}{\text{cl}}

\newcommand{\up}{\uparrow}
\newcommand{\dn}{\downarrow}

\newcommand{\supn}{^{(n)}}

%%%%%%%%%%%%%%%%%%%%%%%%%%%%%%%%%%%%%%%%%%%%%%%%%%%%%%%%%%%%%%%
%
%		Section structure modified by PAP on 2025-05-20
%
%%%%%%%%%%%%%%%%%%%%%%%%%%%%%%%%%%%%%%%%%%%%%%%%%%%%%%%%%%%%%%%

%%%%%%%%%%%%%%%%%%%%
%
%		intro.tex	 	-	Introduction
%
%%%%%%%%%%%%%%%%%%%%

\section{Introduction}

\subsection{Background and main results}
The Intermediate long wave equation (ILW) is the nonlinear dispersive equation
\begin{equation}
\label{ILW}
u_t + \tfrac{1}{\delta} u_x + 2uu_x + T^\delta u_{xx} = 0 
\end{equation}
where $u$ is a real-valued function of $(x,t) \in \R^2$, $\delta>0$ is the fluid depth parameter,  and $T^\delta$ is the singular integral operator
\begin{equation}
\label{T}
(T^\delta f)(x) = \text{P.V.} \frac{1}{2\delta} \int_{-\infty}^\infty \coth \frac{\pi(y-x)}{2\delta} f(y) \, dy. 
\end{equation}
Kubota, Ko, and Dobbs \cite{KKD:1978} introduced \eqref{ILW} as a model of weakly nonlinear wave propagation in a fluid of finite depth which interpolates between the Benjamin-Ono (BO) equation, corresponding to deep water, and the Korteweg-de Vries (KdV) equation, corresponding to shallow water. Sharp well-posedness results of these equations on the line and on the circle are known (see \cite{killip2019kdv}, \cite{kappeler2006well}, \cite{killip2024sharp}, \cite{gerard2023sharp}, 
\cite{gerard2023low}, \cite{gassot2025global}), with the exception of ILW on the line, for which well-posedness has been established on $L^2(\mathbb R)$ (\cite{ifrim2025lifespan}).

All of these equations are completely integrable in the sense of admitting a Lax representation. For rigorous theories, the integrability of KdV on the line is classically established via inverse scattering of the Schr\"odinger operator (\cite{DeiftTrubowitz1979}, \cite{BealsCoifman1985}), while on the circle it is given via global Birkhoff coordinates (\cite{kappeler2013kdv}).
Rigorous integrability theories of BO have also been developed, either by inverse scattering (\cite{coifman1990scattering}, \cite{Wu:2016}, \cite{Wu:2017}), global Birkhoff coordinates (\cite{gerard2021integrability}), or even an exact solution formula (\cite{gerard2023explicit}).

In comparison, rigorous integrability theory of ILW is missing, although formal inverse scattering problems have been discovered. 
\begin{comment}\footnote{For the BO equation on the line, see however the papers of Coifman-Wickerhauser \cite{CW:1990} and the more recent work of the third author \cite{Wu:2016,Wu:2017}; a rigorous analysis of complete integrability for the periodic Benjamin-Ono equation has recently been given by G\'{e}rard and Kappeler \cite{GK:2019}}.  This is the first in a series of papers devoted to the completely integrable method for the ILW equation.

The integral operator $T$ may be regarded as a Fourier multiplier with singular symbol
\begin{equation}
\label{sT}
\sigma(T)(\xi) = i \coth (\delta \xi).
\end{equation}
so that $\delta^{-1} \dee_x + T \dee_x^2$ acts as a Fourier multiplier with symbol 
$$ i \left(-\xi^2 \coth(\delta \xi) + \delta^{-1} \xi \right) \sim
\begin{cases}
-i\dfrac{\delta \xi^3}{3}, 	&	\xi \to 0,\\
\\
-i\xi^2 \sgn(\xi),				&	\xi \to \infty.
\end{cases}
$$
Thus, formally, the small-$\delta$ limit of \eqref{ILW} is the KdV equation
\begin{equation}
\label{KdV}
u_t + 2uu_x + (\delta/3) u_{xxx} = 0
\end{equation}
while the large-$\delta$ limit of \eqref{ILW} is the Benjamin-Ono equation
\begin{equation}
\label{BO}
u_t + 2uu_x + H\left( u_{xx} \right) = 0
\end{equation}
where $H$ is the Hilbert transform, a Fourier multiplier with symbol $i \sgn(\xi)$. 

Henceforth, we will take $\delta=1$.
\end{comment}
Satsuma, Ablowitz and Kodama \cite{SAK:1979} showed that \eqref{ILW} can be written as the compatibility condition for a linear spectral problem and an evolution equation and formulated an inverse scattering approach to integrability.   Chen-Lee \cite{CL:1979} and Joseph-Egri \cite{JE:1978} exploited complete integrability of ILW to find multisoliton solutions. Ablowitz, Kodama, and Satsuma  \cite{KAS:1982,KAS:1981} formulated the inverse scattering problem and studied the direct and inverse scattering transforms, while Santini, Ablowitz and Fokas \cite{SAF:1984} used inverse scattering to examine the BO limit of the ILW equation. A helpful exposition of inverse scattering for the ILW may be found in \cite[Chapter 4.2]{AC:1991}. We refer the reader to Saut's review \cite{Saut:2018}  of further PDE and inverse scattering results for the BO and ILW equations. 

Some early development of the rigorous scattering method of ILW was obtained in \cite{Klipfel:2020}. Our goal in the current paper is to provide a rigorous study of the whole direct scattering problem of ILW formulated in the literature, mostly under a small norm assumption on $u$. The only exception is in Section \ref{sec: finite discrete spec}, where we prove finiteness of the discrete spectrum  with only decay but no small norm assumptions. We now make a brief informal description of the direct scattering transform of ILW, referring to the sections and results in later parts of the paper for details. Set $\delta=1$.  Let $u$ be a real-valued function with suitable decay and small norm. We consider the spectral problem of the operator 
\begin{equation}
    L_u = e^D(D-\tilde u)e^D,
\end{equation}
where $D=\frac1i\partial_x$, and $\tilde u$ is the operator of multiplication by $u$. Refer to Section \ref{sec: lax pair} for a rigorous definition of $L_u$. For $k>0$, $k\ne\frac12$, there exist unique $L^\infty$ solutions (Sections \ref{sec: lax pair} and \ref{sec: existence}) $e^{-ikx}_{\pm\infty}$ satisfying 
\begin{equation}
    L_u e^{-ikx}_{\pm\infty}=-ke^{-2k}e^{-ikx}_{\pm\infty},
\end{equation}
with the asymptotic limit 
\begin{equation}
    e^{-ikx}_{\pm\infty}\sim e^{-ikx} \quad \text{as } x\to\pm\infty.
\end{equation}
These Jost solutions are related to each other by the transmission and reflection coefficients. Indeed, there exist $\tau(k)$ and $\rho(k)$ (Sections \ref{sec: lax pair} and \ref{sec: scat data}) such that 
\begin{equation}
    \tau(k)e^{-ikx}_{-\infty}=e^{-ikx}_{+\infty}+\rho(k)e^{-ik^*x}_{+\infty},
\end{equation}
where $k^*$ is the unique positive number such that $k^*\ne k$, $-k^* e^{-2k^*}=-ke^{-2k}$. $L_u$ has simple bound states $\varphi_j\in L^2$ with eigenvalues $\eta_j<-\frac1{2e}$, $j=1,\dots, N$ (Section \ref{sec: bound states}). Let $\zeta_j = -\frac12W_{-1}(2\eta_j)$, where $W_{-1}$ is the $-1$-st branch of the Lambert $W$ function. $\varphi_j$ can be normalized such that  $\varphi_j e^{i\overline{\zeta_j}x}\to 1$ as $x\to+\infty$. Define the norming constants to be 
\begin{equation}
    C_j = \left(\int_{\R}|\varphi_j|^2~dx\right)^{-1}.
\end{equation}
$\{\rho(k),\zeta_j,C_j\}$ for ${k>0,k\ne\frac12,j=1,\dots,N}$ together form the scattering data $\Sigma$. 
Let 
\begin{equation}
    N_1(x,k)=e^D(e^{ikx}e^{-ikx}_{+\infty}).
\end{equation}
We show that $N_1(x,k)$ can be extended to a meromorphic function $N_1(x,\zeta)$ for $\zeta\in \C\setminus \R^+$ such that $N_1(x,\zeta)$ satisfies a nonlocal Riemann-Hilbert (RH) problem with singular data $\Sigma$. In particular, $N_1(x,k)=N_1(x,\zeta)\big|_{\zeta=k-i0}$ (Lemma \ref{lem: anal cont R+}),  and $N(x,\zeta)$ satisfies (Proposition \ref{prop: RH})
\begin{enumerate}[(i)]
    \item $N(x,\zeta)$ has boundary values on $\R^+$ with
\begin{equation}
    N_1(x,k+i0)-N_1(x,k-i0)=\rho(k)e^{ix\lambda}N_1(x,k^*-i0),
\end{equation}
where $\lambda=Z^{-1}(k)$ (see Definition \ref{def: Z inv}). 
\item The poles of $N_1(x,\zeta)$ are at each $\zeta_j$ with residue
\begin{equation}
    \underset{\zeta=\zeta_j}{\text{Res}}~ N_1(x,\zeta) = iC_j e^{\lambda_j (1+ix) }N_1(x,\overline{\zeta_j}),
\end{equation}
where $\lambda_j = Z^{-1}(\zeta_j)$.
\item
\begin{equation}
    N_1(x,\zeta)\to 1
\end{equation}
as $\zeta\to\infty$ in $\mathbb C\setminus \R$. 
\end{enumerate}
The potential $u$ can be recovered from $N_1(x,\zeta)$ by (Proposition \ref{prop: rec u})
\begin{equation}
    u(x) = -2 \real\lim_{\zeta\to\infty} \zeta(N(x,\zeta)-1).
\end{equation}
We also show in Appendix \ref{app: A} that for a sufficiently smooth and decaying solution $u(x,t)$, $\Sigma(t)$ evolves by 
\begin{align}
     \rho(k,t) &= \rho(k,0)e^{i(k^*-k)(1-k-k^*)t},\\
     \zeta_j(t) &= \zeta_j(0),\\
     C_j(t) &= C_j(0)e^{2\imag [\zeta_j(1-\zeta_j)]t}.
\end{align}
If the above nonlocal RH problem can be solved from the scattering data $\Sigma$, we may construct a solution to the Cauchy problem of ILW by the inverse scattering transform. We leave the study of solvability of the nonlocal RH problem to a future paper.

\subsection{Notations and conventions}
For the reader's convenience, we compile here a list of commonly used notations and conventions for standard symbols in the paper.
$~$
\medskip

\begin{tabular}{@{}l@{\qquad}l@{}}
 
%% Capital Roman
$G_L(x,\zeta)$ & 	Green's function for left-normalized scattering problems \eqref{def: G}\\
$G_R(x,\zeta)$ & Greens' function for right-normalized scattering problems \eqref{def: G}\\

$e^{-ikx}_{\pm \infty}$ & Jost solutions asymptotic to $e^{-ikx}$ as $x\to\pm\infty$ \eqref{eq: eigen original}\\
$M_1(x,\zeta)$ &	normalized Jost solution asymptotic to $1$ on the left \eqref{jost prty: def M}\\
$M_e(x,k)$ &	normalized Jost solution asymptotic to  $e^{i\lambda x}$ on the left \eqref{jost prty: def M_e}\\
$N_1(x,\zeta)$ &	normalized Jost solution asymptotic to $1$ on the right \eqref{jost prty: def N}\\
$N_e(x,k)$ &	normalized Jost solution asymptotic to $e^{i\lambda x}$ on the right \eqref{jost prty: def N_e}\\

%$R(\xi,\zeta)$ & $p(\xi,\zeta)^{-1}$ minus polar parts at $\xi=0$ and $\xi=\lam(\zeta)$, equal to $R(\xi,\zeta,0)$\\

$T$ &		the convolution operator $-\frac12\coth\frac{\pi x}{2}*=i\coth(D)$ \\
$T_\pm$ &	$T \pm i I=\frac{2i}{1-e^{\mp 2D}}$\\
$T_L(\zeta), T_R(\zeta)$ & the operators $f\mapsto G_L(\cdot,\zeta)*(uf)$ and $f\mapsto G_R(\cdot,\zeta)*(uf)$\\

%% Lowercase Roman
$a(\zeta),b(\zeta)$ &	scattering coefficients \eqref{def: a}, \eqref{def: b}\\
$\tau(k),\rho(k)$ & transmission and reflection coefficients \eqref{def: tau}, \eqref{def: rho}\\
%$e_\lam(x)$ &	$\exp(i\lam x)$\\
$\log^+ t$ &  max$(0,\log t)$\\
%$p$ &		Symbol $p(\xi;\zeta) = \xi - \zeta(1-e^{-2\xi})$\\

%%	Upper Case Greek
$\Gamma_L$, $\Gamma_R$ &	integration contour for $G_L$ and $G_R$ (Figure \ref{fig: 1})\\

%% Lower Case Greek
$Z(\lambda)$ & the meromorphic map $\frac{\lambda}{1-e^{-2\lambda}}$\\
$Z^{-1}(\zeta)$ & the principal branch of the inverse of $Z$ (Definition \ref{def: Z inv})\\
$\zeta^*$ &  $Z[-Z^{-1}(\zeta)]$\\

%% Miscellaneous
$\langle x \rangle$ &	$\left(1 + |x|^2 \right)^{1/2}$ \\
$x_\pm$ & $\max(0,\pm x)$\\
$L^{p}_\mu(\R)$ & space of $f$ with 
	$\norm[L^{p}_\mu]{f} \coloneqq  \norm[L^p]{\langle x \rangle^{\mu} f(x)} < \infty$\\

$L^p_{\mu,\pm}(\R)$ & space of $f$ with 
    $\norm[L^p_{\mu,\pm}]{f} \coloneqq \|(1+x_\pm)^\mu f(x)\|_{L^p}<\infty$\\

$H^s_\mu(\mathbb R)$ & space of $f$ with
    $\norm[H^s_{\mu}]{f} \coloneqq \|\langle x\rangle ^\mu f(x)\|_{H^s}<\infty$\\

$H^{k,k}(\mathbb R)$ & space of $f$ with $\|f\|_{H^{k,k}}:=\sum_{j=0}^k\|\langle x \rangle^j f\|_{H^{k-j}}<\infty$\\
$\mathbb H^p(S)$ & the $L^p$ Hardy space on the strip $S$\\
$\mathbb H^\infty_{-N}(S)$ & space of holomorphic $F$ on $S$ with $(1+|z|)^{-N}F(z)\in L^\infty(S)$\\
$\mathbb H^\infty_{-N,\pm}(S)$ & space of holomorphic $F$ on $S$ with $(1+x_\pm)^{-N}F(x+iy)\in L^\infty(S)$\\
 ${W_n}$ & The $n$-th branch of the Lambert $W$ function\\
 $\hat f$, $\mathcal F f$ & Fourier transform of $f$\\
 $\check f$, $\mathcal F^{-1}f$ & inverse Fourier transform of $f$\\
 $D$ & $\frac{1}{i}\partial_x$\\
 $\tilde u$ & the operator of multiplication by $u$
\end{tabular}

\bigskip

Our convention for the Fourier transform and its inverse are
$$\hat f(\xi) = \int_{\R} f(x)e^{-ix\xi}~dx,\quad \check f(x) = \frac{1}{2\pi}\int_{\R} f(\xi)e^{ix\xi}~d\xi.$$

\subsection{Acknowledgements}
This work was supported by a grant from the Simons Foundation/SFARI (359431, PAP). Y.W. was supported in part by NSF grant DMS-2006212. Joel Klipfel \cite{Klipfel:2020} obtained results that 
%are 
were
useful for the early development of this project. 

							%%	Introduction
%%%%%%%%%%%%%%%%%%%
%
%		strip.tex
%
%%%%%%%%%%%%%%%%%%%

\section{Function theory on the strip}

To prepare for the presentation of a Lax pair for ILW, in this section we develop the theory for functions analytic in the strip 
$$S = \left\{ z \in \C~|~ 0 < \imag z < 2 \right\}=\mathbb R+i(0,2).$$
In particular, these functions arise naturally as solutions to the following Riemann-Hilbert problem with periodic jump contours. We want to look for a function $F_p(z)$ that is analytic in $S_p=\bigcup_{k\in \mathbb Z} (S+2ik)$, the complement of the contours $\bigcup_{k\in\mathbb Z}(\mathbb R+2ik)$, periodic across different copies of $S$:
\begin{equation}
    F_p(z+2ik)=F_p(z)\quad \text{for all }k\in\mathbb Z,
\end{equation}
with boundary values $F_p(x+2ik\pm i0)$ from either side of the contours, such that the jump is given by a predetermined function $f$:
\begin{equation}
    F_p(x+2ik+i0)-F_p(x+2ik-i0)=2if(x) \quad \text{for all }x\in \mathbb R.
\end{equation}
It is easy to see, at least formally, that a solution is given by periodizing the Cauchy kernel:
\begin{equation}
    F_p(z)=\frac1{\pi }\int_\R \sum_{k\in \mathbb Z}\frac{f(s)}{s+2ik-z}~ds.
\end{equation}
The formal kernel can then be matched to the Mittag-Leffler expansion of $\coth(z)$:
\begin{equation}
    \frac{\pi }{2}\coth\left(\frac{\pi z}{2}\right)=\frac{1}{z}+\sum_{k=1}^\infty\left(\frac{1}{z+2ik}+\frac{1}{z-2ik}\right).
\end{equation}
This suggests that we define the regular and singular integrals with a coth kernel:
\begin{equation}
     \calS f (z) = -\frac12\left[\coth\frac{\pi (\cdot+iy)}{2}*f\right](x)=\frac{1}{2} \int_{\R} \coth\frac{\pi(s-z)}2 f(s) \, ds
\end{equation}
with $z=x+iy\in S_p$,
and we set
$$ Tf (x) = \text{P.V.} \int_{\R} \frac{1}{2} \coth \frac{\pi (s-x)}2 f(s) \, ds $$
with $x\in\R$,
where P.V. denotes the principal value. 
We further define:
\begin{equation}\label{def: T pm}
    T_\pm f=(T\pm i)f,~(T_i f)(x)=(\mathcal S f)(x+i) = -\tfrac12\tanh\tfrac{\pi x}{2}*f(x).
\end{equation}
Note that $\coth\tfrac{\pi z}{2}$ is bounded on $\{z\in\mathbb C:\epsilon\le \imag z\le 2-\epsilon\}$ for any $\epsilon\in (0,1)$, and $\coth\tfrac{\pi x}{2}-\frac{2}{\pi x}$ bounded and smooth for $x\in\mathbb R$. It follows from the $L^p$ and weak $(1,1)$ boundedness of the Hilbert transform that 
    \begin{equation}\label{eq: T bounded}
        \|T f\|_{L^p+L^\infty}\le C\|f\|_{L^1\cap L^p},
    \end{equation}
for $1<p<\infty$, and 
\begin{equation}\label{eq: T bounded L1}
    \|T f\|_{L^{1,\infty}+L^\infty}\le C\|f\|_{L^1}.
\end{equation}

We compute the Fourier transform of the above $\coth$ kernels.
\begin{lemma}\label{lem: FT of coth}
    For $0<y<2$,
    \begin{equation}\label{eq: FT1}
        \mathcal F \left[-\frac12\coth\frac{\pi(x+iy)}2\right](\xi) = 2i~\text{P.V.}\frac{e^{(2-y)\xi}}{e^{2\xi}-1}.
    \end{equation}
    Also
    \begin{equation}\label{eq: FT2}
        \mathcal F\left(-\frac12\text{P.V.}\coth\frac{\pi x}2\right)(\xi)=i~\text{P.V.}\coth\xi.
    \end{equation}
\end{lemma}
\begin{proof}
    We prove \eqref{eq: FT1} by computing the inverse Fourier transform of the right hand side. To that end, we note that 
    $$
    \int_\Gamma \frac{e^{(2-y)\xi}e^{ix\xi}}{e^{2\xi}-1}~d\xi=0,
    $$
    where $\Gamma=\Gamma^1_{\eps,R}+\Gamma^1_\eps+\Gamma^2_{\eps,R}+\Gamma^2_\epsilon+\Gamma^3_R$ is the closed contour that is almost the boundary of the rectangle $[-R,R]\times i[0,\pi]$, but dodging the poles $0$ and $i\pi$ by small upper and lower semicircles of radius $\eps$. Here $\Gamma^1_{\eps,R}$ is the lower edge excluding $(-\eps,\eps)$; $\Gamma^2_{\eps,R}$ is the upper edge excluding $(-\eps+i\pi,\eps+i\pi)$; $\Gamma^3_R$ are the two vertical edges; and the rest are the semicircles. It's easy to see that the function values on $\Gamma^2_{\eps,R}$ are proportional to those on $\Gamma^1_{\eps,R}$. It follows that
    \begin{equation}
        \left(1-e^{(2-y+ix)i\pi}\right)\int_{\Gamma^1_{\eps,R}} \frac{e^{(2-y)\xi}e^{ix\xi}}{e^{2\xi}-1}~d\xi=\int_{-\Gamma^1_\eps-\Gamma^2_\eps-\Gamma^3_R} \frac{e^{(2-y)\xi}e^{ix\xi}}{e^{2\xi}-1}~d\xi.
    \end{equation}
    We take the limit as $\epsilon\searrow 0$ and $R\nearrow \infty$. Since $0<y<2$, the integral on $\Gamma^3_R$ goes to 0 locally uniformly in $x$. By residue calculus and weak continuity of the Fourier transform, we get
    \begin{align}
        \mathcal F^{-1}\left[\text{P.V.}\frac{e^{(2-y)\xi}}{e^{2\xi}-1}\right](x)&=\frac{1}{1-e^{(2-y+ix)i\pi}}\left(\frac i2\cdot \frac12+\frac i2 \cdot \frac12e^{(2-y+ix)i\pi}\right)\notag\\
        &=\frac i4\coth\frac{\pi(x+iy)}{2}.
    \end{align}
    We compute the inverse Fourier transform of P.V.$\coth\xi$ in a similar way. Indeed,
    \begin{equation}
        (1-e^{-\pi x})\int_{\Gamma^1_{\eps,R}}e^{ix\xi}\coth \xi ~d\xi = \int_{-\Gamma^1_\eps-\Gamma^2_\eps-\Gamma^3_R}e^{ix\xi}\coth \xi ~d\xi.
    \end{equation}
    By Residue calculus, the integral on $-\Gamma^1_\eps-\Gamma^2_\eps$ converges to $i\pi (1+e^{-\pi x})$ locally uniformly in $x$ as $\eps\searrow 0$. 
    The integral on $-\Gamma^3_R$ is
    \begin{align}
        &~\int_0^\pi [e^{ix(-R+i\eta)}\coth(-R+i\eta)-e^{ix(R+i\eta)}\coth(R+i\eta)]~d\eta\notag\\
        =&~-\int_0^\pi [e^{ix(-R+i\eta)}+e^{ix(R+i\eta)}]~d\eta\label{eq: coth FT 1}\\
        &\quad +\int_0^\pi e^{ix(-R+i\eta)}(\coth(-R+i\eta)+1)~d\eta \label{eq: coth FT 2}\\
        &\quad - \int_0^\pi e^{ix(R+i\eta)}(\coth(R+i\eta)-1) ~d\eta.\label{eq: coth FT 3}
    \end{align}
    \eqref{eq: coth FT 1} equals $-\frac{2\cos(Rx)}x(1-e^{-\pi x})$, which converges weakly in distributions to  0  on $\mathbb R$ as $R\nearrow \infty$. Since $\coth(\mp R+i\eta)$ converges  to $\pm 1$ uniformly for $\eta\in(0,\pi)$ as $R\nearrow \infty$, \eqref{eq: coth FT 2} and \eqref{eq: coth FT 3} converges to 0 locally uniformly in $x$. It follows that $\int_{\Gamma^1_{\eps,R}}e^{ix\xi}\coth \xi ~d\xi$ converges weakly in distributions to $i\pi \coth\frac{\pi x}{2}$ on $\mathbb R\setminus \{0\}$ as $\eps\searrow 0$ and $R\nearrow \infty$. On the other hand, the same integral converges weakly in tempered distributions to $2\pi \mathcal F^{-1}[\text{P.V.}\coth \xi](x)$. It follows that 
    \begin{equation}
        \Delta = \mathcal F^{-1}[\text{P.V.}\coth \xi](x)-\frac i2\text{P.V.}\coth\frac{\pi x}{2}
    \end{equation}
    is a distribution supported on $\{0\}$. It remains to show $\Delta=0$. In fact, recall that $\mathcal F^{-1}[\text{sgn}(\xi)](x)=\text{P.V.}\frac{i}{\pi x}$, $\mathcal F^{-1}[\text{P.V.}\frac1\xi](x)=\frac{i}{2}\text{sgn}(x)$, and P.V.$\coth \xi -\text{P.V.}\frac1{\xi}- \sgn(\xi)\in L^2(\mathbb R)$. It follows that 
    \begin{equation}
        \Delta + \frac{i}{2}\text{P.V.}\coth\frac{\pi x}{2}-\text{P.V.}\frac{i}{\pi x}-\frac i2\text{sgn}(x) \in L^2(\mathbb R).
    \end{equation}
    Since the singularities of the two principle value distributions cancel at zero and $\text{sgn}(x)$ is in $L^2$ near zero, we get $\Delta$ is in $L^2$ near zero. Since $\Delta$ is supported at zero, it must be the zero distribution.
\end{proof}

\begin{lemma}\label{lem: bdry val Sf}
Let $f \in H^s\cap L^1$ or $f\in H^s\cap \dot H^{-1}$ for some $s>\frac12$. Then $\mathcal Sf(z)$ is analytic in $S$ and extends to a bounded uniformly continuous function on $\overline S$. The boundary values of $\calS f (z)$ are given as 
\begin{equation}\label{eq: S limit val}
    (\mathcal Sf)(x+i0)=(T_+f)(x),\quad (\mathcal S f)(x+i2)=(T_-f)(x),
\end{equation}
so that 
\begin{equation}
    (\mathcal S f)(x+i0)-(\mathcal S f)(x+i2)=2if(x).
\end{equation}
\end{lemma}

\begin{proof}
First assume $f\in H^s\cap L^1$. Analyticity of $\mathcal S f$ follows from the analyticity of $\coth\frac{ \pi(s-z)}2$ in $z \in S$ and Morera's Theorem. To compute boundary values of $\mathcal S f$, we note that 
$\coth\frac{\pi z}{2}-\frac{2}{\pi z}$
is bounded and uniformly continuous for Im $z\in [-1,1]$, with the singularity at $z=0$ being removable. It follows that $\mathcal S f(z)$ will extend to a bounded uniformly continuous function on $\overline S$ if $\mathcal C f(z)=\frac{1}{\pi}\int_\mathbb{R} \frac{f(s)}{s-z}~ds$ thus extends in $\{z\in \mathbb C~|~0\le \text{Im }z\le 1\}$ and $\{z\in \mathbb C~|~-1\le \text{Im }z\le 0\}$ respectively, since $\mathcal S f(z)$ is periodic with period $2i$. The uniformly continuous extension of $\mathcal C f$ is standard, since $f\in H^s\cap L^1\subset C^{0,s-\frac12}_{\text{loc}}\cap L^1$. The limiting values of $\mathcal S f$ on the boundary follows from Plemelj's formula.

Now let $f\in H^s\cap \dot H^{-1}$. Thus $\hat f(\xi) = |\xi|g(\xi)$ for some $g\in L^2$. By Lemma \ref{lem: FT of coth} we regard $(\mathcal S f)(x+iy)$ as 
\begin{align}
    2i\mathcal F^{-1}\left(\hat f(\xi) ~\text{P.V.}\frac{e^{(2-y)\xi}}{e^{2\xi}-1} \right)(x) &= 2i\mathcal F^{-1}\left(g(\xi)\frac{|\xi|e^{(2-y)\xi}}{e^{2\xi}-1}\right)(x)\notag\\
    &=\frac{i}{\pi}\int_\mathbb{R}g(\xi)\frac{|\xi|e^{(2-y+ix)\xi}}{e^{2\xi}-1}~d\xi.
\end{align}
Analyticity of $\mathcal Sf$ on $S$ again follows from Morera's theorem. Note that $f\in H^s\cap \dot H^{-1}$ implies 
\begin{equation}
    \bigg|g(\xi)\frac{\xi e^{(2-y)\xi}}{e^{2\xi}-1}\bigg|\le | f(\xi)|\frac{ (e^{2\xi}+1)}{|e^{2\xi}-1|}\in L^1
\end{equation}
Thus $|\xi| g(\xi)\frac{e^{(2-y)\xi}}{e^{2\xi}-1}$ converges to $|\xi| g(\xi)\frac{e^{2\xi}}{e^{2\xi}-1}$ in $L^1$ as $y\searrow 0$, and converges to $|\xi| g(\xi)\frac{1}{e^{2\xi}-1}$ in $L^1$ as $y\nearrow 2$. It follows that $\mathcal S f$ extends to a bounded continuous function on $\overline S$. The limiting values \eqref{eq: S limit val} now follow from Lemma \ref{lem: FT of coth}.

%$$ \int \coth(\pi(z-s)/2) f(s) ds = \int \left[ \frac{2}{\pi(z-s)} + R(z-s) \right] f(s) \, ds $$
%where $R(z)$ is a bounded continuous function in the half-closed strip $\R \times [0,2)$ The remaining statements now follow from the identities
%$$
%\lim_{\pm y \downarrow 0} \int \frac{2}{\pi(x+iy-s)} f(s) \, ds 
%	=	\mp i f(x) +  \frac{2}{\pi} P \int \frac{1}{x-s} f(s) \, ds 
%$$
%and the fact that
%$ \coth(x+i(2-y)-s) = \coth(x-iy-s).$
\end{proof}

We also have the following weaker version of the boundary limits.

\begin{lemma}
    Let $f\in L^1\cap L^p$ for some $1<p<\infty$. Then $\mathcal Sf(z)\in \mathbb H^p(S)+\mathbb H^\infty(S)$. There exists $C>0$ such that 
    \begin{equation}
        \sup_{0<y<2}\|\mathcal Sf(\cdot+iy)\|_{L^p+L^\infty}\le C\|f\|_{L^1\cap L^p}
    \end{equation}
    for all $f\in L^1\cap L^p$, and $\mathcal Sf(x+iy)\to T_\pm f(x)$ as $y\to 0^+$ or $y\to 2^-$ respectively in $L^p+L^\infty$ and almost everywhere. 
\end{lemma}
\begin{proof}
    The proof is similar to the first part of that of Lemma \ref{lem: bdry val Sf}. One uses the $L^p$ theory of $\mathcal C f$ instead of the H\"older theory. 
\end{proof}

Next, we prove an analogue of the Cotlar identity $$fg - (Hf)(Hg) + H(fHg) +H(gHf) = 0$$ satisfied by the Hilbert transform.

\begin{lemma}[Cotlar's Identity for $T$]
\label{lemma:Cotlar}
Let $f,g \in L^1(\R) \cap L^2(\R)$. The identity
\begin{equation}
\label{Cotlar}
fg - (Tf)(Tg) + T(f Tg) + T(gTf) +\frac14 (If)(Ig) = 0
\end{equation}
holds pointwise almost everywhere, where $If=\int_\mathbb{R} f~dx$.
\end{lemma}

\begin{proof}
First suppose that $f, g \in \scrS(\R)$. By the hyperbolic trig addition formula
$\coth(\alpha-\beta)=\frac{\coth\alpha\coth\beta-1}{\coth\beta-\coth\alpha}$, we get for $z\in S$
\begin{align}
    \mathcal Sf(z)\mathcal Sg(z) &=\frac14\int_\mathbb{R}\int_\mathbb{R}\coth\tfrac{\pi(s-z)}{2}\coth\tfrac{\pi(t-z)}{2}f(s)g(t)~ds~dt\notag\\
    &=\lim_{\epsilon\searrow 0}\frac14\int_\mathbb{R}\int_{|s-t|>\epsilon}\bigg[\coth\tfrac{\pi(t-z)}{2}\coth\tfrac{\pi(s-t)}{2}\notag\\
    &\qquad \qquad \qquad \qquad \qquad +\coth\tfrac{\pi(s-z)}{2}\coth\tfrac{\pi(t-s)}{2}+1\bigg]f(s)g(t)~ds~dt\notag\\
    &=\mathcal S(gTf)(z)+\mathcal S(fTg)(z)+\frac14 (If)( Ig).
\end{align}
Take the limit as $z\to x+i0$ and use Lemma \ref{lem: bdry val Sf} to get
\begin{equation}
    (Tf+if)(Tg+ig)=T(gTf)+igTf+T(fTg)+ifTg+\frac14(If)(Ig),
\end{equation}
from which \eqref{Cotlar} follows. Now if $f,g\in L^1\cap L^2$, there exists sequences $f_n, g_n\in \mathscr S$ such that $f_n\to f$, $g_n\to g$ in $L^1\cap L^2$. We have \eqref{Cotlar} with $f$, $g$ replaced by $f_n$, $g_n$. By \eqref{eq: T bounded}, $Tf_n\to Tf$ and $Tg_n\to Tg$ in $L^2+L^\infty$. Thus $f_nTg_n\to fTg$, $g_nTf_n\to gTf$ in $L^1$. \eqref{Cotlar} now follows from \eqref{eq: T bounded L1}.  
\end{proof}

%(representation $F = \coth*(F_+ - F_-)$)

%(singular integral operator theory for Hilbert transform and perturbations - or maybe this belongs in an appendix?)

%\sidenote{This is the revised version of the lemma from Allen's notes}

The following lemma follows from the theory of $\mathbb H^2$ spaces, except that the base space is $L^2+L^\infty_{-N}$ instead of just $L^2$. In fact, for a positive integer $N$, we define $\mathbb H^\infty_{-N}(S)$ to be the set of holomorphic functions $F(z)$ on $S$ such that 
$(1+|z|)^{-N}|F(z)|$ is bounded on $S$.

\begin{lemma}
\label{lemma:FT.strip}
Suppose that for some $N\in \mathbb N$, $F(z)\in \mathbb H^2(S)+\mathbb H^\infty_{-N}(S)$. %that
%\begin{equation}
%    \sup_{0<y<H}\|F(\cdot+iy)\|_{L^p+L^\infty}<\infty,
%\end{equation}
Then the limit as $y\searrow 0$, $y\nearrow 2$ of $F(\cdot+iy)$ exist almost everywhere, with the limits $F(\cdot+i0)$, $F(\cdot+i2)$ in $L^2+L^\infty_{-N}$. Furthermore, for any $y\in [0,2]$, 
\begin{equation}\label{Mpm:Fourier}
    \mathcal F[F(\cdot+iy)](\xi)=e^{-y\xi}\mathcal F[F(\cdot+i0)](\xi).
\end{equation}
%Furthermore, $y\mapsto F(\cdot+iy)$ is continuous from $[0,H]$ to $L^p+L^\infty$.
%Similar conclusions hold if $H<0$.

\end{lemma}

%\marginpar{\scriptsize{\color{purple} Is this still needed? The section it refers to no longer exists.}}
%\begin{remark}
%In case $F_2=0$ the bound on $|F(x+iy)|$ is not needed since $F_1$ is bounded and continuous on $\overline{S}$. A model for the function $F_2$ is the Cauchy integral $G$ of an $L^2$ function $g$ on the upper half-plane, which obeys the bound $|G(z)| \lesssim (\imag z)^{-1/2}$ and converges to the boundary value $C_+ g$ in $L^2$ sense. Lemma \eqref{lemma:FT.strip} will be applied in section \ref{sec:solutions} to study the analytic continuation of Jost functions to the strip $S$. The convolution $G_L*(uM)$, analytically continued to $S$, has the above properties.
%\end{remark}

\begin{remark}
The two sides of \eqref{Mpm:Fourier} are equal as distributions. Although the right hand side is apriori not a tempered distribution, the equality implies that it actually is.
\end{remark}

\begin{proof}

Write $F=F_1+F_2$ with $F_1\in \mathbb H^2(S)$, $F_2\in\mathbb H^\infty_{-N}(S)$. We have $\frac{F_1(z)}{(z-3i)^N}\in \mathbb H^\infty(S)$. It follows that $F_1(\cdot+iy)$ converges almost everywhere and in $L^2(\mathbb R)$, and $F_2(\cdot+iy)$ converges almost everywhere, as $y\searrow 0$ or $y\nearrow 2$, by the theory of $\mathbb H^p$ spaces. It follows that $F(\cdot+iy)$ converges as tempered distributions as $y\searrow 0$ or $y\nearrow 2$. So we only need to prove 
\begin{equation}\label{Mpm:Fourier 1}
    e^{a\xi}\mathcal F[F(\cdot+ia)](\xi)=e^{b\xi}\mathcal F[F(\cdot+ib)](\xi)
\end{equation}
for $a,b\in (0,2)$.
We integrate $F(z) e^{-iz\xi}$ on the rectangular contour with the horizontal lines at height $\imag z = a$ and $\imag z=b$, and with the two vertical lines at $\real z = -R$ and $\real z =R$.
It follows that
\begin{align*}
&e^{a\xi}\int_{-R}^R e^{-ix\xi} F(x+ia) \, dx 
		- e^{b\xi}\int_{-R}^R e^{-ix \xi}  F(x+bi) \, dx \\		= &-i \int_a^{b} e^{iR\xi}{e^{t\xi}} F(-R+it) \, dt + i \int_a^{b} e^{-iR\xi} e^{t\xi} F(R+it) \, dt
\end{align*}
The left hand side obviously converges to $e^{a\xi}\mathcal F[F(\cdot+ia)](\xi)-e^{b\xi}\mathcal F[F(\cdot+ib)](\xi) $ in $\calD'(\R)$ as $R\to\infty$.
We claim that the right hand side converges to $0$ in $\calD'(\R)$ as $R \to \infty$ . We consider the first right hand term since the other can be analyzed similarly. Integrating the first right-hand term against a test function 
$\varphi \in C_0^\infty(\R)$ and integrating by parts, we get
\begin{align} \label{eq: Mpm: Fourier 2}
&-i\int_{-\infty}^\infty \varphi(\xi) \int_a^{b} e^{iR\xi} e^{t\xi} F(-R+it) \, dt \, d\xi \notag\\
=&	\frac{1}{R} \int_a^{b} F(-R+it) 
				\int_{-\infty}^\infty e^{iR\xi} \left( t e^{t\xi} \varphi(\xi) + e^{t\xi} \varphi'(\xi) \right)\, d\xi
			\, dt
\end{align}
Assume $N=0$ for the moment so that $F_2\in \mathbb H^\infty(S)$. By standard $\mathbb H^2$ theory $F(-R+it)$ is uniformly bounded for $R$ large and $t\in [a,b]$. It follows that \eqref{eq: Mpm: Fourier 2} is bounded by
$$ \frac{C}{R} \int_{-\infty}^\infty (1+b)e^{b|\xi|}\left( |\varphi(\xi)| + |\varphi'(\xi)| \right) \, d\xi$$
which vanishes as $R \to \infty$ for any $\varphi \in C_0^\infty(\R)$. For higher values of $N$, simply integrate by parts multiple times to create higher powers of $R$ on the denominator to balance out the power growth of $F(-R+it)$. We omit the details.
%It now follows that, as distributions in $\calD'(\R)$, 
%$$ \widehat{F_+}(\xi) = e^{(2-\eps)\xi} \widehat{F}(\xi+(2-%\eps)i) $$
%where the right-hand side denotes the distribution Fourier transform with respect to $x$ of the distribution $F(x+iy)$.
%Using properties (i) and (ii) we have 
%$$\widehat{F}(\dotarg+(2-\eps)i) \to \widehat{F_-}%(\dotarg)$$ in $\calD'(\R)$
%as $\eps \downarrow 0$, from which \eqref{Mpm:Fourier} follows.
\end{proof}

In view of Lemma \ref{lemma:FT.strip}, we may extend the definition of $e^{-hD}$ to functions with exponential growth at infinity, which are typically not tempered distributions.
\begin{definition}\label{def: e^D}
Let $F(z)$ be holomorphic on $0<\imag z<h$, and let the limit of $F(\cdot+iy)$ as $y\searrow 0$ or $y\nearrow h$ exist almost everywhere, with the limits denoted by $F(\cdot+i0)$ and $F(\cdot+ih)$. We define    
\begin{equation}
    (e^{-hD} F(\cdot+i0))(x)=F(x+ih).
\end{equation}
A similar definition can be given when $h$ is negative.
\end{definition}

% \begin{lemma}
%     Let $u\in H^s\cap L^1$ for $s>\frac32$. Then there exists a bounded continuous function $F(x+iy)$ on the strip $\{x+iy~|~0\le y\le 2$ such that $F$ is holomorphic inside the strip, and $F(x+i0)=(T_+u_x)(x)$, $F(x+i2)=(T_-u_x)(x)$.
% \end{lemma}
\begin{lemma}\label{lem: T pm 0}
    Let $f\in H^s\cap L^1$ or $f\in H^s\cap \dot H^{-1}$ for some $s>\frac12$, 
    then
    \begin{equation}\label{eq: multi T id}
        e^{-D}\widetilde{T_+f}e^D=e^{D}\widetilde{T_-f}e^{-D}=\widetilde{T_i f},
    \end{equation}
    where the operators above act on functions analytic in the strip $\{z\in \mathbb C~|~-1<\imag z<1\}$ with almost everywhere boundary values.
\end{lemma}
\begin{proof}
    Denote by $F(z)$ the function that is analytic in the strip $-1<\imag z <1$ with almost everywhere boundary values. Then by Lemma \ref{lem: bdry val Sf},
    \begin{align}
        e^{-D}\widetilde{T_+f}e^D (F(\cdot+i0))&=e^{-D}[\mathcal Sf(\cdot+i0)F(\cdot-i)]\notag\\
        &=\mathcal Sf(\cdot+i)F(\cdot+i0)\notag\\
        &=\widetilde{T_i f}F(\cdot+i0).
    \end{align}
    The other equality in \eqref{eq: multi T id} can be proven similarly.
\end{proof}							%%	Function theory
%%%%%%%%%%%%%%%%%%%
%
%		lax.tex	- 	Lax Representation
%
%%%%%%%%%%%%%%%%%%%

\section{Lax representation}\label{sec: lax pair}

\subsection{The Lax operators}
The Lax pair of ILW was presented in \cite{KAS:1981, KAS:1982} as an overdetermined system of equations, with the ILW given as a compatibility equation for the system. However, to our needs we prefer a more classical operator form of the Lax pair. One could try to extract the Lax operators from the overdetermined system, but the connection cannot be seen easily without a significant modification of the setup. Instead, we define the Lax operators directly and demonstrate that the Lax equation is equivalent to the ILW equation. After defining the self-adjoint $L_u$ operator using the theory of semibounded quadratic forms, we engage in some initial study of the spectrum. %In this section, we first treat the operators on a formal level, where we assume that $u$ has $C_0^\infty$ Fourier transform in space, and let the Lax operators $L_u$ and $B_u$ act on functions with compactly supported Fourier transform. 
In Appendix \ref{app: B}, we show how the Lax operators for ILW will converge to those of the KdV and BO equations, as the depth parameter $\delta$ tends to $0$ or $\infty$. 

Let us now present the Lax operators. For a given function $f$ on $\mathbb R$, denote by $\tilde f$ the operator of multiplication by $f$. %For a given function $m$ on $\mathbb R$, denote by $m(D)$ the Fourier multiplier $F(m(D)\varphi)(\xi)=m(\xi)\hat\varphi(\xi)$. 
We define for $u$  with suitable regularity and decay:
\begin{align}
L_u &= e^D(D-\tilde u) e^D,\label{Lax pair: L}\\
B_u&=iD^2+iD+i\widetilde{T_i  u_x}\label{Lax pair: B}\\
&= e^{-D}(iD^2+iD+i\widetilde{T_+u_x})e^{D} \label{Lax pair: B 1}\\
&=e^{D}(iD^2+iD+i\widetilde{T_-u_x})e^{-D}. \label{Lax pair: B 2}
\end{align}
The equivalence of the above several forms of $B_u$ is an application of Lemma \ref{lem: T pm 0}. We use Definition \ref{def: e^D} to interpret $e^{\pm D}$. Thus the domain of $L_u$ apparently includes only functions analytic in the strip $\{z\in\mathbb C~|~-2<\imag z<0\}$. It may also seem that $L_u$ can be defined only if $u(x)$ has analytic extensions to the strip $\{z\in\mathbb C~|~-1<\imag z<0\}$. As will be shown in the following, the definition can be extended to include much more general $u$, using the theory of unbounded quadratic forms. This makes it convenient to apply the spectral theory of self-adjoint operators. However, when such a concern is not of immediate significance, one can essentially work with the operator $e^{-D}L_u= (D-\tilde u)e^D$, where the analytic extension of $u$ plays no role. Note from \eqref{def: T pm} that $T_i$ is a convolution operator with a real-valued kernel. So $B_u$ is skew-adjoint when $u$ is real-valued with suitable regularity and decay.

In the following, we state a version of the Lax equation with the above interpretation of $e^{-D}L_u$ and $e^{-D}B_u$.

\begin{lemma}\label{lem: Lax pair}
Let $u\in C^0_tH^2_x\cap C^1_tL^2_x$.  Then $u$ satisfies
\begin{equation}
    u_t+u_{x}+2uu_x+Tu_{xx}=0
\end{equation} if and only if  
\begin{equation}\label{eq: Lax}
    -\partial_t e^{-D}L_u +e^{-D}[L_u,B_u]=0.
\end{equation}
Here the left hand side of \eqref{eq: Lax} can be regarded as an operator acting on analytic functions on the strip $\{z\in\mathbb C ~|~-1<\imag z<1 \}$ with sufficiently smooth almost everywhere boundary values.
\end{lemma}
\begin{proof}
By \eqref{Lax pair: B 1} and \eqref{Lax pair: B 2},  
\begin{align}\label{eq: Lax commutator}
e^{-D}[L_u,B_u] &= [D-\tilde u, iD^2+iD] e^D\notag\\
&\quad +i\big((D-\tilde u)\widetilde{T_+u_x}-\widetilde{T_-u_x}(D-\tilde u)\big)e^D\notag\\
&=[-i\partial_x^2+\partial_x,\tilde u]e^D + (\widetilde{T_+u_{xx}}+2i\widetilde{u_x}\partial_x +2\widetilde{uu_x})e^D\notag\\
&=(\widetilde{u_x}+2\widetilde{uu_x}+\widetilde{Tu_{xx}})e^D.
%&= -i[e^D\tilde u e^D,D^2]-i[e^D\tilde u e^D, D]\notag\\
%&\quad +i[De^{2D},e^{-D}\widetilde{T_+u_x}e^{D}]-i[e^D\tilde u e^D,e^{-D}\widetilde{T_+u_x}e^{D}]\notag\\
%&=e^D(-%i\widetilde{u_{xx}}+2\widetilde{u_x}D)e^D+e^D\widetilde{u_x}e^D\notag\\
%&\quad + e^{D}(\widetilde{T_+u_{xx}}-2\widetilde{u_x}D)e^D+e^D2\widetilde{uu_x}e^D\notag\\
%&=e^D(\widetilde{u_x}+2\widetilde{uu_x}+\widetilde{Tu_{xx}})e^D.
\end{align}
%The above calculation is mostly straightforward, except for a few steps that require the following identity:
%\begin{equation}
%e^{-D}\widetilde{T_+u_{x}}e^{D}=e^D\widetilde{T_-u_x}e^{-D}.
%\end{equation}
%This identity can be proven by showing equality of the Fourier transforms of $e^{-2D}\widetilde{T_+u_{x}}e^{2D}\varphi$ and of $\widetilde{T_-u_x}\varphi$.
The assertion of the lemma is clear in view of \eqref{eq: Lax commutator}.
\end{proof}

For later rigorous studies, we need to describe the domain of $L_u$ more carefully. For $u\in L^\infty$ and real, we will define $L_u$ through the quadratic form
\begin{equation}
    q_u(\varphi,\psi)= \langle |D|^{\frac12}e^D\varphi,\text{sgn}(D)|D|^{\frac12}e^D \psi\rangle -\langle \tilde u e^D \varphi,e^D\psi\rangle,
\end{equation}
with form domain $Q=\{\varphi\in L^2~|~ |D|^{\frac12}e^D\varphi\in L^2\}$. We have 
\begin{lemma}\label{lem: domain of Lu}
    Let $u\in L^\infty$ be real-valued. Then there is a unique self-adjoint operator $L_u$ on $L^2$ whose associated quadratic form is $q_u$. The domain $D(L_u)$ of $L_u$ can be described as follows. $\varphi\in D(L_u)$ if and only if $\varphi\in L^2$, $e^D\varphi\in H^1$, and there exists a $\psi\in L^2$ such that $e^{-D}\psi\in L^2$ and
    \begin{equation}\label{eq: Lu phi = psi}
        De^D\varphi- \tilde u e^D\varphi = e^{-D}\psi.
    \end{equation}
    In this case, $L_u\varphi=\psi$.
\end{lemma}
\begin{proof}
    Since $\xi e^{2\xi}\ge -\frac1{2e}$ when $\xi\le 0$, we have for all $\varphi\in Q$
    \begin{align}
        |2\langle \tilde u e^D\varphi,e^D\varphi\rangle |&\le 2\|u\|_{L^\infty}\langle e^D \varphi, e^D\varphi \rangle\notag\\
        &=\int_\R 2\|u\|_{L^\infty} e^{2\xi}|\widehat\varphi(\xi)|^2~d\xi\notag\\
        &\le \int_\R (\xi e^{2\xi}+\tfrac1{2e}+2\|u\|_{L^\infty}e^{2\|u\|_{L^\infty}})|\widehat\varphi(\xi)|^2~d\xi\notag\\
        &=\langle |D|^{\frac12}e^D\varphi,\text{sgn}(D)|D|^{\frac12}e^D \varphi\rangle+(\tfrac1{2e}+2\|u\|_{L^\infty}e^{2\|u\|_{L^\infty}})\|\varphi\|_{L^2}^2\notag\\
        &=q_u(\varphi,\varphi)+\langle \tilde u e^D\varphi,e^D\varphi\rangle + (\tfrac1{2e}+2\|u\|_{L^\infty}e^{2\|u\|_{L^\infty}})\|\varphi\|_{L^2}^2.
    \end{align}
    It follows that for $M=\frac1{2e}+2\|u\|_{L^\infty}e^{2\|u\|_{L^\infty}}$, we have
    \begin{equation}\label{eq: semi-bdd}
        q_u(\varphi,\varphi) + M\|\varphi\|_{L^2}^2\ge 0.
    \end{equation}
    In other words, $q_u$ is semi-bounded. In fact, we have the slightly stronger inequality
    \begin{equation}\label{eq: form bdds perturb term}
        |\langle \tilde u e^D\varphi,e^D\varphi\rangle|\le q_u(\varphi,\varphi)+M\|\varphi\|_{L^2}^2.
    \end{equation}
    We further estimate 
    \begin{align}
        \langle |D|^{\frac{1}2}e^D\varphi, |D|^{\frac12}e^D\varphi\rangle &= \int_\R |\xi|e^{2\xi}|\widehat\varphi(\xi)|^2~d\xi\notag\\
        &\le \int_\R (\xi e^{2\xi}+\tfrac{1}{e})|\widehat\varphi(\xi)|^2~d\xi\notag\\
        &=\langle |D|^{\frac12}e^D\varphi,\text{sgn}(D)|D|^{\frac12}e^D \varphi\rangle+ \tfrac{1}e \|\varphi\|_{L^2}^2\notag\\
        &=q_u(\varphi,\varphi)+\langle \tilde u e^D \varphi, e^D\varphi\rangle +\tfrac1e\|\varphi\|_{L^2}^2.
    \end{align}
    \eqref{eq: form bdds perturb term} now implies that 
    \begin{equation}\label{eq: form bdds norm}
        \langle |D|^{\frac{1}2}e^D\varphi, |D|^{\frac12}e^D\varphi\rangle\le 2q_u(\varphi,\varphi)+(\tfrac1e +M)\|\varphi\|_{L^2}^2.
    \end{equation}
    On the other hand, we have
    \begin{align}\label{eq: form bdd by norm}
        q_u(\varphi,\varphi)&=\int_\R \xi e^{2\xi}|\widehat\varphi(\xi)|^2~d\xi-\langle \tilde u e^D\varphi, e^D\varphi\rangle\notag\\
        &\le \int_\R (|\xi|+\|u\|_{L^\infty}) e^{2\xi}|\widehat\varphi(\xi)|^2~d\xi\notag\\
        &\le 2\int_\R |\xi|e^{2\xi}|\widehat\varphi(\xi)|^2~d\xi+ 2\|u\|_{L^\infty} e^{2\|u\|_{L^\infty}}\|\varphi\|_{L^2}^2.
    \end{align}
    Now assume a sequence $\{\varphi_n\}\in Q$ satisfies 
    \begin{equation}\label{eq: Cauchy in form norm}
        q_u(\varphi_n-\varphi_m,\varphi_n-\varphi_m)+ (M+1)\|\varphi_n-\varphi_m\|_{L^2}^2\to 0
    \end{equation}
    as $n,m\to\infty$. Using \eqref{eq: semi-bdd}, we obtain that \eqref{eq: Cauchy in form norm} implies $\{\varphi_n\}$ is a Cauchy sequence in $L^2$. Denote by $\varphi$ the $L^2$ limit of $\varphi_n$. A subsequence of $\{\varphi_n\}$ converges to $\varphi$ a.e. Now \eqref{eq: form bdds norm} implies that
    $
        \int_\R |\xi|e^{2\xi}|\widehat{\varphi_n}(\xi)-\widehat{\varphi_m}(\xi)|^2~d\xi\le 1
    $
    for $n,m$ sufficiently large. Letting $m$ pass to $\infty$ along the subsequence mentioned above, we obtain by Fatou's lemma
    \begin{equation}
        \int_{\R}|\xi|e^{2\xi}|\widehat{\varphi_n}(\xi)-\widehat\varphi(\xi)|^2~d\xi\le 1.
    \end{equation}
    Since $\varphi_n\in Q$, $\int_\R |\xi |e^{2\xi}|\widehat\varphi_n(\xi)|^2~d\xi<\infty$. It follows that $\varphi\in Q$. Similarly, for any $\epsilon>0$, we have $\int_\R |\xi|e^{2\xi}|\widehat{\varphi_n}(\xi)-\widehat{\varphi_m}(\xi)|^2~d\xi\le  \epsilon$ for $n,m$ sufficiently large. Arguing similarly as above, we get
    \begin{equation}
        \int_\R |\xi|e^{2\xi}|\widehat{\varphi_n}(\xi)-\widehat{\varphi}(\xi)|^2~d\xi\le \epsilon
    \end{equation}
    for $n$ sufficiently large. \eqref{eq: form bdd by norm} now implies that $q_u(\varphi_n-\varphi,\varphi_n-\varphi)\to 0$ as $n\to\infty$. Thus $q_u$ is closed. By Theorem VIII.15 in \cite{reed1980methods}, there exists a unique self-adjoint operator $L_u$ such that $q_u$ is the quadratic form associated with $L_u$. Moreover, $\varphi\in L^2$ is in the domain $D(L_u)$ if and only if $\varphi\in Q$ and there exists a $\psi\in L^2$ such that 
    \begin{equation}\label{eq: Lu phi = psi weak}
        q_u(\upsilon, \varphi)=\langle \upsilon,\psi\rangle 
    \end{equation}
    for all $\upsilon\in Q$. Such a $\psi$ is unique and we have $L_u\varphi=\psi$. Let $\varphi\in D(L_u)$ with $L_u\varphi = \psi$. We want to show that $e^D\varphi \in H^1$, $e^{-D}\psi\in L^2$ and \eqref{eq: Lu phi = psi} holds. In fact, take $\upsilon=\chi_{[0,N]}(D)D\varphi$, where $\chi_E$ denotes the characteristic function on the set $E$. Since $\varphi\in Q$, so is $\upsilon$. We also have
    \begin{equation}
        \|e^D\varphi\|_{L^2}^2=\int_\R e^{2\xi}|\widehat\varphi|^2~d\xi\le e\|\varphi\|_{L^2}^2+\int_\R |\xi e^{2\xi}\widehat\varphi|^2~d\xi<\infty.
    \end{equation}
    Thus $e^D\varphi\in L^2$. \eqref{eq: Lu phi = psi weak} gives
    \begin{align}
        &~\int_0^N\xi^2 e^{2\xi}|\widehat\varphi(\xi)|^2~d\xi\notag\\
        =&~\langle |D|^{\frac12}e^D\upsilon,\text{sgn}(D)|D^{\frac12}e^D\varphi\rangle \notag\\
        =&~\langle \tilde u e^D \upsilon,e^D\varphi\rangle + \langle \upsilon,\psi\rangle\notag\\
        \le &~ \|u\|_{L^\infty}\int_0^N ( e^\xi \widehat\varphi)( \xi e^\xi \overline{\widehat\varphi})~d\xi+\int_0^N (\xi \widehat\varphi )\overline{\widehat\psi}~d\xi\notag\\
        \le &~\|u\|_{L^\infty}\|e^D\varphi\|_{L^2}\left(\int_0^N\xi^2 e^{2\xi}|\widehat\varphi|^2~ d\xi\right)^{\frac12}+\|\psi\|_{L^2}\left(\int_0^N\xi^2 e^{2\xi}|\widehat\varphi|^2~ d\xi\right)^{\frac12}.
    \end{align}
    It follows that 
    \begin{equation}
        \left(\int_0^N\xi^2 e^{2\xi}|\widehat\varphi|^2~ d\xi\right)^{\frac12}\le \|u\|_{L^\infty}\|e^D\varphi\|_{L^2}+\|\psi\|_{L^2}.
    \end{equation}
    Passing to the limit as $N\to \infty$ gives
    \begin{equation}
        \left(\int_0^\infty\xi^2 e^{2\xi}|\widehat\varphi|^2~ d\xi\right)^{\frac12}\le \|u\|_{L^\infty}\|e^D\varphi\|_{L^2}+\|\psi\|_{L^2}.
    \end{equation}
    On the other hand, since $\xi^2 e^{2\xi}\le e^2$ when $\xi\le 0$, we get
    \begin{equation}
        \int_{-\infty}^0 \xi^2 e^{2\xi}|\widehat\varphi|^2~ d\xi\le e^2\|\varphi\|_{L^2}^2.
    \end{equation}
    Thus $De^D\varphi \in L^2$ and $e^D\varphi \in H^1$. 
    Now take $\upsilon=\chi_{[-N,0]}(D) e^{-2D}\psi$. We obviously have $\upsilon\in Q$. \eqref{eq: Lu phi = psi weak} gives
    \begin{align}
        &~\int_{-N}^0 e^{-2\xi}\left|\widehat\psi\right|^2~d\xi\notag\\
        =&~\int_{-N}^0 \xi e^{2\xi} e^{-2\xi}\widehat\psi \overline{\widehat \varphi}~d\xi-\langle ue^D\upsilon, e^D\varphi\rangle\notag\\
        \le &~\int_{-N}^0 (e^{-\xi}\widehat\psi)(\xi e^{\xi}\overline{\widehat\varphi})~d\xi+\|u\|_{L^\infty} \|e^D\upsilon\|_{L^2}\|e^D\varphi\|_{L^2}\notag\\
        \le &~\left(\int_{-N}^0 e^{-2\xi}\left|\widehat\psi\right|^2~d\xi\right)^{\frac12}(\left\|De^D\varphi\right\|_{L^2}+\|u\|_{L^\infty} \|e^D\varphi\|_{L^2})
    \end{align}
    It follows that
    \begin{equation}
        \left(\int_{-N}^0 e^{-2\xi}\left|\widehat\psi\right|^2~d\xi\right)^{\frac12}\le \left\|De^D\varphi\right\|_{L^2}+\|u\|_{L^\infty} \|e^D\varphi\|_{L^2}.
    \end{equation}
    Taking the limit as $N\to\infty$, we get
    \begin{equation}
        \left(\int_{-\infty}^0 e^{-2\xi}\left|\widehat\psi\right|^2~d\xi\right)^{\frac12}\le \left\|De^D\varphi\right\|_{L^2}+\|u\|_{L^\infty} \|e^D\varphi\|_{L^2}.
    \end{equation}
    On the other hand,
    \begin{equation}
        \int_0^\infty e^{-2\xi}\left|\widehat\psi\right|^2~d\xi\le \int_0^\infty \left|\widehat\psi\right|^2~d\xi\le \|\psi\|_{L^2}^2.
    \end{equation}
    Thus $e^{-D}\psi\in L^2$. Now take $\upsilon$ to be an arbitrary function such that $\widehat\upsilon\in C_0^\infty$. Since we have already shown $e^D\varphi\in H^1$, $e^{-D}\psi\in L^2$, \eqref{eq: Lu phi = psi weak} implies
    \begin{equation}
        \langle e^D\upsilon, De^{D}\varphi-\tilde u e^D\varphi\rangle = \langle e^D\upsilon,e^{-D}\psi\rangle.
    \end{equation}
    Since the set of all such $e^D\upsilon$ is dense in $L^2$, we get \eqref{eq: Lu phi = psi}.

    On the contrary, suppose $\varphi \in L^2$, $e^D\varphi\in H^1$, and there exists a $\psi\in L^2$ such that $e^{-D}\psi\in L^2$ and \eqref{eq: Lu phi = psi} holds, we want to show that $\varphi\in Q$, and \eqref{eq: Lu phi = psi weak} holds for all $\upsilon\in Q$. Indeed, by the definition of $Q$, we obviously have $\varphi\in Q$. For any $\upsilon\in Q$, we have $e^D\upsilon\in L^2$. Take the $L^2$ inner product of \eqref{eq: Lu phi = psi} with $e^D\upsilon$ to get
    \begin{align}
        \langle e^D \upsilon, D e^D \varphi\rangle -\langle \tilde u e^D\upsilon,e^D\varphi\rangle = \langle e^D\upsilon, e^{-D}\psi\rangle,
    \end{align}
    which is evidently equivalent to \eqref{eq: Lu phi = psi weak} if one rewrites the first and the last terms using Fourier transform.
\end{proof}

\begin{remark}\label{rmk: density of D(Lu) in Q(Lu)}
    By the proof of Theorem VIII.15 in \cite{reed1980methods}, $D(L_u)$ is dense in $Q$ under the norm
    $$\sqrt{q_u(\cdot,\cdot)+(M+1)\|\cdot\|_{L^2}^2}.$$
\end{remark}

\begin{remark}\label{rmk: domain of Lu}
Note that $\varphi\in L^2$ and $e^{\pm D}\varphi\in L^2$ if and only if $\sup_{y\in (-1, 1)}\|e^{-y\xi}\widehat\varphi(\xi)\|_{L^2}<\infty$. By
%\todo[size=\tiny]{LaTeX doesn't like this citation, which has whitespaces}
Theorem 2.3 in \cite{stein1971introduction}, this happens if and only if $\varphi\in \mathbb H^2(\R+i(-1, 1))$. 
\end{remark}

\begin{lemma}\label{lem: ess spec of Lu}
    Let $u\in L^p\cap L^\infty$ for some $p\in [1,\infty)$ be real-valued. Then $\sigma(L_u)=\sigma_{\text{disc}}(L_u)\cup \sigma_{\text{ess}}(L_u)$, where  $\sigma_{\text{ess}}(L_u)=[-\tfrac1{2e},\infty)$, and $\sigma_{\text{disc}}(L_u)\subset [-M,-\tfrac{1}{2e})$ for some $M>0$, and is a discrete set of eigenvalues which may accumulate only at $-\tfrac{1}{2e}$.
    %{\color{red} describe qualitatively the essential spectrum and the discrete spectrum of $L_u$ here}
\end{lemma}
\begin{proof}
    We will apply Weyl's essential spectrum theorem (Theorem XIII.14 in \cite{RSIV:1978}) to $L_u$ as a perturbation of $L_0=De^{2D}$. By the Fourier transform, it's easy to see that $\sigma(L_0)=\sigma_{\text{ess}}(L_0)=[-\tfrac{1}{2e},\infty)$. We only need to show that $(L_u+i)^{-1}-(L_0+i)^{-1}$ is compact. We first note that
    \begin{align}
        &(L_u+i)^{-1}-(L_0+i)^{-1}\notag\\
        =&(L_u+i)^{-1}-i(L_0+i)^{-1}(L_u+i)^{-1}-(L_0+i)^{-1}+i(L_0+i)^{-1}(L_u+i)^{-1}\notag\\
        =& (L_0+i-i)(L_0+i)^{-1}(L_u+i)^{-1}-(L_0+i)^{-1}(L_u+i-i)(L_u+i)^{-1}\notag\\
        =&L_0(L_0+i)^{-1}(L_u+i)^{-1}-(L_0+i)^{-1}L_u(L_u+i)^{-1}\notag\\
        =&[L_0(L_0+i)^{-1}-(L_0+i)^{-1}L_u](L_u+i)^{-1}.\notag
    \end{align}
    For any $\varphi\in L^2$, $(L_u+i)^{-1}\varphi\in D(L_u)$. By Lemma \ref{lem: domain of Lu}, $\langle D\rangle e^D (L_u+i)^{-1}\varphi\in L^2$. Moreover, $e^{-D}L_u(L_u+i)^{-1}\varphi\in L^2$. We can thus continue the above calculation to get
    \begin{align}
        &(L_u+i)^{-1}-(L_0+i)^{-1}\notag\\
        =&[L_0(L_0+i)^{-1}e^{-D}-(L_0+i)^{-1}e^D e^{-D}L_ue^{-D}]e^D(L_u+i)^{-1}\notag\\
        =& e^D (L_0+i)^{-1}[D-e^{-D}L_u e^{-D}]e^D(L_u+i)^{-1}\notag\\
        =& e^D(L_0+i)^{-1}~\tilde u ~e^D(L_u+i)^{-1}\notag\\
        =&[e^D (L_0+i)^{-1}](\tilde u \langle D\rangle^{-\frac12})[\langle D\rangle ^{\frac12}e^D(L_u+i)^{-1}].\label{eq: pertb resolv}
    \end{align}
    Here we have used $e^{-D}L_u e^{-D}=D-\tilde u$, which is equivalent to \eqref{eq: Lu phi = psi}. We observe that \eqref{eq: pertb resolv} is the product of three operators. Among them, $e^D(L_0+i)^{-1}$ is obviously bounded on $L^2$. We will show that $\tilde{u} \langle D\rangle^{-\frac12}$ is compact on $L^2$, and that $\langle D\rangle ^{\frac12}e^D(L_u+i)^{-1}$ is bounded on $L^2$. By interpolation, we may assume $u\in L^p$ with $p>2$ without loss of generality. Theorem 4.1 in \cite{Simon:2005} implies that $\tilde u \langle D\rangle^{-\frac12}$ is in the trace class $\mathcal I_p$ with trace norm bounded by 
    $$\|\tilde u \langle D\rangle^{-\frac12}\|_{\mathcal{I}_p}\le C\|u \|_{L^p}\|\langle \cdot\rangle ^{-\frac12}\|_{L^p}.$$
    Thus it is compact on $L^2$. For the last assertion, we note that $A=\langle D\rangle^{\frac12} e^D (L_u+i)^{-1}$ is defined on the entire $L^2$, since $(L_u+i)^{-1}$ maps into $D(L_u)\subset Q$, and $A^*=(L_u-i)^{-1}\langle D\rangle ^{\frac{1}{2}}e^{D}$ has a dense domain in $L^2$. The result follows from the following lemma.
\end{proof}

\begin{lemma}\label{lem: generalized H-T}
    Let $A:H\to H$ be a linear operator defined on an entire Hilbert space $H$. If $A^*$ is densely defined, then $A$ is bounded.
\end{lemma}
\begin{proof}
    We only need to prove that $A$ has a closed graph. For that, let $\{x_n\}$ be a sequence in $H$ such that $x_n\to x\in H$, $Ax_n\to y\in H$. For any $z\in D(A^*)$, we have 
    \begin{align*}
        (Ax,z)&=(x,A^*z)\\
        &=\lim_{n\to\infty}(x_n,A^*z)\\
        &=\lim_{n\to\infty} (Ax_n,z)\\
        &=(y,z).
    \end{align*}
    Since $D(A^*)$ is dense in $H$, we have $Ax=y$.
\end{proof}

Later it will be useful to approximate the function $u$ in $L_u$ by $C_0^\infty(\R)$ functions $u_n$. To have appropriate control of the spectrum, we need to show $L_{u_n}$ converges to $L_u$ in the norm resolvent sense. To estimate the difference of the resolvents $(L_u+i)^{-1}-(L_{u_n}+i)^{-1}$, we may consider formally the resolvent identity
\begin{align*}
    (L_u+i)^{-1}-(L_{u_n}+i)^{-1}&=(L_u+i)^{-1}(L_{u_n}-L_u)(L_{u_n}+i)^{-1}\\
    &=(L_u+i)^{-1}e^D(u-u_n)e^D(L_{u_n}+i)^{-1},
\end{align*}
and use the $L^2$ boundedness of $e^D(L_{u_n}+i)^{-1}$ and $(L_u+i)^{-1}e^D$. However, such an identity is difficult to justify directly, as it requires one to feed an output of $(L_{u_n}+i)^{-1}$ into $L_u$. In general, the output of $(L_{u_n}+i)^{-1}$ is only in $D(L_{u_n})$, which may have a small or trivial intersection with $D(L_u)$, so the composition may be undefined. To resolve this difficulty, we essentially need to extend the domain of $L_u$ and show that the identity $(L_u+i)^{-1}L_u = I-i(L_u+i)^{-1}$ continues to hold for the extended operator.
\begin{lemma}\label{lem: extended resolvent inv}
    Let $u\in L^\infty$ be real-valued, and let $L_u$ be the self-adjoint operator defined by Lemma \ref{lem: domain of Lu}, then $(L_u+i)^{-1}e^D \langle D\rangle$ is closable with a bounded closure on $L^2$, and for any $\varphi\in L^2$ such that $e^D\varphi\in L^2$, we have
    \begin{equation}\label{eq: generalized resolvent id}
        \overline{(L_u+i)^{-1}e^D \langle D\rangle} ~\langle D\rangle^{-1}(D-\tilde u)e^D\varphi = \varphi -i(L_u+i)^{-1}\varphi.
    \end{equation}
\end{lemma}
\begin{proof}
    $(L_u+i)^{-1}$ is defined on the entire $L^2$ space. The domain of $(L_u+i)^{-1}e^D\langle D\rangle$ consists of all functions $\varphi\in L^2$ such that $ e^D\langle D\rangle \varphi \in L^2$, which is dense in $L^2$. The $L^2$-boundedness of $(L_u+i)^{-1}e^D\langle D\rangle$ would follow from that of $\langle D\rangle e^D (L_u-i)^{-1}$ and duality. Since $(L_u-i)^{-1}$ maps into $D(L_u)$, which is contained in the domain of $e^D\langle D\rangle $ (see Lemma \ref{lem: domain of Lu}), it follows that $\langle D\rangle e^D (L_u-i)^{-1}$ is defined on the entire $L^2$ space with a densely defined adjoint. By Lemma \ref{lem: generalized H-T}, it is bounded. Thus $\overline{(L_u+i)^{-1}e^D \langle D\rangle} $ is defined on the entire $L^2$ space and is bounded. 
    
    We now turn to the proof of \eqref{eq: generalized resolvent id}. First observe that \eqref{eq: generalized resolvent id} holds if $\varphi\in D(L_u)$. Indeed, for such a $\varphi$, we have that $e^D\varphi \in H^1$, and there exists a $\psi\in L^2$ such that $e^{-D}\psi\in L^2$ and \eqref{eq: Lu phi = psi} holds and $L_u\varphi = \psi$. Thus 
    $$\langle D\rangle ^{-1} (D-\tilde u)e^D\varphi = \langle D\rangle ^{-1} e^{-D}\psi,$$
    which is in the original domain of $(L_u+i)^{-1}e^D\langle D\rangle$. Thus
    \begin{align*}
        \overline{(L_u+i)^{-1}e^D \langle D\rangle} ~\langle D\rangle^{-1}(D-\tilde u)e^D\varphi &= (L_u+i)^{-1}\psi\\
        &=\varphi - i(L_u+i)^{-1}\varphi.
    \end{align*}
    The last equality above follows by rewriting $L_u\varphi = \psi$ as 
    $$ \psi =(L_u+i)\varphi - i\varphi.$$
    To prove \eqref{eq: generalized resolvent id} for general $\varphi$, we  claim that for every $\varphi\in L^2$ such that $e^D\varphi\in L^2$, there exists a sequence of $\varphi_n\in D(L_u)$ such that $\varphi_n \to \varphi$ and $e^D\varphi_n\to e^D\varphi$ both in $L^2$. Assuming this claim, we first write \eqref{eq: generalized resolvent id} with $\varphi$ replaced by $\varphi_n$. Since $\overline{(L_u+i)^{-1}e^D \langle D\rangle} $, $\langle D\rangle ^{-1} (D-\tilde u)$ and $(L_u+i)^{-1}$ are all bounded on $L^2$, we may take the limit to recover the equation for $\varphi$. It remains to prove the above claim. 
    Let $X=\{\varphi\in L^2~|~e^D\varphi \in L^2\}$ with the norm $\|\varphi\|_X=\|\varphi\|_{L^2}+\|e^D \varphi\|_{L^2}$. We need to prove that $D(L_u)$ is dense in $X$ under its norm. We first observe that $Q\subset X$ is dense in $X$ under its norm, since it contains all $L^2$ functions with compactly supported Fourier transform.
    It suffices to show that $D(L_u)$ is dense in $Q$ under the $X$-norm.
    By Remark \ref{rmk: density of D(Lu) in Q(Lu)}, $D(L_u)$ is dense in $Q(L_u)$ under the norm 
    \begin{equation}\label{eq: H_+1 norm}
        \sqrt{q_u(\cdot,\cdot)+(M+1)\|\cdot \|_{L^2}^2}.
    \end{equation}
    By \eqref{eq: semi-bdd}, the above norm dominates the $L^2$ norm. By \eqref{eq: form bdds norm}, we see that \eqref{eq: H_+1 norm} also dominates the norm $\||D|^{\frac12}e^D\cdot\|_{L^2}$. On the other hand, it's easy to see that $\|\cdot\|_{L^2}+\||D|^{\frac12}e^D\cdot\|_{L^2}$ dominates the $X$-norm on $Q$. The proof is complete. 
\end{proof}

\begin{lemma}\label{lem: resolvent diff}
    Let $u,v\in L^\infty$ be real-valued, then
    \begin{equation}\label{eq: resolvent diff}
        (L_u+i)^{-1}-(L_v+i)^{-1}=\overline{(L_u+i)^{-1}e^D \langle D\rangle} \langle D\rangle ^{-1}(\tilde v-\tilde u)e^D(L_v+i)^{-1}.
    \end{equation}
\end{lemma}
\begin{proof}
    For any $\psi \in L^2$, $(L_v+i)^{-1}\psi \in D(L_v)\subset Q$. Thus $e^D(L_v+i)^{-1}\psi\in L^2$. So the right hand side of \eqref{eq: resolvent diff} is defined on $L^2$. Denoting $\varphi = (L_v+i)^{-1}\psi$, we have $L_v\varphi =\psi-i\varphi$, so that $e^{-D}(\psi-i\varphi)\in L^2$, and
    \begin{equation}
        (D-v)e^D\varphi = e^{-D}(\psi-i\varphi).
    \end{equation}
    We also have
    \begin{align}
       &~ (\tilde v-\tilde u)e^D (L_v+i)^{-1}\psi \notag\\
       = &~(D-\tilde u)e^D (L_v+i)^{-1}\psi-(D-\tilde v)e^D (L_v+i)^{-1}\psi\notag\\
        =&~(D-\tilde u)e^D (L_v+i)^{-1}\psi-e^{-D}(\psi-i(L_v+i)^{-1}\psi).\label{eq: resolvent diff 1}
    \end{align}
    By Lemma \ref{lem: extended resolvent inv}, 
    \begin{align}
        &~\overline{(L_u+i)^{-1} e^D\langle D\rangle }\langle D\rangle ^{-1}(D-\tilde u)e^D(L_v+i)^{-1}\psi\notag\\
        =&~(L_v+i)^{-1}\psi-i(L_u+i)^{-1}(L_v+i)^{-1}\psi.\label{eq: resolvent diff 2}
    \end{align}
    On the other hand $\langle D\rangle^{-1} e^{-D}(\psi-(L_v+i)^{-1}\psi)$ is in the original domain of $(L_u+i)^{-1}e^D\langle D\rangle$, so that
    \begin{align}
        &~\overline{(L_u+i)^{-1}e^D\langle D\rangle}\langle D\rangle^{-1} e^{-D}(\psi-(L_v+i)^{-1}\psi)\notag\\
        =&~(L_u+i)^{-1}(\psi-i(L_v+i)^{-1}\psi)\notag\\
        =&~(L_u+i)^{-1}\psi-i(L_u+i)^{-1}(L_v+i)^{-1}\psi.\label{eq: resolvent diff 3}
    \end{align}
    \eqref{eq: resolvent diff} now follows from \eqref{eq: resolvent diff 1}, \eqref{eq: resolvent diff 2}, and \eqref{eq: resolvent diff 3}.
\end{proof}

We can now show the lemma on norm resolvent convergence.

\begin{lemma}\label{lem: Lu norm resolv conv}
    Let $u, u_n\in L^2\cap L^\infty$ and $u_n\to u$ in $L^2$. Then $(L_{u_n}+i)^{-1}\to (L_u+i)^{-1}$ in $L^2$ operator norm.  
\end{lemma}
\begin{proof}
    By Lemma \ref{lem: resolvent diff}, 
    \begin{equation}
        (L_u+i)^{-1}-(L_{u_n}+i)^{-1}=\overline{(L_u+i)^{-1}e^D \langle D\rangle} \langle D\rangle ^{-1}(\tilde u_n-\tilde u)e^D(L_{u_n}+i)^{-1}.
    \end{equation}
    From Lemma \ref{lem: extended resolvent inv}, we have that $(L_u+i)^{-1}e^D\langle D \rangle $ is bounded on $L^2$. By Theorem 4.1 in \cite{Simon:2005} and duality, we know that the $L^2$ operator norm of $\langle D\rangle^{-1}(\tilde u_n-\tilde u)$ is bounded by $C\|u_n-u\|_{L^2}$. A simple application of Lemma \ref{lem: generalized H-T} will give us the $L^2$ boundedness of $e^D(L_{u_n}+i)^{-1}$. It remains to prove, however, that the bounds are uniform in $n$. For that we let $\|\psi\|_{L^2}=1$ and let $\varphi_n = (L_{u_n}+i)^{-1}\psi$. Since $L_{u_n}$ is self-adjoint, the operator norm of $(L_{u_n}+i)^{-1}$ is no more than 1, so $\|\varphi_n\|_{L^2}\le 1$. Thus we only need a uniform $L^2$ bound on the large and positive frequency part of $e^D\varphi_n$. Since $L_{u_n}\varphi =\psi-i\varphi_n $, we have from Lemma \ref{lem: domain of Lu} that $e^{-D}(\psi-i\varphi_n )\in L^2$, and
    \begin{equation}
        (D-u_n)e^D\varphi_n = e^{-D}(\psi-i\varphi_n).
    \end{equation}
    Fourier transform the above equation to get for $\xi\ge 1$
    \begin{equation}
        \widehat {e^D\varphi_n}=\frac1{\xi}\widehat {u_n}*\widehat{e^D\varphi_n}+\frac{e^{-\xi}}{\xi}(\hat\psi-i\widehat{\varphi_n})
    \end{equation}
    Denote by $\chi_{\ge M}$ the characteristic function of $[M,\infty)$. For $M\ge 1$ we get
    \begin{align}
        \|\chi_{\ge M}(D)e^D\varphi_n\|_{L^2}&\le \frac1{\sqrt M}\|\widehat{u_n}*\widehat{ e^D\varphi_n}\|_{L^\infty}+\|\psi\|_{L^2}+\|\varphi_n\|_{L^2}\notag\\
        &\le \frac{\|u_n\|_{L^2}}{\sqrt M}(\|\chi_{\ge M}(D)e^D\varphi_n\|_{L^2}+\|\chi_{< M}(D)e^D\varphi_n\|_{L^2})+2.\label{eq: hi freq e^D phi_n}
    \end{align}
    Since $u_n\to u$ in $L^2$, $\sup_n \|u_n\|_{L^2}<\infty$. Taking $\sqrt M=\max( 2\sup\|u_n\|_{L^2},1)$, and noting that $\|\chi_{<M}(D)e^D\varphi_n\|\le e^M\|\varphi_n\|_{L^2}\le e^M$, we get from \eqref{eq: hi freq e^D phi_n}
    \begin{equation}
        \|\chi_{\ge M}(D)e^D \varphi_n\|_{L^2}\le e^M+4.
    \end{equation}
    The proof is now complete.
\end{proof}

\subsection{Spectral problem and an outline of the paper}
We now consider $L^\infty$ eigenfunctions of $L_u$ with eigenvalue in the essential spectrum, which will be useful in scattering theory. For any $k\in\mathbb R$, we have
$$
L_0e^{-ikx}=-ke^{-2k}e^{-ikx}.
$$
Denote by $e^{-ik x }_{\pm\infty}$ the $L^\infty$ solution to
\begin{equation}\label{eq: eigen original}
     L_ue^{-ikx}_{\pm\infty}= -ke^{-2k}e^{-ikx}_{\pm\infty}
\end{equation}
with the asymptotic property
\begin{equation}\label{eq: asym original}
    e^{-ikx}_{\pm\infty} \sim e^{-ikx}\quad \text{ as }x\to\pm\infty.
\end{equation}
The existence and properties of such scattering functions will be established in  Sections \ref{sec: prop of jost} and \ref{sec: existence}.
We need to properly interpret $L_u$ acting on the $L^\infty$ function $e^{-ikx}_{\pm\infty}$. In view of the $L^2$ domain of $L_u$ in Lemma \ref{lem: domain of Lu} and definition of $e^{hD}$ given in Defintion \ref{def: e^D}, we interpret $L_u\varphi = \psi$, for $\varphi,\psi\in L^\infty$, as follows: $\varphi\in \mathbb H^\infty(\mathbb R+ i(-1,0))$, $\psi\in \mathbb H^\infty(\mathbb R + i(0,1))$, with 
\begin{equation}
    (D-u)\varphi(x-i)=\psi(x+i).
\end{equation}
With such an interpretation, we have that $e^{-ikx}_{\pm\infty}\in \mathbb H^\infty(\R+i(-1,1))$, with 
\begin{equation}
    (D-u)e^{-ik(x-i)}_{\pm\infty}=-ke^{-2k}e^{-ik(x+i)}_{\pm\infty}.
\end{equation}
We also analytically continue the asymptotic condition \eqref{eq: asym original} to the slightly stronger condition 
\begin{equation}
    e^{-ik(x+iy)}_{\pm\infty}\sim e^{-ik(x+iy)} \quad \text{ as }x\to\pm\infty \text{ for all }y\in (-1,1).
\end{equation}
Later on we will analytically continue $e^{-ikx}_{\pm\infty}$ in $k$  into the complex $\zeta$-plane (for which $k$ is the real axis). To avoid working with spaces of exponentially growing solutions, we define the normalized Jost solutions $M_1(x,k)$ and $N_1(x,k)$ by 
\begin{equation}\label{eq: M1 vs eik}
    M_1(x,k) = e^D(e^{ikx}  e^{-ikx}_{-\infty}) = e^{ik(x-i)} e^{-ik(x-i)}_{-\infty},
\end{equation}
\begin{equation}\label{eq: N1 vs eik}
    N_1(x,k)= e^D(e^{ikx} e^{-ikx}_{+\infty}) = e^{ik(x-i)} e^{-ik(x-i)}_{+\infty}.
\end{equation}
As we will show in Section \ref{sec: existence}, $M_1,N_1$ exist and belong to $\mathbb H^\infty(\mathbb R+ i(0,2))$ (except for $k=\tfrac12$, where $M_1$ and $N_1$ can have linear growth). They satisfy the equations
\begin{equation}\label{eq: main M1}
    [D-k(1-e^{-2D})]M_1= uM_1,
\end{equation}
\begin{equation}\label{eq: main N1}
    [D-k(1-e^{-2D})]N_1= uN_1,
\end{equation}
and asymptotic conditions
\begin{equation}
    M_1(x+iy,k)\sim 1 \quad \text{ as }x\to-\infty\text{ for all }y\in (0,2), 
\end{equation}
\begin{equation}
    N_1(x+iy,k)\sim 1 \quad \text{ as }x\to+\infty\text{ for all }y\in (0,2).
\end{equation}

Note that $k\mapsto -ke^{-2k}$ is one-to-one from $(-\infty,0)$ to $(0,\infty)$, but two-to-one from $(0,\infty)\setminus \{\tfrac12\}$ to $(-\tfrac{1}{2e},0)$. We denote the other value by $k^*$ so that $-ke^{-2k}=-k^*e^{-2k^*}$. It can be described alternatively as follows. Note that the map $\lambda\mapsto \frac{\lambda}{1-e^{-2\lambda}}$ (with the singularity at $0$ removed) maps $\mathbb R$ to $(0,\infty)$ bijectively. If $k=\frac{\lambda}{1-e^{-2\lambda}}$, then $k^*=\frac{-\lambda}{1-e^{2\lambda}}$. Thus $e^{-ik^*x}_{\pm\infty}$ solve the same equation as $e^{-ikx}_{\pm\infty}$, but is asymptotic to $e^{-ik^*x}$ at $\pm\infty$. As we will show in Section \ref{sec: RH}, when $k>0$, $e^{-ikx}_{-\infty}$ is a linear combination of $e^{-ikx}_{+\infty}$ and $e^{-ik^*x}_{+\infty}$. We denote the coefficients by $a(k)$ and $b(k)$:
\begin{equation}
    e^{-ikx}_{-\infty}=a(k)e^{-ikx}_{+\infty}+b(k)e^{-ik^* x}_{+\infty} \quad \text{ for }k>0, k\ne \tfrac12.
\end{equation}
In scattering theory it is useful to write the above equation as 
\begin{equation}
    \tau(k)e^{-ikx}_{-\infty}= e^{-ikx}_{+\infty}+\rho(k)e^{-ik^* x}_{+\infty} \quad \text{ for }k>0, k\ne \tfrac12.
\end{equation}
 by defining $\tau = \frac1a$, $\rho=\frac{b}{a}$. We will also show in Section \ref{sec: scat data} that for $k<0$, $e^{-ikx}_{-\infty}$ is proportional to $e^{-ikx}_{+\infty}$. We denote the coefficient still by $a$, and write
 \begin{equation}
     e^{-ikx}_{-\infty}=a(k)e^{-ikx}_{+\infty}\quad \text{ for }k<0.
 \end{equation}
In order to study the above relations between Jost solutions,
for $k>0$, $k\ne\tfrac12$, we also define
\begin{equation}
    M_e(x,k)=e^D(e^{ikx}e^{-ik^* x}_{-\infty})=e^{ik(x-i)}e^{-ik^*(x-i)}_{-\infty},
\end{equation}
\begin{equation}
    N_e(x,k)=e^D(e^{ikx}e^{-ik^* x}_{+\infty})=e^{ik(x-i)}e^{-ik^*(x-i)}_{+\infty}.
\end{equation}
It follows that $M_e,N_e\in \mathbb H^\infty(\R+i(0,2))$ and satisfy 
\begin{equation}\label{eq: main Me}
    [D-k(1-e^{-2D})]M_e=uM_e,
\end{equation}
\begin{equation}\label{eq: main Ne}
    [D-k(1-e^{-2D})]N_e=uN_e,
\end{equation}
and
\begin{equation}
    M_e(x+iy,k)\sim e^{i\lambda (x+iy)}\quad \text{ as }x\to -\infty \text{ for all }y\in (0,2),
\end{equation}
\begin{equation}
    N_e(x+iy,k)\sim e^{i\lambda (x+iy)}\quad \text{ as }x\to +\infty \text{ for all }y\in (0,2).  
\end{equation}
Here once again $\lambda$ is the unique real number such that $k=\frac{\lambda}{1-e^{-2\lambda}}$. 

As we will show in Section \ref{sec: existence}, $M_1(x,k)$ and $a(k)$ has meromorphic extension to $M_1(x,\zeta)$, $a(\zeta)$ for $\zeta\in \mathbb C\setminus \mathbb R^+$, with continuous boundary values on $\mathbb R^+\pm i0$ such that the $M_1(x,k)$, $a(k)$ above coincides with $M_1(x,k+i0)$, $a(k+i0)$ for $k>0$. Similarly, $N_1(x,k)$ has meromorphic extension to $N_1(x,\zeta)$ for $\zeta\in \mathbb C\setminus \R^+$, with continuous boundary values on $\mathbb R^+\pm i0$ such that the $N_1(x,k)$ above coincides with $N_1(x,k-i0)$ for $k>0$.

We give a brief outline of the sections to follow. In Section \ref{sec.green-estimates}, we estimate the Green's function inverting the symbol on the left hand side of \eqref{eq: main M1}, \eqref{eq: main N1}, \eqref{eq: main Me}, \eqref{eq: main Ne}, in general when $k$ is continued to $\zeta\in\mathbb C\setminus \mathbb R^+$. These estimates lay the groundwork for existence and properties of the Jost solutions and scattering coefficients. In Section \ref{sec: prop of jost}, we assume $u\in L^1_1\cap L^2_1$, and show  that the integral equations involving the Green's function are equivalent to the differential RH problems such as \eqref{eq: main M1}, \eqref{eq: main N1}, \eqref{eq: main Me}, \eqref{eq: main Ne}. This includes establishing $x$-space analytic continuations of the Jost solutions into the strip $0<\imag z<2$. In Section \ref{sec: existence}, we show existence of the Jost solutions by inverting the integral operators, assuming $u$ has small norm in $L^1_1\cap L^2_1$ . We will show that invertibility are closedly related to the scattering coefficient $a(\zeta)$. The Jost solutions will be solved not only for the spectral parameter $k$ on the real line, but also for its extension to $\zeta$ in the complex plane. In Section \ref{sec: scat data}, we study $a(\zeta)$, together with other useful scattering coefficients. In Section \ref{sec: bound states}, we study the bounds states of the $L_u$ operator, and show finiteness and simplicity of the discrete spectrum, under various assumptions on $u$. In Section \ref{sec: RH}, we show that $\tau(k)$ can be recovered from $\rho(k)$ via an RH problem, that the Jost solutions $M_1$, $N_1$ satisfy certain nonlocal RH problems in the complex spectral parameter $\zeta$. The singular data of the RH problem are given by the scattering data of $L_u$. We also obtain the reconstruction formula for $u$ in terms of the $N_1$. This completes the picture of the direct scattering problem for the ILW equation for small data, and points to the reconstruction of solution to ILW by solving the inverse scattering problem, which we leave to a future paper.								%% 	Lax representation
%%%%%%%%%%%%%%%%%%%
%
%		green.tex	- Green's Function
%
%%%%%%%%%%%%%%%%%%%

\section{Green's function for the spectral problem}
%% Added by PAP 2022.08.04:
\label{sec.green-estimates}

Note that \eqref{eq: main M1}, \eqref{eq: main N1}, \eqref{eq: main Me}, \eqref{eq: main Ne} and their generalizations to $\zeta\in \mathbb C\setminus \mathbb R^+$ can be formally solved by inverting the symbol $D-\zeta(1-e^{-D})$. This prompts us to study the integral kernel Green's functions
\begin{equation}\label{def: G}
G_{L,R}(x,\zeta) = \frac1{2\pi}\int_{\Gamma_{L,R}} \frac{e^{ix\xi}}{\xi-\zeta(1-e^{-2\xi})}~d\xi
\end{equation} 
where the contours $\Gamma_{L,R}$ are shown in Figure \ref{fig: 1}. In particular, both contours go from $-\infty$ to $\infty$ along the real line, dodging zero along a semi-circles that is so small that $\zeta$ is outside the full circle.
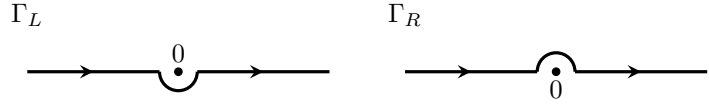
\begin{figure}[ht]  %% change by PAP to eliminate error messages - was [h]
\begin{tikzpicture}[
    scale=0.5,
    ->-/.style={decoration={
      markings,
      mark=at position #1 with {\arrow{stealth}}},postaction={decorate}},
    ->-/.default=0.5
]

%%
%%	\Gamma_L (Semicircle BELOW pole)
%%
\begin{scope}
    % Label at upper left
    \node at (-4, 1.5) {$\Gamma_L$};
    
    % Path from -4 to 4 (Total width 8 units)
    \draw[very thick, black, ->-] (-4,0) -- (-0.5,0);
    \draw[very thick, black] (-0.5,0) arc(180:360:0.5);
    \draw[very thick, black, ->-] (0.5,0) -- (4,0);
    
    % Pole at 0
    \draw[black,fill=black] (0,0) circle(0.10cm) node[anchor=south] {$0$};
\end{scope}

%%
%%	\Gamma_R (Semicircle ABOVE pole)
%%
\begin{scope}[xshift=10cm] % Reduced shift from 15cm to 10cm
    % Label at upper left
    \node at (-4, 1.5) {$\Gamma_R$};
    
    % Path from -4 to 4
    \draw[very thick, black, ->-] (-4,0) -- (-0.5,0);
    \draw[very thick, black] (-0.5,0) arc(180:0:0.5);
    \draw[very thick, black, ->-] (0.5,0) -- (4,0);
    
    % Pole at 0
    \draw[black,fill=black] (0,0) circle(0.10cm) node[anchor=north] {$0$};
\end{scope}

\end{tikzpicture}
\caption{Integration contours for $G$}\label{fig: 1}
\end{figure}

We also define the analytic continuation of $G_{L,R}$ in $x+iy$:
\begin{equation}
G_{L,R}(x+iy,\zeta) = \frac1{2\pi}\int_{\Gamma_{L,R}} \frac{e^{ix\xi}e^{-y\xi}}{\xi-\zeta(1-e^{-2\xi})}~d\xi.
\end{equation}
In the following, most of the analyses are done on $G_L$, while the estimates on $G_R$ can be obtained accordingly, as
\begin{equation}\label{GL-GR}
G_L(x+iy,\zeta)-G_R(x+iy,\zeta)= \frac{i}{1-2\zeta} 
\end{equation}
for $\zeta\ne \frac12\pm i0$.
We also have the symmetry relation
\begin{equation}\label{eq: G_L bar}
    \overline{G_L(x+iy,\zeta)}=G_R(-x+iy,\overline{\zeta}).
\end{equation}
In the above equations, $\zeta$ belongs to the cut plane $\mathbb C\setminus \mathbb R^+$, but we extend it continuously to $\zeta\in \mathbb R^+\pm i0$. When $\zeta=k+i0\in \mathbb R^+ + i0$, we have $G_{L,R}(x,k+i0) = \lim_{\epsilon\searrow 0}G_{L,R}(x,k+i\epsilon)$. Similar convention applies when $\zeta=k-i0\in \R^+-i0$. We also use notation such as Im $\zeta\ge 0$ to mean that Im $\zeta >0$ or $\zeta\in \R^++i0$. For the above reason, we will refer to the set of allowed values for $\zeta$ as $\widetilde{\C}=(\C\setminus \R^+)\cup (\R^+\pm i0)$, where the topology is defined such that a local basis around a point  $x\in \R^++i0$ contains only the upper semi-discs $\{z~|~|z-x|<r,~\imag z\ge 0\}$, and vice versa for a point $x\in \R^+-i0$.

\subsection{\texorpdfstring{Mapping properties of $Z$}{Mapping Properties of Z}}

In order to construct scattering functions under a small norm condition on $u$, we need estimates on the Green's functions that are uniform in $\zeta$. Such estimates can be obtained by deforming the contour of integration. %To state the estimates, we define the following four regions, for a small constant $\epsilon_0>0$, a large constant $R_0>0$ and a small angle $\alpha>0$. The value of these constants will be defined in the proofs of the estimates.
%\begin{align}
%\text{region I}: &\left\{\zeta\in (\C\setminus \R^+) \cup \R^+\pm i0~:~0<\left|\zeta-\frac12\right|<\epsilon_0\right\}. \label{def: re 1}\\
%\text{region II}: &\left\{\zeta\in (\C\setminus \R^+) \cup \R^+\pm i0~:~\left|\zeta-\frac12\right|\ge\epsilon_0, |\zeta|<R_0\right\}.\\
%\text{region III}_\alpha: &\left\{\zeta\in (\C\setminus \R^+) \cup \R^+\pm i0~:~|\zeta|\ge R_0, \text{Re }\zeta>0,|\text{Im }\zeta|\le \tan \alpha ~\text{Re }\zeta\right\}. \label{def: re 3}\\
%\text{region IV}_\alpha: &~\C\setminus (\text{region I}\cup\text{region II}\cup\text{region III}_\alpha). \label{def: re 4}
%\end{align}
%Note that for $\zeta$ in regions I and III$_\alpha$, there is a well-defined $\lambda$ such that $\zeta(\lambda) = \frac{\lambda}{1-e^{-2\lambda}}=\zeta$. To see this in region I, note that $\zeta(0)=\frac12$, and $\zeta'(0)\ne 0$. $\zeta(\lambda)$ is locally one to one near $\lambda=0$ and $\zeta=\frac12$. For $\lambda$ near region III$_\alpha$, the real part of $\lambda$ is large, so that $\frac1{1-e^{-2\lambda}}$ is close to 1. One can easily see that for $\zeta$ in region III$_\alpha$, there is a unique $\lambda$ near region III$_\alpha$ such that $\zeta(\lambda) = \zeta$. For $\zeta$ in these regions, we define $\zeta^* = \zeta(-\lambda)$.
As we deform the contour, the poles of the integrand occurs when $(1-e^{-2\xi})(Z(\xi)-\zeta)=0$, where 
\begin{equation}\label{def: Z}
    Z(\xi) = \frac{\xi}{1-e^{-2\xi}}.
\end{equation}
It will also be important for us to define an involution $\zeta\mapsto\zeta^*$ such that if $\zeta=Z(\lambda)$, $\zeta^*=Z(-\lambda)$. However, since $Z$ is not injective, a complete definition of it will require a description of the Riemann surface of the inverse of $Z$.
%Thus we need to describe mapping properties of $Z$. It is will also 
%To that end and for studies of the singularity of the Green's function, we describe the Riemann surface for $Z^{-1}$. 
We first fix the notation for the domain of $Z$.
Denote by $\C^*=\C\setminus\{0\}$ the punctured plane, and $\mathbb C\setminus i\pi\mathbb Z^*$, where $\mathbb Z^*=\mathbb Z\setminus \{0\}$, the periodically punctured plane. Note that $0\in \mathbb C\setminus i\pi\mathbb Z^*$. 
\begin{lemma}\label{lem: Z holo cover}
    $Z:\C\setminus i\pi \mathbb Z^*\to \mathbb C^{*}$ is a branched holomorphic map with double branch points at 
    \begin{equation}
    \lambda_n = -\tfrac12(W_{-n-1}(-\tfrac1e)+1), \quad n=1,2,\dots, 
    \end{equation}
    and $\overline{\lambda_n}$. Here $W_k$ is the $k$-th branch of the Lambert $W$ function. Furthermore, $Z(\lambda_n)=\lambda_n+\frac12$, $Z(\overline{\lambda_n})=\overline{\lambda_n}+\frac12$. Denote by 
    \begin{equation}
        B=\{\lambda_n,\overline{\lambda_n}~|~ n=1,2,\dots\}
    \end{equation}
    the branch points of $Z$, and 
    \begin{equation}
        A=Z(B)=\{\lambda_n+\tfrac12,\overline{\lambda_n}+\tfrac12~|~ n=1,2,\dots\}
    \end{equation} 
    the critical values. Then $Z:\C\setminus i\pi \mathbb Z^*\setminus Z^{-1}(A)\to \C^*\setminus A$ is a covering map.
\end{lemma}
\begin{proof}
    Since $Z'(\lambda)=\frac{e^{2\lambda}(e^{2\lambda}-2\lambda-1)}{(e^{2\lambda}-1)^2}$, it is easy to see that $Z'(\lambda)=0$ if and only if $\lambda\ne 0$ and $e^{2\lambda}-2\lambda-1=0$. This happens if
    $-(2\lambda+1)e^{-(2\lambda+1)}=-\frac1e$. By the definition of the Lambert $W$ function, $-(2\lambda+1)=W_k(-\frac1e)$ for different $k$. Note that $W_0(-\frac1e)=W_{-1}(\frac1e)=-1$, giving rising to $\lambda=0$. The branch points are therefore given by the other values of $k$. By definition of different branches of $W$, one has $W_{n}(-\frac1e)=\overline{W_{-n-1}(-\frac1e)}$, $n=1,2,\dots$. Thus the formulas for the branch points follow. Since $e^{2\lambda_n}=1+2\lambda_n$, we easily get  
    $Z(\lambda_n)=\lambda_n+\frac12$. Since $Z''(\lambda)=\frac{4 e^{2 \lambda} [e^{2 \lambda} ( \lambda-1) + \lambda+1]}{(e^{2\lambda}-1 )^3}$, we get $Z''(\lambda_n)=\frac{1+2\lambda_n}{\lambda_n}\ne 0$. Thus all the branch points are of multiplicity 2.

    We want to show that $Z:\C\setminus i\pi \mathbb Z^*\setminus Z^{-1}(A)\to \C^*\setminus A$ is a covering map. Let $\zeta\in \C^*\setminus A$ be given. Let $\xi_n, n=1,2,\dots$ be an enumeration of all the points in $\C\setminus i\pi \mathbb Z^*\setminus Z^{-1}(A)$ such that $Z(\xi_n)=\zeta$. Note that $Z$ is meromorphic on compact subsets of $\C$, thus any compact subset of $\C$ can contain only finitely many $\xi_n$'s. Without loss of generality, we may require $\xi_n\to \infty$ as $n\to\infty$. We claim that there exists a $\delta>0$ such that $|\xi_m-\xi_n|>\delta$ for all $m\ne n$. Indeed, let's assume $\xi_m,\xi_n\ne 0$ without loss of generality. It follows from $Z(\xi_m)=\zeta$ and $Z(\xi_n)=\zeta$ that 
    \begin{equation}
        \xi_m-\xi_n=\frac{\xi_m e^{-2\xi_m}}{1-e^{-2\xi_m}}-\frac{\xi_n e^{-2\xi_n}}{1-e^{-2\xi_n}}=\zeta(e^{-2\xi_m}-e^{-2\xi_n}).
    \end{equation}
    \begin{equation}\label{eq: Z(xi_m-xi_n)}
        Z(\xi_m-\xi_n)=\frac{\xi_m-\xi_n}{1-e^{-2(\xi_m-\xi_n)}}=\zeta e^{-2\xi_n}=\zeta-\xi_n.
    \end{equation}
    It is easy to see that $\zeta-\xi_n=\frac12$ implies either $\xi_n=0$ or $\xi_n\in B$, neither of which is the case. It follows that $|\zeta-\xi_n-\frac12|$ is bounded away from zero. This implies $\xi_m-\xi_n$ is bounded away from zero, as $Z(0)=\frac12$. 

    By shrinking $\delta$ by a factor of 2, we see that the discs $D(\xi_m,\delta)$ are disjoint. We claim that by shrinking $\delta$ further if necessary, we may assume $Z$ is uniformly bounded on $\bigcup_{m}D(\xi_m,\delta)$. We denote this uniform upper bound by $M$. Indeed, 
    \begin{equation}\label{eq: Z(xi_m+xi}
        Z(\xi_m+\xi)= \frac{\xi_m+\xi}{1-e^{-2(\xi_m+\xi)}}=\zeta e^{2\xi}\frac{\xi_m+\xi}{\xi_m+\zeta(e^{2\xi}-1)}.
    \end{equation}
    It is easy to see that \eqref{eq: Z(xi_m+xi} is uniformly bounded when $|\xi_m|$ is sufficiently large and $|\xi|<\delta$ is sufficiently small. 

    We next observe that $|Z'(\xi_m)|$ is bounded and bounded away from zero. Indeed, since $\xi_m\notin B$, $Z'(\xi_m)\ne 0$, and
    \begin{equation}\label{eq: Z'(xi_m)}
        Z'(\xi_m)=\frac{1-e^{-2\xi_m}-2\xi_m e^{-2\xi_m}}{(2-e^{-2\xi_m})^2}=\frac{\frac{\xi_m}{\zeta}-2\xi_m(1-\frac{\xi_m}{\zeta})}{(\frac{\xi_m}{\zeta})^2}=\zeta\left(2+\frac{1-2\zeta}{\xi_m}\right)
    \end{equation}
    is obviously bounded and bounded away from zero as $\xi_m\to\infty$. Denote by $\eta$ a positive lower bound of $|Z'(\xi_m)|$, and by enlarging the bound $M$ above, an upper bound of $|Z'(\xi_m)|$.

    By further shrinking $\delta$ if necessary, we claim that for every $\xi$ with $0<|\xi|\le\delta$ and all $m$,  
    \begin{equation}\label{eq: Z annular bd}
        \frac{\eta}{2}|\xi|<|Z(\xi_m+\xi)-\zeta|<2M |\xi|.
    \end{equation}
    Indeed, by the Cauchy integral formula, we can write
    \begin{equation}
        Z(\xi_m+\xi)-\zeta = Z'(\xi_m)\xi+\frac{\xi^2}{2\pi i}\int_{C(\xi_m,\delta_0)}\frac{Z(\tau)}{(\tau-\xi_m)^2(\tau-\xi_m-\xi)}~d\tau,
    \end{equation}
    where $\delta_0$ is the fixed small value of $\delta$ enforced in the above proof. We have $\eta\le |Z'(\xi_m)|\le M$ and 
    \begin{equation}\label{eq: Z'' bound}
        \left|\frac{\xi^2}{2\pi i}\int_{C(\xi_m,\delta_0)}\frac{Z(\tau)}{(\tau-\xi_m)^2(\tau-\xi_m-\xi)}~d\tau\right|\le \frac{ \delta_0 M|\xi|^2}{\delta_0^2(\delta_0-\delta)}\le \frac{2 M\delta |\xi|}{\delta_0^2}.
    \end{equation}
    if $|\xi|\le\delta<\frac{\delta_0}2$. By choosing $\delta$ small enough, we can make \eqref{eq: Z'' bound} smaller than $\frac\eta2|\xi|$, and \eqref{eq: Z annular bd} follows. 

    %By further shrinking $\delta$ if necessary, we may assume $|Z(\lambda)-\zeta|<\frac{|\zeta|}2$ for all $\lambda\in \bigcup_m D(\xi_m,\delta)$. 
    We now claim that $Z$ restricted to $Z^{-1}\left(D(\zeta, \frac{\eta \delta}4)\right)\cap D(\xi_m,\delta)$ is a homeomorphism onto $D(\zeta,\frac{\eta \delta}4)$ for every $m$. Indeed, for each $\zeta'\in D(\zeta,\frac{\eta \delta}4)$, the number of points $\lambda\in D(\xi_m,\delta)$ such that $Z(\lambda)=\zeta'$ is equal to $\text{ind}\left(Z\left(C(\xi_m,\delta)\right),\zeta'\right)$, the winding number of $Z\left(C(\xi_m,\delta)\right)$ around $\zeta'$. By \eqref{eq: Z annular bd}, $\text{dist}(Z(C(\xi_m,\delta)),\zeta)>\frac{\eta\delta}2$. Thus $\zeta'$ and $\zeta$ are in the same connected component of the complement of $Z(C(\xi_m,\delta))$. It follows that 
    \begin{equation}
        \text{ind}(C(\xi_m,\delta),\zeta')=\text{ind}(C(\xi_m,\delta),\zeta)=1.
    \end{equation}
    Thus $Z$ is bijective from $Z^{-1}\left(D(\zeta, \frac{\eta \delta}4)\right)\cap D(\xi_m,\delta)$ to $D(\zeta,\frac{\eta \delta}4)$. Since $Z$ is an unbranched holomorphic function on $\C\setminus i\pi \mathbb Z^*\setminus B$, it is obviously a homeomorphism when restricted to the above set.

    Next, we repeat the above proof with $\zeta$ replaced by $\widetilde{\zeta}$, where $\widetilde\zeta\in D(\zeta,\epsilon)$ for some small $\epsilon$. The goal is to obtain constants $\eta$, $\delta$, $M$ that are uniform for $\widetilde\zeta\in D(\zeta,\epsilon)$. Since most of the arguments can be repeated, we only mention the steps that need to be modified. Enumerate the points in $Z^{-1}(\widetilde\zeta)$ by $\widetilde{\xi_n}$. The right hand side of \eqref{eq: Z(xi_m-xi_n)} now becomes $\widetilde{\zeta}-\widetilde{\xi_m}$. We claim that there exists some $\epsilon>0$ such that $|\widetilde{\zeta}-\widetilde{\xi_m}-\frac12|$ is bounded uniformly away from zero for all $\widetilde\zeta\in D(\zeta,\epsilon)$ and all $m$. If not, there would exist a sequence of points $\{\widetilde\zeta^{(k)}\}$ and preimages $\{\widetilde{\xi}^{(k)}\}$ under $Z$ such that $\widetilde\zeta^{(k)}\to \zeta$, and $\widetilde{\xi}^{(k)}-\widetilde\zeta^{(k)}-\frac12\to 0$. Taking the limit of $Z(\widetilde{\xi}^{k})=\widetilde{\zeta}^{(k)}$, we get $Z(\zeta+\frac12)=\zeta$. This implies that $\zeta\in A$, which contradicts the choice of $\zeta$. The proof above can now be repeated, and we get a $\delta>0$ such that the discs $D(\widetilde{\xi_m},\delta)$ are disjoint for all $\widetilde{\zeta}\in D(\zeta,\epsilon)$. The next step in the proof will establish a uniform upper bound of $Z(\widetilde{\xi_m}+\xi)$ for $|\xi|<\delta$. Such a bound is obvious when $\widetilde{\xi_m}\in D(0,1)$. If $|\widetilde{\xi_m}|>1$, this is also obvious by observing \eqref{eq: Z(xi_m+xi} with $\xi_m$ and $\zeta$ replaced by $\widetilde{\xi_m}$ and $\widetilde{\zeta}$. Next, we need uniform upper and lower bounds on $|Z'(\widetilde{\xi_m})|$. We again have \eqref{eq: Z'(xi_m)} with $\xi_m$ and $\zeta$ replaced by $\widetilde{\xi_m}$ and $\widetilde{\zeta}$. If it's not bounded below, there would exist a sequence of points $\{\widetilde\zeta^{(k)}\}$ and preimages $\{\widetilde{\xi}^{(k)}\}$ under $Z$ such that $\widetilde{\zeta}^{(k)}\to \zeta$, and $2+\frac{1-2\widetilde{\zeta}^{(k)}}{\widetilde{\xi}^{(k)}}\to 0$. This again implies $Z(\zeta+\frac12)=\zeta$, which contradicts the choice of $\zeta$. The upper bound is obvious. The rest of the proof can be carried out without change. As a result, we obtain $\delta>0$ and $\eta>0$  such that for every $m$ and all $\widetilde\zeta\in D(\zeta,\epsilon)$, $Z$ restricted to $Z^{-1}\left(D(\widetilde{\zeta}, \frac{\eta \delta}4)\right)\cap D(\widetilde{\xi_m},\delta)$ is a homeomorphism onto $D(\widetilde{\zeta},\frac{\eta \delta}4)$.
    By further shrinking $\delta$, we may take $\epsilon=\delta$.

    Finally, we will show $Z^{-1}(D(\zeta,\frac{\eta\delta}{4}))\subset \bigcup_n D(\xi_n,\delta)$, so that the proof that $Z:\C\setminus i\pi \mathbb Z^*\setminus Z^{-1}(A)\to \C^*\setminus A$ is a covering map will be complete. Indeed, suppose there exists a $\widetilde{\xi_m}\notin \bigcup_n D(\xi_n,\delta)$ such that $Z(\widetilde{\xi_m})=\widetilde\zeta\in D(\zeta,\frac{\eta\delta}{4})$. Note that we also have $\zeta\in D(\widetilde\zeta,\frac{\eta\delta}{4})$. By the above result, there exists a $\lambda\in D(\widetilde{\xi_m},\delta)$ such that $Z(\lambda)=\zeta$. But this means that $\lambda=\xi_n$ for some $n$, which implies that $\widetilde{\xi_m}\in D(\xi_n,\delta)$, causing a contradiction.
\end{proof}

We can now regard $\C\setminus i\pi \mathbb Z^*$ as the Riemann surface for $Z^{-1}$. More specifically, denote by $f_a$ the holomorphic function germ of $f$ at $a$, where $f$ is holomorphic in a neighborhood of $a$, and let $I$ be the identity map on $\C$. Denote by $Z^{-1}_{\frac12}=Z_*I_0$ the pushforward of $I_0$. This is the holomorphic function germ at $\frac12$ of the local inverse of $Z$. We have
\begin{lemma}\label{lem: glob ana cont}
    $(\mathbb C\setminus i\pi\mathbb Z^*\setminus B, Z, I, 0)$ is the global analytic continuation of $Z^{-1}_{\frac12}$. In other words, a holomorphic function germ $\varphi$ can be obtained from $Z^{-1}_{\frac12}$ by analytic continuation along a curve in $\C$ if and only if there exists a unique point $a\in \C\setminus i\pi \mathbb Z^*\setminus B$ such that $\varphi = Z_*I_a$. 
\end{lemma}
\begin{proof}
    Suppose there is a unique point $a\in \C\setminus i\pi \mathbb Z^*\setminus B$ such that $\varphi = Z_*I_a$. Since $\C\setminus i\pi \mathbb Z^*\setminus B$ is path-connected, take a path $\widetilde{\gamma}:[0,1]\to \C\setminus i\pi \mathbb Z^*\setminus B$ connecting $0$ to $a$. Since $Z$ is a local homeomorphism in a neighborhood of $\widetilde{\gamma}([0,1])$, $Z_*I_{\widetilde{\gamma}(t)}$ is defined and is an analytic continuation of $Z_*I_0=Z^{-1}_{\frac12}$ along $\gamma(t)=Z(\widetilde{\gamma}(t))$ with ending germ $Z_*I_{\widetilde{\gamma}(1)}=Z_*I_a = \varphi$.
    
    Now suppose $\varphi$ can be obtained from $Z^{-1}_{\frac12}$ from an analytic continuation along a curve $\gamma:[0,1]\to \C$, $\gamma(0)=\frac12$. Denote the function germs along $\gamma$ by $\varphi_t$, so $\varphi_0=Z^{-1}_{\frac12}$, and $\varphi_1=\varphi$. Note that $\varphi_t$ is a function germ at $\gamma(t)$, so $\widetilde{\gamma}(t) = \varphi_t(\gamma(t))$ is well-defined. We first show that $\widetilde{\gamma}(t)\in \mathbb C\setminus i\pi \mathbb Z^*$ for all $t\in[0,1]$. If not, by reparametrization, we may assume that $\varphi_1(\gamma(1))\in i\pi \mathbb Z^*$, while $\varphi_t(\gamma(t))\in \mathbb C\setminus i\pi \mathbb Z^*$ for all $0\le t<1$. Consider $Z\circ\varphi_t$, which is another analytic continuation along $\gamma$ for $0\le t<1$, with $Z\circ \varphi_0=Z\circ Z^{-1}_{\frac12}=I_{\frac12}$. By uniqueness of analytic continuation along curves, it is easy to see that $Z\circ \varphi_t= I_{\gamma(t)}$. This implies that  $Z(\widetilde{\gamma}(t))=I_{\gamma(t)}(\gamma(t))=\gamma(t)$ for $0\le t<1$. Thus $\displaystyle\lim_{t\nearrow 1}\gamma(t) = \infty$, since $\widetilde{\gamma}(1)\in i\pi \mathbb Z^*$ is a pole for $Z$. This contradicts $\gamma(1)\in \mathbb C$. We obtain from the above argument that $\widetilde{\gamma}$ is a lift of $\gamma$ into $(\C\setminus i\pi \mathbb Z^*,0)$ under $Z:(\C\setminus i\pi \mathbb Z^*,0)\to (\C,\frac12)$. In particular, this implies $\gamma([0,1])\subset \C^*$. We claim that $\widetilde{\gamma}([0,1])\subset \C\setminus i\pi \mathbb Z^*\setminus B$. Indeed, suppose  $\widetilde{\gamma}(1)=\varphi_{1}(\gamma(1))=\lambda_n\in B$. Since $Z$ has multiplicity 2 at $\lambda_n$, $Z\circ \varphi_{1}$ should have multiplicity at least 2. However, $Z\circ \varphi_{1}=I_{\gamma(1)}$ obviously has multiplicity 1. This contradiction proves our claim.

    By the above discussion, we see that $Z$ is a local homeomorphism in a neighborhood of $\widetilde{\gamma}([0,1])$. Thus $Z_*I_{\widetilde{\gamma}(t)}$ is well-defined. Since $Z_*I_{\widetilde{\gamma}(0)}=Z_* I_0=Z^{-1}_{\frac12}$, $Z_*I_{\widetilde{\gamma}(t)}$ is obviously an analytic continuation along $\gamma$ of $Z^{-1}_{\frac12}$. By uniqueness of analytic continuation along curves, $\varphi=\varphi_1=Z_*I_{\widetilde{\gamma}(1)}$. 
\end{proof}

\begin{figure}[ht]  %% changed by PAP to eliminate error message 2026-09-18
\includegraphics[width=0.9\textwidth]{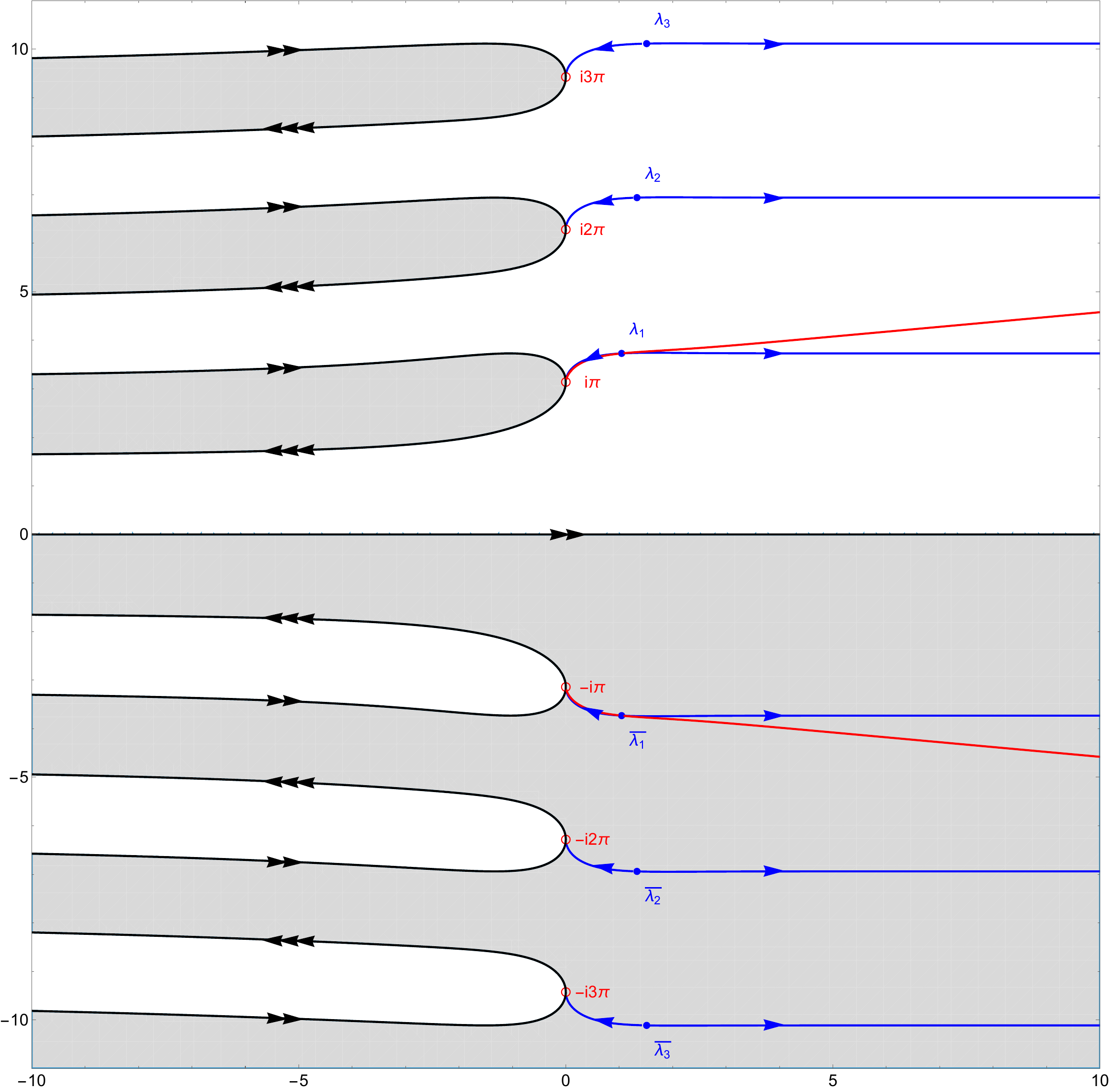}
\caption{Fundamental regions of $Z$}
\label{fig: Z_lambda}
\end{figure}

Figure \ref{fig: Z_lambda} shows the fundamental regions of $Z$ in $\mathbb C\setminus i\pi\mathbb{Z}^*$. We describe them as follows. Each curve marked with a double arrow is mapped to $\mathbb R^+$ from $0$ to $e^{i0}\infty$; each curve marked with a triple arrow is mapped to $\mathbb R^-$ from $e^{i\pi}\infty$ to $0$; the blue curve through $\lambda_n$ (or $\overline{\lambda_n}$) marked with single arrows is doubly folded onto the horizontal ray from $\lambda_n+\tfrac12$ (or $\overline{\lambda_n}+\tfrac12$) to $e^{i0}\infty$. Thus each white region bounded by the plotted curves is mapped biholomorphically to the upper half plane (with one or two cuts if bounded by one or two blue curves), and each grey region bounded by the plotted curves is mapped biholomorphically to the lower half plane (with one or two cuts if bounded by one or two blue curves). As Lemma \ref{lem: glob ana cont} shows, any open subset of $\mathbb C\setminus i\pi\mathbb Z^*\setminus B$ on which $Z$ is one-to-one corresponds to a branch of the inverse of $Z$. To define the branch of $Z$ inverse we need, we single out the following region in the domain of $Z$. %For convenience in the following discussion, we define a particular branch of $Z$ inverse, together with the involution corresponding to $\lambda\mapsto-\lambda$. 

\begin{lemma}\label{lem: Z inv domain}
    There exist uniquely two simple curves $\gamma_1$, $\gamma_2$ in the upper half plane (the black curve marked with triple arrows approaching $i\pi$ and the red unmarked curve in Figure \ref{fig: Z_lambda}) having the following properties: $\gamma_1$ connects $i\pi$ to $\infty$ (not including end points), and is mapped one-to-one to $\mathbb R^-$; $\gamma_2$ connects $i\pi $ to $\infty$ (not including end points) through $\lambda_1$, and is mapped two-to-one (except for $\lambda_1$) to the ray from $\lambda_1+\frac12$ to $\infty$ in the direction $e^{i\alpha_0}$ for some small $\alpha_0>0$ (the red branch cut in the upper half plane in Figure \ref{fig: branch_cut}). Moreover, $\gamma_1\cup\{i\pi\}\cup \gamma_2$ divides $\mathbb C$ into two parts, as does $\overline{\gamma_1}\cup\{-i\pi\}\cup \overline{\gamma_2}$. Let $U$ be the intersection of the parts containing 0 (the region "between" $\gamma_1\cup\{i\pi\}\cup \gamma_2$ and $\overline{\gamma_1}\cup\{-i\pi\}\cup \overline{\gamma_2}$). Let $V$ be $\mathbb C$ with three branch cuts removed: $(-\infty, 0]$, and two others starting at $\lambda_1+\frac12$ and $\overline{\lambda_1}+\frac12$ and going to $\infty$ in the directions $e^{i\alpha_0}$ and $e^{-i\alpha_0}$ (see Figure \ref{fig: branch_cut}). Then $Z$ maps $U$ biholomorphically onto $V$.
\end{lemma}

\begin{figure}[ht]  %% changed by PAP 2026-09-18 to avoid error message 
\includegraphics[width=0.7\textwidth]{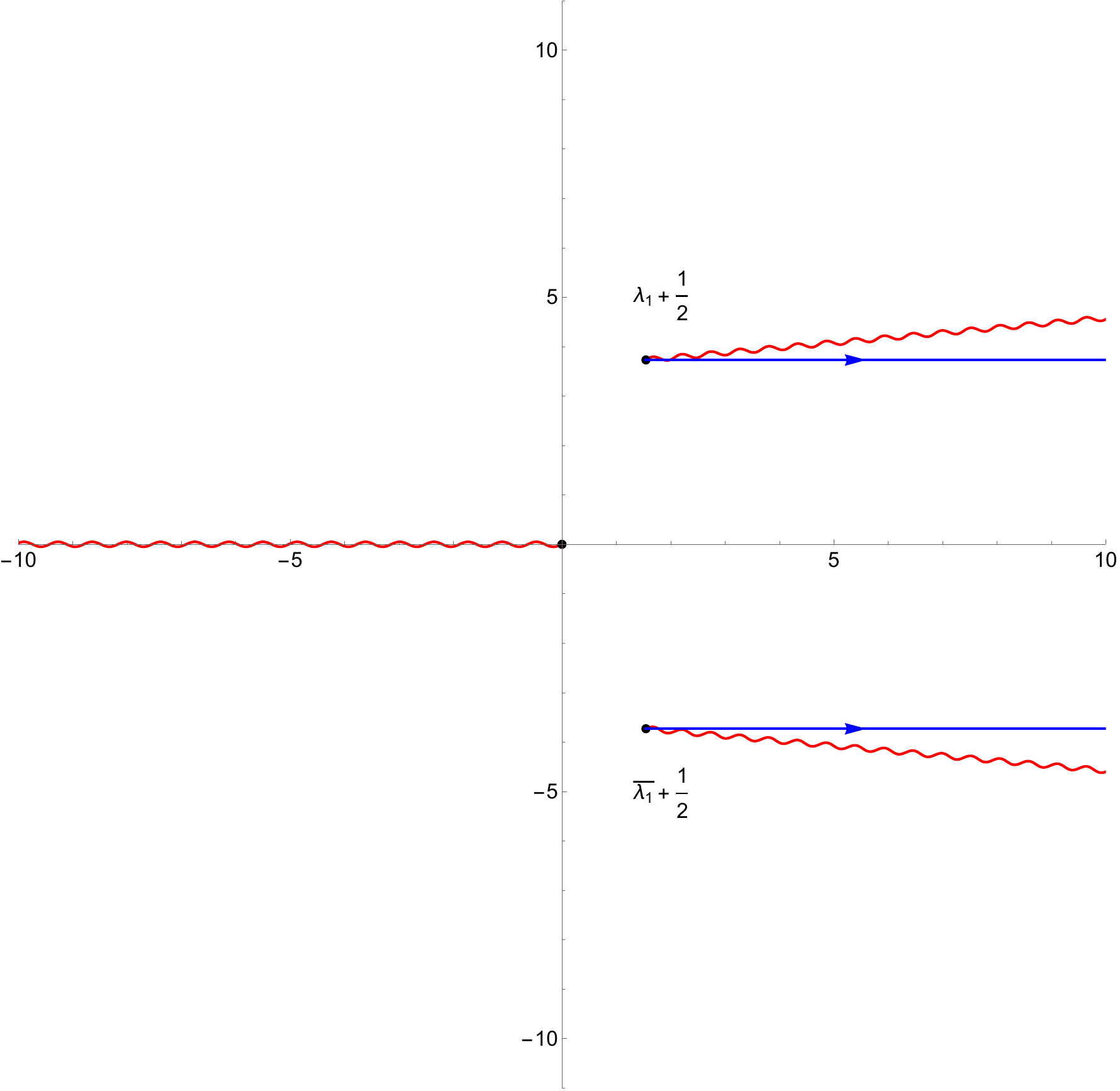}
\caption{$V$ with branch cuts: domain for the principal branch of $Z^{-1}$}
\label{fig: branch_cut}
\end{figure}

\begin{proof}
    First note that $i\pi$ is an order 1 pole. So in a neighborhood of it there is unique lift of the far ends of $\mathbb R^-$ and the branch cut at $\lambda_1+\frac12$. We observe that $\mathbb R^-$ and the branch cut at $\lambda_1+\frac12$ are both in $\mathbb C^*\setminus A$, as long as $\alpha_0$ is sufficiently small. Indeed, by (4.20) in \cite{corless1996lambert}, we get
    \begin{equation}
        \lambda_n=\tfrac12\log(2n\pi)+\tfrac{i\pi}{2}(2n+\tfrac12)+o(1).    
    \end{equation}
    Thus the points $\lambda_n+\tfrac12$ can all be avoided if $\alpha_0$ is small enough.
    By Lemma \ref{lem: Z holo cover} the branch cuts without the end points $0$ and $\lambda_1+\frac12$ can be uniquely lifted to $\mathbb C\setminus i\pi \mathbb Z^*\setminus Z^{-1}(A)$ in continuation of their lifts into the small neighborhood of $i\pi$. Let $\gamma_1$ be the lift of $\mathbb R^-$. The fact that $\gamma_1$ connects to $\infty$ follows from the lower bound $|Z(\lambda)|>\epsilon>0$ if $\real \lambda>\lambda_0$. In fact, $Z(0)=\frac12>0$, and if $|\lambda|>\delta$ for some small $\delta$, $|Z(\lambda)|\ge \frac{\delta}{|1-e^{-2\lambda}|}\ge\frac{\delta}{1+e^{-2\lambda_0}}$ for $\real \lambda>\lambda_0$. In fact, we can get an exact formula for $\gamma_1$, if we rewrite 
    $Z(\lambda)=\zeta$ as
    \begin{equation}
        2(\lambda-\zeta)e^{2(\lambda-\zeta)}=-2\zeta e^{-2\zeta}.
    \end{equation}
    It follows that 
    \begin{equation}
        \lambda=\zeta+\tfrac12 W_k(-2\zeta e^{-2\zeta})
    \end{equation}
    for some branch $k$ of Lambert W function $W_k$. We see that $\lambda=\zeta+\frac12 W_1(-2\zeta e^{-2\zeta})$ gives $\gamma_1$ when $\zeta\in\mathbb R^-$, since it's a lift of $\mathbb R^-$ and $\lambda\to i\pi$ as $\zeta\to-\infty$ in $\mathbb R^-$.
    To see that the lift of the branch cut at $\lambda_1+\frac12$ connects $i\pi$ to $\lambda_1$, we observe that $\lambda = \zeta +\frac12W_{-1}(-2\zeta e^{-2\zeta})$ gives a lift as $\zeta$ moves along the branch cut. Indeed, by standard asymptotics of $W_{-1}$ we know $\lambda \to i\pi$ as $\zeta\to\infty$ along the branch cut. Thus this must be the unique lift in a neighborhood of $i\pi$ of the branch cut. As $\zeta\to \lambda_1+\frac12$ along the branch cut, we get $\lambda\to \lambda_1+\frac12+\frac12W_{-1}(-\frac1e)=\lambda_1$. Since $Z$ has multiplicity 2 at $\lambda_1$, the branch cut near $\lambda_1+\frac12$ has exactly two lifts in a neighborhood of $\lambda_1$, and the other lift can then be continued uniquely by lifting through the covering map $Z$. In fact, one sees that \begin{equation}\label{eq: formula gamma 2}
        \lambda=\zeta+\frac12W_0(-2\zeta e^{-2\zeta})
    \end{equation} gives this lift, and 
    \begin{equation}\label{eq: close to id on gamma 2}
        \lambda-\zeta\to 0 \quad \text{ as }\quad \zeta\to\infty
    \end{equation}  along the branch cut. The union of the above described lifts of the branch cut are now combined to be $\gamma_2$. It follows that $\gamma_2$ connects $i\pi$ to $\infty$ through $\lambda_1$. Now $\gamma_1\cup \{i\pi\}\cup \gamma_2\cup\{\infty\}$ is a simple closed curve on the Riemann sphere, thus it divides $\mathbb C$ into two components. Since $0$ is not on the curve, it belongs to exactly one component, denoted $P$. The same is true to the complex conjugated curve, with the complex conjugated region denoted by $\overline{P}$. Thus $U=A\cap \overline{P}$ is well-defined as an open set in $\mathbb C$. We claim that $U$ is connected. In fact, denote the upper and lower half planes by $\mathbb H^\pm$. Since $\overline{\gamma_1}\cup\{-i\pi\}\cup\overline{\gamma_2}$ is contained in $\mathbb H^-$, $(\mathbb H^+)^{\text{cl}}$, the closure of $\mathbb H^+$, is contained in $\overline{P}$. Thus $U\cap (\mathbb H^+)^{\text{cl}}=P\cap \overline{P}\cap \mathbb (H^+)^{\text{cl}}=P\cap (\mathbb H^+)^{\text{cl}}$. Thus any point $z$ in $U\cap \mathbb H^+$ can be path-connected to the real line with a path in $P\cap (\mathbb H^+)^{\text{cl}}=U\cap (\mathbb H^+)^{\text{cl}}$, since $P$ is path-connected and contains $0$. The same conclusion can be draw for any $z\in U\cap \mathbb H^-$. 
    
    It remains to show that $Z$ maps $U$ to $V$ biholomorphically. This can be obtained by using the argument principle on a contour that goes counterclockwise around the boundary of $U$. More precisely, let $z\in U$ be given. It can be connected to $0$ by a curve $\gamma_0$ in $U$. We take a contour $\gamma_*$ defined as follows. It starts at a point slightly below $\gamma_2$ and sufficiently close to $\infty$, and is the lift of the following curve in Figure \ref{fig: branch_cut}: starting at a point slightly below the branch cut in the first quadrant, it stays close to the lower edge of the branch cut, until it enters a small neighborhood of $\lambda_1+\frac12$. It then follows a small arc around $\lambda_1+\frac12$ to the other side of the branch cut, and stays closely above the it to reach a point close to $\infty$. At this point, the curve traverses a great arc and arrives near the upper edge of $\mathbb R^-$, after which it stays close to $\mathbb R^-$ and reaches a neighborhood of 0. Traversing a small arc, it reaches the other side of $\mathbb R^-$, and basically follows the complex conjugate of the above described curve, until it reaches a point near $\infty$ on the upper edge of the branch cut in the fourth quadrant. Finally it closes on itself by traversing a great arc to reach the starting point. It's easy to see that the lift $\gamma_*$ is a simple closed curve and is either close to $\gamma_1$, $\gamma_2$, $i\pi$, their complex conjugate, or has very positive or very negative real part. Also, $\gamma_*$ doesn't intersect $\gamma_0$. It follows that the winding numbers $n(\gamma_*,z)=n(\gamma_*,0)=1$, and $z$ is inside $\gamma_*$. Also note that $\gamma_*$ is disjoint from $\gamma_1\cup \{i\pi\}\cup \gamma_2$ and its complex conjugate. So the inside of $\gamma_*$ cannot contain a point of $\gamma_1\cup \{i\pi\}\cup \gamma_2$ or its complex conjugate, as they approach $\infty$ at ends. It follows that the inside of $\gamma_*$ is contained in $U$, since it's connected and contains $0$. We now prove that $Z$ maps $U$ onto $V$. We first claim that there are no poles of $Z$ in $U$. Otherwise, call the pole $z$, and construct $\gamma_*$ as above. Since $Z$ never takes the value $0$, $n(Z(\gamma_*),0)=0$ gives the number of poles inside $\gamma_*$. This contradicts the assumption that $z$ is a pole inside $\gamma_*$. Similarly, $Z(z)$ cannot appear in the branch cuts in Figure \ref{fig: branch_cut}, otherwise $n(Z(\gamma_*),Z(z))=0$, implying $Z-Z(z)$ has no zeros inside $\gamma_*$. Thus $Z$ maps $U$ into $V$. Moreover, for any $w\in V$, we can take the above described contour $\gamma_*$ with $n(Z(\gamma_*),w)=1$, so there is one point $z$ inside $\gamma_*$ (which is inside $U$), such that $Z(z)=w$. No other point in $U$ is mapped to $w$, since we can construct $\gamma_*$ sufficiently close to the boundary to enclose other points that are mapped to $w$, but the winding number $n(Z(\gamma_*),w)$ is still 1. This concludes the proof.
\end{proof}
%Modulo points in $Z^{-1}(A)\setminus B$, we can regard the domain of $Z^{-1}$ to be equivalence classes of curves $\gamma:[0,1]\to \mathbb C^*\setminus A$, with $\gamma(0)=\frac12$. Two such curves $\gamma_1$, $\gamma_2$ are equivalent if and only if $\gamma_1(1)=\gamma_2(1)$, and  $[\gamma_1\cdot\gamma_2^{-1}]\in Z_*\pi_1(\mathbb C\setminus i\pi\mathbb Z^*\setminus Z^{-1}(A), 0)$. This happens precisely when the lifts of both paths to $(\mathbb C\setminus i\pi\mathbb Z^*\setminus Z^{-1}(A),0)$ (which exist by Lemma \ref{lem: Z holo cover}), denoted by $\widetilde{\gamma_1}, \widetilde{\gamma_2}$, have the same endpoints: $\widetilde{\gamma_1}(1)=\widetilde{\gamma_2}(1)$. 

\begin{definition}\label{def: Z inv}
   Let $U$, $V$ be given as in Lemma \ref{lem: Z inv domain}. We define $Z^{-1}: V\to U$, the principal branch of $Z$ inverse, to be the inverse map of the restriction of $Z$ to $U$. In addition, we define $*:V\to \mathbb C$ by $\zeta^*=Z(-Z^{-1}(\zeta))$.
   % $Z^{-1}$, the principal Let $[\gamma]$ be an equivalence class of curves in $\C^*\setminus A$ as above. We define $Z^{-1}([\gamma])$ to be the end point of the lift of any of its representative curve. Moreover, we define $[\gamma]^*=Z(-Z^{-1}([\gamma]))$. When there is no ambiguity about the equivalence class of curves $[\gamma]$ connecting $\frac12$ to a point $\zeta$, we will abuse notation and simply write $Z^{-1}(\zeta)$ and $\zeta^*$. 
\end{definition}

\begin{remark}
For $k\in (-\infty,0)$, we refer to $Z^{-1}(k\pm i0)$ and $(k\pm i0)^*$ by taking limits on the branch cut in the usual way. Of course $*$ also has a global analytic continuation to the Riemann surface of $Z^{-1}$. 
\end{remark}

\begin{remark}
The red curve $\gamma_2$ in Figure \ref{fig: Z_lambda} eventually crosses all the blue curves through $\lambda_n$, $n=1,2,\dots$. %Since $U\subset C^*\setminus A$, by Lemma \ref{lem: Z holo cover}, any curve in $U$ starting at $\frac12$ can be uniquely lifted through the covering map $Z$ to a curve in $\mathbb C\setminus i\pi \mathbb Z^*\setminus Z^{-1}(A)$ starting at 0. So the analytic continuations along curves exist. Since $U$ is simply connected, $Z^{-1}$ is well-defined as a biholomorphic map from $U$ to $Z^{-1}(U)$. It's not hard to see that $Z^{-1}(U)$ is close to the region in Figure \ref{fig: Z_lambda} consisting of the union of $\mathbb R$ and the white and grey regions directly above and below it, with the blue curves slightly perturbed (the originally blue curves corresponding to horizontal cuts starting at $\lambda_1+\frac12$ and $\overline{\lambda_1}+\frac12$). 
\end{remark}

%\begin{definition}\label{def: lambda}
%Let $\zeta\in \{0<\imag \zeta<\epsilon\}\cup \R\setminus \{0\}$ for a small $\epsilon>0$.  In this case, we write $\zeta_\up$ for $\zeta$ when it is regarded as the end point of the piecewise linear curve connecting $\frac12$ to $\frac{i\epsilon}{2}$ to $\zeta$. Define $\lambda_\up=Z^{-1}(\zeta_\up)$ (see Definition \ref{def: Z inv}), and $\zeta_\up^* = Z(-\lambda_\up)$. Conjugately, let $\zeta\in \{0>\imag \zeta>-\epsilon\}\cup \R\setminus \{0\}$. In this case, we write $\zeta_\dn$ for $\zeta$ when it is regarded as the end point of the piecewise linear curve connecting $\frac12$ to $-\frac{i\epsilon}{2}$ to $\zeta$. Define $\lambda_\dn=Z^{-1}(\zeta_\dn)$, and $\zeta_\dn^* = Z(-\lambda_\dn)$.  
%\end{definition}

%Note that according to the above definition, $\lambda_\up$ is only defined when $\zeta=\zeta_\up$ is in a subset of the closed upper half plane. The ``reflected" $\zeta^*_\up$ is in the closed lower half plane. In complex conjugation, $\lambda_\dn$ is only defined when $\zeta=\zeta_\dn$ is in a subset of the closed lower half plane, and $\zeta^*_\dn$ is in the closed upper half plane. Also note that if $\zeta_\up\in \R^++i0$ and $\zeta_\dn \in \R^+-i0$ have the same real part, then $\lambda_\up=\lambda_\dn\in \R$. However, if $\zeta_\up =\zeta_\dn=\zeta\in \R^-$, $\lambda_\up$ and $\lambda_\dn $ are different. One is the complex conjugate of the other.

For $\zeta\in V$, denote $\lambda = Z^{-1}(\zeta)$. We have the following elementary identities
\begin{equation}\label{eq: zeta-zeta*}
    \zeta-\zeta^*=\lambda,
\end{equation}
\begin{equation}\label{eq: zeta e^zeta vs zeta* e^zeta*}
    -\zeta e^{-2\zeta}=-\zeta^* e^{-2\zeta^*}.
\end{equation}
Moreover, we have the following asymptotics for $Z^{-1}$.

\begin{lemma}\label{lem: zeta}
If $\zeta\in V$ as in Definition \ref{def: Z inv}, then
%curves $\gamma_{\R^-}$, $\gamma_{\R^+}$ such that 
%\begin{enumerate}[(i)]
%\item $\gamma_{\R^\pm}$ are smooth curves. $\zeta(\gamma_{\R^\pm}) = \R^\pm$.
%As the real part of $\gamma_{\R^-}$ increases from $-\infty$ to $0$, the imaginary part of $\gamma_{\R^-}$ increases from $\frac\pi 2$ to $\pi$. As the real part of $\gamma_{\R^+}$ increases from $-\infty$ to $0$, the imaginary part of $\gamma_{\R^+}$ increases from $\pi$ to a unique maximum, and then decreases to $\pi$.
%\item %Denote by $\Delta_1$ the domain bounded by $\gamma_{\R^-}$, the real line, and $\R^++i\pi$. $\zeta(\xi)$ maps $\Delta_1$ onto the domain bounded by $\R$ and $\zeta(\R^++i\pi)$ biholomorphically. As $\xi$ approaches $-\infty$ within the above mentioned domain, or equivalently, as $\zeta$ approaches 0 from the upper half plane, we have
\begin{equation}\label{eq: zeta near 0}
\lim_{\zeta\to 0}\frac{2Z^{-1}(\zeta)}{\log \zeta}= 1,
\end{equation}
%Here $\imag \log\zeta_\up \in [0,\pi]$. %Moreover, $\zeta(\xi)$ maps $\Delta_1^{\text{cl}}\setminus \{i\pi\}$ to $\left(\zeta(\Delta_1)\right)^{\text{cl}}\setminus \{0\}$ bijectively.
%\item There is a small $\epsilon_0>0$ and a domain $\Omega_{\epsilon_0}$ between $\gamma_{\R^-}$, $\R^-$ and the imaginary axis, such that $\zeta:\Omega_{\epsilon_0}\to \{\zeta\in \C~:~|\zeta|<\epsilon_0, \text{Im } \zeta>0\}$ is biholomorphic.
%\item Let $\epsilon_0$ be given as above. 
%\item %Denote by $\Delta_2$ the domain bounded by $\gamma_{\R^\pm}$. Then $\zeta(\xi)$ maps $\Delta_2$ to the lower half plane biholomorphically. As $\xi$ approaches $-\infty$ within the domain bounded by $\gamma_{\R^\pm}$, or equivalently, as $\zeta$ approaches $0$ from the lower half plane, 
%Let $\zeta$ be in a small lower semi-disc around $0$, and regard it as the end point of the piecewise linear curve connecting $\frac12$ to $i\epsilon$ to $-\epsilon$ to $\zeta$ for some small $\epsilon>0$. Then
%\begin{equation}
%\lim_{\zeta\to 0}\frac{2Z^{-1}(\zeta)}{\log \zeta}= 1.
%\end{equation}
%Here $\imag \log \zeta \in [\pi,2\pi]$.
where $\log\zeta$ is the principal branch. If $\zeta\in V$ and is between the two branch cuts in the first and fourth quadrant, then 
%\item For $\alpha_0>0$ small enough, the wedge domain $\{0<\arg \zeta<\alpha_0\}$ contains no points in $P$. Let $\zeta$ be in this wedge domain, and regard it as the end point of the straight line segment connecting $\frac12$ to $\zeta$. Then% onto the domain bounded by $[\frac12,\infty)$ and $\zeta(\{\arg \xi=\alpha_0\})$ biholomorphically. The image domain is contained in the small wedge $\{0<\arg \zeta<2\alpha_0\}$. In this wedge domain of $\xi$, one has
\begin{equation}\label{eq: zeta near +inf}
\lim_{\zeta\to \infty}Z^{-1}(\zeta)-\zeta = 0.
\end{equation}
%\item $\zeta(\xi)$ is also biholomorphic on the union of the above two domains.

%\item $\zeta$ is also biholomorphic on $\Omega_{\epsilon_0}\cup \Omega_{\epsilon_0,\alpha_0}$.
%\item $\zeta$ maps the domain between $\gamma_{\R^+}$ and $\gamma_{\R^-\}$ biholomorphically onto the lower half plane.
%\end{enumerate}
\end{lemma}
%\begin{remark}
%Since $\overline{Z(\lambda)}=Z(\overline\lambda)$, Lemma \ref{lem: zeta} has a complex conjugate counterpart. 
%\end{remark}
\begin{proof}
    Denote $\lambda = Z^{-1}(\zeta)$. We first observe that as $\zeta\to 0$ in $V$, $\real \lambda\to-\infty$, while $\imag \lambda$ stays bounded. This is clear from Figure \ref{fig: Z_lambda}, but can also be proven rigorously. Indeed, the estimates in the proof of Lemma \ref{lem: Z inv domain} shows $\real \lambda\to-\infty$. By the formula of $\gamma_1$ in terms of $W_1$ in the same proof, the infinite end of $\gamma_1$ asymptotes to the line at height $\frac{i\pi}2$. On the other hand, the infinite end of $\gamma_1$ is in the first quadrant. It follows that $\gamma_1\cup\{i\pi\}\cup \gamma_2$ doesn't intersect the region $Q=\{\real \zeta \le 0, \imag \zeta\ge ik\pi\}$ for some large enough $k$. Since $ik\pi$ is a pole, and we know $U$ contains no pole, we have $ik\pi\in U^c = P^c\cup (\overline P)^c$, where $P$ is the component containing 0 in the complement of $\gamma_1\cup\{i\pi\}\cup \gamma_2$, as in the proof of Lemma \ref{lem: Z inv domain}. As we observed there, $\overline P$ contains the upper half plane. Thus $ik\pi \in P^c$. It follows that $Q$ is contained in the component not containing 0 in the complement of $\gamma_1\cup\{i\pi\}\cup \gamma_2$. In other words, $Q\cup U=\phi$. Similarly $\overline{Q}\cap U=\phi$, thus $\imag \lambda$ stays bounded. \eqref{eq: zeta near 0} now follows easily, as 
    \begin{equation}
        \log\zeta\sim \real \log\zeta =\log|\zeta|=2\real\lambda+\log\frac{|\lambda|}{|1-e^{2\lambda}|}\sim 2\real\lambda
    \end{equation}
    as $\zeta\to 0$.

    We next show that if $\zeta$ is between the two branch cuts in the first and fourth quadrants and $\zeta\to\infty$, then $\real\lambda\to\infty$, and $|\imag \lambda|\lesssim \real\lambda$. In fact, given such $\zeta$, we can form a contour $\gamma_*$ in the domain of $Z$ that starts at the point on $\gamma_2$ with magnitude $2|\zeta|$, following $\gamma_2$ to a point with magnitude $\tfrac{|\zeta|}{2}$, following a circular arc centered at the origin to $\overline{\gamma_2}$, and then along $\overline{\gamma_2}$ to the point of magnitude $2|\zeta|$, and finally along a circular arc centered at the origin back to where it starts. Since $Z(\xi)=\frac{\xi}{1-e^{-2\xi}}$ is very close to $\xi$ in this region, and the inside of $\gamma_*$ is connected, disjoint from the boundary of $U$, and intersects $\mathbb R$, it's easy to see that the inside of $\gamma_*$ is contained in $U$, and the winding number $n(Z(\gamma_*),\zeta)=1$. It follows from the argument principle that $\lambda$ is inside $\gamma_*$. By the description \eqref{eq: formula gamma 2}, \eqref{eq: close to id on gamma 2} of $\gamma_2$, we see that $\real\lambda\to\infty$ and $|\imag\lambda|\lesssim\real\lambda$ as $\zeta\to\infty$. \eqref{eq: zeta near +inf} now follows easily since
    \begin{equation}
        \zeta -\lambda= \frac{\lambda e^{-2\lambda}}{1-e^{-2\lambda}}\to 0
    \end{equation}
    as $\real\lambda\to\infty$.
\end{proof}

\subsection{\texorpdfstring{Green's function estimates uniform in $\zeta$}{Green's function estimates uniform in zeta}}

We now start the estimates of the Green's functions $G_L, G_R$. Due to \eqref{GL-GR}, we only show estimates of $G_L$.

\begin{proposition}\label{prop: green est}
For any sufficiently small  $\epsilon_0>0$, there exist $\epsilon_1>0$ and $C>0$ such that 
\begin{enumerate}[(i)]
\item If Re $\zeta\ge -\epsilon_0$, $0<|\text{Im }\zeta|\le \epsilon_0$, then
%\item If $\zeta\in$ region I,
\begin{equation}\label{eq: G Lp est 1}
G_L(x,\zeta) = \begin{cases}\frac i{1-2\zeta}\chi_{\R^+}(x)+\frac{i}{1-2\zeta^*}\chi_{\R^+}(x)e^{i\lambda x}+H(x,\zeta), &\text{if Im }\zeta> 0,\\ \frac i{1-2\zeta}\chi_{\R^+}(x)-\frac{i}{1-2\zeta^*}\chi_{\R^-}(x)e^{i\lambda x}+H(x,\zeta), &\text{if Im }\zeta< 0.\end{cases}
\end{equation}
Here $\lambda=Z^{-1}(\zeta)$, with $Z^{-1}$ defined in Definition \ref{def: Z inv} and
\begin{align}
H(x,\zeta) =&~\mathcal F^{-1}\left[\tfrac1{\xi-\zeta(1-e^{-2\xi})} - \tfrac{1}{1-2\zeta}\tfrac{1}{\xi}-\tfrac1{1-2\zeta^*}\tfrac1{\xi-\lambda}\right]\\
=&~\tfrac{\chi_{\R^+}(x)}{2\pi}\int_{\gamma_1\cup\{i\pi\}\cup\gamma_2}\tfrac{e^{ix\xi}}{\xi-\zeta(1-e^{-2\xi})}~d\xi\notag\\&~+\tfrac{\chi_{\R^-}(x)}{2\pi}\int_{\overline{\gamma_1}\cup\{-i\pi\}\cup\overline{\gamma_2}}\tfrac{e^{ix\xi}}{\xi-\zeta(1-e^{-2\xi})}~d\xi,\label{def: H}
\end{align}
%\begin{equation}
%H_{2}(x,\zeta) = \tfrac{\chi_{\R^+}(x)}{2\pi}\int_{\R+i\frac\pi2}\tfrac{e^{ix\xi}}{\xi-\zeta(1-e^{-2\xi})}~d\xi+\tfrac{\chi_{\R^-}(x)}{2\pi}\int_{\R-i\pi}\tfrac{e^{ix\xi}}{\xi-\zeta(1-e^{-2\xi})}~d\xi,
%\end{equation}
where $\gamma_1,\gamma_2$ are defined in Lemma \ref{lem: Z inv domain}, with estimates
\begin{equation}\label{eq: G est 2}
\|e^{\epsilon_1 |x|}H(x,\zeta)\|_{L^2_x}\le C.
\end{equation}

\item If $\dist(\zeta,\mathbb R^+)\ge\epsilon_0$, we can write $G_L$ as
%\item If $\zeta\in$ region II,
%\begin{equation}
%|G_L(x,\zeta)|\le C\left(1+\log^+\frac1{|x|}\right).
%\end{equation}
%\item If $\zeta\in $ region III$_\alpha$, we again define $H_{1}(x,\zeta)$ by \eqref{eq: est re 1} and $\lambda$ as above, but this time we have
%\begin{equation}
%\|H_{1}(\cdot,\zeta)\|_2\le \frac{C}{\sqrt{|\zeta|}}.
%\end{equation}
%\item If $\zeta\in$ region IV$_\alpha$, 
\begin{equation}\label{eq: G Lp est 3}
G_L(x,\zeta) = G_L(x,\zeta)\chi_{\mathbb R^-}(x)+\left(\tfrac{i}{1-2\zeta}+G_R(x,\zeta)\right)\chi_{\R^+}(x),
\end{equation}
%Here
%\begin{equation}
%H_2(x,\zeta) = \tfrac{\chi_{\R^+}(x)}{2\pi}\int_{\Gamma_R}\tfrac{e^{ix\xi}}{\xi-\zeta(1-e^{-2\xi})}~d\xi+\tfrac{\chi_{\R^-}(x)}{2\pi}\int_{\Gamma_L}\tfrac{e^{ix\xi}}{\xi-\zeta(1-e^{-2\xi})}~d\xi
%\end{equation}
%where $\Gamma_{L,R}$ are as in \eqref{def: G}.
with estimates
\begin{equation}\label{eq: G est 4}
\|e^{\epsilon_1 |x|}G_L(x,\zeta)\|_{L^2_x(-\infty,0)}\le C, \quad \|e^{\epsilon_1 |x|}G_R(x,\zeta)\|_{L^2_x(0,\infty)}\le C.
\end{equation}
\end{enumerate}
%Similar estimates for $G_R(x,\zeta)$ hold. The only difference is that for $G_R$, one replaces all appearances of $\frac{i}{1-2\zeta}\chi_{\R^+}(x)$ in the above estimates by $\frac{-i}{1-2\zeta}\chi_{\R^-}(x)$.
\end{proposition}
\begin{remark}
$H$ is defined and depends continuously on $\zeta\in V$, even though $G_L$, $G_R$ jump as $\zeta$ moves across $\mathbb R^+$. In fact, we have for $k\in\R^+\setminus \{\frac12\}$, 
 \begin{equation}\label{eq: G up-down}
     G_{L,R}(x,k+i0)-G_{L,R}(x,k-i0)=\frac{ie^{i\lambda x}}{1-2k^*}.
 \end{equation}
 The other terms  in \eqref{eq: G Lp est 1} are defined for $\zeta\in V$, $\zeta\ne \frac12$ (equivalently $\lambda\ne 0$). As $\zeta$ approaches $\frac12$ from the upper half plane, the function $\frac{i}{1-2\zeta}\chi_{\R^+}(x)+\frac{i}{1-2\zeta^*}\chi_{\R^+}(x)e^{i\lambda x}$ has a removable singularity and approaches $\left(-x+i\frac23\right)\chi_{\R^+}(x)+i\frac23$. It will be understood here and in the following that 
\begin{equation}
G_L(x,\tfrac12+i0)=\left(-x+i\tfrac23\right)\chi_{\R^+}(x)+H(x,\tfrac12).
\end{equation}
 In comparison, as $\zeta$ approaches $\frac12$ from the lower half plane, the corresponding function $\frac{i}{1-2\zeta}\chi_{\R^+}(x)+\frac{i}{1-2\zeta^*}\chi_{\R^-}(x)e^{i\lambda x}$, and as a result $G_L(x,\zeta)$, has a pole and blows up. The situation for $G_R$ is reversed. We therefore use a similar convention and define 
 \begin{equation}
     G_R(x,\tfrac12-i0)=(x-i\tfrac23)\chi_{\mathbb R^-}(x)+H(x,\tfrac12).
 \end{equation}
\end{remark}
\begin{remark}\label{rmk: G at k pm 0}
    For $k\in[-\epsilon_0,0)\subset \mathbb R$, the limits $H(x,k\pm i0)$ exists, obeying the same uniform estimate \eqref{eq: G est 2}. So \eqref{eq: G Lp est 1} remains valid for $\zeta=k\pm i0$, $k\in [-\epsilon_0,0)$.
\end{remark}
%\begin{remark}
%When $\zeta_\up\in \R^++i0$ and $\zeta\in \R^+-i0$ have the same real part, $H_\up(x,\zeta_\up)$ and $H_\dn(x,\zeta)$ are the same function. In fact, it's easy to see that when $\zeta\in \R^+$, $\frac1{\xi-\zeta(1-e^{-2\xi})}$ has no pole between $\R+i\frac\pi 2$ and $\R+i\pi$, and between $\R-i\frac\pi2$ and $\R-i\pi$.
%\end{remark}
%\begin{remark}\label{rmk: green est 2}
%The $\epsilon_0$ in Proposition \ref{prop: green est} can be taken arbitrarily small, while the constant $C$ will depend on $\epsilon_0$ and can become large for very small $\epsilon_0$.
%\end{remark}

\begin{proof}
The main argument is to shift the contour of integration so that it stays away from the poles of the integrand. We only show the proof for $x>0$, and the case $x<0$ can be treated similarly. We first consider $\zeta$ in a small neighborhood of $\R^+$ in $V$. If $\imag \zeta> 0$, we will shift the contour to $\R+i\pi$; if $\imag \zeta< 0$, we will shift the contour to $\R+i\frac\pi2$. Finally, if $\zeta$ stays at a positive distance away from $\R^+$, we will shift the contour to $\R+i\epsilon$ for some small $\epsilon>0$.

We first show \eqref{def: H}. The denominator of the integrand in $G_L$ is $\xi-\zeta(1-e^{-2\xi})=(1-e^{-2\xi})\left(Z(\xi)-\zeta\right)$. The zeros of this function are $0$ and the preimages of $\zeta$ under $Z$. It is clear from Figure \ref{fig: Z_lambda} that if $\zeta\in V, \imag \zeta>0$,  $\lambda=Z^{-1}(\zeta)$ and 0 are the only poles of the integrand between $\Gamma_L$ and $\gamma_1\cup\{i\pi\}\cup\gamma_2$; if $\zeta\in V, \imag\zeta<0$, then 0 is the only pole of the integrand between $\Gamma_L$ and $\gamma_1\cup\{i\pi\}\cup\gamma_2$, from which \eqref{eq: G Lp est 1} follows. The above descriptive contour shift in Figure \ref{fig: Z_lambda} can be made rigorous if needed by using the argument principle in a way similar to the proof of Lemma \ref{lem: Z inv domain}, by knowledge of the image of the contours under $Z$. The details are omitted.

%In each case, the contour chosen will dodge the region where the preimages of $\zeta$ may appear (see figure on page 45 of my handwritten notes), although as $\zeta\to +\infty$ near $\R^+$, the preimages of $\zeta$ will approach $i\pi$. Note that even if the distance of $\xi$ and the preimage of $\zeta$ tends to zero, the function $\xi-\zeta(1-e^{-2\xi})$ may not. In fact, when $\zeta$ is near $\R^+$ and $\xi$ is near $i\pi$, the imaginary part of this function is safely away from zero.

Now let's estimate the $L^2$ norm of $H$. First consider the case $\imag \zeta\ge 0$. Let $\zeta=\mu+i\epsilon$, where $\mu>-\mu_0$, $0\le \epsilon<\epsilon_0$ for some small positive $\mu_0, \epsilon_0$. In this case, we can further shift the contour in the definition of $H$ from $\gamma_1\cup\{i\pi\}\cup\gamma_2$ to $\R+i\pi$. This is again clear from Figure \ref{fig: Z_lambda}, but can also be rigorously established by finding the winding number around $\zeta$ of the image curve under $Z$ of a large rectangle whose upper and lower edges are along $\R+i\pi$ and $\Gamma_L$ (dodging $i\pi$ by a small lower semicircle). The curves all have simple explicit formulas, and it's easy to see that the winding number is 1. It follows that there are two poles of the integrand between $\Gamma_L$ and $\mathbb R+i\pi$, one of which is obviously 0, while the other one is mapped to $\zeta$ under $Z$. To see that the other pole is at $\lambda=Z^{-1}(\zeta)$, lift the line segment connecting $\frac12$ to $\zeta$ through $Z$ to a curve starting at $0$. Since the line segment is disjoint from the image under $Z$ of either the rectangle described above or the branch cuts in Figure \ref{fig: branch_cut}, the lift is disjoint from either $\mathbb R+i\pi$ or $\gamma_1\cup\{i\pi\}\cup\gamma_2$. Thus the pole is at $\lambda$. We can now write 
\begin{equation}\label{eq: H contour R+i pi}
H (x,\zeta)=\frac{e^{-\pi x}}{2\pi}\int_\R \frac{e^{ix\xi}}{\xi+i\pi-\zeta(1-e^{-2\xi})}~d\xi.
\end{equation}
We only need to estimate the $L^{2}$ norm of $\frac1{\xi+i\pi-\zeta(1-e^{-2\xi})}$. That amounts to estimating
\begin{align}
&~\int_\R \frac{1}{(|\xi-\mu(1-e^{-2\xi})|+|\pi-\epsilon(1-e^{-2\xi})|)^{2}}~d\xi\notag\\
\lesssim&~ \int_\R \frac{1}{|\xi-\mu(1-e^{-2\xi})|^{2}+1}~d\xi\notag\\
=&~\int_\R\frac{1}{|\eta+\mu e^{-2\mu}e^{-2\eta} |^{2}+1}~d\eta\qquad\qquad (\eta = \xi-\mu)\notag\\
\lesssim &~\int_\R \frac{1}{1+|y|^{2}}~dy=C.
\end{align}
The last step in the above estimates needs more explanation. We let $y=\eta+\mu e^{-2\mu}e^{-2\eta}$. When $-\mu_0<\mu<0$, $\frac{dy}{d\eta} = 1-2\mu e^{-2\mu}e^{-2\eta}\ge 1$. In this case $\eta$ depends monotonically on $y$, and $0\le \frac{d\eta}{dy}\le 1$, so the estimates follow. When $\mu>0$, $\frac{dy}{d\eta}= 1-2\mu e^{-2\mu}e^{-2\eta}$ has a unique zero at $\eta_0 = -\mu+\frac12\log(2\mu)$. On either side of $\eta_0$, $\eta$ depends monotonically on $y$, and $\left|\frac{dy}{d\eta}\right|=|1-e^{-2(\eta-\eta_0)}|\gtrsim 1$ when $|\eta-\eta_0|\ge 1$. Thus the estimates follow.

Next, we consider the case $\imag \zeta\le 0$ and let $\zeta=\mu-i\epsilon$, where $\mu>-\mu_0$, $0\le\epsilon<\epsilon_0$ for some small positive $\mu_0,\epsilon_0$. In this case, we shift the contour of integration to $\R+i\frac\pi 2$. The explanation is similar to the case when $\zeta=\mu+i\epsilon$ above, and is omitted. We have 
\begin{equation}
H(x,\zeta)=e^{-\frac{\pi x}2}\int_\R\frac{e^{ix\xi}}{\xi+i\frac\pi2-\zeta(1+e^{-2\xi})}~d\xi. \notag
\end{equation}
We estimate the $L^{2}$ norm of $\frac{1}{\xi+i\frac\pi2-\zeta(1+e^{-2\xi})}$ as follows:
\begin{align*}
&~\int_\R \frac{1}{(|\xi-\mu(1+e^{-2\xi})|+|\frac\pi 2+\epsilon (1+e^{-2\xi})|)^{2}}~d\xi\\
\lesssim &~\int_\R\frac{1}{|\xi-\mu(1+e^{-2\xi})|^{2}+1}~d\xi\\
\lesssim &~\int_\R\frac{1}{|\eta-\mu e^{-2\mu}e^{-2\eta}|^{2}+1}~d\eta\\
\lesssim &~\int_\R \frac{1}{1+|y|^{2}}~dy=C
\end{align*}
by a similar estimate as above.

\begin{comment}
The $L^2$ norm of $\partial_\zeta H$ can be estimated in a similar way. In particular, when $\imag \zeta\ge 0$, we write as in \eqref{eq: H contour R+i pi}
\begin{align}
    \partial_\zeta H(x,\zeta)&=\frac{e^{-\pi x}}{2\pi}\int_\R \frac{e^{ix\xi}(1-e^{-2\xi})}{[\xi+i\pi-\zeta(1-e^{-2\xi})]^2}~d\xi\notag\\
    &=\frac{e^{-\pi x}}{2\pi}\int_\R \frac{e^{ix\xi}}{[\xi+i\pi-\zeta(1-e^{-2\xi})][Z(\xi)+\frac{i\pi}{1-e^{-2\xi}}-\zeta]}~d\xi
\end{align}
Let $\zeta=\mu+i\epsilon$ with $\mu>-\mu_0$, $0\le \epsilon<\epsilon_0$ as above,
\begin{equation}
    |Z(\xi)+\tfrac{i\pi}{1-e^{-2\xi}}-\zeta|\ge \tfrac{\pi}{1-e^{-2\xi}}-\epsilon
\end{equation}
\end{comment}

Finally, let $\text{dist}(\zeta,\R^+)\ge\epsilon_0>0$. Again we only consider the case $x>0$. Shift the contour of integration in $G_R$ to $\R+i\epsilon$. We may choose $\epsilon\in(0,\frac{\epsilon_0}4)$ so small that $\dist(Z(\R+i\epsilon),\mathbb R^+)<\frac{\epsilon_0}{2}$. We have 
\begin{equation}
G_R(x,\zeta) = \frac{e^{-\epsilon x}}{2\pi}\int_\R\frac{e^{ix\xi}}{\xi+i\epsilon-\zeta(1-e^{-2\xi-2i\epsilon})}~d\xi.
\end{equation}
We estimate the $L^{2}$ norm of 
\begin{equation}
\frac{1}{\xi+i\epsilon-\zeta(1-e^{-2\xi-2i\epsilon})}=\frac{e^{2i\epsilon}}{(e^{2i\epsilon}-e^{-2\xi})(Z(\xi+i\epsilon)-\zeta)}
\end{equation}
Choose $\xi_0>0$ large enough so that 
\begin{equation}
\left|(\xi+i\epsilon)\left(\frac1{1-e^{-2\xi-2i\epsilon}}-1\right)\right|<\frac{\epsilon_0}4
\end{equation}
when $\xi>\xi_0$. In this range of $\xi$, 
\begin{equation}
\left|\frac{\xi+i\epsilon}{1-e^{-2\xi-2i\epsilon}}-\zeta\right|\approx |\xi+i\epsilon-\zeta|\gtrsim |\xi-\text{Re }\zeta|+{\epsilon_0} ,
\end{equation}
and
$|e^{2i\epsilon}-e^{-2\xi}|\gtrsim 1$. On the other hand, when $\xi<-1$, $|e^{2i\epsilon}-e^{-2\xi}|\approx e^{-2\xi}$, and 
$$\left|\frac{\xi+i\epsilon}{1-e^{-2\xi-2i\epsilon}}-\zeta\right|\ge\frac{\epsilon_0} 2.$$
When $-1<\xi<\xi_0$, the function is uniformly bounded. In summary, we conclude that
\begin{equation}
\int_\R \frac{1}{|\xi+i\epsilon-\zeta(1-e^{-2\xi-2i\epsilon})|^{2}}\lesssim \int_{-\infty}^0e^{4\xi}~d\xi+\int_\R\frac{1}{1+|\xi|^{2}}~d\xi.
\end{equation}
The proof is now complete.
\end{proof}

\subsection{\texorpdfstring{Continuous dependence on $\zeta$}{Continuous dependence on zeta}}

The above uniform estimates on the Green's functions will be used to establish  existence of Jost solutions. By using different contour shifts, we can get further continuity estimates of $G_L$ and $G_R$ in $\zeta$, which will be useful in showing continuous dependence of the Jost solutions on $\zeta$. We start with estimates on the $H$ function. %In particular, we will find that $G_L$ and $G_R$ are continuous in suitable spaces for $\zeta$ in a small neighborhood of $0$, and for $\zeta$ on either side of the positive real line.  More precisely, we have 
We define for $\zeta_1,\zeta_2\in V$
\begin{equation}
    \dist_V(\zeta_1,\zeta_2)=\inf\{\text{length}(\gamma)~|~\gamma:[0,1]\to V, \gamma(0)=\zeta_1,\gamma(1)=\zeta_2\}.
\end{equation}
\begin{proposition}\label{prop: G cont}
Let $\epsilon_0>0$ be sufficiently small. 
Let $\zeta,\zeta_1,\zeta_2\in V$ and let $H$ be defined as in \eqref{def: H}. Then there exists $C>0$ such that %for  $\real \zeta>-\epsilon_0$,  $|\imag\zeta|<\epsilon_0$, $\zeta\ne 0$, we have
%\begin{equation}\label{eq: G cont 1}
%G_L(x,\zeta) =  \begin{cases}\frac {i}{1-2\zeta}\chi_{\R^+}(x)+\frac{i}{1-2\zeta_\up ^*}\chi_{\R^+}(x)e^{i\lambda_\up x}+H_\up(x,\zeta), &\text{if Im }\zeta\ge 0,\\ \frac {i}{1-2\zeta}\chi_{\R^+}(x)-\frac{i}{1-2\zeta_\dn^*}\chi_{\R^-}(x)e^{i\lambda_\dn x}+H_\dn(x,\zeta), &\text{if Im }\zeta\le 0,\end{cases}
%\end{equation}
%and we have the following continuity estimates.
\begin{enumerate}[(i)]
\item If $0<|\zeta|<\epsilon_0$,
\begin{equation}\label{G cont near 0}
\|H(x,\zeta)\|_{L^2}\le\frac{C}{\sqrt{\log^+|\zeta|}}.
\end{equation}
\item If $0<|\zeta_2|\le|\zeta_1|<\epsilon_0$, $\dist_V(\zeta_1,\zeta_2)<\frac12|\zeta_1|$, we have
\begin{equation}\label{eq: G cont 2}
\|H(x,\zeta_1)-H(x,\zeta_2)\|_{ L^2\cap L^\infty}\le \frac{C\dist_V(\zeta_1,\zeta_2)}{|\zeta_1|\log^+|\zeta_1|}.
\end{equation}
%and
%\begin{equation}\label{eq: G cont 3}
%\|G_L^2(\cdot,\zeta)\|_2\le  \frac{C}{\sqrt{|\log|\zeta||}},
%\end{equation}
%\begin{equation}\label{eq: G cont 4}
%\|G_L^2(\cdot,\zeta)*f\|_2\le \frac{C}{\left|\log |\zeta|\right|}\|f\|_2.
%\end{equation}
\item If $\real\zeta_j>\frac{\epsilon_0} 2$, $|\imag \zeta_j|\le \epsilon_0$, $j=1,2$,% Im $\zeta_1 \cdot \text{Im }\zeta_2\ge 0$,
\begin{equation}\label{eq: G cont 6}
\|H(x,\zeta_1)-H(x,\zeta_2)\|_{L^2\cap L^\infty}\le C|\zeta_1-\zeta_2|.
\end{equation}
\end{enumerate}
Finally, let $\zeta_1, \zeta_2 \in \mathbb C$.

\begin{enumerate}[(i)]
\setcounter{enumi}{3}
\item If $\dist(\zeta_j,\mathbb R^+)\ge \epsilon_0$, $|\zeta_1-\zeta_2|<\epsilon_0$, one has
\begin{equation}\label{eq: G cont 7}
\|G_L(x,\zeta_1)-G_L(x,\zeta_2)\|_{L^\infty}\le C|\zeta_1-\zeta_2|.
\end{equation}
%as long as Im $\zeta_1 \cdot \text{Im }\zeta_2\ge 0$.
\end{enumerate}

\end{proposition}

Before proving this proposition, we first derive two  corollaries describing continuity of the Green's functions in a way that's convenient for later applications. As it turns out, the limit for $G_L$ is a bit more singular as $\zeta$ approaches $\mathbb R^--i0$ or $\frac12+i0$, and similarly for $G_R$ at the conjugate values.

\begin{corollary}\label{cor: G cont 1}
For every $\zeta_1\in( \mathbb C\setminus \R^+)\cup (\R^+ +  i0)\setminus \{\frac12+i0\}$, we have 
\begin{equation}
\|G_L(x,\zeta_1)-G_L(x,\zeta_2)\|_{(L^2_{-1,+}+L^\infty_{-1,+})}\to 0
\end{equation}
as $\zeta_2\to\zeta_1$ from within the set $( \mathbb C\setminus \R^+)\cup (\R^+ +  i0)$.
Correspondingly, for every $\zeta_1\in( \mathbb C\setminus \R^+)\cup (\R^+ -  i0)\setminus \{\frac12-i0\}$, we have 
\begin{equation}
\|G_R(x,\zeta_1)-G_R(x,\zeta_2)\|_{(L^2_{-1,-}+L^\infty_{-1,-})}\to 0
\end{equation}
as $\zeta_2\to\zeta_1$ from within the set $( \mathbb C\setminus \R^+)\cup (\R^+ -  i0)$.

\end{corollary}
\begin{comment}
    The weight here can be any s<0, doesn't have to be -1. I only put -1 for simplicity, as the weight on u is enough to balance out the -1 anyway.
\end{comment}
\begin{proof}
We only show the proof regarding $G_L$. Firstly if $\zeta_1=0$, $G_L(x,0)=i\chi_{\R^+}(x)$. By \eqref{eq: G Lp est 1} and \eqref{G cont near 0}, 
\begin{align}
\|G_L(x,\zeta_2)-G_L(x,0)\|_{L^2+L^\infty}&\le C\left|\frac{1}{1-2\zeta_2}-1\right|+\frac{C}{|\zeta_2^*|}+\frac{C}{\sqrt{\log^+|\zeta_2|}}\notag\\
&\le \frac{C}{\sqrt{\log^+|\zeta_2|}}\to 0
\end{align}
as $\zeta_2\to 0$. If $\zeta_1$ is in a small neighborhood of 0, and $\imag \zeta_1>0$, we can assume that $\zeta_2$ also has positive imaginary part. By \eqref{eq: G Lp est 1} and \eqref{eq: G cont 2}, as $\zeta_2\to\zeta_1$,  
\begin{align}
\|G_L(x,\zeta_1)-G_L(x,\zeta_2)\|_{L^\infty}&\le \left|\tfrac1{1-2\zeta_1}-\tfrac1{1-2\zeta_2}\right|+ \left|\tfrac1{1-2\zeta_1^*}-\tfrac1{1-2\zeta_2^*}\right|+\tfrac{C|\zeta_1-\zeta_2|}{|\zeta_1|\log^+|\zeta_1|}\notag\\
&\quad +\left|\tfrac{1}{1-2\zeta_1^*}\right|\left\|e^{i\lambda_1 x}\chi_{\R_+}(x)\left(e^{i[\lambda_1-\lambda_2] x}-1\right)\right\|_{L^\infty}.
\end{align}
The last term above can be estimated as 
\begin{align}
\left|e^{i\lambda_1 x}\chi_{\R_+}(x)\left(e^{i[\lambda_1-\lambda_2] x}-1\right)\right|&\le \chi_{\R_+}(x) e^{-\imag \lambda_1 x}\min\left(2,|\lambda_1-\lambda_2| x\right) \notag\\
&\le 2e^{-\imag \lambda_1 x_0}+|\lambda_1-\lambda_2| x_0
\end{align}
for any $x_0>0$. Taking $x_0=|\lambda_1-\lambda_2|^{-\frac12}$, and noting the fact that $\imag \lambda_1>0$, we can bound the above by 
\begin{equation}
2e^{-\imag \lambda_1 |\lambda_1-\lambda_2|^{-\frac12}}+|\lambda_1-\lambda_2| ^{\frac12}\to 0
\end{equation}
as $\zeta_2\to\zeta_1$. As a consequence, $\|G_L(x,\zeta_1)-G_L(x,\zeta_2)\|_{L^\infty}\to 0$ as $\zeta_2\to\zeta_1$. The case when $\zeta_1$ is near 0 and in the lower half plane can be handled in a similar way, with $\chi_{\R^+}(x)$ replaced with $\chi_{\R^-}(x)$. The above argument still works since $\lambda_1, \lambda_2$ are now in the lower half plane and $x<0$. Now suppose $\zeta_1$ is near zero and on $\R^-$. If $\zeta_2\to \zeta_1$ and $\imag \zeta_2\ge 0$, we can use the first representation for $G_L(x,\zeta_1)$ and $G_L(x,\zeta_2)$ in \eqref{eq: G Lp est 1} (see Remark \ref{rmk: G at k pm 0}. Note that even in this case, $\imag \lambda_1>0, \imag \lambda_2>0$. So the above argument can again be repeated to show $\|G_L(x,\zeta_1)-G_L(x,\zeta_2)\|_{L^\infty}\to 0$. On the other hand, if $\zeta_2\to \zeta_1$ and $\imag \zeta_2\le 0$, we can use the second representation for $G_L(x,\zeta_1)$ and $G_L(x,\zeta_2)$ in \eqref{eq: G Lp est 1} and repeat the above argument to show convergence.

Let us now consider the case when $\zeta_1$ is in a small neighborhood of $ \R^++i0$ away from the origin and $\zeta_1\ne \frac12+i0$. In this case, we must have $\imag \zeta_1\ge 0$. If $\zeta_2\to \zeta_1$ and $\zeta_2\in ( \mathbb C\setminus \R^+)\cup (\R^+ +  i0)$, we must have $\imag \zeta_2\ge 0$. Use the first representation of $G_L(x,\zeta_1)$ and $G_L(x,\zeta_2)$ in \eqref{eq: G Lp est 1} and \eqref{eq: G cont 6} to get
\begin{align}\label{eq: G cont in zeta near 1/2}
\|G_L(x,\zeta_1)-G_L(x,\zeta_2)\|_{L^\infty_{-1,+}}&\le \left|\tfrac1{1-2\zeta_1}-\tfrac1{1-2\zeta_2}\right|+ \left|\tfrac1{1-2\zeta_1^*}-\tfrac1{1-2\zeta_2^*}\right|+C|\zeta_1-\zeta_2|\notag\\
&\quad +\left|\tfrac{1}{1-2\zeta_1^*}\right|\left\|\chi_{\R_+}(x)\left(e^{i[\lambda_1-\lambda_2] x}-1\right)\right\|_{L^\infty_{-1,+}}.
\end{align}
Note that we no longer have an exponential gain from $e^{-\imag \lambda_1 x}$, since $\imag \lambda_1$ may now be zero. We can estimate the last term above as
\begin{align}
\frac{\chi_{\R^+}(x)}{1+x}\left|e^{i[\lambda_1-\lambda_2] x}-1\right|&= \frac{\chi_{\R^+}(x)|\lambda_1-\lambda_2|x}{1+x}\frac{\left|e^{i[\lambda_1-\lambda_2] x}-1\right|}{|\lambda_1-\lambda_2|x}\notag\\
&\lesssim |\lambda_1-\lambda_2|.
\end{align}
 Thus $\|G_L(x,\zeta_1)-G_L(x,\zeta_2)\|_{L^\infty_{-1,+}}\to 0$ as $\zeta_2\to\zeta_1$.

%Next we consider the case $\zeta_1=\frac12+i0$.

Now if $\zeta_1$ is in a small neighborhood of $\R^+-i0$ and away from the origin, and if $\imag \zeta_1<0$, we can take $\epsilon_0$ in Proposition \ref{prop: G cont} so small that $\zeta_1$, $\zeta_2$ both belong to the region covered in case (iv), where the result follows obviously. Finally if $\zeta_1$ is away from $\R^+$, we can again assume $\zeta_1$, $\zeta_2$ both belong to the region covered in case (iv) of Proposition \ref{prop: G cont}, and the result follows.
\end{proof}

Note that in the above corollary, we need to exclude the cases $\zeta_1\in \R^+-i0$ or $\zeta_1=\frac12+i0$ for $G_L$, and similar but conjugate values for $G_R$. As $\zeta\to \frac12+i0$, $G_L(x,\zeta)-G_L(x,\frac12+i0)$ contains a term that essentially resembles $\frac{e^{i\lambda x}-1-i\lambda x}{\lambda}\chi_{\R^+}(x)$. Its $L_{-1,+}^\infty$ norm, or the $L^\infty$ norm of 
$
\frac{e^{i\lambda x}-1-i\lambda x}{\lambda(1+x_+)}\chi_{\R^+}(x)
$
doesn't tend to 0 when $\lambda\to 0$. On the other hand, if $\zeta_2\to\zeta_1\in (\R^+-i0)\setminus \{\frac12-i0\}$, $G_L(x,\zeta_1)-G_L(x,\zeta_2)$ contains a term $\left(e^{i\lambda_1 x}-e^{i\lambda_2 x}\right)\chi_{\R^-}(x)$. Its $L_{-1,+}^\infty$ norm is the same as its $L^\infty$ norm, which doesn't tend to zero as $\zeta_2\to \zeta_1$. However, we still have uniform convergence on compact sets. In particular, we have
\begin{lemma}\label{lem: G weak cont 1/2}
As $\zeta\to\frac12+i0$, $G_L(x,\zeta)-G_L(x,\tfrac12+i0)-H(x,\zeta)+H(x,\tfrac12+i0)$ is bounded in $L_{-1,+}^\infty$, and tends to zero uniformly on compact sets. As $\zeta_2\to \zeta_1 \in( \R^+-i0)\setminus \{\frac12-i0\}$, $G_L(x,\zeta_2)-G_L(x,\zeta_1)-H(x,\zeta_2)+H(x,\zeta_1)$ is bounded in $L^\infty$, and tends to zero uniformly on compact sets.

Correspondingly, As $\zeta\to\frac12-i0$, $G_R(x,\zeta)-G_R(x,\tfrac12-i0)-H(x,\zeta)+H(x,\tfrac12-i0)$ is bounded in $L^\infty_{-1,-}$, and tends to zero uniformly on compact sets. As $\zeta_2\to \zeta_1 \in( \R^++i0)\setminus \{\frac12+i0\}$, $G_R(x,\zeta_2)-G_R(x,\zeta_1)-H(x,\zeta_2)+H(x,\zeta_1)$ is bounded in $L^\infty$, and tends to zero uniformly on compact sets.
\end{lemma}
\begin{proof}
We only show the proof for $G_L$. As $\zeta\to\frac12+i0$,
\begin{align}
&~G_L(x,\zeta)-G_L(x,\tfrac12+i0) - H(x,\zeta)+H(x,\tfrac12+i0)\notag\\
= &~\left(\frac{i}{1-2\zeta}+\frac{i}{1-2\zeta^*}e^{i\lambda x}+x-i\frac23\right)\chi_{\R^+}(x)\notag\\
=&~i\chi_{\R^+}(x)\left(\frac{1}{1-2\zeta}+\frac1{1-2\zeta^*}-\frac23\right) -x\chi_{\R^+}(x)\left(\frac{\lambda}{1-2\zeta^*}-1\right)\label{G 1/2 diff line 1}\\
&\qquad + i\chi_{\R^+}(x)\frac{\lambda }{1-2\zeta^*}\left(\frac{e^{i\lambda x}-1}{\lambda}-ix\right).
\end{align}
It's easy to see that $\lambda\to 0$, $\frac{1}{1-2\zeta}+\frac1{1-2\zeta^*}-\frac23\to 0$, and $\frac{\lambda}{1-2\zeta^*}-1\to 0$ as $\zeta\to \frac12+i0$. Thus \eqref{G 1/2 diff line 1} tends to zero in $L^\infty_{-1,+}$. Since $\lambda$ has nonnegative imaginary part, $\chi_{\R^+}(x) \frac{e^{i\lambda x}-1}\lambda$ is bounded in $L^\infty_{-1,+}$. On the other hand, it's easy to see that
\begin{equation}
\chi_{\R^+}(x)\left(\frac{e^{i\lambda x}-1}{\lambda}-ix\right)=\chi_{\R^+}(x)\left(\frac{e^{i\lambda x} - 1 -i\lambda x}{i\lambda x}ix\right)
\end{equation}
converges to zero uniformly on compact sets.

In the case when $\zeta_2\to \zeta_1\in( \R^+-i0)\setminus \{\frac12-i0\}$
\begin{align}
&~G_L(x,\zeta_2)-G_L(x,\zeta_1) -H(x,\zeta_2)+H(x,\zeta_1)\notag\\
=&~ \left(\tfrac{i}{1-2\zeta_2}-\tfrac{i}{1-2\zeta_1} \right)\chi_{\R^+}(x) + \left(\tfrac{ie^{i\lambda_2 x}}{1-2\zeta_2^*}-\tfrac{ie^{i\lambda_1 x}}{1-2\zeta_1^*}\right)\chi_{\R^-}(x)\notag\\
=&~\left(\tfrac{i}{1-2\zeta_2}-\tfrac{i}{1-2\zeta_1} \right)\chi_{\R^+}(x) +ie^{i\lambda_2 x} \left(\tfrac{1}{1-2\zeta_2^*}-\tfrac{1}{1-2\zeta_1^*}\right)\chi_{\R^-}(x)\\
&\qquad + \tfrac{i}{1-2\zeta_1^*}\left(e^{i\lambda_2 x}-e^{i\lambda_1 x}\right)\chi_{\R^-}(x)
\end{align}
Note that $\lambda_1$ and $\lambda_2$ has negative imaginary parts. It follows that $e^{i\lambda_1 x}\chi_{\R^-}(x)$ and $e^{i\lambda_2 x}\chi_{\R^-}(x)$ are bounded. On the other hand,
\begin{equation}
\left(e^{i\lambda_2 x}-e^{i\lambda_1 x}\right)\chi_{\R^-}(x)  = e^{i\lambda_1 x}\left(e^{i[\lambda_2- \lambda_1]x}-1\right)\chi_{\R^-}(x)
\end{equation}
tends to zero uniformly on compact sets as $\zeta_2\to \zeta_1$. The assertion now follows.
\end{proof}

We now return to the proof of 
Proposition \ref{prop: G cont}. It will be proven in several steps via the following lemmas. We start by showing a few estimates regarding the relative sizes of the terms in $\xi-\zeta(1-e^{-2\xi})$.

\begin{lemma}\label{lem: rel size est}
For any $\mu\in\R^+$, let $\lambda=Z^{-1}(\mu)\in \R$. %Note that $\mu$ being small implies $\lambda$ large in size and negative, while $\mu$ being large implies $\lambda$ large and positive. 
There is a small number $\mu_0>0$ and a large number $\mu_1=Z(\lambda_1)>0$ such that 
\begin{enumerate}[(i)]
\item Suppose $0<\mu<\mu_0$, then 

if $\lambda-1<\xi<\lambda+1$, 
\begin{equation}\label{eq: p lam}
|\lambda(\xi-\lambda)|<\left|\xi-\mu(1-e^{-2\xi})\right|<4|\lambda(\xi-\lambda)|;
\end{equation}

if $\xi<\lambda-1$, $\mu(1-e^{-2\xi})$ dominates, i.e.
\begin{equation}\label{eq: p lam-1}
|\xi|<\frac12\left|\mu(1-e^{-2\xi})\right|;
\end{equation}

if $\xi>\lambda+1$, $\xi$ dominates, i.e.
\begin{equation}\label{eq: p lam+1}
\left|\mu(1-e^{-2\xi})\right|<\frac12|\xi|;
\end{equation}

if $\xi<2\lambda$, $\mu(1-e^{-2\xi})$ dominates super polynomially, i.e. for every $N>0$, there is a $C_N>0$ such that
\begin{equation}\label{eq: p 2lam}
|\xi|^N<C_N\left|\mu(1-e^{-2\xi})\right|.
\end{equation}

if $\xi>\frac\lambda 2$, $\mu e^{-2\xi}$ is super polynomially small compared with $\frac{1}{|\xi|+|\lambda|}$, i.e. for every $N>0$, there is a $C_N>0$ such that
\begin{equation}\label{eq: p lam/2}
\mu e^{-2\xi}\le \frac{C_N}{(|\xi|+|\lambda|)^N}.
\end{equation}

\item Suppose $\mu>\mu_1$, $\xi>\lambda_1$, then 
\begin{equation}\label{eq: p pos lam}
\frac12|\xi-\lambda|<\left|\xi-\mu(1-e^{-2\xi})\right|<2|\xi-\lambda|.
\end{equation}
\end{enumerate}
\end{lemma}
\begin{proof}
    Suppose $0<\mu<\mu_0$, for a small $\mu_0$ to be determined. For $\xi\in(\lambda-1,\lambda+1)$, we have for some $\xi_1\in(\lambda-1,\lambda+1)$:
    \begin{align}
        \left|\frac{\xi-\mu(1-e^{-2\xi})}{\lambda(\xi-\lambda)}\right|&=\frac{|Z(\xi)-Z(\lambda)||1-e^{-2\xi}|}{|\lambda(\xi-\lambda)|}\notag\\
        &=\frac{|Z'(\xi_1)||e^{-2\xi}-1|}{|\lambda|}\notag\\
        &=\frac{|1-2\xi_1-e^{2\xi_1}|}{|\lambda|}\frac{e^{-2\xi_1}(e^{-2\xi}-1)}{(e^{-2\xi_1}-1)^2}\to 2
    \end{align}
    as $\mu_0\searrow 0$ or $\lambda\to-\infty$. Thus \eqref{eq: p lam} follows for $\mu_0$ sufficiently small.

    Recall that $Z$ maps $\mathbb R$ to $\mathbb R^+$ as an increasing function. As a result, if $\xi<\lambda-1$, 
    \begin{align}
        \left|\frac{\xi}{\mu(1-e^{-2\xi})}\right|=\frac{Z(\xi)}{\mu}< \frac{Z(\lambda-1)}{Z(\lambda)}&=\frac{1-e^{-2\lambda}}{\lambda}\frac{\lambda-1}{1-e^{-2(\lambda-1)}}\notag\\
        &\to e^{-2}<\frac12
    \end{align}
    as $\lambda\to-\infty$. Thus \eqref{eq: p lam-1} follows.
    Similarly, if $\xi>\lambda+1$, 
    \begin{align}
        \left|\frac{\mu(1-e^{-2\xi})}{\xi}\right|=\frac{\mu}{Z(\xi)}< \frac{Z(\lambda)}{Z(\lambda+1)}&=\frac{\lambda}{1-e^{-2\lambda}}\frac{1-e^{-2(\lambda+1)}}{\lambda+1}\notag\\
        &\to e^{-2}<\frac12
    \end{align}
    as $\lambda\to-\infty$. Thus \eqref{eq: p lam+1} follows. 

    Now assume $\xi<2\lambda$, we have
    \begin{align}
        \left|\frac{\xi^N}{\mu(1-e^{-2\xi})}\right|\lesssim e^{2(\xi-\lambda)}\frac{|\xi|^N}{|\lambda|}&\lesssim_N e^{2(\xi-\lambda)}\frac{|\xi-\lambda|^N+|\lambda|^N}{|\lambda|}\notag\\
        &\lesssim_N e^{2(\xi-\lambda)}|\xi-\lambda|^{N-1}+e^{2\lambda}|\lambda|^{N-1}\notag\\
        &\lesssim_N1.
    \end{align}
    Similarly, if $\xi>\frac\lambda2$,
    \begin{align}
        \mu e^{-2\xi}(|\xi|+|\lambda|)^N&\lesssim_N|\lambda|e^{-2(\xi-\lambda)}\left(|\xi-\lambda|^N+2|\lambda|^N\right)\notag\\
        &\lesssim_{N}|\lambda|e^{\frac{\lambda}{2}}e^{-(\xi-\lambda)}|\xi-\lambda|^N+|\lambda|^{N+1}e^{\lambda}\notag\\
        &\lesssim_N 1.
    \end{align}

    Finally, we switch to the second case, when $\mu>\mu_1=Z(\lambda_1)$, and $\xi>\lambda_1$. By the mean value theorem, there is a $\xi_1$ between $\xi$ and $\lambda$ such that 
    \begin{align}
        \left|\frac{\xi-\mu(1-e^{-2\xi})}{\xi-\lambda}\right|&=(1-e^{-2\xi})\frac{|Z(\xi)-Z(\lambda)|}{|\xi-\lambda|}\notag\\
        &=(1-e^{-2\xi})|Z'(\xi_1)|\to 1
    \end{align}
    as $\mu_1\to\infty$. Thus \eqref{eq: p pos lam} follows.
\end{proof}

For the next lemma, we need the function 
\begin{equation}
    Z_1(\lambda) = \frac{\lambda}{1+e^{-2\lambda}}.
\end{equation}
It is easy to see with a direct calculation that $Z_1$ regarded as a function on $\mathbb R$ is strictly decreasing on $(-\infty, -\frac12(1+W_0(\frac1e))]$, with range $[-\frac12W_0(\frac1e),0)$, and is strictly increasing on $[-\frac12(1+W_0(\frac1e)),\infty)$, with range $[-\frac12W_0(\frac1e),\infty)$. We denote the inverses of the two pieces by $Z_{1,-}^{-1}$ and $Z_{1,+}^{-1}$ respectively.
\begin{lemma}\label{lem: rel size est 1}
There is a small $\mu_0>0$ and a large $\mu_1>0$ such that
\begin{enumerate}[(i)]
    \item Suppose $0<\mu<\mu_0$, and let $\lambda = Z_{1,-}^{-1}(-\mu)$, then
    
    if $\lambda-1<\xi<\lambda+1$, 
\begin{equation}
|\lambda(\xi-\lambda)|<\left|\xi+\mu(1+e^{-2\xi})\right|<4|\lambda(\xi-\lambda)|;
\end{equation}

if $\xi<\lambda-1$, 
\begin{equation}
|\xi|<\frac12\left|\mu(1+e^{-2\xi})\right|;
\end{equation}

if $\xi>\lambda+1$, 
\begin{equation}
\left|\mu(1+e^{-2\xi})\right|<\frac12(1+|\xi|);
\end{equation}

if $\xi<2\lambda$,  for every $N>0$, there is a $C_N>0$ such that
\begin{equation}
|\xi|^N<C_N\left|\mu(1+e^{-2\xi})\right|.
\end{equation}

if $\xi>\frac\lambda 2$, for every $N>0$, there is a $C_N>0$ such that
\begin{equation}
\mu e^{-2\xi}\le \frac{C_N}{(|\xi|+|\lambda|)^N}.
\end{equation}

\item Suppose $\mu>\mu_1$. Let $\lambda=Z_{1,+}^{-1}(\mu)$, $\lambda_1=Z_{1,+}^{-1}(\mu_1)$.
 If  $\xi>\lambda_1$, then 
\begin{equation}
\frac12|\xi-\lambda|<\left|\xi-\mu(1+e^{-2\xi})\right|<2|\xi-\lambda|.
\end{equation}
\end{enumerate}
%(i) of Lemma \ref{lem: rel size est} continues to hold (with the extra restriction $|\xi|\ge 1$ in \eqref{eq: p lam+1}) if $-\mu(1-e^{-2\xi})$ is replaced everywhere by $\mu(1+e^{-2\xi})$, and $\lambda\ll -1$ is the unique large negative number such that $\frac{\lambda}{1+e^{-2\lambda}}=-\mu$.

\end{lemma}

\begin{proof}
    We omit the details of the proof of this lemma, as it involves estimates very similar to the proof of Lemma \ref{lem: rel size est}.
\end{proof}
%\begin{remark}\label{rem: rel size est}
%(ii) of Lemma \ref{lem: rel size est} continues to hold if $-\mu(1-e^{-2\xi})$ is replaced everywhere by $-\mu(1+e^{-2\xi})$, and $\lambda>0$ is the unique number such that $\frac{\lambda}{1+e^{-2\lambda}}=\mu$.
%\end{remark}

We will also need a complex version of (i) of Lemma \ref{lem: rel size est}. The following lemma is stated for $\zeta$ in the upper half plane, but its complex conjugated counterpart also holds.
\begin{lemma}
For $\zeta\in V$, let $\lambda=Z^{-1}(\zeta)$ (see Definition \ref{def: Z inv}). If $|\zeta|<\epsilon_0$ sufficiently small, we have:
\begin{enumerate}[(i)]

\item if $\real\xi\le \real\lambda-1$, $-\pi\le \imag \xi\le \pi$, then $\zeta(1-e^{-2\xi})$ dominates, i.e.
\begin{equation}\label{eq: p lam-1 comp}
|\xi|\le \frac12\left|\zeta(1-e^{-2\xi})\right|;
\end{equation}

\item if $\real\xi\ge \real\lambda+1$, $-\pi\le \imag \xi\le \pi$, then $\xi$ dominates, i.e.
\begin{equation}\label{eq: p lam+1 comp}
\left|\zeta(1-e^{-2\xi})\right|\le \frac12|\xi|;
\end{equation}

\item if $\real\lambda - 1\le\real \xi \le \real\lambda+1$, $-\frac{\pi}4\le \imag \xi\le \frac{3\pi} 4$, there is a constant $C>0$ such that
\begin{equation}\label{eq: p lam comp}
\frac1C |\lambda(\xi-\lambda)|\le \left|\xi-\zeta(1-e^{-2\xi})\right|\le C|\lambda(\xi-\lambda)|;
\end{equation}

\item if $\real\xi\le 2\real\lambda$, $\zeta(1-e^{-2\xi})$ dominates super polynomially, i.e. for every $N>0$, there is a $C_N>0$ such that
\begin{equation}\label{eq: p 2lam comp}
|\xi|^N\le C_N\left|\zeta(1-e^{-2\xi})\right|.
\end{equation}

\item if $\real\xi\ge\real\frac\lambda 2$, $\zeta e^{-2\xi}$ is super polynomially small compared with $\frac{1}{|\xi|+|\lambda|}$, i.e. for every $N>0$, there is a $C_N>0$ such that
\begin{equation}\label{eq: p lam/2 comp}
|\zeta e^{-2\xi}|\le \frac{C_N}{(|\xi|+|\lambda|)^N}.
\end{equation}
\end{enumerate}
\end{lemma}
\begin{remark}
    The asymmetry of the range of $\imag \xi$ in (iii) is necessary. As is shown in Figure \ref{fig: Z_lambda}, if $\imag \xi$ is allowed to vary in the symmetric range $(-\frac{3\pi}{4},\frac{3\pi}{4})$, there may exist $\xi_1\ne \lambda$ such that $Z(\xi_1)=Z(\lambda)=\zeta$, for which $\xi_1-\zeta(1-e^{-2\xi_1})=0$ but $\xi_1-\lambda\ne 0$.
\end{remark}
\begin{proof}
In the following proof, we always assume $-\pi\le \imag \xi\le \pi$. By Lemma \ref{lem: zeta}, if $\epsilon_0$ is chosen small enough, we have $0\le \imag \lambda\le \frac\pi 2+\frac\pi 8$, and $\lambda\sim \frac12\log \zeta$ as $\zeta\to 0$ from the upper half plane. 
Thus $\real\lambda \sim \frac12\log|\zeta|$. For $\real\xi\le \real\lambda-1$, $-\pi\le \imag \xi\le \pi$, we compute
\begin{align}
\frac{|2\xi|}{|\zeta(1-e^{-2\xi})|}&= \frac{2|\xi|}{|\lambda|e^{-2\real (\xi-\lambda)}}\left|\frac{1-e^{2\lambda}}{1-e^{2\xi}}\right|\notag\\
&\le \left(2e^{2\real(\xi-\lambda)} +\frac{2|\xi-\lambda|e^{2\real(\xi-\lambda)}}{|\lambda|}\right)\left|\frac{1-e^{2\lambda}}{1-e^{2\xi}}\right|\notag\\
&\le \left(2e^{-2}+2\frac{|\real(\xi-\lambda)|+2\pi}{|\lambda|}e^{2\real(\xi-\lambda)}\right)\left|\frac{1-e^{2\lambda}}{1-e^{2\xi}}\right|.\label{comp lambda-1 est 1}
\end{align}
Since $2e^{-2}<1$, $\real \lambda\sim \frac12\log|\zeta|$, and $\real(\xi-\lambda)\le -1$, \eqref{comp lambda-1 est 1} can be made less than 1 if we choose $\epsilon_0$ small enough. This completes the proof of (i). Similarly, for any $N>0$, and $\real\xi<2\real\lambda$, we have 
\begin{align}
\frac{|\xi|^N}{|\zeta(1-e^{-2\xi})|}&= \frac{|\xi|^N}{|\lambda|e^{-2\real (\xi-\lambda)}}\left|\frac{1-e^{2\lambda}}{1-e^{2\xi}}\right|\notag\\
&\lesssim_N \left(|\lambda|^{N-1}e^{2\real(\xi-\lambda)} +\frac{|\xi-\lambda|^Ne^{2\real(\xi-\lambda)}}{|\lambda|}\right)\notag\\
&\lesssim_N \left(|\lambda|^{N-1}e^{2\real\lambda}+\frac{|\real(\xi-\lambda)|^N+(2\pi)^N}{|\lambda|}e^{2\real(\xi-\lambda)}\right)\notag\\
&\lesssim_N 1.
\end{align}
This completes the proof of (iv).

If $\real\lambda+1\le \real \xi\le \frac{\real\lambda}2$, we can estimate
\begin{align}
\frac{2|\zeta(1-e^{-2\xi})|}{|\xi|}&=\frac{2|\lambda| e^{-2\real(\xi-\lambda)}}{|\xi|}\left|\frac{1-e^{2\xi}}{1-e^{2\lambda}}\right|\notag\\
&\le \left(2e^{-2\real(\xi-\lambda)}+2\frac{|\real(\xi-\lambda)|+2\pi}{|\xi|}e^{-2\real(\xi-\lambda)}\right)\left|\frac{1-e^{2\xi}}{1-e^{2\lambda}}\right|\notag\\
&\le \left(2e^{-2}+2\frac{|\real(\xi-\lambda)|+2\pi}{|\xi|}e^{-2\real(\xi-\lambda)}\right)\left|\frac{1-e^{2\xi}}{1-e^{2\lambda}}\right|,
\end{align}
which can be made less than 1 for $\xi$ in the above mentioned range and $\epsilon_0$ chosen small enough. When $\real\xi\ge \frac{\real\lambda}{2}$, we have
\begin{equation}
\frac{|1-e^{-2\xi}|}{|\xi|}\lesssim (1+e^{-\real\lambda}).
\end{equation}
Thus 
\begin{equation}
\frac{2|\zeta(1-e^{-2\xi})|}{|\xi|}\lesssim \frac{2|\lambda|(1+e^{-\real\lambda})}{|1-e^{-2\lambda}|},
\end{equation}
which can be made less than 1 if $\epsilon_0$ is chosen sufficiently small. This completes the proof of (ii). Similarly, for any $N>0$ and $\real\xi\ge \frac{\real\lambda}{2}$, we have
\begin{align}
|\zeta e^{-2\xi} (|\xi|+|\lambda|)^N|&\lesssim_N |\lambda| e^{-2\real(\xi-\lambda)}(|\xi-\lambda|^N+|\lambda|^N)\notag\\
&\lesssim_N |\lambda| e^{-2\real(\xi-\lambda)} |\xi-\lambda|^N e^{-2\real(\xi-\lambda)}+|\lambda|^{N+1}e^{-2\real(\xi-\lambda)}\notag\\
&\lesssim_N |\lambda| e^{\real\lambda} +|\lambda|^{N+1}\notag\\
&\lesssim_N 1.
\end{align}
This completes the proof of (iv).

Denote $R=\{\xi\in \C~|~|\real \xi-\real\lambda|\le 1, -\frac\pi4\le \imag \xi\le \frac{3\pi}4\}$. Noting that $\lambda-\zeta(1-e^{-2\lambda})=0$, we have
\begin{align}
\frac{\left|\xi-\zeta(1-e^{-2\xi})\right|}{|\lambda(\xi-\lambda)|}&\le \frac{\sup_R |1+2\zeta e^{-2\xi}|}{|\lambda|}\notag\\
&\le \frac{1+2|\zeta|e^{-2(\real \lambda-1)}}{|\lambda|}\notag\\
&\le \frac{1}{|\lambda|}+\frac{2 e^{-2(\real \lambda-1)}}{e^{-2\real \lambda}-1}\lesssim 1.
\end{align}
This proves the second inequality in \eqref{eq: p lam comp}. Since $$\xi-\zeta(1-e^{-2\xi})=(1-e^{-2\xi})\left(Z(\xi)-\zeta\right),$$ by the same argument as for the contour shift in the proof of Proposition \ref{prop: green est} (see also Figure \ref{fig: Z_lambda}), we know that it has no zero in $R$. By the maximum modulus principle, we only need to bound $\frac{\lambda(\xi-\lambda)}{\xi-\zeta(1-e^{-2\xi})}$ on $\partial R$. When $\xi\in\partial R$ and $\real\xi = \real\lambda-1$, we can use \eqref{eq: p lam-1 comp} to get
\begin{equation}
\frac{|\lambda(\xi-\lambda)|}{|\xi-\zeta(1-e^{-2\xi})|}\lesssim \frac{|\lambda|}{|\zeta(1-e^{-2\xi})|}\lesssim\frac{|\lambda|}{|\zeta(1-e^{-2\lambda})|}= 1.
\end{equation}
If  $\xi\in\partial R$ and $\real\xi = \real\lambda+1$, we can use \eqref{eq: p lam+1 comp} to get
\begin{equation}
\frac{|\lambda(\xi-\lambda)|}{|\xi-\zeta(1-e^{-2\xi})|}\lesssim \frac{|\lambda|}{|\xi|}\lesssim 1.
\end{equation}
To estimate the bound on the remaining edges of $R$, we need the following elementary inequality:
\begin{lemma}
Let $z_1,z_2\in \C$ be such that $\cos(\arg z_1-\arg z_2)\le 1-2\delta$ for some $\delta\in(0,\frac12)$. Then
\begin{equation}\label{law of cos}
|z_1-z_2|\ge \sqrt\delta(|z_1|+|z_2|).
\end{equation}
\end{lemma}
\begin{proof}
\begin{align*}
|z_1-z_2|^2&=|z_1|^2+|z_2|^2-2|z_1z_2|\cos(\arg z_1-\arg z_2)\\
&\ge |z_1|^2+|z_2|^2-2|z_1||z_2|(1-2\delta)\\
&=\delta(|z_1|^2+|z_2|^2+2|z_1||z_2|)+(1-\delta)(|z_1|^2+|z_2|^2-2|z_1||z_2|)\\
&\ge \delta(|z_1|+|z_2|)^2.
\end{align*}
\end{proof}
We continue our proof of case (iii).
If $\xi\in \partial R$ and $\imag \xi = \frac{3\pi}4$,
\begin{align}
|\xi-\zeta(1-e^{-2\xi}) |&=\left| \xi-\lambda e^{-2(\xi-\lambda)}\frac{1-e^{2\xi}}{1-e^{2\lambda}}\right|\notag\\
&=|\lambda|\left|\frac{1-e^{2\xi}}{1-e^{2\lambda}}\right|\left|\frac{\xi}{\lambda}\frac{1-e^{2\lambda}}{1-e^{2\xi}}-e^{-2(\xi-\lambda)}\right|\notag\\
&\gtrsim |\lambda|\left|\frac{\xi}{\lambda}\frac{1-e^{2\lambda}}{1-e^{2\xi}}-e^{-2(\xi-\lambda)}\right|.\label{comp near lam est 1}
\end{align}
By taking $\epsilon_0$ small enough, we can make sure $\frac{\xi}{\lambda}\frac{1-e^{2\lambda}}{1-e^{2\xi}}$ is in a small neighborhood of $1$. On the other hand, since $0\le\imag \lambda\le \frac\pi2+\frac\pi8$, we have
\begin{equation}\label{green: exp arg est}
\arg\left(e^{-2(\xi-\lambda)}\right) = -2\imag (\xi-\lambda)\in (-\tfrac{3\pi}2,-\tfrac\pi 4).
\end{equation}
Noting that $\cos (-\frac\pi4)<\frac34$, we may assume that 
\begin{equation}
\cos \left(\arg \frac{\xi}{\lambda}\frac{1-e^{2\lambda}}{1-e^{2\xi}}-\arg e^{-2(\xi-\lambda)}\right)<\frac34.
\end{equation}
By \eqref{law of cos} and \eqref{comp near lam est 1}, we get
\begin{equation}
|\xi-\zeta(1-e^{-2\xi}) |\gtrsim |\lambda|.
\end{equation}
Thus
\begin{equation}
\frac{|\lambda(\xi-\lambda)|}{|\xi-\zeta(1-e^{-2\xi})|}\lesssim \frac{|\lambda|}{|\lambda|}=1.
\end{equation}
If $\xi\in R$ and $\imag \xi=-\frac\pi 4$, we can repeat the above estimates, replacing \eqref{green: exp arg est} by 
\begin{equation}
\arg\left(e^{-2(\xi-\lambda)}\right) = -2\imag (\xi-\lambda)\in (\tfrac{\pi}2,\tfrac{7\pi} 4).
\end{equation}
The proof of (iii) is now complete.
\end{proof}

We have similar estimates on the relative sizes of $\xi$ and $\mu e^{-2\xi}$.
\begin{lemma}
There exists a $\mu_0\in (0,1)$ small enough, such that for all $\mu\in(0,\mu_0)$, if we let $\tilde \lambda$ be the unique large and negative number such that $|\tilde\lambda| = \mu e^{-2\tilde\lambda}$, then

if $\tilde\lambda-1<\xi<\tilde\lambda+1$, there is a constant $C>0$ such that
\begin{equation}
\frac1 C |\xi|\le \mu e^{-2\xi}\le C|\xi|;
\end{equation}

if $\xi<\tilde\lambda-1$, $\mu e^{-2\xi}$ dominates, i.e.
\begin{equation}
|\xi|<\frac12 \mu e^{-2\xi};
\end{equation}

if $\xi>\tilde\lambda+1$, $1+|\xi|$ dominates, i.e. 
\begin{equation}
\mu e^{-2\xi}<\frac12 (1+|\xi|);
\end{equation}

if $\xi<2\tilde\lambda$, $\mu e^{-2\xi}$ dominates super polynomially, i.e. for every $N>0$, there is a $C_N>0$ such that
\begin{equation}\label{eq: xi e super poly 1}
|\xi|^N\le C_N \mu e^{-2\xi};
\end{equation}

if $\xi>\frac{\tilde \lambda}{2}$, $\mu e^{-2\xi}$ is super polynomially small compared with $\frac1{1+|\xi|}$, i.e.  for every $N>0$, there is a $C_N>0$ such that
\begin{equation}\label{eq: xi e super poly 2}
\mu e^{-2\xi}\le \frac{C_N}{1+|\xi|^N}.
\end{equation}
\end{lemma}
\begin{proof}
    The proof is straightforward using similar ideas as in the above lemmas, and is omitted.
\end{proof}

We now show part (i) of Proposition \ref{prop: G cont} for $\zeta$ near $0$ and in the upper half plane.
\begin{lemma}\label{lem: green cont case 1}
For any $\epsilon_0>0$ sufficiently small, if $|\real \zeta|<\epsilon_0$, $|\imag \zeta|<\epsilon_0$, part (i) of Proposition \ref{prop: G cont} holds.
\end{lemma}
\begin{proof}
Assume in addition that $\imag \zeta\ge 0$. Let $\lambda=Z^{-1}(\zeta)$. By Lemma \ref{lem: zeta}, we have $0\le \imag \lambda<\frac{\pi}{2}+\frac{\pi}{8} = \frac{5\pi}{8}$, if the $\epsilon_0$ above is chosen sufficiently small. %By \eqref{eq: zeta near 0}, $\lambda \sim \frac12\log \zeta$.

%If $x>0$, we shift the integration contour to the sum of the following three pieces
%\begin{equation}
%\Gamma_{L_1} = \{\real\xi \le \real \lambda-1,\imag \xi =0\},
%\end{equation}
%\begin{equation}
%\Gamma_{L_2} = \left\{\real\xi = \real \lambda-1,0\le \imag \xi \le \frac{3\pi }{4}\right\},
%\end{equation}
%\begin{equation}
%\Gamma_{L_3} = \left\{\real\xi \ge \real \lambda-1,\imag \xi =\frac{3\pi}4\right\}.
%\end{equation}
%If $x<0$, we shift the integration contour to the sum of $\Gamma_{L_1}$ and
%\begin{equation}
%\Gamma_{L_2'} = \left\{\real\xi = \real \lambda-1,-\frac\pi 4\le \imag \xi \le 0\right\},
%\end{equation}
%\begin{equation}
%\Gamma_{L_3'} = \left\{\real\xi \ge \real \lambda-1,\imag \xi =-\frac{\pi}4\right\}.
%\end{equation}
It then follows from \eqref{def: H} and a contour shift argument similar to that in the proof of Proposition \ref{prop: green est} that for $x>0$
\begin{align}\label{Green H1 def}
H(x,\zeta) = \frac1{2\pi}\int_{\R+i\frac{3\pi}4}\frac{e^{ix\xi}}{[\xi-\zeta(1-e^{-2\xi})]}~d\xi,
%&:=H_{11}(x,\zeta)+H_{12}(x,\zeta)+H_{13}(x,\zeta).
\end{align}
and for $x<0$
\begin{align}\label{Green H1 def x neg}
H(x,\zeta)=\frac1{2\pi}\int_{\R-i\frac\pi 4}\frac{e^{ix\xi}}{[\xi-\zeta(1-e^{-2\xi})]}~d\xi.
%&:=H_{11}(x,\zeta)+H_{12}(x,\zeta)+H_{13}(x,\zeta).
\end{align}
%By Lemma \ref{lem: zeta}, the deformation from $\Gamma_L$ to $\R+i\frac{3\pi}4$ passes one pole at $\xi=0$ and another pole at $\xi=\lambda$, while the deformation from $\Gamma_L$ to $\R-i\frac\pi 4$ doesn't pass by any poles. It follows that
%\begin{align}
%G_L(x,\zeta)  = \frac{i}{1-2\zeta}\chi_{\R^+}(x) + \frac{i}{1-2\zeta^*}e^{i\lambda x}\chi_{\R^+}(x)+ H_1(x,\zeta).
%\end{align}
%where 
%\begin{align}
%&~H_1(x,\zeta) -H_1(x,0)\notag\\
% = &~\frac{e^{-\pi x}}{2\pi}\int_\R\frac{e^{ix\xi}(\mu+i\epsilon)(1-e^{-2\xi})}{\left[\xi-\mu(1-e^{-2\xi})+i\left[\pi - \epsilon(1-e^{-2\xi})\right]\right][\xi+i\pi]}~d\xi.\label{eq: H1 cont case 1}
%\end{align}
%Observe that $|\frac{i}{1-2\zeta}-i|\lesssim |\zeta|$, and $\left|\frac{1}{1-2\zeta^*}\right|\lesssim \frac{1}{|\log|\zeta||}$ by \eqref{eq: zeta near 0} and \eqref{eq: zeta near +inf}. Also, $\imag \lambda>0$ implies $e^{i\lambda x}$ is bounded. What remain to be estimated are $H_L^\infty$ and $G_L^2$. 
Denote by $H_{1}(x,\zeta)$ the part of integral in \eqref{Green H1 def} with $\real \xi<\real\lambda-1$, by $H_{2}(x,\zeta)$ the part of integral with $\real\lambda-1<\real \xi<\real\lambda+1$, and by $H_{3}(x,\zeta)$ the part of integral  with $\real \xi>\real\lambda+1$.
By \eqref{eq: p lam-1 comp}, when $\xi\in \R+i\frac{3\pi}4$, $\real \xi<\real\lambda-1$, 
\begin{equation}
\left|\frac{1}{[\xi-\zeta(1-e^{-2\xi})]}\right|\lesssim \frac{1}{|\zeta(1-e^{-2\xi})|}\lesssim \frac{e^{2\xi}}{|\zeta|}. 
\end{equation}
Thus 
\begin{equation}
\|H_{1}(x,\zeta)\|_{L^2\cap L^\infty}\lesssim \frac{e^{2\real\lambda}}{|\zeta|}\lesssim \frac1{\log^+|\zeta|}.
\end{equation}
Here we used $\zeta=\frac{\lambda}{1-e^{-2\lambda}}$ and $\lambda\sim \frac12\log \zeta$. By \eqref{eq: p lam comp}, for $\xi\in \R+i\frac{3\pi}4$, with $\real\lambda-1<\real \xi<\real\lambda+2$, we have
\begin{equation}\label{GL2 est 1}
\left|\frac{ 1}{[\xi-\zeta(1-e^{-2\xi})]}\right|\lesssim \frac1{|\lambda||\xi-\lambda|}\lesssim \frac{1}{|\lambda|}.
\end{equation}
Here we have used $|\xi-\lambda|\ge |\imag(\xi-\lambda)|\ge\frac{3\pi}4-\frac{5\pi}8=\frac\pi8$ for $\xi\in \R+i\frac{3\pi}4$. It follows that
\begin{equation}
|H_{2}(x,\zeta)|\lesssim \int_{\real\lambda-1}^{\real\lambda+1} \frac{e^{-\frac{4\pi}3 x}e^{-ix\eta}}{|\lambda|}~d\eta\lesssim e^{-\frac{4\pi}3 x} \frac{1}{\log^+|\zeta|}.
\end{equation}
Thus 
\begin{equation}
\|H_{2}(x,\zeta)\chi_{\R^+}(x)\|_{L^2\cap L^\infty}\lesssim \frac1{\log^+|\zeta|}.
\end{equation}
Noting that 
\begin{equation}
\int_{\R+i\frac{3\pi}4} \frac{e^{ix\xi}}\xi ~d\xi = 0
\end{equation}
for $x>0$,
we write for $x>0$
\begin{align}
H_{3}(x,\zeta)&=H_{3}(x,\zeta)-\frac1{2\pi}\int_{\R+i\frac{3\pi}4} \frac{e^{ix\xi}}\xi ~d\xi\notag\\
&=\frac1{2\pi}\int_{\real\lambda+1+i\frac{4\pi}3}^{\infty+i\frac{4\pi}3}\frac{e^{ix\xi}\zeta(1-e^{-2\xi})}{[\xi-\zeta(1-e^{-2\xi})]\xi}~d\xi + \frac1{2\pi}\int_{-\infty+i\frac{3\pi}4}^{\real \lambda +1+i\frac{3\pi}4}\frac{e^{ix\xi}}{\xi}~d\xi.\label{eq: H31 H32}%\notag\\
%&:=\chi_{\R^+}(x)[H_{31}(x,\zeta)+H_{32}(x,\zeta)].
\end{align}
We define $H_{31}$ and $H_{32}$ to be the two terms in \eqref{eq: H31 H32}.
%For $\xi\in \Gamma_{L_3}$ and $|\real \xi-\real\lambda|\le 1$, we have \eqref{GL2 est 1} since $|\xi-\lambda|\ge |\imag(\xi-\lambda)|\ge \frac{3\pi }{4}-\frac{5\pi}{8}=\frac\pi 8$. Therefore for $\xi$ in this range
%\begin{align}
%\left|\frac{e^{ix\xi}\zeta(1-e^{-2\xi})}{[\xi-\zeta(1-e^{-2\xi})]\xi}\right|&=e^{-\frac{3\pi}4x}\left|\frac1{\xi-\zeta(1-e^{-2\xi})}-\frac1\xi\right|\notag\\
%&\lesssim e^{-\frac{3\pi}4x}\left(\frac{1}{|\lambda|}+\frac1{|\xi|}\right)\notag\\
%&\lesssim e^{-\frac{3\pi}4x}\frac{1}{\log^+|\zeta|}.
%\end{align}
For $\xi\in \R+i\frac{3\pi}4$ and $\real \xi-\real\lambda> 1$, we have by \eqref{eq: p lam+1 comp}
\begin{equation}
\frac1{|\xi-\zeta(1-e^{-2\xi})|}\lesssim \frac1{|\xi|}.
\end{equation}
Thus for  $\xi$ in this range, we have 
\begin{equation}
\left|\frac{e^{ix\xi}\zeta(1-e^{-2\xi})}{[\xi-\zeta(1-e^{-2\xi})]\xi}\right|\lesssim e^{-\frac{3\pi}4x}\left(\frac{|\zeta|}{|\xi|^2}+\frac{|\zeta e^{-2\xi}|}{|\xi|^2}\right).
\end{equation}
The first term above is easy to control. We estimate the integral of the last term above as
\begin{align}
\int_{\real\lambda+1+i\frac{3\pi}4}^{\infty+i\frac{3\pi}4}\frac{|\zeta e^{-2\xi}|}{|\xi|^2}~d\xi =&~\int_{\real\lambda+1+i\frac{3\pi}4}^{\infty+i\frac{3\pi}4}\frac{|\zeta e^{-2\lambda}e^{-2(\xi-\lambda)}|}{|\xi|^2}~d\xi \label{H131 tail est}\\
\lesssim&~ \int_{\real \lambda+1}^\infty \frac{|\lambda| e^{-2(\xi-\real\lambda)}}{|\xi+i\frac{3\pi}4|^2}~d\xi\notag\\
\lesssim&~ \int_{1}^\infty \frac{|\lambda| e^{-2\xi}}{|\xi+\real\lambda|^2+1}~d\xi\notag\\
\lesssim &~\int_1^{\log |\lambda|} \frac{e^{-2\xi}}{|\lambda|}~d\xi + \int_{\log|\lambda|}^\infty |\lambda| e^{-2\xi}~d\xi\notag\\
\lesssim &~\frac1{|\lambda|}\approx \frac1{\log^+|\zeta|}.
\end{align}
With the above estimates, we may conclude that
\begin{align}
\chi_{\R^+}(x)|H_{31}(x,\zeta)|&\lesssim\chi_{\R^+}(x)e^{-\frac{3\pi}4x}\left(\frac1{\log^+|\zeta|}+|\zeta|\right)\notag\\
&\lesssim\chi_{\R^+}(x)e^{-\frac{3\pi}4x}\frac1{\log^+|\zeta|}.
\end{align}
On the other hand, 
\begin{align}
\|H_{32}(x,\zeta)\|_{L^2}&\lesssim \left(\int_{-\infty}^{\real \lambda+1} \frac1{|\xi|^2+1}~d\xi\right)^{\frac12}\notag\\
&\lesssim \frac1{\sqrt{|\real \lambda|}}\lesssim \frac1{\sqrt{\log^+|\zeta|}}.
\end{align}
In summary, we have
\begin{equation}
\|\chi_{\R^+}(x)H_{3}(x,\zeta)\|_{L^2}\lesssim \frac1{\log^+|\zeta|}.
\end{equation}
Now consider the case when $x<0$. In this case, we can break the integral in \eqref{Green H1 def x neg} into the three pieces $H_{1}$, $H_{2}$, $H_{3}$ in a similar way as before, with the real part of $\xi$ belonging to the three intervals broken up by $\real\lambda\pm 1$. We can then estimate $H_{1}$, $H_{2}$ in the same way as above, while using
%\begin{equation}
%|H_{12}(x,\zeta)|\lesssim \int_{-\frac\pi 4}^0\frac{e^{-x\eta}}{|\lambda|}~d\eta\lesssim \frac{1}{\log^+|\zeta|}\frac{e^{\frac\pi 4 x}-1}{x}.
%\end{equation}
%Thus 
%\begin{equation}
%\|H_{12}(x,\zeta)\chi_{\R^-}(x)\|_{L^2}\lesssim \frac1{\sqrt{\log^+|\zeta|}}.
%\end{equation} 
\begin{equation}
\int_{\R-i\frac\pi 4}\frac{e^{ix\xi}}{\xi}~d\xi=0
\end{equation}
for $x<0$ to rewrite
\begin{align}
H_{3}(x,\zeta)&=H_{3}(x,\zeta)-\frac1{2\pi}\int_{\R-i\frac{\pi}4} \frac{e^{ix\xi}}\xi ~d\xi\notag\\
&=\frac1{2\pi}\int_{\real\lambda+1-i\frac{\pi}4}^{\infty+i\frac\pi 4}\frac{e^{ix\xi}\zeta(1-e^{-2\xi})}{[\xi-\zeta(1-e^{-2\xi})]\xi}~d\xi + \frac1{2\pi}\int_{-\infty-i\frac{\pi}4}^{\real \lambda +1-i\frac{\pi}4}\frac{e^{ix\xi}}{\xi}~d\xi
\end{align}
and estimate the terms as before. For $\zeta$ in the lower half plane, take the complex conjugate of the above argument. The proof is complete.
\end{proof}

%\begin{lemma}
%There is an $\epsilon_0>0$ small, such that if $|\real \zeta|<\epsilon_0$, $-\epsilon_0< \imag \zeta\le 0$, part (i) of Proposition \ref{prop: G cont} holds.
%\end{lemma}
Next we show (ii) of Proposition \ref{prop: G cont} for the continuity estimate for $\zeta$ near zero in the upper half plane. The case for when $\zeta$ is in the lower half plane is completely analogous.

\begin{lemma}\label{lem: G cont case ii}
For any $\epsilon_0>0$ sufficiently small, case (ii) of Proposition \ref{prop: G cont} holds for $\zeta_{1},\zeta_2$ in the specified range.
\end{lemma}
\begin{proof}
We first assume $\zeta_1,\zeta_2$ are in the upper half plane. We claim that if we choose $\epsilon_0$ small enough, $|\real (\lambda_1-\lambda_2)|\lesssim 1$. In fact, by definition of $\lambda$, we have
$$\log|\zeta|=\log|\lambda|-\log|1-e^{-2\lambda}|=\log|\lambda|+2\real \lambda-\log|1-e^{2\lambda}|.$$
It follows that
\begin{equation}\label{lam1-lam2}
2\real(\lambda_1-\lambda_2) = \log\frac{|\lambda_2|}{|\lambda_1|}+\log\frac{|\zeta_1|}{|\zeta_2|}-\log\left|\frac{1-e^{2\lambda_2}}{1-e^{2\lambda_1}}\right|
\end{equation}
Let's assume $|\zeta_2|\le|\zeta_1|$ without loss of generality. When $|\zeta_1-\zeta_2|\le \frac12|\zeta_1|=\max(|\zeta_1|,|\zeta_2|)$, we have $\frac12\le \frac{|\zeta_2|}{|\zeta_1|}\le 1$. By Lemma \ref{lem: zeta}, $\lambda\sim \frac12\log\zeta$ as $\zeta\to 0$. We see easily that \eqref{lam1-lam2} is bounded for $\epsilon_0$ sufficiently small. Denote by $\Delta$ an upper bound of $|\real (\lambda_1-\lambda_2)|$.   
%
%the sum of 
%\begin{equation}
%\Gamma_{L_1} = \{\real\xi \le \min(\real\lambda_1,\real\lambda_2)-1,\imag \xi =0\},
%\end{equation}
%\begin{equation}
%\Gamma_{L_2} = \left\{\real\xi = \min(\real\lambda_1,\real\lambda_2)-1,0\le \imag \xi \le \frac{3\pi }{4}\right\},
%\end{equation}
%\begin{equation}
%\Gamma_{L_3} = \left\{\real\xi \ge \min(\real\lambda_1,\real\lambda_2)-1,\imag \xi =\frac{3\pi}4\right\}.
%\end{equation}
We shift the integration contours as in \eqref{Green H1 def} and \eqref{Green H1 def x neg}
%the sum of $\Gamma_{L_1}$ and
%\begin{equation}
%\Gamma_{L_2'} = \left\{\real\xi = \min(\real\lambda_1,\real\lambda_2)-1,-\frac\pi 4\le \imag \xi \le 0\right\},
%\end{equation}
%\begin{equation}
%\Gamma_{L_3'} = \left\{\real\xi \ge \min(\real\lambda_1,\real\lambda_2)-1,\imag \xi =-\frac{\pi}4\right\}.
%\end{equation}
and write for $x>0$
\begin{align}
H(x,\zeta_1)-H(x,\zeta_2) &= \frac1{2\pi}\int_{\R+i\frac{3\pi}4}\frac{e^{ix\xi}(\zeta_1-\zeta_2)(1-e^{-2\xi})}{[\xi-\zeta_1(1-e^{-2\xi})][\xi-\zeta_2(1-e^{-2\xi})]}~d\xi,%\notag\\
%&:=H_{11}(x,\zeta)+H_{12}(x,\zeta)+H_{13}(x,\zeta).
\end{align}
and for $x<0$
\begin{align}
H(x,\zeta_1)-H(x,\zeta_2) &=\frac1{2\pi}\int_{\R-i\frac{\pi}4}\frac{e^{ix\xi}(\zeta_1-\zeta_2)(1-e^{-2\xi})}{[\xi-\zeta_1(1-e^{-2\xi})][\xi-\zeta_2(1-e^{-2\xi})]}~d\xi.%\notag\\
%&:=H_{11}(x,\zeta)+H_{12}(x,\zeta)+H_{13}(x,\zeta).
\end{align}
As in the proof of Lemma \ref{lem: green cont case 1}, we break the integrals into three pieces $H_1$, $H_2$, $H_3$, where the dividing points are at $\real\xi=\min(\real\lambda_1,\real\lambda_2)-1$ and $\real\xi=\max(\real\lambda_1,\real\lambda_2)+1$. We denote the three pieces of integration domains by $\Gamma_{L_1}$, $\Gamma_{L_2}$ and $\Gamma_{L_3}$ respectively.
Let us now focus on the case $x>0$. For $\xi \in \Gamma_{L_1}$, by \eqref{eq: p lam-1 comp}, 
\begin{align}
\left|\frac{(\zeta_1-\zeta_2)(1-e^{-2\xi})}{[\xi-\zeta_1(1-e^{-2\xi})][\xi-\zeta_2(1-e^{-2\xi})]}\right|&\lesssim \left|\frac{(\zeta_1-\zeta_2)(1-e^{-2\xi})}{\zeta_1(1-e^{-2\xi})\zeta_2(1-e^{-2\xi})}\right|\notag\\
&\lesssim \left|\frac{\zeta_1-\zeta_2}{\zeta_1\zeta_2}e^{2\xi}\right|.
\end{align}
It follows that 
\begin{align}
\left\|\frac{(\zeta_1-\zeta_2)(1-e^{-2\xi})}{[\xi-\zeta_1(1-e^{-2\xi})][\xi-\zeta_2(1-e^{-2\xi})]}\right\|_{(L^1\cap L^2)(\Gamma_{L_1})}&\lesssim \left|\frac{\zeta_1-\zeta_2}{\zeta_1\zeta_2} e^{2\real \lambda_2}\right|\notag\\
&\lesssim \left|\frac{\zeta_1-\zeta_2}{\zeta_1 \log^+|\zeta_1|} \right|,
\end{align}
\begin{equation}
\|\chi_{\mathbb R^+}(x)H_{1}(x,\zeta)\|_{L^2\cap L^\infty}\lesssim \left|\frac{\zeta_1-\zeta_2}{\zeta_1\log^+|\zeta_1|}\right|.
\end{equation}
When $\xi \in \Gamma_{L_2}$, by \eqref{eq: p lam comp} (which continues to hold when $|\real\xi-\real\lambda|<\Delta+1$ for fixed $\Delta$), we have 
\begin{equation}
\left|\frac{e^{ix\xi}(\zeta_1-\zeta_2)(1-e^{-2\xi})}{[\xi-\zeta_1(1-e^{-2\xi})][\xi-\zeta_2(1-e^{-2\xi})]}\right|\lesssim\left| \frac{\zeta_1-\zeta_2}{\zeta_1\lambda_1}\right|e^{-\frac{3\pi}4x}.
\end{equation}
%\begin{align}
%|H_{12}(x,\zeta)|&\lesssim\int_0^{\frac{3\pi}4} e^{-x\eta}\left|\frac{(\zeta_1-\zeta_2)[\zeta_1(1-e^{-2\xi})-\xi+\xi]}{\zeta_1\lambda_1\lambda_2}\right|~d\eta\notag\\
%&\lesssim \left|\frac{\zeta_1-\zeta_2}{\lambda_1\zeta_1}\right|\int_0^{\frac{3\pi}4} e^{-x\eta}~d\eta\notag\\
%&\lesssim \left|\frac{\zeta_1-\zeta_2}{\lambda_1\zeta_1}\right|\frac{1-e^{-\frac{3\pi}4x}}{x}.
%\end{align}
Thus 
\begin{equation}
\|\chi_{\R^+}(x)H_{2}(x,\zeta)\|_{L^2\cap L^\infty}\lesssim \left|\frac{\zeta_1-\zeta_2}{\zeta_1\log^+|\zeta_1|}\right|.
\end{equation}
%For $\xi\in \Gamma_{L_3}$, when $\min(\real \lambda_1,\real \lambda_2)-1\le \real \xi\le \max(\real \lambda_1,\real \lambda_2)+1$, by \eqref{eq: p lam comp} again, we have
%\begin{equation}
%\left|\frac{e^{ix\xi}(\zeta_1-\zeta_2)(1-e^{-2\xi})}{[\xi-\zeta_1(1-e^{-2\xi})][\xi-\zeta_2(1-e^{-2\xi})]}\right|\lesssim\left| \frac{\zeta_1-\zeta_2}{\zeta_1\lambda_1}\right|e^{-\frac{3\pi}4x}.
%\end{equation}
For $\xi\in \Gamma_{L_3}$, we have by \eqref{eq: p lam+1 comp}
\begin{equation}
\left|\frac{e^{ix\xi}(\zeta_1-\zeta_2)(1-e^{-2\xi})}{[\xi-\zeta_1(1-e^{-2\xi})][\xi-\zeta_2(1-e^{-2\xi})]}\right|\lesssim\left| \frac{\zeta_1-\zeta_2}{\zeta_1}\right|\frac{|\zeta_1(1-e^{-2\xi})|}{|\xi|^2}e^{-\frac{3\pi}4x}.
\end{equation}
We can estimate in the same way as we did for \eqref{H131 tail est} to get
\begin{equation}
|\chi_{\R^+}(x)H_{3}(x,\zeta)|\lesssim \left| \frac{\zeta_1-\zeta_2}{\zeta_1\lambda_1}\right|e^{-\frac{3\pi}4x}.
\end{equation}
Thus 
\begin{equation}
\|\chi_{\R^+}(x)H_{3}(x,\zeta)\|_{L^2\cap L^\infty}\lesssim \left| \frac{\zeta_1-\zeta_2}{\zeta_1\log^+|\zeta_1|}\right|.
\end{equation}
The proof for the case $x<0$ is similar.

If $\zeta_1,\zeta_2$ are both in the lower half plane, use the complex conjugate of the above argument. If only one of $\zeta_1$, $\zeta_2$ is in the upper half plane, since $\dist_V(\zeta_1,\zeta_2)<\frac12|\zeta_1|$, the line segment connecting $\zeta_1$, $\zeta_2$ must intersect $\mathbb R^+$ at some $\zeta_3$, with $|\zeta_1|\approx|\zeta_2|\approx|\zeta_3|$. We just use \eqref{eq: G cont 2} to estimate the norms of $H(x,\zeta_1)-H(x,\zeta_3)$ and $H(x,\zeta_2)-H(x,\zeta_3)$, and combine the estimates.
\end{proof}

Next we show (iii) of Proposition \ref{prop: G cont} for $\zeta$ near a large positive number.

\begin{lemma}
For any sufficiently large $\mu^*>0$ and small $\epsilon_0>0$, if $\real\zeta_1,\real\zeta_2>\mu^*$, $|\imag \zeta_1|,|\imag \zeta_2|\le \epsilon_0$, Case (iii) of Proposition \ref{prop: G cont} holds.
\end{lemma}
\begin{proof}
First assume $\imag\zeta_1,\imag\zeta_2\ge 0$. We again shift the contour of integration by an argument similar to the proof of Proposition \ref{prop: green est}. This time if $x>0$, we shift the contour to $\mathbb R+i\pi$ and write
\begin{equation}\label{eq: H cont x>0 case 3} 
H(x,\zeta) = \frac{e^{-\pi x}}{2\pi}\int_\R\frac{e^{ix\xi}}{\xi-\zeta(1-e^{-2\xi})+i\pi}~d\xi.
\end{equation}
Let $\zeta_{j}=\mu_j+i\epsilon_j$, $j=1,2$, where $\mu_j>\mu^*$, $0\le \epsilon_j\le \epsilon_0$. We have
\begin{align}
&~H(x,\zeta_1) -H(x,\zeta_2) \notag\\
= &~\frac{e^{-\pi x}(\zeta_1-\zeta_2)}{2\pi}\int_\R\frac{e^{ix\xi}(1-e^{-2\xi})}{[\xi-\zeta_1(1-e^{-2\xi})+i\pi][\xi-\zeta_2(1-e^{-2\xi})+i\pi]}~d\xi. \label{eq: H cont case 3 1}
\end{align}
We only need to show that the integral in \eqref{eq: H cont case 3 1} is bounded by a constant. Let $\lambda_1$ be the large number given in (ii) of Lemma \ref{lem: rel size est}. Choose $\mu^*\ge 2\lambda_1$.  Writing the integrand in \eqref{eq: H cont case 3 1} as
\begin{equation}
    \frac{e^{ix\xi}}{(1-e^{-2\xi})[Z(\xi)-\zeta_1-\frac{i\pi}{1-e^{-2\xi}}][Z(\xi)-\zeta_2-\frac{i\pi}{1-e^{-2\xi}}]},
\end{equation} we can bound the integral from $-\infty$ to $\lambda_1+1$ by a constant multiple of 
%\begin{equation}
%\int_{-\infty}^{\lambda_1+1}\frac{e^{ix\xi}(1-e^{-2\xi})}{[\xi-\zeta_1(1-e^{-2\xi})+i\pi][\xi-\zeta_2(1-e^{-2\xi})+i\pi]}~d\xi
%\end{equation}
\begin{equation}
\int_{-\infty}^{-1} \frac{e^{2\xi}}{|Z(\xi) - \mu_1||Z(\xi)-\mu_2|}~d\xi + \int_{-1}^{\lambda_1+1}\frac1{\left|\xi+i\pi\right|^2}~d\xi\lesssim 1.
\end{equation}
Without loss of generality, assume $|\zeta_1-\zeta_2|\le 1$. By (ii) of Lemma \ref{lem: rel size est}, the part of the integral in \eqref{eq: H cont case 3 1} from $\lambda_1+1$ to $\infty$ is bounded by a constant multiple of  
\begin{equation}
\int_{|\xi-\bar\mu|>2}\frac1{|\xi-\bar\mu|^2}~d\xi + \int_{|\xi-\bar\mu|<2}\frac1{\left|i\pi \right|^2}~d\xi\lesssim 1.
\end{equation}
Here $\bar\mu = \frac{\mu_1+\mu_2}{2}$.

If $x<0$, we shift the contour of integration to $\R-i\frac\pi 2$ and write
\begin{equation}\label{eq: H cont x<0 case 3}
H(x,\zeta) = \frac{e^{\frac{\pi x}2}}{2\pi}\int_\R\frac{e^{ix\xi}}{\xi-\zeta(1+e^{-2\xi})-i\frac{\pi}{2}}~d\xi.
\end{equation}
\eqref{eq: H cont x<0 case 3} can be treated in the same way \eqref{eq: H cont x>0 case 3} was, as we note that $\left|\frac{\xi}{1+e^{-2\xi}}-\mu\right|\gtrsim \mu$ for $\xi<0$, and Lemma \ref{lem: rel size est 1} for $\xi$ large and positive.

Finally if $\zeta_1$, $\zeta_2$ are not both in the upper half plane, we can complete the proof in the same way as that of Lemma \ref{lem: G cont case ii}.
\end{proof}

%\begin{lemma}
%There are large $\mu^*>0$ and small $\epsilon_0>0$ such that if $\real\zeta>\mu^*$, $-\epsilon_0<\imag \zeta\le 0$, (ii) of Proposition \ref{prop: G cont} holds.
%\end{lemma}
 We can now fill the middle gap left open by the previous lemmas.
\begin{lemma}
Let $0<\epsilon_0<\mu^*$ be given, where $\epsilon_0$ is sufficiently small and $\mu^*$ is sufficiently large. If $\epsilon_0\le\real \zeta_1,\real\zeta_2\le \mu^*$, $|\imag \zeta_1|,|\imag \zeta_2|\le\epsilon_0$, case (iii) of Proposition \ref{prop: G cont} holds.
\end{lemma}
\begin{proof}
First assume $\imag \zeta_1,\imag\zeta_2\ge 0$. If $x>0$, we shift the contour of integration to $\R+i\pi$ and again have \eqref{eq: H cont x>0 case 3} and \eqref{eq: H cont case 3 1}. Let $\lambda_0=Z^{-1}(\epsilon_0)$, $\lambda^*=Z^{-1}(\mu^*)$. Let $\zeta_{j} = \mu_j+i\epsilon_j$ with $\epsilon_0\le\mu_j\le \mu^*$, $0\le \epsilon_j\le \epsilon_0$. We split the integral into three pieces:
\begin{align}
&~\left|\int_{-\infty}^{\lambda_0-1}\frac{e^{ix\xi}(1-e^{-2\xi})}{[\xi-\zeta_1(1-e^{-2\xi})+i\pi][\xi-\zeta_2(1-e^{-2\xi})+i\pi]}~d\xi\right|\notag\\
\lesssim &~\int_{-\infty}^{\lambda_0-1}\frac{e^{2\xi}}{|Z(\xi)-\mu_1||Z(\xi)-\mu_2|}~d\xi\lesssim 1.
\end{align}
\begin{align}
&~\left|\int_{\lambda_0-1}^{\lambda^*+1}\frac{e^{ix\xi}(1-e^{-2\xi})}{[\xi-\zeta_1(1-e^{-2\xi})+i\pi][\xi-\zeta_2(1-e^{-2\xi})+i\pi]}~d\xi\right|\notag\\
\lesssim &~\int_{\lambda_0-1}^{\lambda^*+1} \frac{1}{|i[\pi - \epsilon_1(1-e^{-2\xi})]||i[\pi - \epsilon_2(1-e^{-2\xi})]|}~d\xi\lesssim 1.
\end{align}
In fact, $\pi -\epsilon_j(1-e^{-2\xi})>\pi$ when $\xi<0$ and $\pi-\epsilon_j(1-e^{-2\xi})>\frac\pi 2$ when $\xi>0$
if $\epsilon_0$ is sufficiently small. The last piece is estimated by using \eqref{eq: p pos lam} as
\begin{align}
&~\left|\int_{\lambda^*+1}^\infty\frac{e^{ix\xi}(1-e^{-2\xi})}{[\xi-\zeta_1(1-e^{-2\xi})+i\pi][\xi-\zeta_2(1-e^{-2\xi})+i\pi]}~d\xi\right|\notag\\
\lesssim &~\int_{\lambda^*+1}^\infty \frac{1}{|\xi-\lambda_1||\xi-\lambda_2|}~d\xi\lesssim 1.
\end{align}
Here $\lambda_j=Z^{-1}(\mu_j)$.
The other cases, when $x<0$ or when $\imag\zeta_j\le 0$ can be treated similarly. The details are omitted.
\end{proof}

We close the continuity estimates by proving
\begin{lemma}
Case (iv) of Proposition \ref{prop: G cont} holds.
\end{lemma} 
\begin{proof}
We assume $\imag \zeta_j\ge 0$. The case when both $\zeta_j$ are in the lower half plane can be proven similarly. If $x>0$, we estimate using $G_R(x,\zeta)$ and shift the contour of integration to $\Gamma_\epsilon=\R\setminus (-\epsilon,\epsilon)\cup C_\epsilon$. Here $C_\epsilon$ is the upper semicircle of radius $\epsilon$ centered at the origin, and $\epsilon$ is chosen so small that $\epsilon>\frac{\epsilon_0}{8}$ and $\imag [Z(C_\epsilon)]<\frac{\epsilon_0}2$. We have
\begin{equation}
G_R(x,\zeta) = \frac1{2\pi}\int_{\Gamma_\epsilon}\frac{e^{ix\xi}}{\xi-\zeta(1-e^{-2\xi})}~d\xi.
\end{equation}
We just need to show \eqref{eq: G cont 7} with $G_L$ replaced by $G_R$. As before, we just need a uniform bound on
\begin{align}
&~\int_{\Gamma_\epsilon}\frac{e^{ix\xi}(1-e^{-2\xi})}{[\xi-\zeta_1(1-e^{-2\xi})][\xi-\zeta_1(1-e^{-2\xi})]}~d\xi\notag\\
 =&~ \int_{\Gamma_\epsilon}\frac{e^{ix\xi}}{(1-e^{-2\xi})[Z(\xi)-\zeta_1][Z(\xi)-\zeta_2]}~d\xi. \label{eq: H cont x>0 case 6}
\end{align}
\eqref{eq: H cont x>0 case 6} is bounded by a multiple of 
\begin{align}
&~\int_{-\infty}^{-\epsilon}\frac{e^{2\xi}}{\epsilon_0^3}~d\xi+\int_{C_\epsilon}\frac1{(\epsilon_0/2)^3}~|d\xi|+ \frac{\lambda_1}{\epsilon_0^3}+ \int_{\lambda_1}^\infty \frac1{|\xi-\zeta_1||\xi-\zeta_2|}~d\xi\notag\\
\lesssim &~1+\int_{\lambda_1}^\infty \frac1{|\xi-\zeta_1||\xi-\zeta_2|}~d\xi.
\end{align}
Here $\lambda_1$ is chosen large enough so that $|Z(\xi)-\xi|<\frac{\epsilon_0}{4}$ for $\xi>\lambda_1$. By \eqref{eq: zeta near +inf}, this can be done. Recall that $|\zeta_1-\zeta_2|\le \epsilon_0<1$. By the range of $\zeta_j$, we have for $\xi>\lambda_1$, $|\xi-\zeta_j|\gtrsim |\xi-\real\zeta_j +i\epsilon_0|\approx |\xi-\bar\mu + i\epsilon_0|$, where $\bar\mu=\frac12\real(\zeta_1+\zeta_2)$. Thus
\begin{equation}
\int_{\lambda_1}^\infty \frac1{|\xi-\zeta_1||\xi-\zeta_2|}~d\xi\lesssim \int_\R\frac1{|\xi-\bar\mu+i\epsilon_0|^2}~d\xi\lesssim 1.
\end{equation}
If $x<0$, we use a lower semicircle in the contour, and the proof is similar.
\end{proof}

\subsection{\texorpdfstring{Differentiability in $\zeta$ on $\mathbb R^+\pm i0$}{Differentiability in zeta}}
Using complex analysis, we may avoid extra estimates of differentiability of Green's functions in $\zeta$  when $\zeta\in \mathbb C\setminus \R^+$. However, such estimates are needed when $\zeta\in\mathbb R^+\pm i0$. For simplicity, we denote the difference quotient of a function $f(x)$ by 
\begin{equation}
    \Delta^h_x f(x)=\frac{f(x+h)-f(x)}{h}.
\end{equation}

\begin{lemma}\label{lem: Del G - DG on R+}
    Let $k\in \mathbb R^+$, $k\ne \frac12$. Then $\partial_k G_L(\cdot,k+i0)\in L^\infty_{-1,+}$, and as $h\to 0$, $\Delta^h_k G_L(\cdot,k+i0)-\partial_k G_L(\cdot,k+i0)$ is bounded in  $L^\infty_{-1,+}$ and converges to 0 uniformly on compact sets.
    Similarly, $\partial_k G_R(\cdot,k-i0)\in L^\infty_{-1,-}$, and as $h\to 0$, $\Delta^h_k G_R(\cdot,k-i0)-\partial_k G_R(\cdot,k-i0)$ is bounded in $L^\infty_{-1,-}$ and converges to 0 uniformly on compact sets.
\end{lemma}
\begin{comment}
    Just like in Cor \ref{cor: G cont 1}, the weight just needs to be any s<-1. I put -2 only because the assumption on u is enough to balance out that weight anyway.
\end{comment}
\begin{proof}
    We only show the proof for $G_L$. Let $x>0$. Formally differentiating \eqref{eq: H contour R+i pi}, we get 
    \begin{equation}\label{eq: Dk H}
        \partial_k H(x,k)=\frac{e^{-\pi x}}{2\pi}\int_\R \frac{e^{ix\xi}(1-e^{-2\xi})}{(\xi+i\pi -k(1-e^{-2\xi}))^2}~d\xi,
    \end{equation}
    \begin{equation}
        \partial_k^2 H(x,k)=\frac{e^{-\pi x}}{\pi}\int_\R \frac{e^{ix\xi}(1-e^{-2\xi})^2}{(\xi+i\pi -k(1-e^{-2\xi}))^3}~d\xi.
    \end{equation}
    By using the mean value theorem on the difference quotient, we see that $H(x,k)$ is differentiable in $k$, and 
    \begin{equation}\label{eq: Del H - DH}
        \|\Delta_k^h H(\cdot,k)-\partial_k H(\cdot,k)\|_{L^\infty(0,\infty)}\to 0
    \end{equation}
    as $h\to 0$. If $x<0$, we deform the contour in the definition of $H$ to $\R-i\pi$ and prove similarly. We now consider 
    \begin{align}
        I(x,k+i0)&=G_L(x,k+i0)-H(x,k)\notag\\
        &=\left(\tfrac{i}{1-2k}+\tfrac{i}{1-2k^*}e^{i\lambda x}\right)\chi_{\R^+}(x),\label{def: I}
    \end{align}
    where $\lambda=Z^{-1}(k)$.
    It's clear that $\partial_k I(x,k+i0)\in L^\infty_{1,+}$ for $k\in \R^+$, $k\ne \frac12$. To estimate $\Delta_k^h I(x,k+i0)$ in $L^\infty_{-1,+}$, we only need to estimate in $L^\infty$:
    \begin{align}
        %&\frac{\chi_{\R^+}(x)}{x(1+x)}\left|\frac1h\left(e^{i Z^{-1}(k+h) x}-e^{iZ^{-1}(k)x}\right)-ix (Z^{-1})'(k)e^{i\lambda x}\right|\notag\\
        %= &\frac{\chi_{\R^+}(x)}{1+x}\left|\frac{e^{ix\tilde h}-1-ix\tilde h +ix[\tilde h-(Z^{-1})'(k)h]}{xh}\right|,
        &\qquad \frac{\chi_{\R^+}(x)}{x}\left|\frac1h\left(e^{i Z^{-1}(k+h) x}-e^{iZ^{-1}(k)x}\right)\right|\notag\\
        &= \chi_{\R^+}(x)\left|\frac{e^{ix\tilde h}-1 }{xh}\right|\notag\\
        &\lesssim_k 1.
    \end{align}
    where we denote $\tilde h = Z^{-1}(k+h)-Z^{-1}(k)$. Finally, uniform convergence on compact sets of $\Delta_k^h G_L(\cdot,k+i0)-\partial_k G_L(\cdot,k+i0)$ follows from \eqref{eq: Del H - DH} and the fact that the terms in the parentheses of \eqref{def: I} are analytic in $(x,k)$.
    
    %\begin{equation}
     %   \frac{\chi_{\R^+}(x)}{1+x}\left|\frac{e^{ix\tilde h}-1-ix\tilde h }{xh}\right|=\frac{\chi_{\R^+}(x)x|h|}{1+x}\left|\frac{e^{ix\tilde h}-1-ix\tilde h }{(xh)^2}\right|\lesssim_k |h|,
    %\end{equation}
    %and
   % \begin{equation}
    %    |\tilde h-(Z^{-1})'(k)h|=|Z^{-1}(k+h)-Z^{-1}(k)-(Z^{-1})'(k)h|\lesssim_k h^2.
    %\end{equation}
    %It follows that $\|\Delta_k^h I(x,k+i0)-\partial_k I(x,k+i0)\|_{L^\infty_{-2,+}}\lesssim_k |h|$ as $h\to 0$.
\end{proof}

We provide estimates on $\partial_k G_L$, $\partial_k G_R$ that are uniform in $k$ as $k\searrow 0$ and $k\nearrow \infty$.
\begin{lemma}\label{lem: Dk G uniform}
    There exists $C>0$ such that 
    \begin{equation}\label{eq: Dk G k small}
        \|\partial_k G_L(\cdot,k+i0)\|_{L^\infty_{-1,+}}\le \tfrac{C}{k},~\|\partial_k G_R(\cdot,k+i0)\|_{L^\infty_{-1,-}}\le \tfrac{C}{k}
    \end{equation}
    for $0<k<\frac14$, and 
    \begin{equation}\label{eq: Dk G k large}
        \|\partial_k G_L(\cdot,k+i0)\|_{L^\infty_{-1,+}}\le C, ~\|\partial_k G_R(\cdot,k+i0)\|_{L^\infty_{-1,-}}\le C
    \end{equation}
    for $k>1$.
\end{lemma}
\begin{proof}
    We only show the proof for $G_L$. Given $\mu_0<\mu_1$ in $\mathbb R^+$, there exists $K>0$ such that for $\mu_0\le k\le \mu_1$, if $\xi<-K$, 
    \begin{equation}
        2|\xi|<k|1-e^{-2\xi}|;
    \end{equation}
    if $\xi>K$,
    \begin{equation}
        2k|1-e^{-2\xi}|\le |\xi|.
    \end{equation}
    It follows easily from the above estimates and \eqref{eq: Dk H} that there exists a $C=C(\mu_0,\mu_1)$ such that
    \begin{equation}
        \|\partial_k H(\cdot,k)\|_{L^\infty}\le C
    \end{equation}
    for $\mu_0\le k\le \mu_1$.

    Now consider $0<k<\mu_0$, where $\mu_0$ is given by Lemma \ref{lem: rel size est} (i). Let $\lambda= Z^{-1}(k)$. Lemma \ref{lem: rel size est} (i) suggests we break the integral in \eqref{eq: Dk H} into three parts.
    \begin{align}
        &~\int_{\xi<\lambda-1}\frac{|1-e^{-2\xi}|}{|\xi-k(1-e^{-2\xi})+i\pi|^2}~d\xi\notag\\
        \lesssim &~\int_{\xi<\lambda-1}\frac{|1-e^{-2\xi}|}{k^2|1-e^{-2\xi}|^2}~d\xi\notag\\
        \lesssim&~ \frac{e^{2\lambda}}{k^2}\approx \frac{1}{k|\lambda|}\approx\frac{1}{k|\log k|}.
    \end{align}
    \begin{align}
        &~\int_{\lambda-1<\xi<\lambda+1}\frac{|1-e^{-2\xi}|}{|\xi-k(1-e^{-2\xi})+i\pi|^2}~d\xi\notag\\
        \lesssim &~\int_{\lambda-1<\xi<\lambda+1}\frac{e^{-2\lambda}}{1+|\lambda(\xi-\lambda)|^2}~d\xi\notag\\
        =&~\frac{e^{-2\lambda}}{|\lambda|}\int_{-|\lambda|}^{|\lambda|}\frac{1}{1+\eta^2}~d\eta\notag\\
        \lesssim &~\frac{e^{-2\lambda}}{|\lambda|}\approx \frac{1}{k}.
    \end{align}
    \begin{align}
        &~\int_{\xi>\lambda+1}\frac{|1-e^{-2\xi}|}{|\xi-k(1-e^{-2\xi})+i\pi|^2}~d\xi\notag\\
        \lesssim &~\int_{\xi>\lambda+1}\frac{|1-e^{-2\xi}|}{1+\xi^2}~d\xi\notag\\
        \lesssim &\int_{\lambda}^1 \frac{e^{-2\xi}}{\xi^2}~d\xi+1\lesssim \frac{e^{-2\lambda}}{|\lambda|^2}\approx\frac{1}{k|\log k|}.
    \end{align}
    In summary, 
    $ \|\partial_k H(\cdot,k)\|_{L^\infty}\lesssim \frac{1}{k}$
    for $0<k<\mu_0$.

    Now let's consider $k>\mu_1$, where $\mu_1$ is given as in Lemma \ref{lem: rel size est} (ii). This time we split the integral in \eqref{eq: Dk H} into the following pieces and estimate:
        \begin{align}
            &~\int_{-\infty}^{-1}\frac{|1-e^{-2\xi}|}{|\xi-k(1-e^{-2\xi})+i\pi|^2}~d\xi\notag\\
            \lesssim &~\int_{-\infty}^{-1}\frac{e^{2\xi}}{k^2}~d\xi\lesssim \frac1{k^2}.
        \end{align}
    Let $\lambda_1=Z^{-1}(\mu_1)$, then
    \begin{align}
        &~\int_{-1}^{\lambda_1}\frac{|1-e^{-2\xi}|}{|\xi-k(1-e^{-2\xi})+i\pi|^2}~d\xi\notag\\
            \lesssim &~\int_{-1}^{\lambda_1}\frac{|1-e^{-2\xi}|}{\pi^2}~d\xi\lesssim_{\lambda_1}1.
    \end{align}
    By Lemma \ref{lem: rel size est} (ii)
    \begin{align}
        &~\int_{\lambda_1}^\infty \frac{|1-e^{-2\xi}|}{|\xi-k(1-e^{-2\xi})+i\pi|^2}~d\xi\notag\\
            \lesssim &~ \int_{\lambda_1}^\infty\frac{1}{1+|\xi-\lambda|^2}~d\xi \lesssim 1.
    \end{align} 
    In summary $\|\partial_k H(\cdot,k)\|_{L^\infty}\lesssim 1$ for $k>\mu_1$. 

    \eqref{eq: Dk G k small} and \eqref{eq: Dk G k large} now follow by taking a $k$-derivative explicitly in \eqref{def: I}. Some key estimates are 
    \begin{equation}\label{eq: k* bounds k small}
        |k^*|\approx |\log k|, ~|\tfrac{dk^*}{dk}|\approx \tfrac{1}{k},~|\tfrac{d\lambda}{dk}|\approx \tfrac{1}{k} 
    \end{equation}
    for $k<\frac14$, and 
    \begin{equation}\label{eq: k* bounds k large}
        |k^*|\approx ke^{-2k}, ~|\tfrac{dk^*}{dk}|\approx ke^{-2k},~|\tfrac{d\lambda}{dk}|\approx 1 
    \end{equation}
    for $k>1$.
\end{proof}

The estimates at $k=\frac12$ have one extra order of growth in $x$.

\begin{lemma}\label{lem: Del G - DG at 1/2}
    $\partial_k G_L(\cdot,\frac12+i0)\in L^\infty_{-2,+}$. As $h\to 0$, $\Delta_k^h G_L(\cdot,\frac12+i0)-\partial_k G_L(\cdot,\frac12+i0)$ is bounded in $L^\infty_{-2,+}$ and converges to 0 uniformly on compact sets. 

    Similarly, $\partial_k G_R(\cdot,\frac12-i0)\in L^\infty_{-2,-}$. As $h\to 0$, $\Delta_k^h G_R(\cdot,\frac12-i0)-\partial_k G_R(\cdot,\frac12-i0)$ is bounded in $L^\infty_{-2,-}$ and converges to 0 uniformly on compact sets. 
\end{lemma}
\begin{proof}
    The estimates on $\partial_k H$ continues to work when $k=\frac12$. We only need to look at $I$ as defined by \eqref{def: I}. Taking the limit as $k\to \frac12$ or $\lambda\to 0$, we have 
    \begin{equation}
        I(x,\tfrac12+i0)=\left(-x+i\tfrac23\right)\chi_{\R^+}(x),
    \end{equation}
    and
    \begin{equation}
        \partial_k I(x,\tfrac12+i0)=\left(-\tfrac23 x -ix^2\right)\chi_{\R^+}(x).
    \end{equation}
    We obviously have $\partial_k I(x,\frac12+i0)\in L^\infty_{-2,+}$. To estimate $\Delta_k^h I(x,\tfrac12+i0)$, we consider as $k\to \frac12$ or $\lambda\to 0$
    \begin{align}
       & \frac{\chi_{\R^+}(x)}{|\lambda|(1+x)^2}\left|\frac{i}{1-2k}+\frac{i}{1-2k^*}e^{i\lambda x}+x-i\frac23\right|\notag\\
       \lesssim &\frac{\chi_{\R^+}(x)}{\lambda^2(1+x)^2}\left|e^{i\lambda x}-1+\frac{1-2k^*}{i}x+\frac{1-2k^*}{1-2k}+1-\frac{2(1-2k^*)}{3}\right|\notag\\
       \lesssim & \frac{\chi_{\R^+}(x)}{\lambda^2(1+x)^2}\left(\left|e^{i\lambda x}-1-i\lambda x\right|+\lambda^2 x+\lambda^2\right)\notag\\
       \lesssim &1.
    \end{align}
    Uniform convergence on compact sets of $\Delta_k^h I(x,\tfrac12+i0)-\partial_k I(x,\tfrac12+i0)$ is obvious, as the function in front of $\chi_{\R^+}$ in the definition of $I$ is analytic in $(x,k)$.
\end{proof}

\subsection{\texorpdfstring{Decay estimates as $\zeta\to\infty$}{Decay estimates as zeta goes to infinity}}

The Green's functions also have certain decay properties as $\zeta$ tends to infinity. We start by showing the following decay estimates for the Fourier symbol of $G_*$.

Recall that the principal argument $\arg z\in (-\pi,\pi]$.
\begin{lemma}
For any $\alpha>0$ sufficiently small, and $R>0$ sufficiently large, there exists a $C>0$  %be a sufficiently small angle. Let $\Omega_{<\alpha}$ be the domain outside a large disc around the origin, and that the angle with $\R^+$ is less than $\alpha$. Let $\Omega_{>\alpha}$ be the domain outside the large disc and not in $\Omega_{<\alpha}$. There exists a $C>0$ such that 
such that
\begin{enumerate}[(i)]
\item If $|\arg\zeta|\ge \alpha$, $|\zeta|>R$,
\begin{equation}\label{eq: G sym decay 1}
\left|\frac1{\xi-\zeta(1-e^{-2\xi})} - \frac{1}{1-2\zeta}\frac{1}{\xi}\right|\le\frac{C}{|\zeta|+|\xi|}, 
\end{equation}
\begin{equation}\label{eq: G sym decay 2}
\left| \frac1{\xi-\zeta(1-e^{-2\xi})} - \frac{1}{1-2\zeta}\frac{2}{1-e^{-2\xi}}\right|\le \frac{C\langle\xi\rangle}{|\zeta|(|\zeta|+|\xi|)}.
\end{equation}

\item If $|\arg\zeta|<\alpha$, $|\zeta|>R$, let $\lambda=Z^{-1}(\zeta)$ (see Definition \ref{def: Z inv}). Then
\begin{equation}\label{eq: G sym decay 3}
\left|\frac1{\xi-\zeta(1-e^{-2\xi})} -\frac1{1-2\zeta^*}\frac1{\xi-\lambda}- \frac{1}{1-2\zeta}\frac{1}{\xi}\right|\le\frac{C}{|\zeta|+|\xi|},
\end{equation}
\begin{equation}\label{eq: G sym decay 4}
\left|\frac1{\xi-\zeta(1-e^{-2\xi})} -\frac1{1-2\zeta^*}\frac1{\xi-\lambda}-\frac1{1-2\zeta}\frac{2}{e^{2\xi}-1}\right|\le\frac{C\langle \xi\rangle}{|\zeta|(|\zeta|+|\xi|)}.
\end{equation}
\end{enumerate}
\end{lemma}
\begin{proof}
First consider case (i). We can write the left hand side of \eqref{eq: G sym decay 1} in several equivalent ways:
\begin{align}
&~\left|\frac1{\xi-\zeta(1-e^{-2\xi})} - \frac{1}{1-2\zeta}\frac{1}{\xi}\right|\notag\\
=&~\left|\frac{1}{[Z(\xi)-\zeta](1-e^{-2\xi})}-\frac1{1-2\zeta}\frac1{\xi}\right| \label{eq: G sym decay 11}\\
=&~\left|\frac{\zeta}{(1-2\zeta)[Z(\xi)-\zeta]}\frac{1-2Z(\xi)}{\xi}\right|. \label{eq: G sym decay 12}
\end{align}
If $\xi<-1$, the two terms in \eqref{eq: G sym decay 11} are bounded by 
$$
\frac{C}{|\zeta|}e^{2\xi}\lesssim \frac{1}{|\zeta|+|\xi|}
$$
and
$$
\frac{C}{|\zeta||\xi|}\lesssim \frac{1}{|\zeta|+|\xi|}
$$
respectively. If $-1\le\xi\le1$, \eqref{eq: G sym decay 12} is bounded by $\frac{C}{|\zeta|}$. If $\xi>1$, the second term in \eqref{eq: G sym decay 11} is again bounded by $\frac{C}{|\zeta||\xi|}$, while the first term is bounded by $\frac{C}{|\xi-\zeta|}$. Now since $\zeta = |\zeta|e^{i\theta}$, with $\alpha<\theta<2\pi-\alpha$, and $\xi>1$, there exists a $C$ depending on $\alpha$ such that $|\xi-\zeta|\ge C(|\xi|+|\zeta|)$. The bound \eqref{eq: G sym decay 1} is now established.

We now write the left hand side of \eqref{eq: G sym decay 2} as 
\begin{equation}\label{eq: G sym decay 21}
\left|\frac{1-2Z(\xi)}{(1-2\zeta)[Z(\xi)-\zeta](1-e^{-2\xi})}\right|.
\end{equation}
If $\xi<-1$, \eqref{eq: G sym decay 21} is bounded by 
$$
\frac{Ce^{2\xi}}{|\zeta|^2 }\lesssim \frac{1\langle \xi\rangle}{|\zeta|(|\zeta|+|\xi|)}.
$$
If $-1\le\xi\le 1$, $\frac{1-2Z(\xi)}{1-e^{-2\xi}}$ is continuous. Thus \eqref{eq: G sym decay 21} is bounded by $\frac{C}{|\zeta|^2}$. If $\xi>1$, \eqref{eq: G sym decay 21} is bounded by 
$$
\frac{C|\xi|}{|\zeta||\xi-\zeta|}\lesssim\frac{|\xi|}{|\zeta|(|\zeta|+|\xi|)}.
$$
The bound \eqref{eq: G sym decay 2} is now established.

Now we consider case (ii). If $\xi<-1$, the terms in \eqref{eq: G sym decay 3} can be estimated in a similar way as  above for \eqref{eq: G sym decay 1}, except for the additional term $\frac1{1-2\zeta^*}\frac1{\xi-\lambda}$, which is bounded by 
$$
\frac{C}{|\xi-\lambda|}\lesssim \frac{1}{|\zeta|+|\xi|}.
$$
Note that $\lambda$ is large and in a sector about the positive real line, while $\xi$ is negative in the above estimate. If $-1\le\xi\le 1$, the estimate is similar to that of \eqref{eq: G sym decay 1}. The additional term  $\frac1{1-2\zeta^*}\frac1{\xi-\lambda}$ is bounded by $\frac{C}{|\zeta|}$. If 
$\xi>1$,
$$
\left|\frac{1}{1-2\zeta}\frac1\xi\right|\lesssim \frac{1}{|\zeta||\xi|}\lesssim \frac{1}{|\zeta|+|\xi|}.
$$
Thus we only focus on the remaining terms in \eqref{eq: G sym decay 3}.
If $1<\xi<\frac{|\lambda|}{2}$,  both $\frac{1}{\xi-\zeta(1-e^{-2\xi})}$ and $\frac1{1-2\zeta^*}\frac1{\xi-\lambda}$ are bounded by $\frac{C}{|\zeta|}\approx \frac{C}{|\zeta|+|\xi|}$. If $\frac{|\lambda|}{2}<\xi<2|\lambda|$, we write
\begin{align}
&~\frac{1}{\xi-\zeta(1-e^{-2\xi})}-\frac1{1-2\zeta^*}\frac1{\xi-\lambda}\notag\\
=&~-\frac{[\xi-\zeta(1-e^{-2\xi})]-(1-2\zeta^*)(\xi-\lambda)}{(1-2\zeta^*)(1-e^{-2\xi})(\xi-\lambda)[Z(\xi)-\zeta]}\label{eq: G sym decay 22}
\end{align}
If we let $f(\xi) = \xi-\zeta(1-e^{-2\xi})$, the numerator of \eqref{eq: G sym decay 22} is of the form \begin{equation}\label{eq: G sym decay 23}
f(\xi)-f(\lambda)-f'(\lambda)(\xi-\lambda).
\end{equation}
By Taylor's theorem, \eqref{eq: G sym decay 23} is bounded by 
\begin{equation}
C\sup_{\eta = \theta\xi+(1-\theta)\lambda}|f''(\eta)||\xi-\lambda|^2\lesssim |\zeta|e^{-|\lambda|} |\xi-\lambda|^2.
\end{equation}
Thus \eqref{eq: G sym decay 22} is bounded by 
\begin{equation}\label{eq: G sym decay 24}
C|\zeta|e^{-|\lambda|}\left|\frac{\xi-\lambda}{Z(\xi)-\zeta}\right|.
\end{equation}
To estimate \eqref{eq: G sym decay 24}, we write
\begin{equation}\label{eq: G sym decay 25}
\frac{\xi-\lambda}{Z(\xi)-\zeta}=\frac1{Z'(\lambda) +\frac{Z(\xi)-\zeta}{\xi-\lambda}-Z'(\lambda)}
\end{equation}
Again by Taylor's theorem, we have
\begin{equation}\label{eq: G sym decay 26}
\left|\frac{Z(\xi)-\zeta}{\xi-\lambda}-Z'(\lambda)\right|\lesssim\sup_{\eta=\theta\xi+(1-\theta)\lambda}|Z''(\eta)||\xi-\lambda|\lesssim  |\lambda|e^{-|\lambda|}|\lambda|.
\end{equation}
\eqref{eq: G sym decay 26} is very small when $|\lambda|$ is large, while $Z'(\lambda)\approx 1$ when $|\lambda|$ is large. As a result, \eqref{eq: G sym decay 25} is bounded, and \eqref{eq: G sym decay 24} is bounded by 
$$C|\zeta|e^{-|\lambda|}\lesssim \frac{1}{|\zeta|}\approx \frac{1}{|\zeta|+|\xi|}.$$ 
Finally, if $\xi>2|\lambda|$, both $\frac{1}{\xi-\zeta(1-e^{-2\xi})}$ and $\frac1{1-2\zeta^*}\frac1{\xi-\lambda}$ are bounded by $\frac{C}{|\xi|}\approx \frac{1}{|\zeta|+|\xi|}$. \eqref{eq: G sym decay 3} is now established.

The left hand side of \eqref{eq: G sym decay 4} can be written as 
\begin{equation}\label{eq: G sym decay 40}
\frac1{\xi-\zeta(1-e^{-2\xi})}  -\frac1{1-2\zeta}\frac{2}{1-e^{-2\xi}} +\frac2{1-2\zeta}- \frac1{1-2\zeta^*}\frac1{\xi-\lambda}.
\end{equation}
If $\xi<1$, the first two terms can be handled in a similar way as in \eqref{eq: G sym decay 2}, while the last two terms can be written as
\begin{equation}\label{eq: G sym decay 41}
\frac2{1-2\zeta}- \frac1{1-2\zeta^*}\frac1{\xi-\lambda} = \frac{1}{1-2\zeta}\left(\frac{2\xi}{\xi-\lambda}+\frac{2(\zeta-\lambda)-1+4\zeta^*\lambda}{(1-2\zeta^*)(\xi-\lambda)}\right).
\end{equation}
Observe that 
\begin{equation}
\left|\frac{2\xi}{\xi-\lambda}\right| \lesssim \frac{\langle \xi\rangle}{|\zeta|+|\xi|},
\end{equation}
and $\zeta-\lambda$, $\zeta^*\lambda$ are bounded. It follows that \eqref{eq: G sym decay 41} is bounded by $\frac{C\langle \xi\rangle}{|\zeta|+|\xi|}$. If $1<\xi<\frac{|\lambda|}2$, the first two terms of \eqref{eq: G sym decay 40} can be written as \eqref{eq: G sym decay 21}, which is bounded by $\frac{C|\xi|}{|\zeta|^2}\approx \frac{|\xi|}{|\zeta|(|\zeta|+|\xi|)}$. The last two terms of \eqref{eq: G sym decay 40} can be treated as \eqref{eq: G sym decay 41}, which is also bounded by $\frac{C|\xi|}{|\zeta|^2}\approx \frac{C|\xi|}{|\zeta|(|\zeta|+|\xi|)}$. If $\xi>\frac{|\lambda|}2$, the term $\frac1{1-2\zeta}\frac{2}{e^{2\xi}-1}$ in \eqref{eq: G sym decay 4} is bounded by 
\begin{equation}
\frac{C}{|\zeta|e^{2\xi}}\lesssim \frac{1}{|\zeta|}\lesssim \frac{|\xi|}{|\zeta|(|\zeta|+|\xi|)}.
\end{equation}
The first two terms in \eqref{eq: G sym decay 4} can be treated in the same way as was done for \eqref{eq: G sym decay 22}, thus is bounded by $\frac{C}{|\zeta|+|\xi|}\lesssim \frac{|\xi|}{|\zeta|(|\zeta|+|\xi|)}$. \eqref{eq: G sym decay 4} is now established.
\end{proof}
We can now use the above symbol estimates to show the corresponding estimates on the decay of $G_*$ and $H$ (defined in \eqref{def: H}).

\begin{proposition}\label{prop: G decay}
Let $f$ be a Schwartz class function For any $\alpha>0$ sufficiently small, and $R>0$ sufficiently large, there exists a $C>0$ such that 
\begin{enumerate}[(i)]
\item If $|\arg\zeta|\ge \alpha$, $|\zeta|>R$, %let
%\begin{equation}\label{eq: G decay 4 split}
%H_3(x,\zeta)=G_L(x,\zeta)-\frac{i}{1-2\zeta}\chi_{\mathbb R^+}(x).
%\end{equation}
then
\begin{equation}\label{eq: G decay 4}
\|G_L(\cdot,\zeta)\|_{L^2(-\infty,0)}\le \frac{C}{\sqrt{|\zeta|}},\quad \|G_R(\cdot,\zeta)\|_{L^2(0,\infty)}\le \frac{C}{\sqrt{|\zeta|}}.
\end{equation}
Furthermore,
\begin{equation}\label{eq: G decay 5}
\left\|G_L(\cdot,\zeta)*f - \frac{i}{2\zeta-1}\left(T_+f-\frac12\int_\R f\right)\right\|_{L^\infty} \le \frac{C}{|\zeta|}\left\|\frac{\langle\xi\rangle}{|\zeta|+|\xi|}\widehat{f}(\xi)\right\|_{L^1}.
\end{equation}
%\begin{equation}\label{eq: G decay 6}
%\left\|G_R(\cdot,\zeta)*f - \frac{i}{2\zeta-1}\left(T_+f+\frac12\int_\R f~dx\right)\right\|_\infty \le \frac{C}{|\zeta|^2}\left\|\langle\xi\rangle\widehat{f}(\xi)\right\|_1.
%\end{equation}

\item  If $|\arg\zeta|<\alpha$, $|\zeta|>R$, let $\lambda=Z^{-1}(\zeta)$ (see Definition \ref{def: Z inv}). Then 
\begin{equation}\label{eq: G decay 1}
\|H(\cdot,\zeta)\|_{L^2}\le \frac{C}{\sqrt{|\zeta|}}.
\end{equation}
Furthermore,
\begin{align}
\bigg\|\left(G_L(\cdot,\zeta) -\frac{\pm i\chi_{\R^\pm}}{1-2\zeta^*}e^{i\lambda\cdot}\right)*f - &\frac{i}{2\zeta-1}\left(T_-f-\frac12\int_\R f\right)\bigg\|_{L^\infty} \notag\\
&\le \frac{C}{|\zeta|}\left\|\frac{\langle\xi\rangle}{|\zeta|+|\xi|}\widehat{f}(\xi)\right\|_{L^1}.\label{eq: G decay 2}
\end{align}
Here $H$ is defined by \eqref{def: H}; the plus sign is taken when $\zeta$ is in the upper half plane, and the minus sign is taken when $\zeta$ is in the lower half plane.
%\begin{equation}\label{eq: G decay 3}
%\left\|G_R(\cdot,\zeta)*f -\left(\frac{i\chi_{\R^+}}{1-2\zeta^*}e^{i\lambda\cdot}\right)*f - \frac{i}{2\zeta-1}\left(T_-f+\frac12\int_\R f~dx\right)\right\|_\infty \le \frac{C}{|\zeta|^2}\|\langle\xi\rangle\widehat{f}(\xi)\|_1.
%\end{equation}

\end{enumerate}
\end{proposition}
\begin{proof}
First consider case (i). Note that
\begin{align}\label{eq: H3 Fourier int}
&\chi_{\mathbb R^-}(x)G_L(x,\zeta)+\chi_{\mathbb R^+}(x)G_R(x,\zeta)\notag\\
= &~\frac1{2\pi}\int_{\Gamma_L}e^{ix\xi}\left(\frac{1}{\xi-\zeta(1-e^{-2\xi})}-\frac{1}{(1-2\zeta)\xi}\right)~d\xi.
\end{align}
Since the integrand is bounded near 0, one can change the contour from $\Gamma_L$ to $\mathbb R$. By \eqref{eq: G sym decay 1} and Plancherel's theorem, 
\begin{equation}
\|\chi_{\mathbb R^-}(\cdot)G_L(\cdot,\zeta)+\chi_{\mathbb R^+}(\cdot)G_R(\cdot,\zeta)\|_{L^2}\le C\bigg\|\frac1{|\xi|+|\zeta|}\bigg\|_2=\frac C{\sqrt{|\zeta|}}.
\end{equation}
Next we observe that 
$$
G_L(x,\zeta)=\frac1{2\pi}P.V.\int_\mathbb{R}\frac{e^{ix\xi}}{\xi-\zeta(1-e^{-2\xi})}~d\xi+\frac{i}{2(1-2\zeta)},
$$
and recall from \eqref{def: T pm} and Lemma \ref{lem: FT of coth} that $\widehat{T_+f }= P.V.\frac{2i }{1-e^{-2\xi}}\hat f$.
Denote by $\kappa(T_+)$ the convolution kernel of $T_+$. The above formulas imply that the Fourier transform of $G_L(\cdot,\zeta)-\frac{i}{2\zeta-1}\left(\kappa(T_+)-\frac12\right)$ is
\begin{equation}
\frac1{\xi-\zeta(1-e^{-2\xi})}-\frac2{(1-2\zeta)(1-e^{-2\xi})}.
\end{equation}
\eqref{eq: G decay 5} now follows from \eqref{eq: G sym decay 2}.

The proof of \eqref{eq: G decay 1} and \eqref{eq: G decay 2} is similar. We only mention the argument for \eqref{eq: G decay 2}. If $|\arg\zeta|<\alpha$, $|\zeta|>R$, one can write the Fourier transform of $G_L(x,\zeta)-\frac{\pm i\chi_{\R^\pm}(x)}{1-2\zeta^*}e^{i\lambda x}-\frac{i}{2\zeta-1}(\kappa(T_-)-\frac12)$ as
\begin{equation}
\frac1{\xi-\zeta(1-e^{-2\xi})}-\frac1{1-2\zeta^*}\frac{1}{\xi-\lambda}-\frac1{1-2\zeta}\frac2{e^{2\xi}-1}. 
\end{equation}
Now \eqref{eq: G decay 2} follows from \eqref{eq: G sym decay 4}.
\end{proof}
\begin{remark}
Other decay estimates resulting from \eqref{eq: G sym decay 2} and \eqref{eq: G sym decay 4} may require less decay on $\hat f$. For instance, if we replace the $L^\infty$ norm in \eqref{eq: G decay 5} by the $L^2$ norm, we may conclude that it is of order $o_{f}\left(\frac{1}{|\zeta|}\right)$ by applying the dominated convergence theorem.
\end{remark}

\subsection{\texorpdfstring{Pointwise estimates for fixed $\zeta$}{Pointwise estimates for fixed zeta}}

 Before constructing the Jost functions, we will first show their analytic continuation properties in $x$-space. Such proofs rely on the corresponding  continuation properties of the Green's functions. Here we prove estimates on the symbol of the Green's kernel that are useful for such purposes. This is achieved by simply subtracting off the poles in the symbol. Such a direct approach makes it easy to apply Fourier analysis to the remainder. Note that we do not attempt to get estimates that are uniform in $\zeta$ in the following proposition, so the constants could blow up as $\zeta$ approaches the boundary of the domains.

\begin{proposition}\label{prop: remainder}
Let $\xi\in \mathbb R$ and $y\in[0,2]$. For $\zeta\in V$, let $\lambda = Z^{-1}(\zeta)$ (see Lemma \ref{lem: Z inv domain} and Definition \ref{def: Z inv}), and define %$\zeta\ne \frac12$, $0\le y\le 2$, $\xi$ in a small neighborhood of $\R$, $\xi\ne 0$, $\xi\ne \lambda$, define
\begin{equation}\label{eq: def R11}
R_1(\xi,y,\zeta) = \frac{e^{-y\xi}}{\xi-\zeta(1-e^{-2\xi})}-\frac1{(1-2\zeta)\xi} - \frac{e^{-\lambda y}}{(1-2\zeta^*)(\xi-\lambda)}.
\end{equation}
For $\zeta\in\mathbb C\setminus \mathbb R^+$, define
\begin{equation}\label{eq: def R2}
R_2(\xi,y,\zeta) = \frac{e^{-y\xi}}{\xi-\zeta(1-e^{-2\xi})}-\frac1{(1-2\zeta)\xi}.
\end{equation}
Then for each $\zeta$ in the above ranges of definitions of $R_1$ and $R_2$, there is a constant $C(\zeta)>0$ such that for $i=1,2$, $\xi\in \R$ and $y,y_1,y_2\in [0,2]$, one has
\begin{equation}\label{eq: R bound}
|R_i(\xi,y,\zeta)|\le C(\zeta)\left(\frac{1}{\langle \xi\rangle}+e^{(2-y)\xi}\chi_{\R^-}(\xi)\right),
\end{equation}
\begin{align}\label{R diff y}
&~|R_i(\xi,y_1,\zeta)-R_i(\xi,y_2,\zeta)| \\ 
\le &~C(\zeta)\bigg[\frac{|y_1-y_2|\chi_{\R^+}(\xi)}{\langle \xi\rangle}\notag\\
&\qquad \qquad +(|y_1-y_2|+e^{2-\min(y_1,y_2)}|e^{(y_1-y_2)\xi}-1|)\chi_{\R^-}(\xi)\bigg],\notag
\end{align}
\begin{equation}\label{eq: deriv R bound}
|\partial_\xi R_i(\xi,y,\zeta)|\le C(\zeta)\left(\frac{1}{\langle \xi\rangle^2}+(2-y)e^{(2-y)\xi}\chi_{\R^-}(\xi)\right),
\end{equation}
\begin{equation}\label{eq: d_y R bound}
    |\partial_y^k R_i(\xi,y,\zeta)|\le C(\zeta)\left(1+|\xi|^k e^{(2-y)\xi}\chi_{\R^-}(\xi)\right),
\end{equation}
\begin{equation}\label{eq: d_zeta R bound}
    |\partial_\zeta \partial_\xi R_2(\xi,y,\zeta)|\le C(\zeta)\left(\frac{1}{\langle \xi\rangle^2}+(2-y)e^{(2-y)\xi}\chi_{\R^-}(\xi)\right).
\end{equation}
Moreover, the constant $C(\zeta)$ for $R_1$ can be taken to be uniform when $\zeta$ is in a compact subset of $V$, and similarly for $R_2$ when $\zeta$ is in a compact subset of $\C\setminus [0,\infty)$.
\end{proposition}
\begin{remark}
\eqref{eq: def R11} is orignally defined only for $\xi\ne0, \lambda$, and $\zeta\ne\frac12$.   
But $R_1(\xi,y,\zeta)$ can be extended analytically to include the removable singularities $\xi=0$, $\xi=\lambda$, and $\zeta=\frac12$. In particular, by taking the limit of \eqref{eq: def R11} we get
\begin{align}\label{eq: def R12}
R_1\left(\xi,y,\frac12\right) %= \frac{e^{-y\xi}}{\xi-\frac12(1-e^{-2\xi})}-\frac{1}{2\zeta'(0)\xi^2}-\left(\frac{1-y}{2\zeta'(0)}-\frac{\zeta''(0)}{4[\zeta'(0)]^2}\right)\frac1\xi\\
=\frac{e^{-y\xi}}{\xi-\frac12(1-e^{-2\xi})}- \frac1{\xi^2}-\left(\frac23-y\right)\frac1\xi.
\end{align}
\end{remark}
\begin{proof}
For fixed $\zeta$ and $y$, $R_1$ and $R_2$ can be thought of as meromorphic functions of $\xi$. These meromorphic functions are analytic near the real line, as the principal parts of the poles on the real line are subtracted off. Let us consider the case for $R_2$, for which $\zeta\notin \R^+$. We can write 
\begin{equation}
R_2(\xi,y,\zeta) = \frac{(1-2\zeta)(e^{-y\xi}-1)+\zeta\left(\frac{1-e^{-2\xi}}{\xi}-2\right)}{(1-2\zeta)[Z(\xi)-\zeta](1-e^{-2\xi})}.
\end{equation}
It follows that 
\begin{equation}
|R_2(\xi,y,\zeta)|\le C(\zeta), ~|\partial_\xi R_2(\xi,y,\zeta)|\le C(\zeta)
\end{equation} 
for $-1<\xi<1$. When $\zeta>1$, we can write
\begin{equation}\label{eq: xi far from 1 form}
\frac{e^{-y\xi}}{\xi-\zeta(1-e^{-2\xi})} = \frac{e^{-y\xi}}{(1-e^{-2\xi})(Z(\xi)-\zeta)},
\end{equation}
which is bounded by 
\begin{equation}
\frac{1}{|Z(\xi)-\zeta|}\le \frac{C(\zeta)}{|\xi|}
\end{equation}
for $\xi>1$.
When $\xi<-1$, \eqref{eq: xi far from 1 form} is bounded by
\begin{equation}
\frac{Ce^{(2-y)\xi}}{|Z(\xi)-\zeta|}\le C(\zeta) e^{(2-y)\xi}.
\end{equation}
The bound on the $\partial_\xi$ derivative of \eqref{eq: xi far from 1 form} can be proven in a similar way. In fact, we can write \eqref{eq: xi far from 1 form} as
\begin{equation}
e^{(2-y)\xi}\frac{e^{-2\xi}}{(1-e^{-2\xi})(Z(\xi)-\zeta)}.
\end{equation}
The derivative of $e^{(2-y)\xi}$ will produce the $(2-y)e^{(2-y)\xi}$ contribution in \eqref{eq: deriv R bound}. The derivartive of the rest will produce an exponentially small term which can be absorbed into the first term in \eqref{eq: deriv R bound}.
 Estimates on the $\partial_y$ derivatives of $R_2$ and the other inequalities of $R_1$ can be proven in a similar way. In the case $\zeta>0$, $\zeta\neq \frac12$, one just treats the singularity at $\lambda$ in the same way as the singularity at $0$.
\end{proof}

We will also need the following more specific estimate for $R_1(\xi,y,\frac12)$.

\begin{lemma}\label{lem: inv Fourier R1}
    For any $0<\epsilon<2$, there exists a $C(\epsilon)>0$ such that for all $0\le y\le 2-\epsilon$,
    \begin{equation}
        \|\mathcal F^{-1}[R_1(\cdot,y,\tfrac12)]\|_{L^\infty}\le C(\epsilon).
    \end{equation}
\end{lemma}
\begin{proof}\label{eq: 189}
    From \eqref{eq: R bound} and the form \eqref{eq: def R12} it's easy to see that
    \begin{equation}
        \left|R_1(\xi,y,\tfrac12)+(\tfrac23-y)\tfrac{\chi_{|\xi|>1}}{\xi}\right|\le C\left(\tfrac{1}{\langle\xi\rangle^2}+e^{(2-y)\xi}\chi_{\mathbb R^-}(\xi)\right).
    \end{equation}
    Since the right hand side of \eqref{eq: 189} is bounded in $L^1$ for $0\le y\le 2-\epsilon$, it remains to show $\mathcal F^{-1}\left(\tfrac{\chi_{|\xi|>1}}{\xi}\right)\in L^\infty$. The latter is a straightforward calculation, as 
    \begin{equation}
        \mathcal F^{-1}\left(\tfrac{\chi_{|\xi|>1}}{\xi}\right)(x) = \frac i{\pi}\text{sgn}(x)\int_{|x|}^\infty\frac{\sin \xi}{\xi}~d\xi.
    \end{equation}
\end{proof}
\section{Properties of Jost solutions}
\label{sec: prop of jost}

In this section, we show that the differential RH problems such as \eqref{eq: main M1}, \eqref{eq: main N1}, \eqref{eq: main Me}, \eqref{eq: main Ne}, and their continuations to complex values of $\zeta$, are equivalent to Fredholm integral equations of the form
\begin{align}
    M_1&=1+G_L*(uM_1),\\
    N_1&=1+G_R*(uN_1),\\
    M_e&=e^{i\lambda \cdot}+G_L*(uM_e),\\
    N_e&=e^{i\lambda \cdot}+G_R*(uN_e).
\end{align}
A first step in doing so is to prove that solutions to the above integral equations have analytic continuations in the $x$-variable to a strip of height 2. This we do in Section \ref{sec: ana cont}.

To help analyze the $x$-space properties, we consider the analytically continued Green's functions $G_L(x+iy,\zeta)$, $G_R(x+iy,\zeta)$, and split them into a main part and a remainder part by subtracting off the poles from the Green's kernel. In particular, When $\zeta\in \C\setminus\R^+$, we write
\begin{equation}\label{G_L split 2}
G_L(x+iy,\zeta) = \tfrac{i}{1-2\zeta}\chi_{\R^+}(x) + \mathcal F^{-1}[R_2(\xi,y,\zeta)](x),
\end{equation}
\begin{equation} \label{G_R split 2}
G_R(x+iy,\zeta) = -\tfrac{i}{1-2\zeta}\chi_{\R^-}(x)+\mathcal F^{-1}[R_2(\xi,y,\zeta)](x),
\end{equation}
where $R_2$ is given as in \eqref{eq: def R2}.
When $k\in \R^+$ and $k\ne \frac12$, we write 
\begin{align}\label{G_L split +}
G_L(x+iy,k+i0) &= \tfrac{i}{1-2k}\chi_{\R^+}(x)+ \tfrac{i}{1-2k^*}e^{-\lambda y}e^{i\lambda x}\chi_{\R^+}(x)\notag\\
&\qquad +\mathcal F^{-1}[R_1(\xi,y,k)](x),
\end{align}
\begin{align}\label{G_L split -}
G_L(x+iy,k-i0) &= \tfrac{i}{1-2k}\chi_{\R^+}(x)-\tfrac{i}{1-2k^*}e^{-\lambda y}e^{i\lambda x}\chi_{\mathbb R^-}(x)\notag\\
&\qquad +\mathcal F^{-1}[R_1(\xi,y,k)](x),
\end{align}
\begin{align}\label{G_R split +}
G_R(x+iy,k+i0) &= -\tfrac{i}{1-2k}\chi_{\R^-}(x)+ \tfrac{i}{1-2k^*}e^{-\lambda y}e^{i\lambda x}\chi_{\R^+}(x)\notag\\
&\qquad +\mathcal F^{-1}[R_1(\xi,y,k)](x),
\end{align}
\begin{align}\label{G_R split -}
G_R(x+iy,k-i0) &= -\tfrac{i}{1-2k}\chi_{\mathbb R^-}(x)-\tfrac{i}{1-2k^*}e^{-\lambda y}e^{i\lambda x}\chi_{\mathbb R^-}(x)\notag\\
&\qquad +\mathcal F^{-1}[R_1(\xi,y,k)](x),
\end{align}
where $\lambda=Z^{-1}(k)$ (see Definition \ref{def: Z inv}), and $R_1$ is given as in \eqref{eq: def R11}.
Finally, note that \eqref{G_L split +} and \eqref{G_R split -} have limits as $k\to \frac12$, and we write
\begin{equation}\label{G_L split 1/2}
G_L(x+iy,\tfrac12+i0)=-x\chi_{\R^+}(x)+i(\tfrac23-y)\chi_{\R^+}(x)+\mathcal F^{-1}[R_1(\xi,y,\tfrac12)](x),
\end{equation}
\begin{equation}\label{G_R split 1/2}
G_R(x+iy,\tfrac12-i0)=x\chi_{\mathbb R^-}(x)-i(\tfrac23-y)\chi_{\mathbb R^-}(x)+\mathcal F^{-1}[R_1(\xi,y,\tfrac12)](x),
\end{equation}
where $R_1(\xi,y,\tfrac12)$ is given as in \eqref{eq: def R12}.

\subsection{\texorpdfstring{Analytic continuation in $x$-space}{Analytic continuation in x-space}}
\label{sec: ana cont}

Through the following Lemmas, we show analytic continuation in $x$-space of the solutions to the integral equations. For results at most values of $\zeta$, we assume $u\in L^1\cap L^2$. However, for the result at $\zeta=\frac12\pm i0$, we require an extra degree of decay on $u$ and assume $u\in L^1_1\cap L^2_1$ (see Lemma \ref{lem: anal cont 1/2}). %We remark that $H^{2,2}\subset L^1_1\cap H^s_1$. We will use the following fact to get a simpler proof of analyticity:
We will establish the results in three steps, first for $\zeta\in \mathbb C\setminus \mathbb R^+$, then for $\zeta\in \mathbb (\mathbb R^+\pm i0)\setminus \{\frac12\pm i0\}$, and finally for $\zeta=\frac12\pm i0$.
\begin{lemma}\label{lem: anal cont 1}
Let $u\in L^1\cap L^2$. Fix $\zeta\in \C\setminus \R^+$. Let $M_1(\cdot,\zeta), N_1(\cdot,\zeta)\in L^\infty$ be solutions to 
\begin{equation}\label{jost prty: def M}
M_1(x,\zeta) = 1+G_L(\cdot,\zeta)*[u(\cdot)M_1(\cdot,\zeta)](x),
\end{equation}

\begin{equation}\label{jost prty: def N}
N_1(x,\zeta) = 1+G_R(\cdot,\zeta)*[u(\cdot)N_1(\cdot,\zeta)](x).
\end{equation}
Then $\partial_x M_1, \partial_x N_1\in L^2$. Moreover, $M_1(x,\zeta), N_1(x,\zeta)$ are non-tangential boundary values of holomorphic functions $M_1(x+iy,\zeta), N_1(x+iy,\zeta)\in \mathbb H^2(S)+\mathbb H^\infty(S)$, where $S=\{x+iy~|~0<y<2\}$, and 
%\begin{enumerate}[(i)]
%\item 
\begin{align}
M_1(x+iy,\zeta) &=1 + o(1) & & \text{ as }x\to -\infty~ \text{ for each }y\in [0,2),\label{jost prty: M1 limit left}\\
M_1(x+iy,\zeta) &= a(\zeta) + o(1) & & \text{ as }x\to+\infty~\text{ for each }y\in [0,2),\label{jost prty: M1 limit right}\\
N_1(x+iy,\zeta)& =1 + o(1) & & \text{ as }x\to+\infty~\text{ for each }y\in [0,2), \label{jost prty: N1 limit right}\\
N_1(x+iy,\zeta)&=\breve{a}(\zeta) +o(1) & & \text{ as }x\to-\infty~\text{ for each }y\in [0,2).\label{jost prty: N1 limit left}
\end{align}
where the implied limits are uniform for $y$ in compact subsets of $[0,2)$. Here 
\begin{equation}\label{def: a}
a(\zeta)=1+\frac{i}{1-2\zeta}\int_\R u(x)M_1(x,\zeta)~dx,
\end{equation}
\begin{equation}\label{def: breve a}
\breve a(\zeta) = 1-\frac{i}{1-2\zeta}\int_\R u(x)N_1(x,\zeta)~dx.
\end{equation}

%\item $$\|M_1(\cdot+iy)-M_1(\cdot+i0)\|_\infty\to 0,~\|N_1(\cdot+iy)-N_1(\cdot+i0)\|_\infty\to 0$$
%as $y\searrow 0$.
%$$\|M_1(\cdot+iy)-M_1(\cdot+i2)\|_\infty\to 0,~\|N_1(\cdot+iy)-N_1(\cdot+i2)\|_\infty\to 0$$
%as $y\nearrow 2$.
%\end{enumerate}
\end{lemma}
%\begin{remark}
%The above Lemma is true for either $M_1$ or $N_1$. We do not need $M_1$ and $N_1$ to exist at the same time for the results to hold.
%\end{remark}
\begin{proof}
We only show the proof for $M_1$, as that of $N_1$ will be analogous. For simplicity of notation, we also suppress the functional dependence on $\zeta$ when we refer to Jost and Green's functions in this proof. Using the splitting formula \eqref{G_L split 2}, we can rewrite \eqref{jost prty: def M} as 
\begin{equation}\label{jost prty: M1 split}
M_1(x)=1+\frac{i}{1-2\zeta}\int_{-\infty}^x u(s)M_1(s)~ds+\mathcal F^{-1}[R_2(\xi,0)\widehat{uM_1}(\xi)](x).
\end{equation}
We get from \eqref{jost prty: M1 split}
\begin{equation}\label{jost prty: D M1 split}
D M_1(x) = \tfrac1{1-2\zeta}u(x)M_1(x)+\mathcal F^{-1}[\xi R_2(\xi,0)\widehat{uM_1}(\xi)](x).
\end{equation}
Note that $\widehat{uM_1}\in L^2$.  It follows from \eqref{jost prty: D M1 split} and \eqref{eq: R bound} that $\partial_x M_1\in L^2$. Let $P_{\le N}=\chi(\frac{D}{N})$ be a Littlewood-Paley type projector, where $\chi$ is a standard cutoff function supported on $|\xi|\le 2$, and $\chi(\xi)=1$ for $|\xi|\le 1$. Let $P_{>N}=I-P_{\le N}$. We rewrite \eqref{jost prty: def M} as 
\begin{equation}\label{jost prty: M1 split sub 1}
    M_1(x)=1+G_L*(P_{\le 1}(uM_1))+G_L*(P_{>1}(uM_1))
\end{equation}
Using \eqref{G_L split 2} again, we have
\begin{align}\label{jost prty: M1 split sub 2}
    [G_L*(P_{\le 1}(uM_1))](x)=&\frac{i}{1-2\zeta}\int_{-\infty}^x P_{\le 1}(uM_1)~ds \notag\\
    &\quad +\mathcal F^{-1}[R_2(\xi,0)\chi(\xi)\widehat{uM_1}(\xi)](x)
\end{align}
It's easy to see that
\begin{equation}
    P_{\le 1}(uM_1)(x+iy) = \frac{1}{2\pi}\int_\R \chi(\xi)\widehat{uM_1}(\xi)e^{i(x+iy)\xi}~d\xi
\end{equation}
is entire with estimates $\|P_{\le 1}(uM_1)\|_{L^1}\le C\|uM_1\|_{L^1}$ and
\begin{equation}\label{eq: entire P(uM) bound}
    |P_{\le 1}(uM_1)(x+iy)|\le C e^{2y}\|uM_1\|_{L^2}.
\end{equation}
It follows that $\int_{-\infty}^x P_{\le 1}(uM_1)~ds$ can be continuously extended to an $\mathbb H^\infty(S)$ function
$\int_{-\infty}^{x+iy} P_{\le 1}(uM_1)~ds$. Similarly, $\mathcal F^{-1}[R_2(\xi,0)\chi(\xi)\widehat{uM_1}(\xi)](x)$ can be extended to
\begin{equation}\label{eq: M1 holo 2}
    \mathcal F^{-1}[R_2(\xi,0)\chi(\xi)\widehat{uM_1}(\xi)](x+iy) =\frac{1}{2\pi}\int_\R R_2(\xi,0)\chi(\xi)\widehat{uM_1}(\xi) e^{i(x+iy)\xi}~d\xi
\end{equation}
Using \eqref{eq: R bound}, similar estimates as above show that \eqref{eq: M1 holo 2} is in $ \mathbb H^\infty(S)$.
Note that $P_{\le 1}$ is a convolution operator. It follows that
\begin{equation}
    G_L*(P_{>1}(uM_1))=G_L*(P_{>\frac12}P_{>1}(uM_1))=(P_{>\frac12}G_L)*(P_{>1}(uM_1)).
\end{equation}
By \eqref{def: G} we see that
\begin{equation}
    G_L(x)- \mathcal F^{-1}\left[\text{P.V.}\tfrac{1}{\xi-\zeta(1-e^{-2\xi})}\right](x)=\tfrac{i}{2(1-2\zeta)}.
\end{equation}
Thus
\begin{equation}
    \mathcal F[P_{>\frac12}G_L](\xi)=\frac{1-\chi(2\xi)}{\xi-\zeta(1-e^{-2\xi})},
\end{equation}
\begin{equation}
    \mathcal F\left[(P_{>\frac12}G_L)*(P_{>1}(uM_1))\right]=\frac{(1-\chi(\xi))\widehat{uM_1}(\xi)}{\xi-\zeta(1-e^{-2\xi})}.
\end{equation}
Note that 
$$\left|\frac{(1-\chi(\xi))e^{-y\xi}}{\xi-\zeta(1-e^{-2\xi})}\right|\le C(\zeta)$$
for all $y\in [0,2]$. Thus for the same values of $y$,
\begin{equation}
   \| \mathcal F\left[G_L*(P_{>1}(uM_1))\right](\xi)e^{-y\xi}\|_{L^2}\le C(\zeta)\|uM_1\|_{L^2}.
\end{equation}
By standard $\mathbb H^2$ theory, $G_L*(P_{>1}(uM_1))$ is the non-tangential boundary value of a holomorphic function in $\mathbb H^2(S)$. We conclude that $M_1(x)$ is the non-tangential boundary value of a function $M_1(x+iy)\in \mathbb H^2(S)+\mathbb H^\infty(S)$.

We now rewrite the values of $M_1(x+iy)$ for $0\le y\le 2$ as
\begin{align}
M_1(x+iy)&=1+G_L(x+iy)*(uM_1)(x)\notag\\
&=1+\frac{i}{1-2\zeta}\int_{-\infty}^x u(s)M_1(s)~ds+\mathcal F^{-1}[R_2(\xi,y)\widehat{uM_1}(\xi)](x).\label{jost prty: M1 ext}
\end{align}
Note that the $M_1$ on the right hand side of \eqref{jost prty: M1 ext} is the original $M_1(x)$. To see that $M_1(x+iy)$ is indeed given by the right hand side of \eqref{jost prty: M1 ext}, we just need to show that it is holomorphic on $S$, and the non-tangential lower boundary value is $M_1(x)$. The latter is easy to see, as $R_2(\xi,y)\widehat{uM_1}\to R_2(\xi,0)\widehat{uM_1}$ in $L^1$ by \eqref{eq: R bound}. To see the analyticity, we could try to asseble the above decompositions, but instead we argue more directly as follows. Take a sequence of $f_n\in C_0^\infty$ such that $f_n\to uM_1$ in $L^1\cap L^2$. We have
\begin{equation}\label{eq: G_L * fn}
    G_L(x+iy)*f_n(x)=\frac{i}{1-2\zeta}\int_{-\infty}^x f_n(s)~ds+\mathcal F^{-1}[R_2(\xi,y)\widehat{f_n}(\xi)](x)
\end{equation}
The first term in \eqref{eq: G_L * fn} obviously converges in $L^\infty$ to the integral term in \eqref{jost prty: M1 ext}. By \eqref{eq: R bound}, the last term in \eqref{eq: G_L * fn} converges to the last term in \eqref{jost prty: M1 ext} in $L^\infty$ uniformly for $y$ in any compact subset of $[0,2)$. Thus it suffices to show analyticity of $G_L(x+iy)*f_n(x)$. Since $f_n\in C_0^\infty$, $\widehat{f_n}$ is entire. We have
\begin{align*}
G_L(x+iy)*f_n(x)&=\frac{i}{1-2\zeta}\int_{-\infty}^x f_n(s)~ds+\mathcal F^{-1}[R_2(\xi,y)\widehat{f_n}(\xi)](x)\\
&=\frac1{2\pi} \int_{\Gamma_L} \frac{e^{ix\xi}}{1-2\zeta}\frac{\widehat{f_n}(\xi)}{\xi}~d\xi +\frac1{2\pi}\int_\R e^{ix\xi}R_2(\xi,y)\widehat{f_n}(\xi)~d\xi\\
&=\frac1{2\pi} \int_{\Gamma_L} \frac{e^{ix\xi}}{1-2\zeta}\frac{\widehat{f_n}(\xi)}{\xi}~d\xi +\frac1{2\pi}\int_{\Gamma_L} e^{ix\xi}R_2(\xi,y)\widehat{f_n}(\xi)~d\xi\\
&=\frac1{2\pi}\int_{\Gamma_L}\frac{e^{i(x+iy)\xi}\widehat{f_n}(\xi)}{\xi-\zeta(1-e^{-2\xi})}~d\xi.
\end{align*}
The analyticity of the above function in $x+iy$ now follows easily from Morrera's theorem.

It remains to prove \eqref{jost prty: M1 limit right}. From \eqref{jost prty: M1 ext}, it suffices to prove that for any given $\delta\in (0,2) $, $$\mathcal F^{-1}[R_2(\xi,y)\widehat{uM_1}(\xi)](x)\to 0$$ uniformly for $y\in [0,2-\delta]$ as $|x|\to\infty$. For any $\epsilon>0$, let $\varphi \in C^\infty_0$ such that $\|\widehat{uM_1}-\varphi\|_{L^2}< \epsilon$. Since 
\begin{align*}
\frac1{2\pi}\int_\R R_2(\xi,y)\varphi(\xi)e^{ix\xi}~d\xi &= \frac1{2\pi}\int_\R R_2(\xi,y)\varphi(\xi)\left(\frac1{ix}\partial_\xi e^{ix\xi}\right)~d\xi\\
&=-\frac1{2\pi ix}\int_\R \left(\varphi(\xi)\partial_\xi R_2(\xi,y)+R_2(\xi,y)\varphi'(\xi)\right)e^{ix\xi}~d\xi.
\end{align*}
By \eqref{eq: R bound}, \eqref{eq: deriv R bound} and $\varphi\in C_0^\infty$ we see from the above that 
$$\left|\frac1{2\pi}\int_\R R_2(\xi,y)\varphi(\xi)e^{ix\xi}~d\xi \right|\le \frac{ C(\zeta,\varphi,\delta) }{|x|}.$$
 Thus there exists an $M=M(\zeta,\varphi)>0$ such that for all $|x|>M$ and $y\in [0,2]$, we have $|\mathcal F^{-1}[R_2(\xi,y)\varphi(\xi)](x)|<\epsilon$. Using \eqref{eq: R bound} again, we have $\|R_2(\xi,y)\|_{L^2}\le C(\zeta,\delta)$ for all $y\in [0,2-\delta]$. Thus
 \begin{align}\label{eq: inv Fourier R2 phi - uM}
     |\mathcal F^{-1}[R_2(\xi,y)(\varphi-\widehat{uM_1})(\xi)](x)|&<C(\zeta,\delta)\|\widehat{uM_1}-\varphi\|_{L^2}\notag\\
     &< C(\zeta,\delta)\epsilon
 \end{align}
for all $x\in \R$ and $y\in [0,2-\delta]$. Thus $|\mathcal F^{-1}[R_2(\xi,y)\widehat{uM_1}(\xi)](x)|<(C(\zeta,\delta)+1)\epsilon$ when $|x|>M$ and $y\in [0,2-\delta]$. The proof is now complete.
\end{proof}

We next turn to the cases where $\zeta$ is on $\R^+\pm i0$.

\begin{lemma}\label{lem: anal cont R+}
Let $u\in L^1\cap L^2$. Fix $k\in \R^+$, $k\ne\frac12$. Let $M_1(\cdot,k\pm i0), N_1(\cdot,k\pm i0)\in L^\infty$ be solutions to \eqref{jost prty: def M}, \eqref{jost prty: def N} with $\zeta=k\pm i0$, and $M_e(\cdot,k\pm i0), N_e(\cdot,k\pm i0)\in L^\infty$ be solutions to 
\begin{equation}\label{jost prty: def M_e}
M_e(x,k\pm i0) = e^{i\lambda x} +G_L(\cdot,k\pm i0)*[u(\cdot)M_e(\cdot,k\pm i0)](x),
\end{equation}
\begin{equation}\label{jost prty: def N_e}
N_e(x,k\pm i0) = e^{i\lambda x} +G_R(\cdot,k\pm i0)*[u(\cdot)N_e(\cdot,k\pm i0)](x).
\end{equation}
Here $\lambda=Z^{-1}(k)$ (see Definition \ref{def: Z inv}). Then $\partial_x M_1, \partial_x N_1, \partial_x M_e, \partial_x N_e\in L^2+L^\infty$. Moreover, $M_1(x,k\pm i0), N_1(x,k\pm i0), M_e(x,k\pm i0), N_e(x,k\pm i0)$ are non-tangential boundary values of holomorphic functions $M_1(x+iy,k\pm i0), N_1(x+iy,k\pm i0), M_e(x+iy,k\pm i0), N_e(x+iy,k\pm i0) \in \mathbb H^2(S)+\mathbb H^\infty(S)$, and 
\begin{align}
M_1(x+iy,k+i0) &=1+o(1) & & \text{ as }x\to-\infty,\label{eq: M1 k+i0 limit left}\\
M_1(x+iy,k+i0) &= a(k+i0)+b(k+i0)e^{i\lambda(x+iy)}+o(1) & & \text{ as }x\to+\infty,\\
N_1(x+iy,k-i0)&=1+o(1)& & \text{ as }x\to+\infty,\label{eq: N1 k-i0 limit right}\\
M_e(x+iy,k+i0) &= e^{i\lambda(x+iy)}+o(1)& & \text{ as }x\to-\infty,\label{eq: Me k+i0 limit left}\\
N_e(x+iy,k-i0) &=e^{i\lambda(x+iy)}+o(1)& & \text{ as }x\to+\infty,\label{eq: Ne k-i0 limit right}
\end{align}
for each $y\in [0,2)$,
where the implied limits are uniform for $y$ in compact subsets of $[0,2)$. Here  $a(k+i0)$ is given as in \eqref{def: a}, and%, $\breve a(\zeta)$ are given as in \eqref{def: a} and \eqref{def: breve a}, and 
\begin{equation}\label{def: b}
b(k\pm i0) = \frac{i}{1-2k^*}\int_\R u(x)M_1(x,k\pm i0)e^{-i\lambda x}~dx.
\end{equation}
\end{lemma}
\begin{remark}
We will not need the following, but one can also compute other limits of the Jost solutions:
    \begin{align*}
        N_1(x+iy,k+i0)&=\breve{a}(k+i0)+o(1)& & \text{ as }x\to-\infty,\\
        N_1(x+iy,k+i0)&=1+\breve b(k+i0)e^{i\lambda(x+iy)}+o(1)& & \text{ as }x\to+\infty,\\
        M_e(x+iy,k+i0) &=a_e(k+i0)+[1+b_e(k+i0)]e^{i\lambda(x+iy)}+o(1)& & \text{ as }x\to+\infty,\\
        N_e(x+iy,k+i0)&=\breve{a}_e(k+i0)+e^{i\lambda( x+iy)}+o(1)& & \text{ as }x\to-\infty,\\
N_e(x+iy,k+i0)&=[1+\breve{b}_e(k+i0)]e^{i\lambda (x+iy)}+o(1)& & \text{ as }x\to+\infty,\\
M_1(x+iy,k-i0) &=1-b(k-i0)e^{i\lambda(x+iy)}+o(1) & & \text{ as }x\to-\infty,\\
M_1(x+iy,k-i0) &= a(k-i0)+o(1)& & \text{ as }x\to+\infty,\\
N_1(x+iy,k-i0)&=\breve{a}(k-i0)-\breve b(k-i0) e^{i\lambda(x+iy)}+o(1) & & \text{ as }x\to-\infty,\\
M_e(x+iy,k-i0) &=[1-b_e(k-i0)]e^{i\lambda(x+iy)} +o(1) & & \text{ as }x\to-\infty,\\
M_e(x+iy,k-i0) &=a_e(k-i0)+e^{i\lambda(x+iy)}+o(1)& & \text{ as }x\to+\infty,\\
N_e(x+iy,k-i0) &=\breve a_e(k-i0) +[1-\breve b_e(k-i0)]e^{i\lambda(x+iy)}+o(1)& & \text{ as }x\to-\infty,
    \end{align*}
    for each $y\in [0,2)$,
where
    \begin{align}
\breve b(k\pm i0) &= \frac{i}{1-2k^*}\int_\R u(x)N_1(x,k\pm i0)e^{-i\lambda x}~dx,\label{def: breve b}\\
a_e(k\pm i0) &= \frac{i}{1-2k}\int_\R u(x)M_e(x,k\pm i0)~dx,\\
b_e(k\pm i0) &= \frac i{1-2k^*}\int_\R u(x)M_e(x,k\pm i0)e^{-i\lambda x}~dx,  \label{def: b_e}\\
\breve a_e(k\pm i0) &= -\frac{i}{1-2k}\int_\R u(x)N_e(x,k\pm i0)~dx,  \label{def: breve a_e}\\ 
\breve b_e(k\pm i0) &=\frac i{1-2k^*}\int_\R u(x)N_e(x,k\pm i0)e^{-i\lambda x}~dx.\label{def: breve b_e}
\end{align}
\end{remark}
\begin{proof}
This lemma can be proven by basically the same arguments as in Lemma \ref{lem: anal cont 1}. We only show the proof for $M_1(x,k\pm i0)$, and the other functions can be handled similarly. We will use the splitting formulas \eqref{G_L split +} to \eqref{G_R split -} to replace \eqref{G_L split 2} and \eqref{G_R split 2}, and note that $R_1$ and $R_2$ satisfy the same estimates in Proposition \ref{prop: remainder}. In particular, for $M_1(x)=M_1(x,k+i0)$, when $k\in \R^+$ and $k\ne \frac12$, \eqref{jost prty: M1 split} can be replaced by 
\begin{align}\label{jost prty: M1 split k-i0}
M_1(x) =&~1+\frac{i}{1-2k}\int_{-\infty}^x u(s)M_1(s)~ds\\
&\quad +\frac{ie^{i\lambda x}}{1-2k^*}\int_{-\infty}^x u(s)M_1(s)e^{-i\lambda s}~ds+\mathcal F^{-1}[R_1(\xi,0)\widehat{uM_1}](x).\notag
\end{align}
\eqref{jost prty: D M1 split} is replaced by
\begin{align}
D M_1(x) = &~\left(\frac1{1-2k}+\frac1{1-2k^*}\right)u(x)M_1(x)\notag\\
&\quad +\frac{i\lambda e^{i\lambda x}}{1-2k^*}\int_{-\infty}^x u(s)M_1(s)e^{-i\lambda s}~ds+\mathcal F^{-1}[\xi  R_1(\xi,0)\widehat{uM_1}](x).\label{jost prty: D M1 R+}
\end{align} 
\eqref{jost prty: M1 split sub 1} is replaced by 
\begin{equation}
    M_1(x) = 1+G_L*(P_{\le 2|\lambda|}(uM_1))+G_L*(P_{>2|\lambda|}(uM_1)).
\end{equation}
\eqref{jost prty: M1 split sub 2} is replaced by 
\begin{align}\label{eq: G_L * P< sub 1}
    G_L*(P_{\le 2|\lambda|}(uM_1)) &=\frac{i}{1-2\zeta}\int_{-\infty}^x P_{\le 2|\lambda|}(uM_1)~ds+\frac{ie^{i\lambda x}}{1-2k^*}\int_{-\infty}^x P_{\le 2|\lambda|}(uM_1)(s)e^{-i\lambda s}~ds\notag\\
    &\quad + \mathcal F^{-1}\left[R_1(\xi,0)\chi\left(\tfrac{\xi}{2|\lambda|}\right)\widehat{uM_1}\right].
\end{align}
\eqref{jost prty: M1 ext} is replaced by 
\begin{align}
M_1(x+iy)= &~1+\frac{i}{1-2k}\int_{-\infty}^x u(s)M_1(s)~ds+\frac{ie^{i\lambda (x+iy)}}{1-2k^*}\int_{-\infty}^x u(s)M_1(s)e^{-i\lambda s}~ds\notag\\
&\quad +\mathcal F^{-1}[R_1(\xi,y)\widehat{uM_1}](x).\label{jost prty: M1 ext R+}
\end{align}
On the other hand, for $M_1(x)=M_1(x,k-i0)$, if $k\in \R^+$ and $k\ne \frac12$, \eqref{jost prty: M1 split k-i0} can be replaced by 
\begin{align}
M_1(x) =&~1+\frac{i}{1-2k}\int_{-\infty}^x u(s)M_1(s)~ds\notag\\
&\quad -\frac{ie^{i\lambda x}}{1-2k^*}\int_x^{\infty} u(s)M_1(s)e^{-i\lambda s}~ds+\mathcal F^{-1}[R_1(\xi,0)\widehat{uM_1}](x).
\end{align}
\eqref{jost prty: D M1 R+} becomes
\begin{align}
D M_1(x) = &~\left(\frac1{1-2k}+\frac1{1-2k^*}\right)u(x)M_1(x)\notag\\
&\quad -\frac{i\lambda e^{i\lambda x}}{1-2k^*}\int_x^{\infty} u(s)M_1(s)e^{-i\lambda s}~ds+\mathcal F^{-1}[\xi  R_1(\xi,0)\widehat{uM_1}](x).
\end{align} 
\eqref{eq: G_L * P< sub 1} becomes
\begin{align}
    G_L*(P_{\le 2|\lambda|}(uM_1)) &=\frac{i}{1-2\zeta}\int_{-\infty}^x P_{\le 2|\lambda|}(uM_1)~ds-\frac{ie^{i\lambda x}}{1-2k^*}\int_x^{\infty} P_{\le 2|\lambda|}(uM_1)(s)e^{-i\lambda s}~ds\notag\\
    &\quad + \mathcal F^{-1}\left[R_1(\xi,0)\chi\left(\tfrac{\xi}{2|\lambda|}\right)\widehat{uM_1}\right].
\end{align}
\eqref{jost prty: M1 ext} becomes
\begin{align}
M_1(x+iy)= &~1+\frac{i}{1-2k}\int_{-\infty}^x u(s)M_1(s)~ds\notag\\
&\quad -\frac{ie^{i\lambda (x+iy)}}{1-2k^*}\int_x^{\infty} u(s)M_1(s)e^{-i\lambda s}~ds+\mathcal F^{-1}[R_1(\xi,y)\widehat{uM_1}](x).
\end{align}
The rest of the proof can be given in the same way as in Lemma \ref{lem: anal cont 1}.
\end{proof}
We finally turn to the case when $k=\frac12$. Only $M_1(x,\frac12+i0)$ and $N_1(x,\frac12-i0)$ will be considered. Note that we require stronger decay of $u$ in this case. We also define $\mathbb H^\infty_{-1,\pm}(S)$ to be the space of holomorphic functions $F$ on $S$ such that 
\begin{equation}
    |F(x+iy)|\le C(1+x_{\pm})
\end{equation}
for some $C>0$ and all $x+iy\in S$. If $F\in \mathbb H^\infty_{-1,\pm}(S)$, then $F(z)(1+e^{\pm z})^{-1}\in \mathbb H^\infty(S)$. Thus such $F(z)$ has almost everywhere boundary values on $\partial S$.
\begin{lemma}\label{lem: anal cont 1/2}
Let $u\in L^1_1\cap L^2_1$. Let $M_1(\cdot,\frac12+i0) \in L^\infty_{1,+}$, $N_1(\cdot,\frac12-i0)\in L^\infty_{1,-}$ be solutions to \eqref{jost prty: def M}, \eqref{jost prty: def N} with $\zeta=\frac12+i0$ and $\zeta = \frac12-i0$ respectively. Then $\partial_x M_1, \partial_x N_1\in L^2+L^\infty$. Moreover, $M_1(x,\frac12+i0), N_1(x,\frac12-i0)$ are non-tangential boundary values of holomorphic functions $M_1(x+iy,\frac12+i0)\in \mathbb H^2(S)+\mathbb H^\infty_{-1,+}(S), N_1(x+iy,\frac12-i0) \in \mathbb H^2(S)+\mathbb H^\infty_{-1,-}(S)$, with $M_1(\cdot+iy,\frac12+i0)\in L^\infty_{-1,+}, N_1(\cdot+iy,\frac12-i0)\in L^\infty_{-1,-}$ bounded for each $y\in [0,2)$, and 
\begin{align}
M_1(x+iy,\tfrac12+i0)&=1+o(1) & & \text{ as }x\to-\infty,\label{jost prty: M1 1/2 left}\\
M_1(x+iy,\tfrac12+i0)&=-x \int_\R u(s)M_1(s,\tfrac12+i0)~ds+o\left(|x|\right)& & \text{ as } x\to+\infty,\label{jost prty: M1 1/2 right}\\
N_1(x+iy,\tfrac12-i0)&=x\int_\R u(s)N_1(s,\tfrac12-i0)~ds +o(|x|), & & \text{ as }x\to-\infty,\\
N_1(x+iy,\tfrac12-i0)&=1+o(1), & & \text{ as }x\to+\infty,\label{jost prty: N1 1/2 right}
\end{align}
where the implied limits  are uniform for $y$ in compact subsets of $[0,2)$.
\end{lemma}
\begin{proof}
The proof is similar to that of Lemma \ref{lem: anal cont 1}. We can use \eqref{G_R split 1/2} to write \eqref{jost prty: def M} as 
\begin{align}
M_1(x) = &~1-\int_{-\infty}^x (x-s)u(s)M_1(s)~ds + i\left(\frac23\right)\int_{-\infty}^x u(s)M_1(s)~ds\\
&\quad +\mathcal F^{-1}[R_1(\xi,0)\widehat{uM_1}](x).\notag
\end{align}
Differentiate to get 
\begin{align}\label{jost prty: M1 1/2 diff split}
D M_1(x) = &~i\int_{-\infty}^x u(s)M_1(s)~ds +\frac23u(x)M_1(x)\\
&\quad +\mathcal F^{-1}[\xi R_1(\xi,0)\widehat{uM_1}](x).\notag
\end{align}
Note that $ u\in L^1_1\cap L^2_1$, and $M_1\in L^\infty_{1,+}$. It follows that $uM_1\in L^1_{1,-}\cap L^2_{1,-}$. We can argue as in Lemma \ref{lem: anal cont 1} and obtain  from \eqref{jost prty: M1 1/2 diff split} and \eqref{eq: R bound} that $\partial_x M_1\in L^2+L^\infty$. For analyticity, we write $G_L*(uM_1)$ again as \eqref{jost prty: M1 split sub 1}. But for $G_L=G_L(\cdot,\frac12+i0)$, we have
\begin{align}\label{eq: GL * P(uM) k=1/2}
    [G_L*(P_{\le 1}(uM_1))](x)&=-\int_{-\infty}^x (x-s)P_{\le 1}(uM_1)~ds + i\frac{2}{3}\int_{-\infty}^x P_{\le 1}(uM_1)~ds\notag\\
    &\quad +\mathcal F^{-1}[R_1(\xi,0)\chi(\xi)\widehat{uM_1}](x).
\end{align}
Note that $P_{\le 1}(uM_1)=\check{\chi}*(uM_1)$. By Lemma \ref{lem: weighted Holder Young},
\begin{equation}
    \|P_{\le 1}(uM_1)\|_{L^1_{1,-}}\le \|\check\chi\|_{L^1_{1,-}}\|uM_1\|_{L^1_{1,-}}.
\end{equation}
It follows that when $x<0$, 
\begin{equation}
    \left|\int_{-\infty}^x (x-s)P_{\le 1}(uM_1)~ds\right|\le \int_{-\infty}^x|sP_{\le 1}(uM_1)|~ds\le \|P_{\le 1}(uM_1)\|_{L^1_{1,-}}.
\end{equation}
When $x>0$,
\begin{align}
    \left|\int_{-\infty}^x (x-s)P_{\le 1}(uM_1)~ds\right|&\le \int_{-\infty}^0 (x+|s|)|P_{\le 1}(uM_1)|~ds+2x\int_0^x |P_{\le 1}(uM_1)|~ds\notag\\
    &\lesssim (1+x_+)\|P_{\le 1}(uM_1)\|_{L^1_{1,-}}.
\end{align}
As before, we note that $P_{\le 1}(uM_1)(x+iy)$ is entire with bound \eqref{eq: entire P(uM) bound}. Thus the integral terms in \eqref{eq: GL * P(uM) k=1/2} can be extended holomorphically to 
\begin{equation}
    -\int_{-\infty}^{x+iy}(x+iy-s)P_{\le 1}(uM_1)~ds + i\frac23 \int_{-\infty}^{x+iy}P_{\le 1}(uM_1)~ds.
\end{equation}
which is bounded in size by a constant multiple of $1+x_+$. Thus it is in $\mathbb H^\infty_{-1,+}$.

We now turn to the proof of \eqref{jost prty: M1 1/2 left} and \eqref{jost prty: M1 1/2 right}. With the same argument as in Lemma \ref{lem: anal cont 1}, we can rewrite the values of $M_1(x+iy)$ as  
\begin{align}
M_1(x+iy) &= 1+G_L(x+iy)*(uM_1)(x)\notag\\
&=1-\int_{-\infty}^x (x-s)u(s)M_1(s)~ds + i\left(\frac23-y\right)\int_{-\infty}^x u(s)M_1(s)~ds \label{jost prty: M1 1/2 ext}\\
&\qquad +\mathcal F^{-1}[R_1(\xi,y)\widehat{uM_1}](x).\notag
\end{align}
 By the same argument as in Lemma \ref{lem: anal cont 1}, we get that $\mathcal F^{-1}[ R_1(\xi,y)\widehat{uM_1}](x)$ tends to 0 uniformly in $y\in [0,2-\delta]$ as $x\to\pm\infty$. Since $uM_1\in L^1_{1,-}$, both $\int_{-\infty}^x u(s)M_1(s)~ds$ and $\int_{-\infty}^x su(s)M_1(s)~ds$ tend to 0 as $x\to-\infty$. Also, when $x<0$,
$$\left|x\int_{-\infty}^x u(s)M_1(s)~ds\right|\le \int_{-\infty}^x |su(s)M_1(s)|~ds.$$ 
Thus $x\int_{-\infty}^x u(s)M_1(s)~ds$ tends to 0 as $x\to-\infty$. When $x\to \infty$, it remains only to show that $\int_{-\infty}^x su(s)M_1(s)~ds=o(x)$. Let $\epsilon>0$ be given. Since $uM_1\in L^1$, there exists an $x_0>0$ such that $\int_{x_0}^\infty|u(s)M_1(s)|~ds<\frac\epsilon2$. By the previous discussions, $\int_{-\infty}^{x_0}|su(s)M_1(s)|~ds$ is finite, thus there exists an $M>0$ such that for all $x>M$, $\frac1x |\int_{-\infty}^{x_0}su(s)M_1(s)|~ds<\frac\epsilon 2$. On the other hand, for $x>x_0$,
$$\left|\frac1x \int_{x_0}^x su(s)M_1(s)\right|\le \left|\int_{x_0}^x u(s)M_1(s)~ds\right|<\frac\epsilon2.$$ 
It follows that $\frac1x|\int_{-\infty}^x su(s)M_1(s)|~ds<\epsilon$. The proof is complete.
\end{proof}

If we assume more regularity of $u$, we can obtain more regularity and boundedness of the Jost functions near the boundary of $S$.

\begin{lemma}\label{lem: M1 improved regularity}
    Suppose the assumptions on $u$ in Lemma \ref{lem: anal cont 1} and Lemma \ref{lem: anal cont R+} are strengthened to $u\in L^1_1\cap H^s_1$ for some $s>\frac12$. Then the following improvements in the conclusions can be obtained. $\partial_x M_1, \partial_x N_1$, $\partial_x M_e$, $\partial_x N_e\in L^\infty$. $M_1(x,\zeta),N_1(x,\zeta), M_e(x,k\pm i0), N_e(x,k\pm i0)\in \mathbb H^\infty(S)$, and are uniformly continuous on $\overline{S}$. The asymptotic limits as $x\to\pm\infty$ hold uniformly for all $y\in [0,2]$.
\end{lemma}
\begin{proof}
    We only focus on the case in Lemma \ref{lem: anal cont 1}. The improvement on Lemma \ref{lem: anal cont R+} can be obtained in the same way. By \eqref{eq: R bound}, $|\xi R_2(\xi,0)|\le C(\zeta)$. Recall that $\partial_x M_1\in L^2$. Without loss of generality, assume $s\le 1$. We now use the following version of the Gagliardo-Nirenberg inequality
\begin{equation}
\|uM_1\|_{H^s}\le C(\|u\|_{H^s}\|M_1\|_{L^\infty}+\|u\|_{L^\infty}\|D M_1\|_{L^2})
\end{equation}
to conclude that $uM_1\in H^s$. Thus $\widehat{uM_1}\in L^1$. It follows from \eqref{jost prty: D M1 split} that $D M_1\in L^\infty$. 

To obtain uniform continuity of $M_1(x+iy)$ on $S$, we only need to show the last term in \eqref{jost prty: M1 ext} is uniformly continuous. To that end, note that
\begin{equation*}
\left|\int_\R R_2(\xi,y)\widehat{uM_1}(\xi)(e^{ix_1\xi}-e^{ix_2\xi})\right|~d\xi\le C(\zeta)\int_\R |\widehat{uM_1}(\xi)||e^{i(x_1-x_2)\xi}-1|~d\xi,
\end{equation*}
where the integral on the right hand side only depends on $x_1-x_2$ and tends to 0 as $x_1-x_2\to 0$ by the dominated convergence theorem. On the other hand, by \eqref{R diff y},
\begin{align*}
&~\left|\int_\R (R(\xi,y_1)-R(\xi,y_2))\widehat{uM_1}(\xi)e^{ix\xi}\right|\\
\le&~ C(\zeta)\int_\R|\widehat{uM_1}(\xi)|(|y_1-y_2|+|e^{(y_1-y_2)\xi}-1|\chi_{\R^-}(\xi))~d\xi,
\end{align*}
where the latter integral depends only on $y_1-y_2$ and tends to zero as $y_1-y_2\to 0$. The above estimates show that $M_1(x+iy)$ is uniformly continuous in the closed strip $0\le y\le 2$.

It remains to prove \eqref{jost prty: M1 limit right} holds uniformly for $y\in [0,2]$. We can basically repeat the proof in Lemma \ref{lem: anal cont 1}, with the improved estimate $\widehat{uM_1}\in L^1$. Thus we may choose $\varphi\in C_0^\infty$ such that $\|\widehat{uM_1}-\varphi\|_{L^1}<\epsilon$. \eqref{eq: inv Fourier R2 phi - uM} can be improved to
\begin{align}
     |\mathcal F^{-1}[R_2(\xi,y)(\varphi-\widehat{uM_1})(\xi)](x)|&<C(\zeta)\|\widehat{uM_1}-\varphi\|_{L^1}\notag\\
     &< C(\zeta)\epsilon
 \end{align}
 for all $y\in [0,2]$, since $\|R_2(\cdot,y)\|_{L^\infty}<C(\zeta)$ by \eqref{eq: R bound}. The rest follows as in Lemma \ref{lem: anal cont 1}.
\end{proof}

In the same way, we obtain improved regularity of $M_1$ and $N_1$ at $\zeta=\frac{1}{2}\pm i0$.
\begin{lemma}\label{lem: improved regularity k=1/2}
    Suppose the assumptions on $u$ in Lemma \ref{lem: anal cont 1/2} is strengthened to $u\in L^1_1\cap H^s_1$ for some $s>\tfrac12$. Then the following improvements can be given to the conclusions. $M_1(x+iy,\frac12+i0) \in \mathbb H^\infty_{-1,+}(S)$, $N_1(x+iy,\frac12-i0)\in \mathbb H^\infty_{-1,-}(S)$, and are uniformly continuous on $\overline{S}$. The asymptotic limits as $x\to\pm\infty$ hold uniformly for all $y\in [0,2]$.
\end{lemma}

\subsection{Equivalence of integral and differential equations}

We will again establish the equivalence in three steps for different values of $\zeta$: $\zeta\in \mathbb C\setminus \mathbb R^+$, $\zeta\in (\mathbb R^+\pm i0)\setminus \{\frac12\pm i0\}$, $\zeta=\frac12\pm i0$.

\begin{lemma}\label{lem: equiv int diff}
Let $u\in L^1\cap L^2$. Fix $\zeta\in \C\setminus \R^+$. 
\begin{enumerate}[(a)]
\item Let $M_1(\cdot,\zeta), N_1(\cdot,\zeta)\in L^\infty$ be solutions to \eqref{jost prty: def M}, \eqref{jost prty: def N}, and let $M_1(x+iy,\zeta), N_1(x+iy,\zeta)$ be the analytic extensions given in Lemma \ref{lem: anal cont 1}. Then it follows that they satisfy the asymptotic conditions \eqref{jost prty: M1 limit left}, \eqref{jost prty: N1 limit right} and the following differential equations:
\begin{equation}\label{M diff eqn}
D M_1(x+i0,\zeta)-\zeta[M_1(x+i0,\zeta)-M_1(x+i2,\zeta)]=u(x)M_1(x+i0,\zeta),
\end{equation}
\begin{equation}\label{N diff eqn}
D N_1(x+i0,\zeta)-\zeta[N_1(x+i0,\zeta)-N_1(x+i2,\zeta)]=u(x)N_1(x+i0,\zeta),
\end{equation}

\item Let $M_1(x+iy,\zeta), N_1(x+iy,\zeta) \in \mathbb H^2(S)+\mathbb H^\infty(S)$. Assume that $M_1(x+iy,\zeta), N_1(x+iy,\zeta)$ satisfy the differential equations \eqref{M diff eqn}, \eqref{N diff eqn} and the asymptotic conditions 
\begin{align}
    M_1(x+i0,\zeta)&=1+o(1)\quad  & &\text{ as }x\to-\infty,\\
    N_1(x+i0,\zeta)&=1+o(1)\quad  & &\text{ as }x\to+\infty.
\end{align}
Then it follows that $M_1(x+i0,\zeta)$, $N_1(x+i0,\zeta)$ satisfy \eqref{jost prty: def M}, \eqref{jost prty: def N}.
\end{enumerate}
\end{lemma}
\begin{proof}
We again suppress the dependence on $\zeta$, and show only the argument for $M_1$ in the following proof. First assume $M_1(x)\in L^\infty$ solves \eqref{jost prty: def M}. By \eqref{jost prty: M1 ext}, 
\begin{equation}
M_1(x)-M_1(x+i2)=\mathcal F^{-1}\left([R_2(\xi,0)-R_2(\xi,2)]\widehat{uM_1}(\xi)\right)(x).
\end{equation}
Note that
\begin{align}
R_2(\xi,0)-R_2(\xi,2) &= \tfrac{1-e^{-2\xi}}{\xi-\zeta(1-e^{-2\xi})}\notag\\
&=\tfrac1\zeta\tfrac{\zeta(1-e^{-2\xi})-\xi+\xi}{\xi-\zeta(1-e^{-2\xi})}\notag\\
&=\tfrac1\zeta\left(-1+\tfrac{\xi}{\xi-\zeta(1-e^{-2\xi})}\right)\notag\\
&=\tfrac1\zeta\left(-1+\tfrac1{1-2\zeta}+\xi\left[\tfrac{1}{\xi-\zeta(1-e^{-2\xi})}-\tfrac1{1-2\zeta}\tfrac1\xi\right]\right)\notag\\
&=\tfrac1\zeta\left(-1+\tfrac1{1-2\zeta}+\xi R_2(\xi,0)\right).
\end{align}
It follows that 
\begin{align}
\zeta[M_1(x)-M_1(x+i2)]&=\left(\tfrac1{1-2\zeta}-1\right)u(x)M_1(x) + \mathcal F^{-1}[\xi R_2(\xi,0)\widehat{uM_1}](x)\notag\\
&=-u(x)M_1(x) + \frac1i\partial_x M_1(x),
\end{align}
where we have used \eqref{jost prty: D M1 split} for the last step. The proof of part (a) is now complete.

We now consider the reversed direction. Let $M_1(x+iy)$ satisfy the assumptions in part (b). Let $\varphi(\xi)$ be a smooth cutoff function such that $\varphi(\xi)=1$ when $|\xi|\le 1$ and $\varphi(\xi)=0$ when $|\xi|\ge 2$. For $\delta>0$, denote the mollification of a function $f$ by $f^\delta = \varphi(\delta D) f$. Note that 
\begin{equation}
M_1^\delta(x+iy)=\int_\R \tfrac1\delta \check\varphi(\tfrac s \delta)M_1(x+iy-s)~ds
\end{equation}
is still in $\mathbb H^2(S)+\mathbb H^\infty(S)$. By Lemma \ref{lemma:FT.strip}, we have
\begin{equation}\label{M1 anal Fourier}
\mathcal F[M_1^\delta(\cdot+i2)](\xi)=e^{-2\xi}\mathcal F[M_1^\delta](\xi).
\end{equation}
We mollify \eqref{M diff eqn} to get 
\begin{equation}
D M_1^\delta (x) -\zeta[M_1^\delta (x)-M_1^\delta(x+i2)]=(uM_1)^\delta(x).
\end{equation}
Take its Fourier transform and use \eqref{M1 anal Fourier} to get
\begin{equation}
[\xi-\zeta(1-e^{-2\xi})]\widehat{M_1^\delta} = \mathcal F[(uM_1)^\delta].
\end{equation}
Multiply both sides by the $C^\infty$ function $R_2(\xi,0)+\frac1{1-2\zeta}\frac{1}{\xi-i\epsilon}$, where $R_2$ is given as in \eqref{eq: def R2} and $\epsilon>0$. Since
\begin{equation*}
[\xi-\zeta(1-e^{-2\xi})][R_2(\xi,0)+\tfrac1{1-2\zeta}\tfrac{1}{\xi-i\epsilon}]=1+\tfrac{\xi-\zeta(1-e^{-2\xi})}{(1-2\zeta)\xi}\tfrac{i\epsilon}{\xi-i\epsilon},
\end{equation*}
we get
\begin{align}\label{M1 delta Fourier}
\widehat{M_1^\delta} = &~\tfrac{-i\epsilon}{\xi-i\epsilon}
\tfrac{\xi-\zeta(1-e^{-2\xi})}{(1-2\zeta)\xi}\varphi(\delta \xi)\widehat{M_1}\\
&~\quad + [R_2(\xi,0)+\tfrac1{1-2\zeta}\tfrac{1}{\xi-i\epsilon}]\mathcal F[(uM_1)^\delta].\notag
\end{align}
Note that even with the presence of the $e^{-2\xi}$ terms, both sides are still tempered distributions because they are compactly supported. %% This is the advantage of the compactly supported cutoff. This way don't have to combine e^{-2\xi} with \hat{M_1} and only work the Fourier transform of M_1(x+2i). For that reason, we can dispense with the condition on the left limit of M_1(x+2i).
Define $\psi(\xi)$ by
\begin{equation}
\psi(\xi)=\tfrac{\xi-\zeta(1-e^{-2\xi})}{(1-2\zeta)\xi}\varphi(\delta \xi).
\end{equation} 
Note that $\psi\in C_0^\infty$ and $\psi(0)=1$. By \eqref{jost prty: M1 limit left} and the dominated convergence theorem, we obtain
\begin{equation}\label{mllfd M1 left lim}
\lim_{x\to-\infty} (\check\psi*M_1)(x) = \int_\R\check \psi(s)~ds=\psi(0)=1.
\end{equation}
Let us now take the inverse Fourier transform of \eqref{M1 delta Fourier} to get
\begin{align}\label{M1 delta Fourier 2}
M_1^\delta = \mathcal F^{-1}\left[\tfrac{-i\epsilon}{\xi-i\epsilon}
\psi\widehat{M_1}\right]
+\mathcal F^{-1}\left([R_2(\xi,0)+\tfrac1{1-2\zeta}\tfrac{1}{\xi-i\epsilon}]\mathcal F[(uM_1)^\delta]\right).
\end{align}
A direct calculation gives
\begin{equation}\label{FT of 1/xi}
\mathcal F^{-1}\left(\tfrac{1}{\xi-i\epsilon}\right)(x)=i\chi_{\R^+}(x)e^{-\epsilon x}.
\end{equation}
Thus \eqref{M1 delta Fourier 2} can be written as 
\begin{align}\label{M1 delta pre int}
M_1^\delta(x) = &~\int_0^{\infty} \epsilon e^{-\epsilon s}(\check\psi*M_1)(x-s)~ds\\
&~+\frac{i}{1-2\zeta} \int_{-\infty}^x e^{-\epsilon(x-s)}(uM_1)^\delta(s)~ds+\mathcal F^{-1}[R_2(\xi,0)\mathcal F[(uM_1)^\delta]](x).\notag
\end{align}
The first line of \eqref{M1 delta pre int} can be written as
\begin{equation}\label{M1 delta pre int 1}
\int_0^\infty e^{-s}(\check\psi*M_1)(x-\tfrac s\epsilon)~ds
\end{equation}
Since $\check\psi*M_1(x)$ is bounded and tends to 1 as $x\to-\infty$ (see \eqref{mllfd M1 left lim}), we can use the dominated convergence theorem to conclude that \eqref{M1 delta pre int 1} converges to 1 for every $x$ as $\epsilon\searrow 0$. Since $(uM_1)^\delta\in L^1\cap L^2$, we again apply the dominated convergence theorem to conclude that as $\epsilon\searrow 0$, the second line of \eqref{M1 delta pre int} converges for every $x$ to 
$$\frac i{1-2\zeta}\int_{-\infty}^x (uM_1)^\delta (s)~ds + [R_2(\xi,0)\mathcal F[(uM_1)^\delta]](x)=G_L*(uM_1)^\delta(x),$$
where we have used the splitting formula \eqref{G_L split 2}. As a result of the above limits, we obtain
\begin{equation}
M_1^\delta (x)= 1+G_L*(uM_1)^\delta(x).
\end{equation}
We now send $\delta$ to 0. By \eqref{M diff eqn}, $\partial_xM_1\in L^2+L^\infty$. Thus $M_1$ is uniformly continuous in addition to being bounded. Thus $M_1^\delta$ converges to $M_1$ uniformly. Recall that $uM_1\in L^1\cap L^2$. Thus $(uM_1)^\delta$ converges to $uM_1$ in $L^1\cap L^2$. By the splitting \eqref{G_L split 2}, and the fact that $R_2(\xi,0)\in L^2$ (see \eqref{eq: R bound}), we obtain that $G_L*(uM_1)^\delta$ converges to $G_L*(uM_1)$ uniformly. We have now obtained \eqref{jost prty: def M}.
\end{proof}

\begin{lemma}\label{lem: equiv int diff on R+}
Let $u\in L^1\cap L^2$. Fix $k\in \R^+$, $k\ne \frac12$. 
\begin{enumerate}[(a)]
\item Let $M_1(\cdot,k+i0), N_1(\cdot,k-i0), M_e(\cdot,k+i0), N_e(\cdot,k-i0) \in L^\infty$ be solutions to \eqref{jost prty: def M}, \eqref{jost prty: def N}, \eqref{jost prty: def M_e}, \eqref{jost prty: def N_e} and let $M_1(x+iy,k+i0), N_1(x+iy,k-i0), M_e(x+iy,k+i0),N_e(x+iy,k-i0)$ be the analytic extensions given in Lemma \ref{lem: anal cont R+}. Then it follows that they satisfy the asymptotic conditions \eqref{eq: M1 k+i0 limit left}, \eqref{eq: N1 k-i0 limit right}, \eqref{eq: Me k+i0 limit left}, \eqref{eq: Ne k-i0 limit right}, and the differential equations \eqref{M diff eqn}, \eqref{N diff eqn} and
\begin{align}
&D M_e(x+i0,k+i0)-k[M_e(x+i0,k+i0)-M_e(x+i2,k+i0)]\notag\\
& =u(x)M_e(x+i0,k+i0),\label{M_e diff eqn}
\end{align}
\begin{align}
&D N_e(x+i0,k-i0)-k[N_e(x+i0,k-i0)-N_e(x+i2,k-i0)]\notag\\
& =u(x)N_e(x+i0,k-i0).\label{N_e diff eqn}
\end{align}

\item Let $M_1(x+iy,k+i0), N_1(x+iy,k-i0), M_e(x+iy,k+i0), N_e(x+iy,k-i0)$ be functions in $\mathbb H^2(S)+\mathbb H^\infty(S)$. Assume that $M_1(x+iy,k+i0), N_1(x+iy,k-i0), M_e(x+iy,k+i0), N_e(x+iy,k-i0)$ satisfy the differential equations \eqref{M diff eqn}, \eqref{N diff eqn}, \eqref{M_e diff eqn},  \eqref{N_e diff eqn}, together with the asymptotic conditions
\begin{align}
    M_1(x+i0,k+i0)&= 1+o(1) & & \text{ as }x\to-\infty,\\
    N_1(x+i0,k-i0)&= 1+o(1) & & \text{ as }x\to+\infty,\\
    M_e(x+i0,k+i0)&= e^{i\lambda x} + o(1) & &\text{ as }x\to-\infty,\\
    N_e(x+i0,k-i0)&= e^{i\lambda x} +o(1) & & \text{ as }x\to+\infty,
\end{align}
Then it follows that $M_1(x+i0,k+i0)$, $N_1(x+i0,k-i0)$, $M_e(x+i0,k+i0)$, $N_e(x+i0,k-i0)$ satisfy \eqref{jost prty: def M}, \eqref{jost prty: def N}, \eqref{jost prty: def M_e}, \eqref{jost prty: def N_e}.
\end{enumerate}
\end{lemma}
\begin{remark}\label{rmk: diff int equiv k-i0}
     We will not need this in the following, but the same results hold for $M_1(\cdot, k-i0)$, $N_1(\cdot, k+i0)$, $M_e(\cdot, k-i0)$, $N_e(\cdot,k+i0)$, where the asymptotic conditions in part (b) are given by 
\begin{align}
M_1(x+i0,k-i0)&=1-b(k-i0) e^{i\lambda x}+o(1) & &\text{ as }x\to-\infty, \label{M1 left limit R-i0}\\
N_1(x+i0,k+i0)&=1+\breve{b}(k+i0)e^{i\lambda x}+o(1)& & \text{ as }x\to+\infty,\\
M_e(x+i0,k-i0)&=[1-b_e(k-i0)]e^{i\lambda x}+o(1) & & \text{ as }x\to-\infty, \\
N_e(x+i0,k+i0)&=[1+\breve b_e(k+i0)]e^{i\lambda x}+o(1)& & \text{ as }x\to+\infty,
\end{align}
where $b(k-i0)$, $\breve b(k+i0)$, $ b_e(k-i0)$, $\breve b_e(k+i0)$ are given by \eqref{def: b}, \eqref{def: breve b}, \eqref{def: breve b},  \eqref{def: breve b_e}.
\end{remark}
\begin{proof}
We give a proof analogous to Lemma \ref{lem: equiv int diff}. As before, we focus mainly on $M_1$ and will suppress $k+i0$ in the equations. Recall that $k>0$, $k\ne \frac12$. Assume $M_1(x) \in L^\infty$ solves \eqref{jost prty: def M}. By \eqref{jost prty: M1 ext R+}, we have
\begin{align}
M_1(x)-M_1(x+i2)=&~\frac{ie^{i\lambda x}(1-e^{-2\lambda})}{1-2k^*}\int_{-\infty}^x u(s)M_1(s)e^{-i\lambda s}~ds\\
&\quad + \mathcal F^{-1}\left([R_1(\xi,0)-R_1(\xi,2)]\widehat{uM_1}(\xi)\right)(x).
\end{align}
A direct calculation as before shows
\begin{align}
R_1(\xi,0)-R_1(\xi,2) &= R_2(\xi,0)-R_2(\xi,2) -\tfrac{1-e^{-2\lambda} }{1-2k^*}\tfrac1{\xi-\lambda}\notag\\
&=\tfrac1k\left(\tfrac1{1-2k}+\tfrac1{1-2k^*}-1\right)+\tfrac1{k}\xi R_1(\xi,0).\notag
\end{align}
It follows that
\begin{align*}
k [M_1(x)-M_1(x+i2)]=&~\frac{i\lambda e^{i\lambda x}}{1-2k^*}\int_{-\infty}^x u(s)M_1(s)e^{-i\lambda s}~ds\\
&\quad + \left(\tfrac1{1-2k}+\tfrac1{1-2k^*}-1\right)u(x)M_1(x)\\
&\quad +\mathcal F^{-1}[\xi R_1(\xi,0)\widehat{uM_1}(\xi)](x)\\
=&D M_1(x)-u(x)M_1(x),
\end{align*}
where we have used \eqref{jost prty: D M1 R+} in the last step.

We now consider the reverse direction and assume $M_1(x+iy)$ satisfy all the assumptions of part (b). In other words, the limit at $-\infty$ is given by \eqref{eq: M1 k+i0 limit left}. We mollify \eqref{M diff eqn} and take the Fourier transform as in Lemma \ref{lem: equiv int diff} to obtain 
\begin{equation}
[\xi-k(1-e^{-2\xi})]\widehat{M_1^\delta} = \mathcal F[(uM_1)^\delta].
\end{equation}
Multiply by the $C^\infty$ function $R_1(\xi,0)+\frac1{1-2k}\frac1{\xi-i\epsilon}+\frac1{1-2k^*}\frac{1}{\xi-\lambda-i\epsilon}$. Noting that
\begin{align}
&~[\xi-k(1-e^{-2\xi})][R_1(\xi,0)+\tfrac1{1-2k}\tfrac{1}{\xi-i\epsilon}+\tfrac1{1-2k^*}\tfrac{1}{\xi-\lambda-i\epsilon}]\varphi(\delta\xi)\notag\\
=&~\left[1+\tfrac{i\epsilon}{\xi-i\epsilon}\tfrac{\xi-k(1-e^{-2\xi})}{(1-2k)\xi}+\tfrac{i\epsilon}{\xi-\lambda-i\epsilon}\tfrac{\xi-k(1-e^{-2\xi})}{(1-2k^*)(\xi-\lambda)}\right]\varphi(\delta\xi)\notag\\
=&~\varphi(\delta \xi)+\tfrac{i\epsilon}{\xi-i\epsilon}\psi_1(\xi)+\tfrac{i\epsilon}{\xi-\lambda-i\epsilon}\psi_2(\xi),\label{def: psi_j}
\end{align}
where $\psi_1$ and $\psi_2$ are $C_0^\infty$ functions defined as above so that $\psi_1(0)=1$, $\psi_2(0)=0$, we get 
\begin{align}
\widehat{M_1^\delta}=&~-\tfrac{i\epsilon}{\xi-i\epsilon}\psi_1 \widehat{M_1} -\tfrac{i\epsilon}{\xi-\lambda-i\epsilon}\psi_2 \widehat{M_1}\\
& \quad +[R_1(\xi,0)+\tfrac1{1-2k}\tfrac{1}{\xi-i\epsilon}+\tfrac1{1-2k^*}\tfrac{1}{\xi-\lambda-i\epsilon}]\mathcal F[(uM_1)^\delta].\notag
\end{align}
Using \eqref{FT of 1/xi} and
\begin{equation}
\mathcal F^{-1}\left(\tfrac1{\xi-\lambda-i\epsilon}\right) = i\chi_{\R^+}(x)e^{-\epsilon x}e^{i\lambda x},
\end{equation}
we get as in Lemma \ref{lem: equiv int diff} that 
\begin{align}\label{M1 delta int eq with e}
M_1^\delta(x) = &~\int_0^\infty  e^{- s}(\check{\psi_1}*M_1)(x-\tfrac{s}\epsilon)~ds + \int_0^\infty  e^{-s}e^{i\lambda \tfrac s\epsilon}(\check{\psi_2}*M_1)(x-\tfrac s\epsilon)~ds\\
&\quad + \frac{i}{1-2k}\int_{-\infty}^x e^{-\epsilon (x-s)}(uM_1)^\delta(s)~ds+ \frac{i}{1-2k^*}\int_{-\infty}^x e^{-\epsilon (x-s)}(uM_1)^\delta(s)e^{i\lambda (x-s)}~ds\notag\\
&\quad +\mathcal F^{-1}\left(R_1(\xi,0)\mathcal F[(uM_1)^\delta]\right).\notag
\end{align}
Passing to pointwise limit as $\epsilon\searrow 0$ and arguing as in Lemma \ref{lem: equiv int diff}, we get 
\begin{align*}
M_1^\delta(x)= &~1+ \frac{i}{1-2k}\int_{-\infty}^x (uM_1)^\delta(s)~ds+ \frac{i}{1-2k^*}\int_{-\infty}^x (uM_1)^\delta(s)e^{i\lambda (x-s)}~ds\\
&\quad +\mathcal F^{-1}\left(R_1(\xi,0)\mathcal F[(uM_1)^\delta]\right)(x)\\
=&~1+G_L*(uM_1)^\delta(x)
\end{align*}
Now take the limit as $\delta\searrow 0$, and we get \eqref{jost prty: def M} as before.

We briefly comment on how the above argument can be adapted to treat the case of $M_e$. The proof can basically be repeated, except that we need to rewrite $M_e(x) = M_e(x)-e^{i\lambda x}+e^{i\lambda x}$ and  
\begin{align}
\check{\psi_1}*M_e(x) &= \check{\psi_1}*(M_e-e^{i\lambda\cdot})+\psi_1(\lambda)e^{i\lambda x}=\check{\psi_1}*(M_e-e^{i\lambda\cdot}),\\
\check{\psi_2}*M_e(x) &= \check{\psi_2}*(M_e-e^{i\lambda\cdot})+\psi_2(\lambda)e^{i\lambda x}=\check{\psi_2}*(M_e-e^{i\lambda\cdot})+e^{i\lambda x}.
\end{align}
Since $M_e(x)-e^{i\lambda x}\to 0$ as $x\to -\infty$, we can repeat the argument above to get the desired result.

We finally turn to the case of $M_1(\cdot,k-i0)$ in Remark \ref{rmk: diff int equiv k-i0}. Assume $M(x+iy,k-i0)$ satisfies all assumptions of part (b) including \eqref{M1 left limit R-i0}. We have been suppressing the dependence on $k$ in the above proofs, but in this case, the proof will need to go through the difference of convolutions with $G_L(x,k\pm i0)$, thus in the following we will emphasize the functional dependence on $k\pm i0$ when appropriate. We first obtain \eqref{M1 delta int eq with e} for $M_1(x)=M_1(x,k-i0)$ using the same calculation as before. In this case, however, the limit of the first line of \eqref{M1 delta int eq with e} needs to be calculated differently as we have different end behavior \eqref{M1 left limit R-i0} for $M_1(x,k-i0)$. For $j=1,2$, and $b=b(k-i0)$, we write
\begin{align}
\check{\psi_j}*M_1(x ,k-i0)&=\check{\psi_j}*[M_1(x ,k-i0)+be^{i\lambda x}-be^{i\lambda x}]\notag\\
& = \check{\psi_j}*[M_1(x ,k-i0)+be^{i\lambda x}]-b\psi_j(\lambda)e^{i\lambda x}.
\end{align} 
By their definitions in \eqref{def: psi_j}, $\psi_1(\lambda)=0$, $\psi_2(\lambda)=1$. Thus the first line of \eqref{M1 delta int eq with e} can be written as 
\begin{align}\label{left limit pre 1}
&~ \int_0^\infty  e^{- s}(\check{\psi_1}*[M_1(\cdot)+be^{i\lambda\cdot}])(x-\tfrac{s}\epsilon)~ds \\
&\quad + \int_0^\infty  e^{-s}e^{i\lambda \tfrac s\epsilon}(\check{\psi_2}*[M_1(\cdot)+be^{i\lambda\cdot}])(x-\tfrac s\epsilon)~ds\notag\\
&\quad -b(k-i0)e^{i\lambda x}.\notag
\end{align}
Since $M_1(x,k-i0)+b(k-i0)e^{i\lambda x}\to 1$ as $x\to-\infty$, we can use the same argument as above to conclude that the sum of the first two terms in \eqref{left limit pre 1} tends to 1 pointwise as $\epsilon\searrow 0$. We thus obtain
\begin{align*}
M_1^\delta(x,k-i0)= &~1-b(k-i0)e^{i\lambda x}\\
&\quad+ \frac{i}{1-2k}\int_{-\infty}^x (uM_1)^\delta(s)~ds+ \frac{i}{1-2k^*}\int_{-\infty}^x (uM_1)^\delta(s)e^{i\lambda (x-s)}~ds\\
&\quad +\mathcal F^{-1}\left(R_1(\xi,0)\mathcal F[(uM_1)^\delta]\right)(x)\\
=&~1-b(k-i0)e^{i\lambda x}+G_L(\cdot,k+i0)*(uM_1)^\delta(x).
\end{align*}
Take the limit as $\delta\searrow 0$ to get
\begin{equation}\label{M1 int alternate G_L}
M_1(x,k-i0)=1-b(k-i0)e^{i\lambda x}+G_L(\cdot,k+i0)*(uM_1(\cdot,k-i0))(x).
\end{equation}
By definition, we have $G_L(x,k+i0)=G_L(x,k-i0)+\frac{i}{1-2k^*}e^{i\lambda x}$. \eqref{jost prty: def M} now follows from \eqref{M1 int alternate G_L} if we recall the definition of $b(k-i0)$ in \eqref{def: b}.
\end{proof}

To prepare for proving the last case of the equivalence, we need a lemma for the convergence of the mollified functions to the original in weighted spaces. 

\begin{lemma}\label{lem: approx id weighted}
Let $M_1\in L^\infty_{-1,+}$, $ u\in L^1_{1}$. Let $f^\delta=\varphi(\delta D)f$ be the mollification of $f$ as in Lemma \ref{lem: equiv int diff}. Then $(M_1)^\delta\in L^\infty_{-1,+}$, $(uM_1)^\delta\in L^1_{1,-}$ and $\|(uM_1)^\delta-uM_1\|_{L^1_{1,-}}\to 0$ as $\delta\searrow 0$. If $M_1$ is continuous, then $M_1^\delta$ converges to $M_1$ uniformly on compact sets.
\end{lemma}
\begin{proof}
By Lemma \ref{lem: weighted Holder Young}, we get
$\|M_1^\delta\|_{L^\infty_{-1,+}}\le \|M_1\|_{L^\infty_{-1,+}}\|\check\varphi\|_{L^1_{1,-}}$, $\|(uM_1)^\delta\|_{L^1_{1,-}}\le \|uM_1\|_{L^1_{1,-}}\|\check\varphi\|_{L^\infty_{1,-}}$.
For any $A>0$, $M_1$ is uniformly continuous on $[-A-1,A+1]$. Thus for every $\epsilon>0$, there exists an $s_0>0$ such that $|M_1(x)-M_1(x-s)|<\epsilon$ for all $|s|<s_0$, $x\in [-A,A]$. We have for $x\in [-A,A]$
\begin{align*}
|M_1^\delta(x)-M_1(x)|&\le \int_{|s|<s_0}\tfrac1\delta \left|\check\varphi(\tfrac s\delta)\right||M(x-s)-M(x)|~ds \\
&\quad +  \int_{|s|\ge s_0}\tfrac1\delta \left|\check\varphi(\tfrac s\delta)\right||M(x-s)-M(x)|~ds \\
&\le \epsilon+C\int_{|s|\ge s_0}\tfrac1\delta \left|\check\varphi(\tfrac s\delta)\right|[1+(x-s)_++x_+]~ds\\
&\le \epsilon+C\int_{|s|\ge s_0}\tfrac1\delta \left|\check\varphi(\tfrac s\delta)\right|[1+|s|+2A]~ds\\
&\le  \epsilon+C\int_{|s|\ge \frac {s_0}\delta}\left|\check\varphi( s)\right|[1+2A+|s|]~ds.
\end{align*}
Since $\check\varphi$ is in Schwartz class, the integral above is less than $\epsilon$ for $\delta$ sufficiently small. This proves all the assertions on $M_1$.
Since $uM_1\in L^1_{-1,-}$, we have $(1+x_-)(uM_1)(x)\in L^1$. For every $\epsilon>0$, there exists an $s_0>0$ such that for all $|s|<s_0$,
\begin{equation}
\|[1+(x-s)_-](uM_1)(x-s)-(1+x_-)(uM_1)(x)\|_{L^1}<\epsilon.
\end{equation}
We have for $\delta$ small
\begin{align}
&~\|(1+x_-)[(uM_1)^\delta(x)-(uM_1)(x)\|_{L^1}\notag\\
\le &~\int_{|s|<s_0}\tfrac1\delta\left| \check\varphi(\tfrac s\delta)\right|\int_\R (1+x_-)|(uM_1)(x-s)-(uM_1)(x)|~dx~ds\notag\\
&\quad \int_{|s|\ge s_0}\tfrac1\delta\left| \check\varphi(\tfrac s\delta)\right|\int_\R (1+x_-)|(uM_1)(x-s)-(uM_1)(x)|~dx~ds\notag\\
:=&~I_1+I_2.
\end{align}
We estimate the above two integrals as follows. Writing $$1+x_-=[1+(x-s)_-]+[x_--(x-s)_-]$$ and by the elementary inequality 
\begin{equation}
|(x-s)_--x_-|\le |s|
\end{equation}
we have
\begin{align}
I_1\le &~\int_{|s|<s_0}\tfrac1\delta\left| \check\varphi(\tfrac s\delta)\right|\int_\R |[1+(x-s)_-](uM_1)(x-s)-(1+x_-)(uM_1)(x)|~dx~ds\notag\\
&\quad+\int_{|s|<s_0}\tfrac1\delta\left| \check\varphi(\tfrac s\delta)\right|\int_\R |(x-s)_--x_-||(uM_1)(x-s)|~dx~ds\notag\\
=&~\int_{|s|<s_0}\tfrac1\delta\left| \check\varphi(\tfrac s\delta)\right|\big\|[1+(x-s)_-](uM_1)(x-s)-(1+x_-)(uM_1)(x)\big\|_{L^1}~ds\notag\\
&\quad+\int_{|s|<s_0}\tfrac1\delta\left| \check\varphi(\tfrac s\delta)\right| |s|\|uM_1\|_{L^1}~ds\notag\\
\le &~\epsilon + s_0\|uM_1\|_{L^1}\notag
\end{align}
Decreasing $s_0$ if necessary to make $s_0\|uM_1\|_{L^1}<\epsilon$, we get $I_1\le 2\epsilon$. We can estimate $I_2$ in a similar way and get 
\begin{align}
I_2\le &~2\|(1+x_-)(uM_1)(x)\|_{L^1}\int_{|s|\ge s_0}\tfrac1\delta\left| \check\varphi(\tfrac s\delta)\right| ~ds\notag\\
&\quad+\|uM_1\|_{L^1}\int_{|s|\ge s_0}\tfrac1\delta\left| \check\varphi(\tfrac s\delta)\right| |s|~ds\notag\\
&\le C \int_{|s|\ge \frac{s_0}\delta }\left| \check\varphi(s)\right| ~ds+C\delta \int_{|s|\ge \frac{s_0}\delta }\left| s\check\varphi( s)\right|~ds
\end{align}
We see that $I_2<\epsilon$ if $\delta$ is sufficiently small.
\end{proof}

We now return to the proof of the last case of the equivalence of the differential and integral equations. Once again, we require stronger decay on $u$ for this case.

\begin{lemma}\label{lem: equiv diff int eq 1/2}
Let $ u \in L^1_1\cap L^2_1$. 
\begin{enumerate}[(a)]
\item Let $M_1(\cdot,\frac12+i0)\in L^\infty_{-1,+}, N_1(\cdot,\zeta)\in L^\infty_{-1,-} $ be solutions to \eqref{jost prty: def M}, \eqref{jost prty: def N} with $\zeta=\frac12\pm i0$ respectively, and let $M_1(x+iy,\frac12+i0), N_1(x+iy,\frac12-i0)$ be the analytic extensions given in Lemma \ref{lem: anal cont 1/2}. Then it follows that they satisfy the asymptotic conditions \eqref{jost prty: M1 1/2 left}, \eqref{jost prty: N1 1/2 right} and differential equations \eqref{M diff eqn}, \eqref{N diff eqn} respectively.

\item Let $M_1(x+iy,\frac12+i0) \in \mathbb H^2(S)+\mathbb H^\infty_{-1,+}(S)$, $N_1(x+iy,\frac12-i0)\in \mathbb H^2(S)+\mathbb H^\infty_{-1,-}(S)$. Assume that $M_1(x+iy,\frac12+i0), N_1(x+iy,\frac12-i0)$ satisfy the differential equations \eqref{M diff eqn}, \eqref{N diff eqn} for $\zeta=\frac12\pm i0$ and the asymptotic conditions
\begin{align}
M_1(x+i0,\tfrac12+i0)&=1+o(1), & & \text{ as }x\to-\infty,\label{equiv M1 left limit 1/2}\\
N_1(x+i0,\tfrac12-i0)&=1 +o(1), & & \text{ as }x\to+\infty.
\end{align}
Then it follows that $M_1(x+i0,\frac12+i0)$, $N_1(x+i0,\frac12-i0)$ satisfy \eqref{jost prty: def M}, \eqref{jost prty: def N}. 
\end{enumerate}
\end{lemma}
\begin{proof}
In the following proof, we only consider the function $M_1(x)$, suppressing the dependence on $\zeta=\frac12+i0$. If $M_1(x)\in L^\infty_{-1,+}$ solves \eqref{jost prty: def M}, by \eqref{jost prty: M1 1/2 ext}, 
\begin{align}
M_1(x)-M_1(x+i2)=&~2i\int_{-\infty}^x u(s)M_1(s)~ds\\
&\quad +\mathcal F^{-1}\left([R_1(\xi,0)-R_1(\xi,2)]\widehat{uM_1}\right)(x).\notag
\end{align} 
We compute 
\begin{align}
R_1(\xi,0)-R_1(\xi,2)&=\frac{1-e^{-2\xi}}{\xi-\frac12(1-e^{-2\xi})}-\frac2\xi\\
&=2\xi R(\xi,0)-\frac23.\notag
\end{align}
It follows that
\begin{align}
\frac12[M_1(x)-M_1(x+i2)]=&~i\int_{-\infty}^x u(s)M_1(s)~ds-\frac13 u(x)M_1(x)\\
&\quad +\mathcal F^{-1}[\xi R_1(\xi,0)\widehat{uM_1}](x)\notag\\
=&~DM_1(x)-u(x)M_1(x).\notag
\end{align}
Here we used \eqref{jost prty: M1 1/2 diff split} in the last step.

In the reverse direction, we assume $M_1(x+iy)$ satisfies all the assumptions of part (b). We mollify and take the Fourier transform of \eqref{jost prty: def M} as in Lemma \ref{lem: equiv int diff} to get
\begin{equation}\label{M1 1/2 Fourier}
[\xi-\tfrac12(1-e^{-2\xi})]\widehat{M_1^\delta}=\mathcal F[(uM_1)^\delta].
\end{equation}
Multiply \eqref{M1 1/2 Fourier} by the $C^\infty$ function $R_1(\xi,0)+\frac1{(\xi-i\epsilon)^2}+\frac23\frac{1}{\xi-i\epsilon}$. Note that
\begin{align}
&~[\xi-\tfrac12(1-e^{-2\xi})][R_1(\xi,0)+\tfrac1{(\xi-i\epsilon)^2}+\tfrac23\tfrac{1}{\xi-i\epsilon}]\varphi(\delta\xi) \notag\\
=&~\left[1+\tfrac{2i\epsilon \xi+\epsilon^2}{(\xi-i\epsilon)^2}\tfrac{\xi-\frac12(1-e^{-2\xi})}{\xi^2}+\tfrac{i\epsilon}{\xi-i\epsilon}\tfrac{2}{3}\tfrac{\xi-\frac12(1-e^{-2\xi})}{\xi}\right]\varphi(\delta\xi)\notag\\
=&~\varphi(\delta\xi)+\tfrac{2i\epsilon \xi+\epsilon^2}{(\xi-i\epsilon)^2}\psi_1(\xi)+\tfrac{i\epsilon}{\xi-i\epsilon}\psi_2(\xi).
\end{align}
Here $\psi_1$ and $\psi_2$ are $C_0^\infty$ functions defined by the above equations. We can easily check that $\psi_1(0)=1$, $\psi_2(0)=0$. It follows that 
\begin{align}
\widehat{M_1^\delta}=&~-\tfrac{2i\epsilon \xi+\epsilon^2}{(\xi-i\epsilon)^2}\psi_1 \widehat{M_1}-\tfrac{i\epsilon}{\xi-i\epsilon}\psi_2\widehat{M_1}\\
&\quad + [R_1(\xi,0)+\tfrac1{(\xi-i\epsilon)^2}+\tfrac23\tfrac{1}{\xi-i\epsilon}]\mathcal F[(uM_1)^\delta].\notag
\end{align}
By \eqref{FT of 1/xi} and 
\begin{equation}
\mathcal F^{-1}\left[\tfrac1{(\xi-i\epsilon)^2}\right](x) = -\chi_{\R^+}(x)x e^{-\epsilon x},
\end{equation}
\begin{equation}
\mathcal F^{-1}\left[\tfrac{\xi}{(\xi-i\epsilon)^2}\right](x)=i\chi_{\R^+}(x)\left(e^{-\epsilon x}-\epsilon x e^{-\epsilon x}\right),
\end{equation}
we get
\begin{align}\label{M1 delta 1/2 int}
M_1^\delta(x) = &~\int_0^\infty \epsilon \left(2e^{\epsilon s}-\epsilon s e^{-\epsilon s}\right)(\check{\psi_1}*M_1)(x-s)~ds+ \int_0^\infty \epsilon e^{-\epsilon s} (\check{\psi_2}*M_1)(x-s)~ds\\
&\quad -\int_{-\infty}^x (x-s)e^{-\epsilon(x-s)}(uM_1)^\delta(s)~ds +\frac{2i}3\int_{-\infty}^x e^{-\epsilon (x-s)}(uM_1)^\delta(s)~ds\notag\\
&\quad +\mathcal F^{-1}\left(R_1(\xi,0)\mathcal F[(uM_1)^\delta]\right)(x).\notag
\end{align}
For any $A\in \R$, if $x\le A$, $s\in \R$, we have $(x-s)_+\le A-s$. Thus for $j=1,2$,
\begin{equation}
|\check{\psi_j}(s)M_1(x-s)|\le C|\check{\psi_j}(s)|(1+|A|+|s|).
\end{equation}
It follows that $(\check{\psi_j}*M_1)(x)$ is bounded for $x\le A$. By \eqref{equiv M1 left limit 1/2} and the dominated convergence theorem,
\begin{equation}
\lim_{x\to-\infty}(\check{\psi_j}*M_1)(x)=\int_\R \check{\psi_j}(s)~ds=\psi_j(0).
\end{equation}
The first line on the right hand side of \eqref{M1 delta 1/2 int} can be written as 
\begin{equation}
\int_0^\infty  \left(2e^{ s}- s e^{- s}\right)(\check{\psi_1}*M_1)(x-\tfrac s\epsilon)~ds+ \int_0^\infty e^{-s} (\check{\psi_2}*M_1)(x-\tfrac s\epsilon)~ds,
\end{equation}
which converges to 
\begin{equation}
\psi_1(0)\int_0^\infty  \left(2e^{ s}- s e^{- s}\right)~ds+\psi_2(0) \int_0^\infty e^{-s} ~ds=1
\end{equation}
as $\epsilon\searrow 0$ by the dominated convergence theorem.
%We also claim that $\langle x\rangle (uM_1)^\delta(x)\in L^1(-\infty, A)$ for every $A\in \R$. In fact, 
%\begin{align*}
%&\int_{-\infty}^A \langle x\rangle \left|\int_\R \widecheck{\varphi(\delta\cdot)}(s)u(x-s)M_1(x-s)~ds\right|~dx\\
%\le &~C\int_{-\infty}^A\int_\R (\langle x-s\rangle +\langle s\rangle) \left|\widecheck{\varphi(\delta\cdot)}(s)\right||u(x-s)|(1+|A|+|s|)~ds~dx\\
%\le &~C\int_\R\int_\R (\langle x-s\rangle +\langle s\rangle) \left|\widecheck{\varphi(\delta\cdot)}(s)\right||u(x-s)|(1+|A|+|s|)~ds~dx\\
%\le &~C\int_\R \langle s\rangle(1+|A|+|s|)\left|\widecheck{\varphi(\delta\cdot)}(s)\right|\int_\R \langle x-s\rangle |u(x-s)|~dx~ds\\
%=&~C\int_\R \langle s\rangle(1+|A|+|s|)\left|\widecheck{\varphi(\delta\cdot)}(s)\right|~ds\int_\R \langle x\rangle |u(x)|~dx\\
%<&~\infty.
%\end{align*}
By Lemma \ref{lem: approx id weighted}, $\langle x\rangle (uM_1)^\delta(x)\in L^1(-\infty, A)$ for every $A\in \R$.
Thus the second line of \eqref{M1 delta 1/2 int} converges to 
$$
-\int_{-\infty}^x (x-s)(uM_1)^\delta(s)~ds +\frac{2i}3\int_{-\infty}^x (uM_1)^\delta(s)~ds
$$
by the dominated convergence theorem. By \eqref{G_L split 1/2} this means
\begin{equation}
M_1^\delta(x) = 1+G_L*(uM_1)^\delta(x).
\end{equation}
We now send $\delta$ to 0. From \eqref{M diff eqn} and the conditions on $u$ and $M_1$, we know that $\partial_x M_1(x)$ is bounded on compact sets. Thus $M_1(x)$ is continuous. By Lemma \ref{lem: approx id weighted}, $M_1^\delta \to M_1$ locally uniformly, $\langle x\rangle (uM_1)^\delta\to \langle x \rangle (uM_1)$ in $L^1(-\infty,A)$ for any $A$. Furthermore, since $uM_1\in L^2(\R)$, we have $(uM_1)^\delta \to uM_1$ in $L^2$. By \eqref{G_L split 1/2} and the fact that $R(\xi,0)\in L^2$, we conclude that $G_L*(uM_1)^\delta\to G_L*(uM_1)$ locally uniformly, and we obtain \eqref{jost prty: def M}.
\end{proof}

                 %%  Properties of Jost Solutions
\section{Existence of Jost solutions}\label{sec: existence}

In this section, we show existence of solutions to the integral equations \eqref{jost prty: def M}, \eqref{jost prty: def N}, \eqref{jost prty: def M_e}, \eqref{jost prty: def N_e} for appropriate ranges of $\zeta$, assuming smallness of $u$ in $L^1_1\cap L^2_1$. In order to describe different regimes of $\zeta$ more efficiently, we introduce the following notations. Let
\begin{equation}
\widetilde \C =\C\setminus {\R^+}\cup(\R^+\pm i0)
\end{equation}
be the cut plane glued to two copies of $\R^+$. For $\delta>0$, let
\begin{equation}
D^\pm _\delta = \left\{w\in \C~\big|~|w-\tfrac12|<\delta, ~\pm\text{Im }w>0\right\}\cup (\tfrac12-\delta\pm i0,\tfrac12+\delta\pm i0)
\end{equation}
be the upper and lower half discs of radius $\delta$ about $\frac12$ glued to the corresponding copy of the diameter.

We first need the following uniform structural estimates on the Green's functions, which follows easily from Proposition \ref{prop: green est}. 

\begin{lemma}\label{lem: green uniform split}
$G_L(x,\zeta)$ and $G_R(x,\zeta)$ can be written structurally as
\begin{equation}\label{eq: green uniform split 1}
G_L(x,\zeta) = G_L^1(x,\zeta)\chi_{\R^+}(x)+G_L^2(x,\zeta)+G_L^3(x,\zeta),
\end{equation}
\begin{equation}
G_R(x,\zeta) = G_R^1(x,\zeta)\chi_{\R^-}(x)+G_R^2(x,\zeta)+G_R^3(x,\zeta),
\end{equation}
where there exists a constant $C>0$ such that
\begin{equation}
\|G_{L,R}^1(\cdot,\zeta)\|_{  L^\infty_{-1}}\le C, \quad \|G_{L,R}^3(\cdot,\zeta)\|_{L^2}\le C\quad \text{ for all }\zeta\in \widetilde{\mathbb C},
\end{equation}
and there are constants $C(\zeta)$ depending on $\zeta$ such that
\begin{equation}
    \|G^1_{L,R}(\cdot,\zeta)\|_{L^\infty}\le C(\zeta)\quad \text{ for all }\zeta\in \widetilde{\mathbb C}\setminus \{\tfrac{1}{2}\pm i0\};
\end{equation}
\begin{equation}
    \|G_L^2(\cdot,\zeta)\|_{L^\infty}\le C(\zeta)\quad \text{ for all }\zeta\in \widetilde{\mathbb C}\setminus \{\tfrac{1}{2}-i0\};
\end{equation}
\begin{equation}
    \|G_R^2(\cdot,\zeta)\|{L^\infty}\le C(\zeta) \quad \text{ for all }\zeta\in \widetilde{\mathbb C}\setminus \{\tfrac{1}{2}+i0\}.
\end{equation}
Here $C(\zeta)$ can be taken to be uniform for $\zeta$ on compact subsets of $\widetilde{\C}\setminus\{\frac12\pm i0\}$.
Furthermore, for every $\delta>0$, there exists a $C>0$ such that 
\begin{enumerate}[(i)]
\item For all $\zeta\in \widetilde\C\setminus D^-_\delta$, we have
\begin{equation}
\|G_L^2(\cdot,\zeta)\|_{L^\infty}\le C, 
\end{equation}

\item For all $\zeta\in \widetilde\C\setminus D^+_\delta$, we have
\begin{equation}
\|G_R^2(\cdot,\zeta)\|_{L^\infty}\le C.
\end{equation}
\end{enumerate}
\end{lemma}
\begin{proof}
We show only the proof for $G_L$, and the proof for $G_R$ will be similar. Let $\epsilon_0>0$ be given as in Proposition \ref{prop: green est}. We first consider $\zeta$ in the range $\text{Re } \zeta\ge -\epsilon_0$, $|\text{Im }\zeta|\le \epsilon_0$. If $\text{Im }\zeta\ge 0$, define
\begin{align}
G_L^1(x,\zeta)&=\left(\frac{i}{1-2\zeta}+\frac{i}{1-2\zeta^*}e^{i\lambda x}\right)\chi_{\R^+}(x),\\
    G_L^2(x,\zeta)&=0,\quad G_L^3(x,\zeta)=H(x,\zeta). 
\end{align}
We have \eqref{eq: green uniform split 1} from Proposition \ref{prop: green est}.
Note that when Im $\zeta\ge 0$, Im $\lambda\ge 0$. Thus $|e^{i\lambda x}|\le 1$ when $x>0$. If $\zeta$ is close to $\frac12$, $\lambda$ is close to $0$, and for $x>0$,
\begin{align}
\left|\frac{i}{1-2\zeta}+\frac{i}{1-2\zeta^*}e^{i\lambda x}\right| &\le \left|\frac{i}{1-2\zeta}+\frac{i}{1-2\zeta^*}\right|+\left|\frac{i}{1-2\zeta^*}(e^{i\lambda x}-1)\right|\\
&\le C\langle x\rangle. 
\end{align}
The estimate on $G_L^3$ follows from \eqref{eq: G est 2}. 
If $\text{Im }\zeta\le 0$, define
\begin{align}
G_L^1(x,\zeta) &= \left(\frac{i}{1-2\zeta}+\frac{i}{1-2\zeta^*}e^{-i\lambda x}\right)\chi_{\R^+}(x),\\
G_L^2(x,\zeta) &= -\frac{i}{1-2\zeta^*}\left[e^{-i\lambda x}\chi_{\R^+}(x)+e^{i\lambda x}\chi_{\R^-}(x)\right],\\
G_L^3(x,\zeta)&=H(x,\zeta).
\end{align}
$G_L^1(x,\zeta)$ can be estimated as before, while $|G_L^2(x,\zeta)|\le \frac{C}{|1-2\zeta|}$. Thus there exists a $C>0$ such that for $\zeta\ne D^-_\delta$, $|G_L^2(x,\zeta)|\le C$. 
Finally, for $\zeta$ with $\dist(\zeta,\mathbb R^+)\ge \epsilon_0$, define
\begin{align}
G_L^1(x,\zeta) &= \frac{i}{1-2\zeta},\quad G_L^2(x,\zeta)=0,\notag\\
G_L^3(x,\zeta)&=G_L(x,\zeta)\chi_{\mathbb R^-}(x)+G_R(x,\zeta)\chi_{\mathbb R^+}(x).
\end{align}
The estimates follow from  \eqref{eq: G est 4}.
\end{proof}

For later estimates, it's convenient to have the following weighted inequalities.

\begin{lemma}\label{lem: weighted Holder Young}
Let $p,q,r\in [1,\infty]$ satisfy $\frac1r=\frac1p+\frac1{q}-1$, then there exists a $C>0$ such that for any measurable functions $f$, $g$ on $\mathbb R$,
    \begin{align}
        \|fg\|_{L^p_{1,-}}&\le C\|f\|_{L^p_1}\|g\|_{L^\infty_{-1,+}},\label{eq: weighted Holder L}\\
        \|fg\|_{L^p_{1,+}}&\le C\|f\|_{L^p_1}\|g\|_{L^\infty_{-1,-}}.\label{eq: weighted Holder R}
    \end{align}
    \begin{align}
        \|f*g\|_{L^r_{-1,+}}\le \|f\|_{L^p_{-1,+}}\|g\|_{L^q_{1,-}},\label{eq: weighted Young R}\\
        \|f*g\|_{L^r_{-1,-}}\le \|f\|_{L^p_{-1,-}}\|g\|_{L^q_{1,+}}.\label{eq: weighted Young L}
    \end{align}
    \begin{align}
        \|f*g\|_{L^r_{1,+}}&\le \|f\|_{L^p_{1,+}}\|g\|_{L^q_{1,+}},\label{eq: weighter Young R 1}\\
        \|f*g\|_{L^r_{1,-}}&\le \|f\|_{L^p_{1,-}}\|g\|_{L^q_{1,-}},
    \end{align}
\end{lemma}
\begin{remark}
    \eqref{eq: weighted Holder L}, \eqref{eq: weighted Holder R} can obviously be generalized to other H\"older indices, but we will not need these generalizations.
\end{remark}
\begin{proof}
    \eqref{eq: weighted Holder L} and \eqref{eq: weighted Holder R} are straightforward. We only show \eqref{eq: weighted Young R} and \eqref{eq: weighter Young R 1}. By the elementary inequality
\begin{equation}\label{eq: elementary ineq 1}
1+(x-y)_+\le (1+x_+)(1+y_-),
\end{equation}
we get
\begin{align}
&~(1+x_+)^{-1}\left|\int_\R f(x-y)g(y)~dy\right|\label{wgt Young 1}\\
\le&~\int_\R\frac{1+(x-y)_+}{(1+x_+)(1+y_-)}| (1+(x-y)_+)^{-1}f(x-y)(1+y_-)g(y)|~dy\notag\\
\le&~\int_\R| (1+(x-y)_+)^{-1}f(x-y)(1+y_-)g(y)|~dy\notag
\end{align}
By the standard Young inequality, The $L^r$ norm of \eqref{wgt Young 1} is bounded by 
$$\left\|(1+x_+)^{-1}f(x)\right\|_{L^p}\left\|(1+x_-)g(x)\right\|_{L^q}=\|f\|_{L^p_{-1,+}}\|g\|_{L^q_{1,-}}.$$
This proves \eqref{eq: weighted Young R}. We can prove \eqref{eq: weighter Young R 1} in a similar way, if we replace \eqref{eq: elementary ineq 1} by
\begin{equation}
    1+x_+\le [1+(x-y)_+](1+y_+).
\end{equation}
\begin{comment}
   \begin{align}
        \frac{1}{1+x_+}|(f\chi_{\R^+})*g|(x)&\le\frac{1}{1+x_+}\int_{-\infty}^x |f(x-s)g(s)|~ds\notag\\
        &\le \frac{\|f\|_{L^\infty_{-1}}\|g\|_{L^1_{1,-}}}{1+x_+}\sup_{s\le x}\frac{\langle x-s\rangle}{1+s_-}\notag\\
        &\lesssim \|f\|_{L^\infty_{-1}}\|g\|_{L^1_{1,-}}\sup_{s\le x}\frac{1+x_++s_-}{(1+x_+)(1+s_-)}\notag\\
        &\lesssim\|f\|_{L^\infty_{-1}}\|g\|_{L^1_{1,-}}.
    \end{align}
\end{comment}
\end{proof}

We define the linear operators $T_L(\zeta)$ for $\zeta\in \widetilde{\C}$, $\zeta\ne \frac12-i0$, and $T_R(\zeta)$ for $\zeta\in \widetilde{\C}$, $\zeta\ne \frac12+i0$ by 
\begin{align}
T_L(\zeta)&: L^\infty_{-1,+}\to L^\infty_{-1,+},\notag\\
T_L(\zeta) m&= G_L(\cdot,\zeta)*(um).\label{def: T_L}
\end{align}
\begin{align}
T_R(\zeta)&: L^\infty_{-1,-}\to L^\infty_{-1,-},\notag\\
T_R(\zeta) m&= G_R(\cdot,\zeta)*(um).
\end{align}
Lemmas \ref{lem: green uniform split} and \ref{lem: weighted Holder Young} guarantees the above mapping space properties provided $u\in L^1_1\cap L^2_1$. With these notations, we can write \eqref{jost prty: def M}, \eqref{jost prty: def N}, \eqref{jost prty: def M_e}, \eqref{jost prty: def N_e} as
\begin{align}
(I-T_L(\zeta))M_1(\cdot,\zeta)&=1,\label{jost exist: def M}\\
(I-T_R(\zeta))N_1(\cdot,\zeta)&=1,\label{jost exist: def N}\\
(I-T_L(\zeta))M_e(\cdot,\zeta)&= e^{i\lambda x},\\
(I-T_R(\zeta))N_e(\cdot,\zeta)&= e^{i\lambda x}.
\end{align}
We can reduce the solvability of the integral equations to invertibility of $I-T_L(\zeta)$ and $I-T_R(\zeta)$. More precisely, we have

\begin{lemma}\label{lem: jost exist 1}
Let $u\in  L^1_1\cap L^2_1$. Suppose $\zeta\in\widetilde{\C}\setminus \{\frac12-i0\}$ and that $I-T_L(\zeta)$ is invertible on $L^\infty_{-1,+}$. Then 
\begin{enumerate}[(i)]
    \item \eqref{jost prty: def M} is uniquely solvable for $M_1(x,\zeta)\in L^\infty_{-1,+}$.
    \item If $\zeta=k\pm i0$, $k>0$ (excluding $\zeta=\frac12-i0$), then \eqref{jost prty: def M_e} is uniquely solvable for $M_e(x,k\pm i0)\in L^\infty_{-1,+}$. 
\end{enumerate}
In each of the above cases, as long as $\zeta\ne \frac12\pm i0$, then we have the improved estimates $M_1(x,\zeta)\in L^\infty$ or $M_e(x,\zeta)\in L^\infty$.

Similarly, suppose $\zeta\in\widetilde{\C}\setminus \{\frac12+i0\}$ and that $I-T_R(\zeta)$ is invertible on $L^\infty_{-1,-}$, then 
\begin{enumerate}[(i)]
    \item \eqref{jost prty: def N} is uniquely solvable for $N_1(x,\zeta)\in L^\infty_{-1,+}$.
    \item If $\zeta=k\pm i0$, $k>0$ (excluding $\zeta=\frac12+i0$), \eqref{jost prty: def N_e} is uniquely solvable for $N_e(x,\zeta)\in L^\infty_{-1,-}$. 
\end{enumerate}
In each of the above cases, as long as $\zeta\ne \frac12-i0$, then we have the improved estimates $N_1(x,\zeta)\in L^\infty$ or $N_e(x,\zeta)\in L^\infty$.
\end{lemma}
\begin{proof}
The unique solvability is obvious. We only need to show the improved estimates. By Lemma \ref{lem: green uniform split}, $G_L^1(x,\zeta), G_L^2(x,\zeta)$ are actually in $L^\infty$ for each fixed $\zeta$, as long as $\zeta\ne \frac12\pm i0$. It follows from Lemma \ref{lem: weighted Holder Young} that 
\begin{equation}\label{eq: T(m) L inf}
    \|T_L(\zeta) m\|_{L^\infty}\le C(\zeta)\|um\|_{L^1\cap L^2}\le C(\zeta)\|u\|_{L^1_1\cap L^2_1}\|m\|_{L^\infty_{-1,+}}.
\end{equation} 
The conclusion follows.
\end{proof}
%\begin{lemma}
%For every $\delta>0$, there exists $\epsilon>0$ such that for all $u\in \langle x\rangle^{-1} (L^1\cap L^2)$ with $\|u\|_{\langle x\rangle^{-1} (L^1\cap L^2)}<\epsilon$, 
%
%\begin{enumerate}[(i)]
%\item If $\zeta\in \widetilde{\C}\setminus D^-_\delta$, $$\|T_L(\zeta)\|_{ \mathcal {L} (L^\infty_{-1,+})}<\frac12;$$
%
%\item If $\zeta\in \widetilde{\C}\setminus D^+_\delta$, $$\|T_R(\zeta)\|_{ \mathcal {L} (L^\infty_{-1,-})}<\frac12.$$
%
%\end{enumerate}
%
%\end{lemma}
%\begin{proof}
%
%\end{proof}

We therefore focus our attention on the invertibility of $I-T_L(\zeta)$ and $I-T_R(\zeta)$. 

\begin{theorem}\label{thm: inv T off disc}
Let $\delta>0$ be given. Then there exists $\epsilon>0$ such that for all $u\in L^1_1\cap L^2_1$ with $\|u\|_{L^1_1\cap L^2_1}<\epsilon$, 
\begin{enumerate}[(i)]
\item If $\zeta\in \widetilde{\C}\setminus D^-_\delta$, then $I-T_L(\zeta)$ is invertible on $L^\infty_{-1,+}$, with $$\left\|\left(I-T_L(\zeta)\right)^{-1}\right\|_{L^\infty_{-1,+}\to L^\infty_{-1,+}}\le 2;$$ 

\item If $\zeta\in \widetilde{\C}\setminus D^+_\delta$, then $I-T_R(\zeta)$ is invertible on $L^\infty_{-1,-}$, with $$\left\|\left(I-T_R(\zeta)\right)^{-1}\right\|_{L^\infty_{-1,-}\to L^\infty_{-1,-}}\le 2.$$

\end{enumerate}

\end{theorem}
\begin{proof}
We only show the proof for $T_L$. Using the splitting formula in Lemma \ref{lem: green uniform split}, we have
\begin{equation*}
[T_L(\zeta)m](x) =[  G_L^1(\cdot,\zeta)\chi_{\mathbb R^+}+G_L^2(\cdot,\zeta)+G_L^3(\cdot,\zeta)]*(um)(x)
\end{equation*}
Let $\zeta\in \widetilde{\C}\setminus D^-_\delta$. By Lemmas \ref{lem: green uniform split} and \ref{lem: weighted Holder Young}, we have
\begin{align}
    \|[G_L^1(\cdot,\zeta)\chi_{\R^+}]*(um)\|_{L^\infty_{-1,+}}&\le C\|G_L(\cdot,\zeta)\|_{L^\infty_{-1}}\|um\|_{L^1_{1,-}}\notag\\
    &\lesssim_\delta \| u\|_{L^1_1}\|m\|_{L^\infty_{-1,+}},
\end{align}
\begin{equation}
\|G_L^2(\cdot,\zeta)*(um)\|_{L^\infty}\le \|G_L^2(\cdot,\zeta)\|_{L^\infty}\|um\|_{L^1}\lesssim_\delta\|u\|_{L^1_1}\|m\|_{L^\infty_{-1,+}}, 
\end{equation}
\begin{equation}
\|G_L^3(\cdot,\zeta)*(um)\|_{L^\infty}\le \|G_L^3(\cdot,\zeta)\|_{L^2}\|um\|_{L^2}\lesssim_\delta \|u\|_{ L^2_1}\|m\|_{L^\infty_{-1,+}}.
\end{equation}
Combining the above estimates, we obtain a constant $C=C(\delta)$ such that for all $\zeta\in \widetilde{\C}\setminus D^-_\delta$
\begin{equation}\label{eq: est T_L via u}
\|T_L(\zeta) m \|_{L^\infty_{-1,+}}\le C\|u\|_{L^1_1\cap L^2_1}\|m\|_{L^\infty_{-1,+}}.
\end{equation}
Thus there exists $\epsilon>0$ such that for all $u$ with $\|u\|_{L^1_1\cap L^2_1}<\epsilon$, $$\|T_L(\zeta)\|_{L^\infty_{-1,+}\to L^\infty_{-1,+}}<\tfrac12.$$ Thus $I-T_L(\zeta)$ is invertible on $L^\infty_{-1,+}$ with the claimed norm bound.
\end{proof}

We next study conditions for invertibility of $I-T_L(\zeta)$ for $\zeta\in D^-_\delta\setminus \{\frac12-i0\}$ and of $I-T_R(\zeta)$ for $\zeta\in  D^+_\delta\setminus \{\frac12+i0\}$.

\begin{lemma}\label{lem: compact T}
Let $u\in L^1_1\cap L^2_1$. Each $T_L(\zeta)$ for $\zeta\in\widetilde{\C}\setminus \{\frac12-i0\}$ is compact on $L^\infty_{-1,+}$; each $T_R(\zeta)$ for $\zeta\in\widetilde{\C}\setminus \{\frac12+i0\}$  is compact on $L^\infty_{-1,-}$.
\end{lemma}
\begin{proof}
We only show the proof for $T_L(\zeta)$. To prepare for our proof, we first claim that in order to show that a sequence $f_n\in L^\infty_{-1,+}$ converges, it suffices to show 
\begin{enumerate}[(a)]
\item $\{f_n\}$ converges uniformly on every compact set.
\item For every $\epsilon>0$, there exist $M>0$, $N\in \mathbb N$, and two functions $A_-\in L^\infty(-\infty,-M)$, $A_+\in L^\infty(M,\infty)$ such that for all $n>N$, $x<-M$, we have $|f_n(x)-A_-(x)|<\epsilon$, and for all $n>N$, $x>M$, we have $|\frac{f_n(x)}x-A_+(x)|<\epsilon$.
\end{enumerate}
In fact, if $f_n$ satisfies the above two conditions, let $f$ be the pointwise limit of $f_n$. Then for every $\epsilon>0$, by taking pointwise limit of condition (b) above, we get that there exists $M>0$ and $L^\infty$ functions $A_\pm(x)$, such that for all $x<-M$, $|f(x)-A_-(x)|\le \epsilon$, and for all $x>M$, $|\frac{f(x)}x-A_+(x)|\le\epsilon$. For $x\in [-M,M]$, there exists $N_1>N$, such that when $n>N_1$, $|f_n(x)-f(x)|<\epsilon$ by condition (a). This implies that $f\in L^\infty_{-1,+}$. Furthermore, There exists $N_1>0$ such that for $n>N_1$, $x<-M$, 
$$\frac1{1+x_+}|f_n(x)-f(x)|=|f_n(x)-f(x)|\le |f_n(x)-A_-(x)|+|f(x)-A_-(x)|< 2\epsilon.$$
For $n>N_1$, $x>M$,
$$\frac1{1+x_+}|f_n(x)-f(x)|\le \frac{|f_n(x)-f(x)|}{x}\le \left|\frac{f_n(x)}x-A_+(x)\right|+\left|\frac{f(x)}x-A_+(x)\right|< 2\epsilon.$$
For $n>N_1$, $x\in [-M,M]$,
$$\frac1{1+x_+}|f_n(x)-f(x)|\le |f_n(x)-f(x)|<\epsilon.$$
In summary, when $n>N_1$, $\|f_n-f\|_{L^\infty_{-1,+}}< 2\epsilon$. Thus $f_n$ converges to $f$ in $L^\infty_{-1,+}$, and the claim is proven.

Let $\{m_n\}$ be a bounded sequence in $L^\infty_{-1,+}$, with $\sup_n\|m_n\|_{L^\infty_{-1,+}}\le C$. We want to show $\{T_L(\zeta) m_n\}$ has a convergent subsequence in $L^\infty_{-1,+}$. If $\zeta\in \widetilde{\C}\setminus (\R^+\pm i0)$, use \eqref{G_L split 2} to write 
\begin{equation}\label{T_L m split 1}
[T_L(\zeta)m_n](x) = \frac{i}{1-2\zeta}\int_{-\infty}^x u(s)m_n(s)~ds +\mathcal F^{-1}[R_2(\xi,0,\zeta)\widehat{um_n}(\xi)](x).
\end{equation}
We have 
\begin{equation}
D[T_L(\zeta)m_n](x) = \frac1{1-2\zeta}u(x)m_n(x)+\mathcal F^{-1}[\xi R_2(\xi,0,\zeta)\widehat{um_n}(\xi)](x)
\end{equation}
Note that $um_n$ is bounded in $L^1\cap L^2$, and $\xi R_2(\xi,0,\zeta)$ is bounded. It follows that $\{T_L(\zeta)m_n\}$ is bounded in $H^1[-N,N]\subset C^{0,\frac12}[-N,N]$ for any $N\in \mathbb N$. Since $C^{0,\frac12}[-N,N]$ is compactly embedded in $C[-N,N]$, by passing to a subsequence, we may assume $\{T_L(\zeta)m_n\}$ converges uniformly on any compact set. It remains to show that $\{T_L(\zeta)m_n\}$ satisfies condition (b) above. Note that
\begin{align}
\left|\int_{-\infty}^x u(s)m_n(s)~ds \right|&\le \|m_n\|_{L^\infty_{-1,+}}\int_{-\infty}^x |u(s)|(1+s_+)~ds\notag\\
&\le C\int_{-\infty}^x |u(s)|(1+s_+)~ds.
\end{align}
Since $u\in L^1_1$, for every $\epsilon>0$, there exists an $M>0$ such that for all $x<-M$, 
\begin{equation}
\left|\frac{i}{1-2\zeta}\int_{-\infty}^x u(s)m_n(s)~ds\right|<C\int_{-\infty}^{-M} |u(s)|~ds<\epsilon,
\end{equation}
and for $x>M$,
\begin{equation}
\frac1x\left|\frac{i}{1-2\zeta}\int_{-\infty}^x u(s)m_n(s)~ds\right|< \frac CM\int_{\R}\langle s\rangle |u(s)|~ds<\epsilon.
\end{equation}
Thus we only need to work on the second term in \eqref{T_L m split 1}. Denote $h=\mathcal F^{-1}[R_2(\xi,0,\zeta)]$. We have by the estimates on $R_2$ that $h\in L^2$, and the second term in \eqref{T_L m split 1} can be written as $h*(um_n)$. Since 
\begin{align*}
    \|um_n\|_{L^2}&\le \|u\|_{L^2_1}\|m_n\|_{L^\infty_{-1,+}}\le C\|u\|_{L^2_1},\\
    \|h*(um_n)\|_{L^\infty}&\le C\|h\|_{L^2}\|u\|_{L^2_1},
\end{align*}
it follows that for every $\epsilon>0$, there exists $M>0$ such that for $x>M$,
\begin{equation}
\frac1x|h*(um_n)(x)|\le \frac1M \|h*(um_n)\|_{L^\infty}< \epsilon.
\end{equation}
To study $h*(um_n)(x)$ as $x\to-\infty$, we write
\begin{equation}
h*(um_n)(x)=\int_{|s|>a}h(s) u(x-s)m_n(x-s)~ds+ \int_{|s|<a}h(s) u(x-s)m_n(x-s)~ds
\end{equation}
for some $a>0$ to be determined. Note that
\begin{align*}
&~\left|\int_{|s|>a}h(s) u(x-s)m_n(x-s)~ds\right|\\
\le&~ \left(\int_{|s|>a}|h(s)|^2~ds\right)^{\frac12}\|um_n\|_{L^2}\le C\left(\int_{|s|>a}|h(s)|^2~ds\right)^{\frac12}.
\end{align*}
Since $h\in L^2$, we can choose $a$ large enough so that the above integral is bounded by $\frac\epsilon 2$. Enlarge $M$ if necessary to make $M>a$. For $x<-M$, $|s|<a$, we have $x-s<0$, thus
\begin{align*}
&~\left|\int_{|s|<a}h(s) u(x-s)m_n(x-s)~ds\right|\\
\le &~\|m_n\|_{L^\infty_{-1,+} }\int_{|s|<a}|h(s)u(x-s)|~ds\\
\le &~C\int_{|s|<a}\frac{1}{\langle x-s\rangle }|h(s)\langle x-s\rangle u(x-s)|~ds\\
\le &~\frac{C}{\langle M-a\rangle}\|h*(\langle x\rangle u(x))\|_{L^\infty}\\
\le &~\frac{C}{\langle M-a\rangle}\|h\|_{L^2}\|u\|_{L^2_1}.
\end{align*}
Enlarge $M$ if necessary, we can bound the above integral by $\frac\epsilon 2$. Thus $|h*(um_n)(x)|<\epsilon$ if $x<-M$.

Now suppose $\zeta\in \R^++i0$, $\zeta\ne \frac12+i0$. Using \eqref{G_L split +}, we can split $T_L(\zeta)m_n$ in a way similar as before and argue in the same way for compactness. If $\zeta\in \R^+-i0$, $\zeta\ne \frac12-i0$, we can split $T_L(\zeta)m_n$ using \eqref{G_L split -}, and argue similarly as above for compactness, except for the term
\begin{equation}
e^{i\lambda x}\int_{x}^\infty u(s)m_n(s) e^{-i\lambda s}~ds.
\end{equation}
When $x\to\infty$, this term is uniformly small. However, it is not uniformly small when $x\to-\infty$. Instead, we observe that $\int_\R u(s)m_n(s) e^{-i\lambda s}~ds$ is a bounded sequence in $\R$. By passing to a subsequence, we may assume that it converges to $A_-\in \R$. We then have for $x<-M$
\begin{align*}
&~\left|e^{i\lambda x}\int_{x}^\infty u(s)m_n(s) e^{-i\lambda s}~ds-A_-e^{i\lambda x}\right|\\
\le &~\int_{-\infty }^x |u(s)m_n(s)|~ds +\left|\int_\R u(s)m_n(s) e^{-i\lambda s}~ds-A_-\right|\\
\le &~C\int_{-\infty}^{-M} |u(s)|~ds+\left|\int_\R u(s)m_n(s) e^{-i\lambda s}~ds-A_-\right|,
\end{align*}
which can be made less than any given $\epsilon>0$ by choosing $M>0$ and $n\in \mathbb N$ large enough.

Finally, suppose $\zeta=\frac12+i0$. In this case, we can use \eqref{G_L split 1/2} to split $T_L(\frac12+i0)m_n$. Then convergence of a subsequence can be shown in the same way for all the terms except 
\begin{equation}\label{T_L m split 2}
\int_{-\infty}^x (x-s)u(s)m_n(s)~ds.
\end{equation}
As before, by passing to a subsequence, we can assume \eqref{T_L m split 2} converges uniformly on compact sets, and it remains to show condition (b) above. When $x<-M<0$,
\begin{align*}
&~\left|\int_{-\infty}^x (x-s)u(s)m_n(s)~ds\right|\\
\le &~\|m_n\|_{L^\infty_{-1,+}}\int_{-\infty}^x (|x|+|s|)|u(s)|~ds\\
\le &~C\int_{-\infty}^x 2|s||u(s)|~ds\le C\int_{-\infty}^{-M} \langle s\rangle |u(s)|~ds.
\end{align*}
Since $u\in L^1_1$, the above integral can be made less than any given $\epsilon>0$ by choosing $M$ large enough. To study \eqref{T_L m split 2} as $x\to\infty$, we note that $\int_\R u(s)m_n(s)~ds$ is a bounded sequence in $\R$, and by passing to a subsequence, converges to $A_+\in \R$. Thus for $x>0$
\begin{align*}
&~\left|\frac1x\int_{-\infty}^x (x-s)u(s)m_n(s)~ds-A_+\right|\notag\\
\le &~\int_{x}^\infty |u(s)m_n(s)|~ds+\frac1x\int_{-\infty}^x |s| |u(s)m_n(s)|~ds + \left|\int_\R u(s)m_n(s)~ds-A_+\right|\notag\\
\le &~C\int_{x}^\infty \langle s\rangle |u(s)|~ds+\frac{C}x\int_{-\infty}^x |s|(1+s_+)|u(s)|~ds+ \left|\int_\R u(s)m_n(s)~ds-A_+\right|.
\end{align*}
Thus by choosing $M>0$ and $N\in \mathbb N$ large, for all $x>M$, $n>N$, we can make the first and the last terms in the above sum less than $\frac\epsilon 3$ each. We estimate the remaining term for $x>M>0$ as
\begin{align*}
&~\frac{C}{x}\int_{-\infty}^x |s|(1+s_+)|u(s)|~ds\\
\le &~\frac{C}x \left(\int_{-\infty}^0 \langle s\rangle |u(s)|~ds + \int_0^{\sqrt x}\sqrt x \langle s \rangle |u(s)|~ds+\int_{\sqrt x}^x x \langle s\rangle |u(s)|~ds\right)\\
\le &~\frac{C}M\|u\|_{L^1_1}+\frac{C}{\sqrt M}\|u\|_{L^1_1}+\int_{\sqrt M}^\infty \langle s\rangle |u(s)|~ds
\end{align*}
Since $u\in L^1_1$, the above sum can be made less than $\frac\epsilon3$ by choosing $M>0$ large enough. In summary, we can find $M>0$ and $N\in \mathbb N$ sufficiently large such that $\left|\frac1x\int_{-\infty}^x (x-s)u(s)m_n(s)~ds-A_+\right|<\epsilon$ for all $x>M$, $n>N$. The proof is complete.
\end{proof}

Let $\|u\|_{L^1_1\cap L^2_1}<\epsilon$ be sufficiently small. By Theorem \ref{thm: inv T off disc}, $I-T_L(\zeta)$ is invertible on $L^\infty_{-1,+}$, and $M_1(x,\zeta)$ is defined, for $\zeta \in \widetilde{\C}\setminus D^-_\delta$, and in particular, on $D^+_\delta$. $I-T_R(\zeta)$ is invertible on $L^\infty_{-1,-}$, and $N_1(x,\zeta)$ is defined, for $\zeta\in \widetilde{\C}\setminus D^+_\delta$, and in particular, on $D^-_\delta$. Let $a(\zeta)$ and $\breve a(\zeta)$ be given by \eqref{def: a}, \eqref{def: breve a}. It follows that $a(\zeta)$ is defined for $\zeta \in \widetilde{\C}\setminus (D^-_\delta\cup \{\frac12+i0\})$, and in particular, on $D^+_\delta\setminus \{\frac12+i0\}$. $\breve a(\zeta)$ is defined for $\zeta\in \widetilde{\C}\setminus (D^+_\delta\cup\{\frac12-i0\})$,  and in particular on $D^-_\delta\setminus \{\frac12-i0\}$. The next theorem shows that the invertibility of $I-T_L(\zeta)$ and $I-T_R(\zeta)$ in the critical semi-disc is determined by the scattering functions $\breve a$ and $a$ of each other.

\begin{theorem}\label{thm: inv T on disc}
Assume $u\in L^1_1\cap L^2_1$, and let $\zeta\in \widetilde{\C}\setminus \{\frac12\pm i0\}$ be given. Then
\begin{enumerate}[(i)]
\item Suppose $I-T_R(\zeta)$ is invertible on $L^\infty_{-1,-}$ and $\breve a(\zeta) $ is defined by \eqref{def: breve a}. Then $I-T_L(\zeta)$ is invertible on $L^\infty_{-1,+}$ if and only if $\breve a(\zeta)\ne 0$; 

\item Suppose $I-T_L(\zeta)$ is invertible on $L^\infty_{-1,+}$ and $a(\zeta) $ is defined by \eqref{def: a}. Then $I-T_R(\zeta)$ is invertible on $L^\infty_{-1,-}$ if and only if $a(\zeta)\ne 0$; 
\end{enumerate}
\end{theorem}
\begin{proof}
Let us only consider the first case. Let $\zeta\in \widetilde{\C}\setminus \{\frac12\pm i0\}$. By Lemma \ref{lem: compact T}, $T_L(\zeta)$ is compact on $L^\infty_{-1,+}$. Thus the invertibility of $I-T_L(\zeta)$ is equivalent to the triviality of its kernel. Now $I-T_R(\zeta)$ is invertible on $L^\infty_{-1,-}$ by assumption. By Lemma \ref{lem: jost exist 1}, $N_1(x,\zeta)$ exists. Define $\breve a(\zeta)$ by \eqref{def: breve a}. By \eqref{GL-GR}, we have
\begin{equation}\label{a=0 vs trivial kernel}
N_1(\cdot,\zeta)=1+T_R(\zeta) N_1(\cdot,\zeta)=\breve a(\zeta) + T_L(\zeta)N_1(\cdot,\zeta).
\end{equation} 
Note that since $\zeta\ne \frac12-i0$, $N_1(\cdot,\zeta)\in L^\infty$ by Lemma \ref{lem: jost exist 1}. From \eqref{a=0 vs trivial kernel} we see that if $\breve a(\zeta)=0$, $ N_1(\cdot,\zeta)\ne 0$ is in the kernel of $I-T_L(\zeta)$. On the contrary, suppose a nonzero $m\in L^\infty_{-1,+}$ satisfies 
\begin{equation}
m=T_L(\zeta)m= c+T_R(\zeta)m
\end{equation}
for some number $c\in \C$. Since $\zeta\ne \frac12-i0$, it follows as before that $m\in L^\infty$. Since $I-T_R(\zeta)$ is invertible on $L^\infty_{-1,-}$, we have that $m=cN_1(\cdot,\zeta)$. Since $m$ is nonzero, $c\ne 0$. We have 
\begin{equation}
m=cN_1(\cdot,\zeta) = c\breve a(\zeta) +T_L(\zeta)cN_1(\cdot,\zeta)=c\breve a(\zeta)+T_L(\zeta)m.
\end{equation}
It follows that $c\breve a(\zeta)=0$. Since $c\ne 0$, we conclude that $\breve a(\zeta)=0$. 
\end{proof}
\begin{remark}
The proof above actually shows that the kernels of $I-T_L(\zeta)$ and $I-T_R(\zeta)$ are at most one-dimensional.
\end{remark}
From Theorem \ref{thm: inv T on disc} we see that it is of interest to study the functions $a(\zeta)$ and $\breve a(\zeta)$. This we'll do in the next section. Before going in that direction, we first restate Theorem \ref{thm: inv T on disc} in a form that doesn't presuppose the invertibility of $I-T_L(\zeta)$ or $I-T_R(\zeta)$ for the implications to hold.
\begin{corollary}\label{cor: a=0 vs N1 nonexist}
    Assume $\|u\|_{L^1_1\cap L^2_1}<\epsilon$ for $\epsilon$ sufficiently small, and let $\zeta\in \widetilde{\C}\setminus \{\frac12\pm i0\}$ be given. Then
    \begin{enumerate}[(i)]
        \item $I-T_L(\zeta)$ is invertible on $L^\infty_{-1,+}$ and $a(\zeta)=0$ if and only if $I-T_R(\zeta)$ is not invertible on $L^\infty_{-1,-}$.
        \item $I-T_R(\zeta)$ is invertible on $L^\infty_{-1,-}$ and $\breve a(\zeta)=0$ if and only if $I-T_L(\zeta)$ is not invertible on $L^\infty_{-1,+}$.
    \end{enumerate}
\end{corollary}
\begin{proof}
    The forward implication follows readily from Theorem \ref{thm: inv T on disc}. To get the backward implication, assume $I-T_R(\zeta)$ is not invertible on $L^\infty_{-1,-}$. We note that this can happen only when $\zeta\in D^+_\delta\setminus \{\frac12+i0\}$ by Theorem \ref{thm: inv T off disc}. By the same theorem, we see that for such $\zeta$, $I-T_L(\zeta)$ must be invertible on $L^\infty_{-1,+}$. We now apply Theorem \ref{thm: inv T on disc} again to conclude that $a(\zeta)=0$.
\end{proof}

We also easily get
\begin{corollary}\label{cor: a a breve duality}
    Assume $\|u\|_{L^1_1\cap L^2_1}<\epsilon$ for $\epsilon$ sufficiently small, and let $\zeta\in \widetilde \C\setminus \{\frac12\pm i0\}$ be given. Then
    \begin{enumerate}[(i)]
        \item $a(\zeta)=0$ if and only if $\breve a(\zeta)$ is undefined.
        \item $\breve a(\zeta)=0$ if and only if $a(\zeta)$ is undefined.
        \item $a(\zeta)$ is defined and nonzero if and only if $\breve a (\zeta)$ is defined and nonzero.
    \end{enumerate}
\end{corollary}

                  %%  Existnce of Jost Solutions
\section{Scattering coefficients}\label{sec: scat data}

We have seen from the previous section that the solvability of integral equations for the Jost solutions is closely related to zeros of the scattering functions $a(\zeta)$ and $\breve a(\zeta)$. The invertibility of $I-T_L(\zeta)$, $I-T_R(\zeta)$ and the definability of $a(\zeta)$, $\breve a(\zeta)$ are given by Theorem \ref{thm: inv T off disc}, Theorem \ref{thm: inv T on disc}, and Corollary \ref{cor: a a breve duality}. To better understand the scattering problem, we will study the scattering data in more detail in this section.

\begin{comment}
First recall that for $\zeta\in \widetilde{\C}\setminus \{\frac12\pm i0\}$, if $I-T_L(\zeta)$ is invertible on $L^\infty_{-1,+}$, we define
\begin{equation}
\label{alpha.zeta}
a(\zeta)	
	=	1	+	\frac{i}{1-2\zeta} \int_\R u(x) M_1(x,\zeta) \,  dx .
\end{equation}
For $\zeta+i0\in (\R^++ i0)\setminus \{\frac12+ i0\}$, we can define
\begin{equation}
\label{beta.zeta}
b(\zeta+i0)	
	=	\frac{i}{1-2\zeta^*} \int_\R e^{-i\lambda x} u(x) M_1(x,\zeta+i0) \,  dx.
\end{equation}
For $\zeta\in \widetilde{\C}\setminus \{\frac12\pm i0\}$, if $I-T_R(\zeta)$ is invertible on $L^\infty_{-1,-}$, we define
\begin{equation}
\label{alphab.zeta}
\breve{a}(\zeta)
	=	1	-	\frac{i}{1-2\zeta} \int_\R u(x) N_1(x,\zeta) \,  dx.
\end{equation}
For $\zeta-i0\in (\R^+-i0)\setminus \{\frac12-i0\}$,  we can define
\begin{equation}
\label{betab.zeta}
\breve{b}(\zeta-i0)	
	=	\frac{i}{1-2\zeta^*} \int e^{-i\lambda x} u(x) N_1(x,\zeta-i0) \, dx.
\end{equation} 

\end{comment}
We first show analyticity of $a(\zeta)$, $\breve a(\zeta)$ and continuity of $b(k+i0)$, $\breve b(k-i0)$.
For convenience, we define for $\zeta\in \widetilde{\C}\setminus D^-_\delta$
\begin{equation}\label{def: a_1}
a_1(\zeta)=\int_\R u(x) M_1(x,\zeta) \,  dx,
\end{equation}
and for $k>0$,
\begin{equation}\label{def: b1}
b_1(k+i0)=\int_\R e^{-i\lambda x} u(x) M_1(x,k+i0) \,  dx.
\end{equation}
Here $\lambda=Z^{-1}(k)$ (see Definition \ref{def: Z inv}).
It follows from \eqref{def: a} and \eqref{def: b} that $a(\zeta)=1+\frac{ia_1(\zeta)}{1-2\zeta}$, and $b(k+i0)=\frac{ib_1(k+i0)}{1-2k^*}$. Similarly for $\breve a_1$ and $\breve b_1$, we define for $\zeta\in \widetilde{\C}\setminus D^+_\delta$
\begin{equation}\label{def: breve a_1}
    \breve a_1(\zeta)=\int_\R u(x)N_1(x,\zeta)~dx,
\end{equation}
and for $k>0$,
\begin{equation}
    \breve b_1(k-i0)=\int_\R e^{-i\lambda x}u(x)N_1(x,k-i0)~dx.
\end{equation}
It follows from \eqref{def: breve a} and \eqref{def: breve b} that $\breve a(\zeta) = 1-\frac{i\breve a_1(\zeta)}{1-2\zeta}$, and $\breve b(k-i0)=\frac{i\breve b_1(k-i0)}{1-2k^*}$.

\begin{theorem}\label{thm: a b cont}
Let $u\in L^1_1\cap L^2_1$ with sufficiently small norm. $a_1(\zeta)$ is continuous on $\widetilde{\C}\setminus D^-_\delta$ and analytic on its interior. $b_1(k+i0)$ is continuous on $\R^+$. $\breve a_1(\zeta)$ is continuous on $\widetilde{\C}\setminus D^+_\delta$ and analytic on its interior. $\breve b_1(k-i0)$ is continuous on $\R^+$. 
\end{theorem}

We will only show the proofs for $a_1(\zeta)$ and $b_1(k+i0)$. The argument for $\breve a_1(\zeta)$ and $\breve b_1(k-i0)$ will be completely analogous. 

We start by proving the analyticity properties of $T_L(\zeta)$ and $T_R(\zeta)$. 

We first find the range of $\zeta$ for which $T_L(\zeta)$ and $T_R(\zeta)$ are continuous in the operator norm topology.
\begin{lemma} \label{lem: T anlytc}
Let $u\in L^1_1\cap L^2_1$. Consider $T_L$ and $T_R$ as mappings
$$T_L:\widetilde{\C}\setminus \{\tfrac12-i0\} \to \mathcal {L} (L^\infty_{-1,+}),$$ and $$T_R:\widetilde{\C}\setminus \{\tfrac12+i0\} \to \mathcal L(L^\infty_{-1,-}).$$ Then $T_L$ is continuous on $\widetilde{\C}\setminus \left((\R^+-i0)\cup \{\frac12+i0\}\right)$ and analytic in its interior. $T_R$ is  continuous on $\widetilde{\C}\setminus\left((\R^++i0)\cup \{\frac12-i0\}\right)$ and analytic in its interior. 
\end{lemma}
\begin{remark}\label{rmk: cont at 0}
    Note that $T_L(\zeta)$ and $T_R(\zeta)$ are continuous at $\zeta=0$, with $\zeta\in \widetilde{\C}$ approaching 0.
\end{remark}
\begin{remark}
    Note that $T_L(\zeta)$ is not continuous in the  $\mathcal L(L^\infty)$ operator norm topology at $\zeta\in \R^++i0$, even though it is so for $\zeta$ away from $\R^+\pm i0$. In fact, For $k_1,k_2\in \R^+\setminus \{\frac12\}$, $[T_L(k_1+i0)-T_L(k_2+i0)]m(x)$ contains a term like $$\int_{-\infty}^x \left(e^{i(\lambda_1-\lambda_2)(x-y)}-1\right)u(y)m(y)e^{i\lambda_2y}~dy.$$ If $u$ is supported near zero, when $x$ is close to $\frac\pi{|\lambda_1-\lambda_2|}$, the integral is not small.
\end{remark}
\begin{proof}
We only show proofs for $T_L$. Consider $\zeta_1\in \widetilde{\C}\setminus \left((\R^+-i0)\cup \{\frac12+i0\}\right)$, and let $\zeta_2\to\zeta_1$ from within in the set $ \widetilde{\C}\setminus (\R^+-i0)$. We have by Lemma \ref{lem: weighted Holder Young}
\begin{align*}
&~\|[T_L(\zeta_1)-T_L(\zeta_2)]m\|_{L^\infty_{-1,+}}\\
 \le &~\|G_L(\cdot,\zeta_2)-G_L(\cdot,\zeta_1)\|_{L^2_{-1,+}+L^\infty_{-1,+}}\|um\|_{L^2_{1,-}\cap L^1_{1,-}}\\
\le &~C\|G_L(\cdot,\zeta_2)-G_L(\cdot,\zeta_1)\|_{L^2_{-1,+}+L^\infty_{-1,+}} \|u\|_{L^1_1\cap L^2_1}\|m\|_{L^\infty_{-1,+}}.
\end{align*}
Thus 
\begin{align}\label{eq: T cont in zeta}
&~\|[T_L(\zeta_1)-T_L(\zeta_2)]\|_{L^\infty_{-1,+}\to L^\infty_{-1,+}}\notag\\
\le&~ C\|G_L(\cdot,\zeta_2)-G_L(\cdot,\zeta_1)\|_{L^2_{-1,+}+L^\infty_{-1,+}} \|u\|_{L^1_1\cap L^2_1},
\end{align}
which tends to 0 by Corollary \ref{cor: G cont 1}.
 
Next we show analyticity of $T_L$. Note that we only need to show that for every $m\in L^\infty_{-1,+}$ and every $g\in C_0^\infty(\R)$, $\int_\R [T_L(\zeta)m] (x) g(x)~dx$ is analytic in $\zeta$. Since we have already shown continuity of $T_L(\zeta)$, by Morera's theorem, it suffices to show that for every closed curve $\Gamma$ contained in the interior of $\widetilde{\C}\setminus \left((\R^+-i0)\cup \{\frac12+i0\}\right)$,
\begin{equation}\label{anlytc int}
\int_\Gamma \int_\R [T_L(\zeta)m] (x) g(x)~dx ~d\zeta = 0.
\end{equation}
Note that $[T_L(\zeta)m] (x) = \int_\R G_L(x-y,\zeta)u(y)m(y)~dy$ with $\|G_L(\cdot,\zeta)\|_{L^2+L^\infty}$ bounded uniformly in $\zeta\in\Gamma$. Note that $um \in L^1\cap L^2$. Thus the $L^1(dy)$ norm of $G_L(x-y,\zeta)u(y)m(y)$ is bounded uniformly in $\zeta\in\Gamma$. It follows that $G_L(x-y,\zeta)u(y)m(y)g(x)$ is integrable on the product space of $(x,y,\zeta)$, so we can change the order of integration to get 
\begin{equation}
\int_\Gamma \int_\R [T_L(\zeta)m] (x) g(x)~dx ~d\zeta = \int_\R\int_\R \int_\Gamma G_L(x-y,\zeta)u(y)m(y)g(x)~d\zeta~dy  ~dx.
\end{equation}
Denote by $\Gamma_L^N$ the part of $\Gamma_L$ where $|\real \xi|\le N$, and define 
$$G_L^N(x,\zeta) = \frac1{2\pi}\int_{\Gamma_L^N}\frac{e^{ix\xi}}{\xi-\zeta(1-e^{-2\xi})}~d\xi.$$ 
It's easy to see that $\|G_L^N(\cdot,\zeta)-G_L(\cdot,\zeta)\|_{L^2}\to 0$ uniformly in $\zeta\in\Gamma$ as $N\to\infty$. Thus 
\begin{equation*}
\int_\R \int_\Gamma G_L(x-y,\zeta)u(y)m(y)g(x)~d\zeta~dy = \lim_{N\to\infty}\int_\R \int_\Gamma G_L^N(x-y,\zeta)u(y)m(y)g(x)~d\zeta~dy
\end{equation*}
Now apply Fubini's theorem again to see
\begin{equation}
\int_\Gamma G_L^N(x-y,\zeta)~d\zeta = \frac1{2\pi}\int_{\Gamma_L^N} \int_\Gamma \frac{e^{ix\xi}}{\xi-\zeta(1-e^{-2\xi})}~d\zeta~d\xi=0.
\end{equation} 
It follows now from the above calculation that \eqref{anlytc int} holds.
\end{proof}

We can now prove part of Theorem \ref{thm: a b cont}.

\begin{corollary}
Let $u\in L^1_1\cap L^2_1$ with sufficiently small norm. Then $a_1(\zeta)$ is continuous on $\widetilde{\C}\setminus \left(D^-_\delta\cup(\R^+-i0)\cup \{\frac12+i0\}\right)$ and analytic on its interior. $b_1(k+i0)$ is continuous on $\R^+\setminus \{\frac12\}$. $\breve a_1(\zeta)$ is continuous on $\widetilde{\C}\setminus \left(D^+_\delta\cup(\R^++i0)\cup \{\frac12-i0\}\right)$ and analytic on its interior. $\breve b_1(k-i0)$ is continuous on $\R^+\setminus \{\frac12\}$. 
\end{corollary}
\begin{proof} 
The analyticity of $a_1(\zeta)$, $\breve a_1(\zeta)$ in Theorem \ref{thm: a b cont} now follows directly from Lemma \ref{lem: T anlytc}, since $M_1(\cdot,\zeta) = (I-T_L(\zeta))^{-1}1$, and $N_1(\cdot,\zeta) = (I-T_R(\zeta))^{-1}1$. The continuity properties follow in the same way.
\end{proof}

Before we continue on the proof of continuity of $a$ and $b$. We here give a lemma about regularity of the Jost solutions which will later be used in Section \ref{sec: bound states}.
\begin{lemma}\label{lem: C^inf regularity jost fcn}
    Let $u\in L^1_1\cap L^2_1$ with sufficiently small norm. Assume further that $u\in C_0^\infty(\R)$. If $\zeta\in D^+_\delta\setminus (\R^+ + i0)$, then $M_1(x+iy,\zeta)$ is $C^\infty$ with all partial derivatives bounded for $(x,y)\in\R\times [0,2]$ and $\zeta$ in a compact subset of $D^+_\delta\setminus (\R^+ + i0)$. Similarly, if $\zeta\in D^-_\delta\setminus (\R^+-i0)$, then $N_1(x+iy,\zeta)$ is $C^\infty$ with all partial derivatives bounded for $(x,y)\in\R\times [0,2]$, and $\zeta$ in a compact subset of $D^-_\delta\setminus (\R^+-i0)$. Furthermore, the $M_1(x+iy,\zeta)$ and $N_1(x+iy,\zeta)$ are holomorphic in $(x+iy,\zeta)$ for $(x,y)\in \R\times (0,2)$ and $\zeta$ in the corresponding range.
\end{lemma}
\begin{proof}
    In Lemma \ref{lem: M1 improved regularity}, we have proven that $\partial_xM_1(x,\zeta)\in L^\infty$. Taking the $x$-derivative of \eqref{jost prty: D M1 split}, we get
    \begin{align}\label{eq: D^2 M1}
        \frac{1}{i}\partial_x^2 M_1(x,\zeta)=&\frac{i}{1-2\zeta}\partial_x[u(x)M_1(x,\zeta)]\notag\\
        &\quad + \mathcal F^{-1}[\xi R_2(\xi,0,\zeta)]*\partial_x[u(x)M_1(x,\zeta)].
    \end{align}
    We can now repeat the argument in the proof of Lemma \ref{lem: M1 improved regularity} on the last term of \eqref{eq: D^2 M1} to get $\partial_x^2 M_1(x,\zeta)\in L^\infty$. The higher $x$-derivatives can be estimated accordingly. Taking derivatives of \eqref{jost prty: M1 ext}, we get
    \begin{align}
        \partial_x^k \partial_y^l M_1(x+iy,\zeta)&=\frac{i}{1-2\zeta}\partial_x^{k-1}\partial_y^l[u(x)M_1(x,\zeta)]\notag\\
        &\quad +\mathcal F^{-1}\left[\partial_y^l R_2(\xi,y,\zeta)\mathcal F\left(\partial_s^k[u(s)M_1(s,\zeta)]\right)(\xi)\right](x).
    \end{align}
    The bound follows from \eqref{eq: d_y R bound}, since we already have that $\mathcal F\left(\partial_s^k[u(s)M_1(s,\zeta)]\right)$ is in Schwarz class.

    Next we estimate the $\zeta$-derivatives. By Lemma \ref{lem: T anlytc}, $\zeta\mapsto M_1(x,\zeta)$ is holomorphic as a map into $L^\infty_{-1,+}$. It follows from \eqref{jost prty: M1 ext} that $M_1(x+iy,\zeta)$ is holomorphic in $\zeta$ for fixed $x+iy$. As a result
    \begin{equation}
        \partial_x^k\partial_y^l\partial_\zeta^m M_1(x+iy,\zeta)=\frac{m!}{2\pi i}\int_{C(\zeta,\epsilon)}\frac{\partial_x^k\partial_y^l M_1(x,\tau)}{(\tau-\zeta)^{m+1}}~d\tau,
    \end{equation}
    where $C(\zeta,\epsilon)$ is the circle centered at $\zeta$ with a sufficiently small radius $\epsilon$. The bounds on $\partial_x^k\partial_y^l \partial_\zeta^m M_1(x+iy,\zeta)$ follow from the above bounds on $\partial_x^k \partial_y^l M_1(x,\zeta)$. Finally, we observe that $M_1(x+iy,\zeta)$ is separately holomorphic in $x+iy$ and in $\zeta$. By Osgood's lemma, it is jointly holomorphic in $(x+iy,\zeta)$.
\end{proof}

We are still missing the continuity of $a_1(\zeta)$, $b_1(k+i0)$, $\breve a_1(\zeta)$, $\breve b_1(k-i0)$ at a few specific locations. The reason for this gap is the loss of operator norm convergence of $T_L(\zeta)$ as $\zeta\to \frac12+i0$ or as $\zeta\to \R^+-i0$ and similarly for $T_R(\zeta)$. In fact, the convergence of $T_L(\zeta)$ and $T_R(\zeta)$ in these cases are much weaker, and we need to work on the integral expressions for $a_1(\zeta)$, $b_1(k+i0)$, etc. directly. To that end, we make the following observations. 

$L^\infty_{-1,+} = [L^1_{1,+}]^*$ and $L^\infty_{-1,-}=[L^1_{1,+}]^*$ under the identification $f\mapsto \Lambda_f$ where $\Lambda_f(g)= \int_\R fg ~dx$. For $\zeta\in \widetilde{\C}\setminus \{\frac12-i0\}$, define 
\begin{equation}\label{def: A(zeta)}
A(\zeta) g = u[G_L(-\cdot,\zeta)*g].
\end{equation}
For $\zeta\in \widetilde{\C}\setminus \{\frac12+i0\}$, define 
\begin{equation}\label{def: B(zeta)}
B(\zeta) g = u[G_R(-\cdot,\zeta)*g].
\end{equation}
We have 
\begin{lemma}\label{lem: strong cont}
Let $u\in L^1_1\cap L^2_1$. $A(\zeta): L^1_{1,+}\to L^1_{1,+}$, $B(\zeta):L^1_{1,-}\to L^1_{1,-}$ are bounded, and $T_L(\zeta)=[A(\zeta)]^*$. $T_R(\zeta) = [B(\zeta)]^*$. Moreover, $A(\zeta)\in \mathcal L(L^1_{1,+})$ is strongly continuous in $\zeta\in \widetilde{\C}\setminus \{\frac12-i0\}$, and $B(\zeta)\in \mathcal L(L^1_{1,-})$ is strongly continuous in $\zeta\in \widetilde{\C}\setminus \{\frac12+i0\}$.
\end{lemma}
\begin{proof}
We only show the proof for $T_L(\zeta)$ and $A(\zeta)$.
In fact, by Proposition \ref{prop: green est}, $G_L(\cdot,\zeta)\in L^\infty_{-1,+}+L^2$. Thus $G_L(-\cdot,\zeta)\in L^\infty_{-1,-}+L^2$. 
We have by Lemma \ref{lem: weighted Holder Young} that 
\begin{align}\label{eq: A norm est}
\|u[G_L(-\cdot,\zeta)*g]\|_{L^1_{1,+}}&\le C\|u\|_{L^1_1\cap L^2}\|G_L(-\cdot,\zeta)*m\|_{L^\infty_{-1,-}+L^2}\notag\\
%&\le  C\|u\|_{L^1_1\cap L^2_1}\|G_L(-\cdot,\zeta)\|_{L^\infty_{-1,-}+L^2_{-1,-}}\|m\|_{L^1_{1,+}}\\
&\le  C\|u\|_{L^1_1\cap L^2}\|G_L(-\cdot,\zeta)\|_{L^\infty_{-1,-}+L^2}\|m\|_{L^1_{1,+}}.
\end{align}
This shows $A(\zeta)$ is bounded on $L^1_{1,+}$.
Using the above estimates and by Fubini's theorem, we have for $f\in L^\infty_{-1,+}$ and $g\in L^1_{1,+}$ that 
\begin{align*}
\int_\R [T_L(\zeta)f ]g~dx &=\int_\R\int_\R G_L(x-y,\zeta)u(y)f(y) g(x)~dy~dx\\
&=\int_\R\int_\R G_L(x-y,\zeta)u(y)f(y) g(x)~dx~dy\\
&=\int_\R f[A(\zeta)g]~dy.
\end{align*}
This verifies that $T_L(\zeta)=[A(\zeta)]^*$. 

We only need to show strong continuity of $A(\zeta)$ at $\zeta=\frac12+i0$ and $\zeta\in (\R^+-i0)\setminus \{\frac12-i0\}$. The other cases follow from Lemma \ref{lem: T anlytc} and duality. To estimate $[A(\zeta)-A(\frac12+i0)]$ for given $f\in L^1_{1,+}$, we write 
\begin{equation}
G_L(x,\zeta)-G_L(x,\tfrac12+i0)=H(x,\zeta)-H(x,\tfrac12+i0)+R(x,\zeta)
\end{equation}
where
\begin{equation}
R(x,\zeta) = [G_L(x,\zeta)-H(x,\zeta)]-[G_L(x,\tfrac12+i0)-H(x,\tfrac12+i0)].
\end{equation}
By \eqref{eq: G cont 6} in Proposition \ref{prop: G cont}, $\|H(-x,\zeta)-H(-x,\frac12+i0)\|_{L^2\cap L^\infty}\le C|\zeta-\tfrac12|$. Thus by Lemma \ref{lem: weighted Holder Young}
\begin{align*}
&~\|u[(H(-\cdot,\zeta)-H(-\cdot,\tfrac12+i0))*f]\|_{L^1_{1,+}}\\
\le&~ C\|u\|_{L^2_1}\| (H(-\cdot,\zeta)-H(-\cdot,\tfrac12+i0))*f\|_{L^2_{-1,-}}\\
\le &~ C\|u\|_{L^2_1}\| \|H(-\cdot,\zeta)-H(-\cdot,\tfrac12+i0))\|_{L^2_{-1,-}}\|f\|_{L^1_{1,+}}\\
\le &~ C\|u\|_{L^2_1}\| \|H(-\cdot,\zeta)-H(-\cdot,\tfrac12+i0))\|_{L^2}\|f\|_{L^1_{1,+}}\\
\le &~C|\zeta-\tfrac12|\|u\|_{L^2_1}\| \|f\|_{L^1_{1,+}}\to 0
\end{align*}
as $\zeta\to \tfrac12+i0$. 
By Lemma \ref{lem: G weak cont 1/2}, as $\zeta\to\frac12+i0$, $R(-\cdot,\zeta)$ is bounded in $L^\infty_{-1,-}$ and converges pointwise to zero. It follows that 
\begin{equation}
\|R(-\cdot,\zeta)*f\|_{L^\infty_{-1,-}}\le \|R(-\cdot,\zeta)\|_{L^\infty_{-1,-}}\|f\|_{L^1_{1,+}}\le C\|f\|_{L^1_{1,+}}.
\end{equation}
In fact, for any given $x$, we have from the proof of Lemma \ref{lem: weighted Holder Young}
\begin{align*}
|R(y-x,\zeta)f(y)|&\le \|R(-\cdot,\zeta)\|_{L^\infty_{-1,-}}\frac{1+(x-y)_-}{1+y_+}(1+y_+)|f(y)|\\
&\le \|R(-\cdot,\zeta)\|_{L^\infty_{-1,-}}(1+x_-)(1+y_+)|f(y)|
\end{align*}
Since $(1+y_+)|f(y)|\in L^1$, $[R(-\cdot,\zeta)*f](x)$ converges to zero for every $x$ as $\zeta\to\frac12+i0$ by the dominated convergence theorem. Now 
\begin{align*}
|(1+x_+)u(x)[R(-\cdot,\zeta)*f](x)|&\le C \langle x\rangle |u(x)|| (1+x_-)^{-1}[R(-\cdot,\zeta)*f](x)|\\
&\le C\|R(-\cdot,\zeta)*f\|_{L^\infty_{-1,-}} \langle x\rangle |u(x)|\\
&\le C\|f\|_{L^1_{1,+}} \langle x\rangle |u(x)|.
\end{align*}
Since $\langle x\rangle |u(x)|\in L^1$, we have $\|(1+x_+)u(x)[R(-\cdot,\zeta)*f](x)\|_{L^1}\to 0$ as $\zeta\to\frac12+i0$ by the dominated convergence theorem. In other words, $\|u[R(-\cdot,\zeta)*f]\|_{L^1_{1,+}} \to 0$ as $\zeta\to\frac12+i0$. This completes the proof that $[A(\zeta)-A(\frac12+i0)]f = u[G_L(-\cdot,\zeta)-G_L(-\cdot,\frac12+i0)]*f$ converges to zero in $L^1_{1,+}$. Thus $A(\zeta)$ is strongly continuous at $\zeta=\frac12+i0$. The case for $\zeta\in (\R^+-i0)\setminus \{\frac12-i0\}$ can be proven similarly.
\end{proof}
 
We are now ready to show 

\begin{lemma}\label{lem: a b cont k=1/2}
Let $u\in L^1_1\cap L^2_1$ with sufficiently small norm. $a_1(\zeta)$ is continuous at $\frac12+i0$ and on $(\R^+-i0)\setminus D^-_\delta$. $b_1(k+i0)$ is continuous at $k=\frac12+i0$. $\breve a_1(\zeta)$ is continuous at $\frac12-i0$ and on $(\R^++i0)\setminus D^+_\delta$. $\breve b_1(k-i0)$ is continuous at $k=\frac12$. 
\end{lemma}
\begin{proof}
Consider $\zeta$ close to $\frac12+i0$. We write using the second resolvent formula
\begin{align}
M_1(\cdot,\zeta)-M_1(\cdot,\tfrac12+i0)&=\left[(I-T_L(\zeta))^{-1}-\left(I-T_L\left(\tfrac12+i0\right)\right)^{-1}\right]1\notag\\
&=\left(I-T_L(\tfrac12+i0)\right)^{-1}\left[T_L(\zeta)-T_L(\tfrac12+i0)\right](I-T_L(\zeta))^{-1}1\notag\\
&=\left(I-T_L(\tfrac12+i0)\right)^{-1}\left[T_L(\zeta)-T_L(\tfrac12+i0)\right]M_1(\cdot,\zeta).
\end{align}
Note that $(I-T_L(\frac12+i0))^{-1}=(I-[A(\frac12+i0)]^*)^{-1}=[(I-A(\frac12+i0))^{-1}]^*$. Thus
\begin{align}\label{eq: diff a1}
a_1(\zeta)-a_1(\tfrac12+i0)&=\int_\R u [M_1(\cdot,\zeta)-M_1(\cdot,\tfrac12+i0)]\notag\\
&=\int_\R u[(I-A(\tfrac12+i0))^{-1}]^*\left[A(\zeta)-A(\tfrac12+i0)\right]^* M_1(\cdot,\zeta)\notag\\
&=\int_\R [(A(\zeta)-A(\tfrac12+i0))(I-A(\tfrac12+i0))^{-1}u]M_1(\cdot,\zeta).
\end{align}
Since $A(\zeta)$ is strongly continuous by Lemma \ref{lem: strong cont}, $[(A(\zeta)-A(\tfrac12+i0))(I-A(\tfrac12+i0))^{-1}u]$ converges to zero in $L^1_{1,+}$. In the meantime, $(I-T_L(\zeta))^{-1}1$ is bounded in $L^\infty_{-1,+}$. Thus $a_1(\zeta)-a_1(\tfrac12+i0)$ converges to zero. The cases for the continuity of $a_1(\zeta)$ on $(\R^+-i0)\setminus D^-_\delta$ and for $b_1(k+i0)$ are similar. 
%Using the identity $(I-A(\zeta))^{-1} = I+A(\zeta)(I-A(\zeta))^{-1}$, we can write
%\begin{equation}
%(I-A(\zeta))^{-1} u = u +A(\zeta)(I-A(\zeta))^{-1}u = u+uG_L(-\cdot,\zeta)*[(I-A(\zeta))^{-1}u].
%\end{equation}
%
%From Theorem \ref{thm: inv T off disc}, we see $\|T_L(\zeta)\|_{\mathcal L(L^\infty_{-1,+})}<\tfrac12$ for $\zeta$ near $\tfrac12+i0$. Since $\|A(\zeta)\|_{\mathcal L(L^1_{1,+})}=\|[A(\zeta)]^*\|_{\mathcal L(L^\infty_{-1,+})}<\tfrac12$, we have $$(I-A(\zeta))^{-1} = \sum_{k=0}^\infty [A(\zeta)]^k = I+A(\zeta)\sum_{k=0}^\infty [A(\zeta)]^k .$$
\end{proof}

The above lemma closes the gap left open in the previous lemmas in the proof of Theorem \ref{thm: a b cont}. Thus its proof is now complete. Moreover, we show differentiability of $a_1(k+i0)$, $b_1(k+i0)$, $\breve a_1(k-i0)$, $\breve b_1(k-i0)$. To that end, we recall $A(k+i0)$, $B(k-i0)$ as defined by \eqref{def: A(zeta)} and \eqref{def: B(zeta)}, and further define
\begin{equation}
    \partial_k A(k+i0) g = u[\partial_k G_L(-\cdot,k+i0)*g]
\end{equation}
\begin{equation}
    \partial_k B(k-i0)g = u[\partial_k G_R(-\cdot,k-i0)*g].
\end{equation}
\begin{lemma}\label{lem: A differntiable k>0}
    Let $u\in L^1_1$, $k\in \mathbb R^+$, $k\ne \frac12$. Then $\partial_k A(k+i0):L^1_{1,+}\to L^1_{1,+}$, $\partial_k B(k-i0):L^1_{1,-}\to L^1_{1,-}$ are bounded, such that $\Delta_k^h A(k+i0)\in \mathcal L (L^1_{1,+})$ converges strongly to $\partial_k A(k+i0)$, and $\Delta_h^k B(k-i0)\in \mathcal L (L^1_{1,-})$ converges strongly to $\partial_k B(k-i0)$ as $h\to 0$.
\end{lemma}
\begin{proof}
    We show the proof for $A$ only. By Lemma \ref{lem: Del G - DG on R+} and Lemma \ref{lem: weighted Holder Young}, we get
    \begin{align}\label{eq: Dk G norm est}
        \|u[\partial_k G_L(-\cdot, k+i0)*g]\|_{L^1_{1,+}}&\lesssim \|u\|_{L^1_1}\|\partial_k G_L(-\cdot, k+i0)*g\|_{L^\infty_{-1,-}}\notag\\
        &\lesssim \|u\|_{L^1_1}\|\partial_k G_L(\cdot,k+i0)\|_{L^\infty_{-1,+}}\|g\|_{L^1_{1,+}}\notag\\
        &\lesssim_{k,u} \|g\|_{L^1_{1,+}}
    \end{align}
    and a similar estimate with $\partial_k G_L(-\cdot, k+i0)$ replaced by $\Delta_k^h G_L(-\cdot, k+i0)$. It remains to show the strong convergence. By the same lemmas above, we have as $h\to 0$
    \begin{align}
        &\qquad |(1+x_+)u(x)\left([\Delta_k^h G_L(-\cdot, k+i0)-\partial_k G_L(\cdot,k+i0)]*g\right)(x)|\notag\\
        &\lesssim \langle x\rangle |u(x)|\|\Delta_k^h G_L(\cdot, k+i0)-\partial_k G_L(\cdot,k+i0)\|_{L^\infty_{-1,+}}\|g\|_{L^1_{1,+}}\notag\\
        &\lesssim_{k,g}\langle x \rangle |u(x)| \in L^1.
    \end{align}
    Thus $[\Delta_k^hA(k+i0)-\partial_k A(k+i0)]g\to 0$ in $L^1_{1,+}$ by Lemma \ref{lem: Del G - DG on R+} and the dominated convergence theorem.
\end{proof}

\begin{lemma}\label{lem: A differentiable k=1/2}
    Let $u\in L^1_2\cap L^2$. Then 
    \begin{equation}\label{eq: A bound L^1_2+}
        \|A(\tfrac12+i0)\|_{L^1_{2,+}\to L^1_{2,+}}\lesssim \|u\|_{L^1_2\cap L^2}.
    \end{equation}
    $\partial_k A(\frac12+i0): L^1_{2,+}\to L^1_{2,+}$, $\partial_k B(\frac12-i0): L^1_{2,-}\to L^1_{2,-}$ are bounded, such that $\Delta_k^h A(\frac12+i)\in \mathcal L(L^1_{2,+})$ converges strongly to $\partial_k A(\frac12+i0)$, and $\Delta_k^h B(\frac12-i0)\in \mathcal L(L^1_{2,-})$ converges strongly to $\partial_k B(\frac12-i0)$ as $h\to 0$.
\end{lemma}
\begin{proof}
    By Proposition \ref{prop: green est} and Lemma \ref{lem: weighted Holder Young},
    \begin{equation}
        \|u[G_L(-\cdot,\tfrac k+i0)*g]\|_{L^1_{2,+}}\lesssim \|u\|_{L^1_2\cap L^2}\|G_L(-\cdot, k+i0)*g\|_{L^1_{-2,-}+L^2},
    \end{equation}
    and
    \begin{align}
        \|G_L(-\cdot, k+i0)*g\|_{L^1_{-2,-}+L^2}& \le \|G_L(-\cdot, k+i0)*g\|_{L^1_{-1,-}+L^2}\notag\\
        &\lesssim \|G_L(\cdot, k+i0)\|_{L^\infty_{-1,+}+L^2}\|g\|_{L^1_{1,+}}\notag\\
        &\lesssim \|g\|_{L^1_{2,+}}.
    \end{align}
    This implies \eqref{eq: A bound L^1_2+}. Similarly, by Lemma \ref{lem: Del G - DG at 1/2} and Lemma \ref{lem: weighted Holder Young}, 
    \begin{equation}
        \|u[\partial_k G_L(-\cdot,\tfrac12+i0)*g]\|_{L^1_{2,+}}\lesssim \|u\|_{L^1_2}\|\partial_k G_L(-\cdot, \tfrac12+i0)*g\|_{L^\infty_{-2,-}},
    \end{equation}
    and 
    \begin{align}
        \|\partial_k G_L(-\cdot, \tfrac12+i0)*g\|_{L^\infty_{-2,-}}&\lesssim \|\partial_k G_L(\cdot,\tfrac12+i0)\|_{L^\infty_{-2,+}}\|g\|_{L^1_{2,+}}\notag\\
        &\lesssim \|g\|_{L^1_{2,+}}.
    \end{align}
    It remains to show strong convergence. By the above lemmas again, as $h\to 0$
    \begin{align}
        &\qquad |(1+x_+)^2u(x)\left([\Delta_k^h G_L(-\cdot, \tfrac12+i0)-\partial_k G_L(\cdot,\tfrac12+i0)]*g\right)(x)|\notag\\
        &\lesssim \langle x\rangle^2 |u(x)|\|\Delta_k^h G_L(\cdot, \tfrac12+i0)-\partial_k G_L(\cdot,\tfrac12+i0)\|_{L^\infty_{-2,+}}\|g\|_{L^1_{2,+}}\notag\\
        &\lesssim_{g}\langle x \rangle^2 |u(x)| \in L^1.
    \end{align}
    Thus $[\Delta_k^h A(\frac12+i0)-\partial_k A(\frac12+i0)]g\to 0$ in $L^1_{2,+}$ by Lemma \ref{lem: Del G - DG at 1/2} and the dominated convergence theorem.
\end{proof}

We now show differentiability of the scattering coefficients on $\R^+$.
\begin{theorem}\label{thm: a b diff k=1/2}
    Let $u\in L^1_1\cap L^2_1$ with sufficiently small norm. Then $a_1(k+i0)$, $b_1(k+i0)$ and $\breve a_1(k-i0)$, $\breve b_1(k-i0)$ are differentiable for $k\in \mathbb R^+$, $k\ne \frac12$. Furthermore, if $u\in L^1_2\cap L^2_1$ with sufficiently small norm, then they are also differentiable at $k=\frac12$.
\end{theorem}
\begin{proof}
    Let $k>0$. We write $\Delta_k^h a_1(k+i0)$ similar as in \eqref{eq: diff a1}:
    \begin{equation}
        \Delta_k^h a_1(k+i0)=\int_\R [\Delta_k^h A(k+i0)(I-A( k+i0))^{-1}u]M_1(\cdot,k+h+i0).
    \end{equation}
    Note that $M_1(\cdot,k+h+i0)\to M_1(\cdot,k+i0)$ in $L^\infty_{-1,+}$ by Lemma \ref{lem: T anlytc}. If $k\ne \frac12$ and $u\in L^1_1\cap L^2_1$ with small norm, we use Lemma \ref{lem: A differntiable k>0}. If $k=\frac12$ and $u\in L^1_2\cap L^2_1$, we use Lemma \ref{lem: A differentiable k=1/2}. In either case we get 
    $$\Delta_k^h A(k+i0)(I-A(k+i0))^{-1}u\to \partial_k A(k+i0)(I-A(k+i0))^{-1}u$$
    in $L^1_{1,+}$ as $h\to 0$. It follows that
    \begin{equation}
        \Delta_k^h a_1(k+i0)\to \int_\R [\partial_k A(k+i0)(I-A(k+i0))^{-1}u]M_1(\cdot,k+i0).
    \end{equation}
    Thus $a_1(k+i0)$ is differentiable at $k$. To show differentiability of $b_1(k+i0)$, we write using \eqref{def: b1}
    \begin{align}\label{eq: Del h b1}
        \Delta_k^h b_1(k+i0)
        =&\int_\R u(x)(\Delta_k^h e^{-i\lambda x})M_1(x,k+h+i0)~dx\notag\\
        &\quad+ \int_{\R}u(x)e^{-i\lambda x}\Delta_k^h M_1(x,k+i0)~dx.
    \end{align}
    The second term in \eqref{eq: Del h b1} can be treated in the same way as $a_1$ was, with $u(x)$ replaced by $u(x)e^{-i\lambda x}$. To take the limit of the first term, we first note the elementary inequality
    \begin{equation}
        |\Delta_k^h e ^{-i\lambda x}|\le \frac{|e^{-i\tilde{h}x}-1|}{|h|}\lesssim_k |x|,
    \end{equation}
    where $\tilde h = Z^{-1}(k+h)-Z^{-1}(k)$. By Theorem \ref{thm: inv T off disc}, $M_1(\cdot,k+h+i0)$ is uniformly bounded in $L^{\infty}_{-1,+}$. When $k\ne \frac12$, by \eqref{eq: T(m) L inf}, $M_1(\cdot,k+h+i0)$ is in $L^\infty$. Since the constant $C(k+h+i0)$ in \eqref{eq: T(m) L inf} can be taken to be uniform for $k+h$ away from $\frac12$ by Lemma \ref{lem: green uniform split}, we have that $\|M_1(\cdot,k+h+i0)\|_{L^\infty}$ is uniformly bounded as $h\to 0$. We summarize the above estimates to conclude 
    \begin{equation}
        |u(x)(\Delta_k^h e^{-i\lambda x})M_1(x,k+h+i0)|\lesssim_k |xu(x)|\in L^1
    \end{equation}
    since $u\in L^1_1$. Lemma \ref{lem: T anlytc} implies $M_1(\cdot,k+h+i0)\to M_1(\cdot,k+i0)$ in $L^{\infty}_{-1,+}$. Thus the first term in \eqref{eq: Del h b1} converges by the dominated convergence theorem. If $k=\frac12$, we have the stronger assumption $u\in L^1_2$. We just use the uniform $L^{\infty}_{-1,+}$ bound of $M_1(\cdot,\frac12+h+i0)$ to get
    \begin{equation}
        |u(x)(\Delta_k^h e^{-i\lambda x})M_1(x,\tfrac12+h+i0)|\lesssim_k \langle x\rangle ^2|u(x)|\in L^1,
    \end{equation}
    then convergence follows as above.
\end{proof}

Next, we show a simple relation between $M_1$ and $N_1$, which holds on almost the entire cut plane $\widetilde{\mathbb C}$, even when limits are taken at $\R^+\pm i0$.

\begin{lemma}\label{lem: M=aN}
    Let $\zeta\in \widetilde{\mathbb C}\setminus \{\frac12\pm i0\}$ be such that $a(\zeta)$ and $\breve a(\zeta)$ are both defined and nonzero (cf. Corollary \ref{cor: a a breve duality}), then we have
    \begin{equation}\label{eq: a a breve = 1}
        a(\zeta)\breve a(\zeta)=1,
    \end{equation}
    and the identity
    \begin{equation}\label{eq: M1 = a N1}
        M_1(x,\zeta)=a(\zeta)N_1(x,\zeta).
    \end{equation}
\end{lemma}
\begin{proof}
The proof is a calculation similar to that in Theorem \ref{thm: inv T on disc}. By Corollary \ref{cor: a a breve duality}, $a(\zeta)$ is defined and nonzero if and only if $\breve a(\zeta)$ is. For such a $\zeta$, by Theorem \ref{thm: inv T on disc}, $I-T_L(\zeta)$ is invertible on $L^\infty_{-1,+}$ and $I-T_R(\zeta)$ is invertible on $L^\infty_{-1,-}$, with 
$$[I-T_L(\zeta)]M_1(\cdot,\zeta)=1,$$
$$[I-T_R(\zeta)]N_1(\cdot,\zeta)=1.$$
By \eqref{GL-GR}, we have $[T_L(\zeta)-T_R(\zeta)]m=\frac{i}{1-2\zeta}\int_\R um~dx$, so
\begin{equation}
    M_1(\cdot,\zeta)=1+T_L(\zeta)M_1(\cdot,\zeta)=a(\zeta)+T_R(\zeta)M_1(\cdot,\zeta).
\end{equation}
Thus 
\begin{equation}
    M_1(\cdot,\zeta)=[I-T_R(\zeta)]^{-1}a(\zeta)=a(\zeta)N_1(\cdot,\zeta).
\end{equation}
Similarly we have
\begin{equation}
    N_1(\cdot,\zeta)=[I-T_L(\zeta)]^{-1}\breve a(\zeta)=\breve a(\zeta)M_1(\cdot,\zeta).
\end{equation}
It follows that $M_1(\cdot,\zeta)=a(\zeta)\breve a(\zeta)M_1(\cdot,\zeta)$. Since $M_1(\cdot,\zeta)$ is obviously nonzero, we get $a(\zeta)\breve a(\zeta)=1$.
\end{proof}

We next prove an identity relating different scattering fucntions on $\mathbb R^+$.

\begin{lemma}\label{lem: a b identity}
   Let $u\in L^1_1\cap L^2_1$ be real-valued with sufficiently small norm and let $k\in\R^+\setminus\{\frac12\}$. We have the following relations between scattering functions

\begin{align}
\label{ab.id}
|a(k+i0)|^2 = 1 + \frac{2k^*-1}{1-2k} |b(k+i0)|^2,\\
\intertext{and}
\label{abbb.id}
|\breve{a}(k-i0)|^2 = 1 + \frac{2k^*-1}{1-2k} |\breve{b}(k-i0)|^2.
\end{align}

\end{lemma}
\begin{proof}
For the moment, let us abbreviate $M_1=M_1(x,k+i0)$, $G_L=G_L(x,k+i0)$, and $e_\lambda=e^{i\lambda x}$. To prove \eqref{ab.id}, 
we first note the identity
\begin{equation}
\label{ILW.GL.id}
 \langle G_L * f ,  g \rangle 
 	= \langle f,  G_L*g \rangle 
 			+ \frac{i}{1-2k}\langle f,1\rangle \langle 1,g \rangle 
 			+ \frac{1}{1-2k^*}\langle f, e_\lambda \rangle \langle e_\lambda, g \rangle
\end{equation}
where $\langle u, v \rangle=\int_\R u\bar v~dx$.  To get \eqref{ILW.GL.id}, we use \eqref{GL-GR}, \eqref{eq: G_L bar}, \eqref{eq: G up-down} and write 
\begin{equation}
    \overline{G_L(-x,k+i0)} =  G_L(x,k+i0) -\frac{i}{1-2k} - \frac{i}{1-2k^*} e^{i\lambda x}.
\end{equation} 
Since 
\begin{align*}
    \langle G_L*f,g\rangle&=\int_\R\int_\R G_L(x-y,k+i0)f(y)\overline{g(x)}~dy~dx\\
    &=\int_\R\int_\R f(y)\overline{\overline{G_L(-(y-x),k+i0)}g(x)}~dx~dy,
\end{align*}
\eqref{ILW.GL.id} follows.
We now use the equation $M_1 = 1 + G_L*(uM_1)$ and \eqref{ILW.GL.id} to compute
\begin{align*}
\langle M_1,  uM_1 \rangle 
	&=	\langle 1 + G_L*(uM_1),  uM_1 \rangle	\\
	&=	\langle 1, uM_1 \rangle + \langle uM_1,  G_L*(uM_1) \rangle \\
	&\quad	+	\frac{i}{1-2k} \left| \langle uM_1, 1 \rangle \right|^2
				+	\frac{i}{1-2k^*} \left|   \langle uM_1, e_\lambda \rangle \right|^2\\
	&=	\langle 1, uM_1 \rangle + \langle uM_1,  M_1-1 \rangle \\
	&\quad	+	\frac{i}{1-2k} \left| \langle uM_1, 1 \rangle \right|^2
				+	\frac{i}{1-2k^*} \left|   \langle uM_1, e_\lambda \rangle \right|^2
\end{align*}
Since $u$ is real,  $\langle uM_1, M_1 \rangle = \langle M_1, uM_1 \rangle$.  Using the facts that 
$$\langle uM_1, 1 \rangle =\frac{(1-2k)(a(k+i0)-1)}i , \quad \langle uM_1, e_\lambda \rangle = \frac{(1-2k^*) b(k+i0)}i$$ we obtain \eqref{ab.id}. The proof of \eqref{abbb.id} is similar.
\end{proof}

Observe that \eqref{ab.id} and \eqref{abbb.id} implies 
\begin{equation}\label{eq: a > 1}
    |a(k+i0)|\ge 1, \quad |\breve a(k-i0)|\ge 1
\end{equation}
for $k\in \R^+\setminus \{\frac12\}$. We can thus define the trasmission and reflection coefficients on $\R^+$ as follows.

\begin{definition}
    For $k\in \R^+$, $k\ne \frac12$, we define transmission coefficients $\tau(k)$, $\breve \tau(k)$ as
    \begin{equation}\label{def: tau}
        \tau(k)=\frac{1}{a(k+i0)},\quad \breve \tau(k)=\frac{1}{\breve a(k-i0)},
    \end{equation}
    and the reflection coefficients $\rho(k)$, $\breve \rho(k)$ as
    \begin{equation}\label{def: rho}
        \rho(k)=\frac{b(k+i0)}{a(k+i0)},\quad \breve \rho(k)=\frac{\breve b(k-i0)}{\breve a(k-i0)}.
    \end{equation}
\end{definition}

With the above definitions, \eqref{ab.id} and \eqref{abbb.id} can be restated as
\begin{equation}\label{eq: tau rho uni}
    |\tau(k)|^2+\frac{2k^*-1}{1-2k}|\rho(k)|^2=1,
\end{equation}
\begin{equation}\label{eq: breve tau rho uni}
    |\breve \tau(k)|^2+\frac{2k^*-1}{1-2k}|\breve \rho(k)|^2=1.
\end{equation}

It follows from Theorem \ref{thm: a b diff k=1/2} that the transmission and reflection coefficients are differentiable. In particular, we have
\begin{corollary}\label{cor: tau rho cont}
    Let $u\in L^1_1\cap L^2_1$ be real-valued with sufficiently small norm. Then $\tau(k)$, $\breve \tau(k)$, $\rho(k)$, $\breve \rho(k)$ are differentiable on $\R^+\setminus \{\frac12\}$.
\end{corollary}
\begin{proof}
    This is obvious from Theorem \ref{thm: a b diff k=1/2} and \eqref{eq: a > 1}.
\end{proof}

Furthermore, we have the following bounds on the reflection coefficients as $k\searrow 0$ and as $k\nearrow \infty$.
\begin{lemma}\label{lem: D rho uniform est}
There exists a $C>0$ such that 
    \begin{equation}
        |\partial_k\rho(k)|\le \frac{C}{k|\log k|}, ~|\partial_k\breve\rho(k)|\le \frac{C}{k|\log k|}\quad \text{ for }0<k<\tfrac{1}{4},
    \end{equation}
    and 
    \begin{equation}
        |\partial_k\rho(k)|\le C,~|\partial_k\breve \rho(k)|\le C \quad \text{ for }k>1.
    \end{equation}
\end{lemma}
\begin{proof}
    By Theorem \ref{thm: inv T off disc} and \eqref{jost exist: def M}, $\|M_1(\cdot,k+i0)\|_{L^\infty_{-1,+}}\le C$ for $k\in (0,\frac14)\cap (1,\infty)$. By \eqref{def: a_1}, \eqref{def: b1}, $|a_1(k+i0)|, |b_1(k+i0)|$ are bounded for $k$ in the same range. By the proof of Theorem \ref{thm: a b diff k=1/2}, we can write
    \begin{equation}
        \partial_k a_1(k+i0)=\int_\R [\partial_k A(k+i0)(I-A(k+i0))^{-1}u]M_1(x,k+i0)~dx.
    \end{equation}
    By estiamte \eqref{eq: A norm est} in the proof of Lemma \ref{lem: strong cont}, 
    \begin{equation}
        \|A(k+i0)\|_{L^1_{1,+}\to L^1_{1,+}}\lesssim \|u\|_{L^1_1\cap L^2}\|G_L(\cdot,k+i0)\|_{L^{\infty}_{-1,+}+L^2}.
    \end{equation}
    By Lemma \ref{lem: green uniform split}, we get 
    \begin{equation}
         \|A(k+i0)\|_{L^1_{1,+}\to L^1_{1,+}}\le \tfrac{1}{2}
    \end{equation}
    for $\|u\|_{L^1_1\cap L^2_1}$ sufficiently small and $k\in (0,\frac14)\cap (1,\infty)$. It follows that 
    \begin{equation}
        \left\|(I-A(k+i0))^{-1}\right\|_{L^1_{1,+}\to L^1_{1,+}}\le 2
    \end{equation}
    for $k$ in the same range.
    By estimate \eqref{eq: Dk G norm est} in the proof of Lemma \ref{lem: A differntiable k>0}, we see that
    \begin{equation}
        \|\partial_k A(k+i0)\|_{L^1_{1,+}\to L^1_{1,+}}\lesssim \|u\|_{L^1_1}\|\partial_k G_L(\cdot,k+i0)\|_{L^\infty_{-1,+}}.
    \end{equation}
    Lemma \ref{lem: Dk G uniform} and the above estimates now imply 
    $|\partial_k a_1(k+i0)|\le \frac{C}{k} $
    for $0<k<\frac14$ and 
    $ |\partial_k a_1(k+i0)|\le C$
    for $k>1$.
    Similarly, we can write
    \begin{align}
        \partial_k b_1(k+i0)
        =&\int_\R u(x)(\partial_k e^{-i\lambda x})M_1(x,k+i0)~dx\notag\\
        &\quad+ \int_{\R}[\partial_k A(k+i0)(I-A(k+i0))^{-1}(ue^{-i\lambda \cdot} )]M_1(x,k+i0)~dx
    \end{align}
    and estimate to get
    $|\partial_k b_1(k+i0)|\le \frac{C}{k}$ 
    for $0<k<\frac14$ and $|\partial_k b_1(k+i0)|\le C$
    for $k>1$.
   The claimed bounds on $\partial_k\rho$ now follow from the definition of $\rho$ (see \eqref{def: rho}, \eqref{def: a}, \eqref{def: b}), the above bounds on $a_1$, $b_1$, $\partial_k a_1$, $\partial_k b_1$, and explicit bounds on $\lambda$ and $k^*$ (see \eqref{eq: k* bounds k small}, \eqref{eq: k* bounds k large}).
\end{proof}

With the above lemmas, we can now show the meromorphic dependence of $M_1$ and $N_1$ on $\zeta$ in the whole cut plane $\widetilde{\mathbb C}$. We state this as

 \begin{theorem}\label{thm: general meromorphy of M N}
    Let $\|u\|_{L^1_1\cap L^2_1}<\epsilon$ for $\epsilon$ sufficiently small. Then $N_1(\cdot,\zeta)$ is meromorphic for $\zeta$ in the interior of $\widetilde{\C}$ with continuous boundary values except perhaps at $\zeta=\frac12+i0$. Its poles coincide with zeros of $a(\zeta)$, which can appear only in the interior of $D^+_\delta$ and can accumulate only possibly at $\zeta=\frac12+i0$.
    Similarly, $M_1(\cdot,\zeta)$ is meromorphic for $\zeta$ in the interior of $\widetilde{\C}$ with continuous boundary values except perhaps at $\zeta=\frac{1}{2}-i0$. Its poles coincide with zeros of $\breve a(\zeta)$, which can appear only in the interior of $D^-_\delta$, and can accumulate only possibly at $\zeta=\frac12-i0$.
\end{theorem}
\begin{remark}
    $M_1(\cdot,\zeta)$ and $N_1(\cdot,\zeta)$ are meromorphic with values in $L^\infty$ for $\zeta$ in the interior of $\widetilde{\C}$. The continuity  on $\R^++i0$ is with respect to the $L^\infty_{-1,+}$ norm. The continuity on $\R^+-i0$ is with respect to the $L^\infty_{-1,-}$ norm. The continuity at $\zeta=0$ is with respect to the $L^\infty$ norm.
\end{remark}
\begin{proof}
    By Theorem \ref{thm: inv T off disc} and Lemma \ref{lem: T anlytc}, $N_1(\cdot,\zeta)=(I-T_R(\zeta))^{-1}1\in L^\infty_{-1,-}$ is continuous on $\widetilde{\C}\setminus \left((\R^++i0)\cup D^+_\delta\right)$. $M_1(\cdot,\zeta)=(I-T_L(\zeta))^{-1}1\in L^\infty_{-1,+}$ is continuous on $\widetilde{\C}\setminus \left((\R^+-i0)\cup D^-_\delta\right)$. 
    Since $N_1(\cdot,\zeta)$ is defined on $ \widetilde{\C}\setminus \left((\R^++i0)\cup D^+_\delta\right)$, $a(\zeta)\ne 0$ in the same range by Corollary \ref{cor: a=0 vs N1 nonexist}. On the other hand, $a(\zeta)\ne 0$ for $\zeta\in \mathbb R^++i0$, $\zeta\ne \frac12+i0$ by \eqref{ab.id}. It follows that the zeros of $a$ must be contained in $D^+_\delta$. By the analyticity of $a$, these zeros must be isolated, and may only possibly accumulate on $\R^++i0$. However, Theorem \ref{thm: a b cont} implies that $a$ is continuous on $(\R^++i0)\setminus \{\frac12+i0\}$. Thus accumulation can only possibly happen at $\frac12+i0$. A similar argument shows that zeros of $\breve a$ can only occur in $D^-_\delta$, and can accumulate only possibly at $\frac12-i0$.
    By Corollary \ref{cor: a a breve duality} and Lemma \ref{lem: M=aN}, For $\zeta\in (\R^++i0)\cup D^+_\delta$ away from those zeros of $a$, we have $a(\zeta)\ne 0$ and $N_1(\cdot,\zeta)=\frac{M_1(\cdot,\zeta)}{a(\zeta)}$. Thus the continuity of $N_1(\cdot,\zeta)$ for $\zeta\in (\R^++i0)\cup D^+_\delta$ follows from those of $M_1(\cdot,\zeta)$ and $a(\zeta)$. Since $G(x,0)=i\chi_{\R^+}(x)$, by direct computation one gets that $a(0)=e^{i\int_\R u~dx}\ne 0$. By continuity $a(\zeta)\ne 0$ and $N_1(\cdot,\zeta)=\frac{M_1(\cdot,\zeta)}{a(\zeta)}$ for $\zeta$ near $0$. Thus $N_1(\cdot,\zeta)$ is continuous at $0$ as maps both into $L^\infty_{-1,-}$ and $L^\infty_{-1,+}$, hence is continuous with respect to the $L^\infty$ norm at $\zeta=0$. When $\zeta\in \widetilde{\C}\setminus (\R^+\pm i0)$ and is away from those zeros of $a$ in $D^+_\delta$, $N_1(\cdot,\zeta)$ maps analytically into $L^\infty_{-1,-}$ by Lemma \ref{lem: T anlytc}, although we know that each $N_1(\cdot,\zeta)$ is actually in $L^\infty$ (Lemma \ref{lem: jost exist 1}). To get analyticity with respect to the $L^\infty$ norm, we only need to show that $\int_R N_1(x,\zeta) g(x)~dx$ is analytic for every $g\in L^1$ with compact support, but this is obvious due to the analyticity of $N_1(\cdot,\zeta)$ in the $L^\infty_{-1,-}$ norm.
\end{proof}

Note that even though $a_1(\zeta)$ and $b_1(k+i0)$ are continuous at $\zeta=\frac12+i0$ and $k=\frac12$, $a(\zeta)$ and $b(k+i0)$ in general are not, since their definitions  involve division by $1-2\zeta$ and $1-2k^*$. To distinguish different behaviors of $a(\zeta)$ near $\frac12+i0$ and $\breve a(\zeta)$ near $\frac12-i0$, we make the following definition. To emphasize the dependence of the Jost functions and scattering functions on $u$, we denote them by a superscript $u$, such as $a_1^u(\zeta)$, $M_1^u(x,\zeta)$ etc.
\begin{definition}\label{def: genericity}
    Let $\epsilon>0$ be sufficiently small, and $(L^1_1\cap L^2_1)_\epsilon$ the ball centered at 0 of radius $\epsilon$ in $L^1_1\cap L^2_1$. A scattering potential $u\in (L^1_1\cap L^2_1)_\epsilon$ is called upper-generic if $a_1^u(\frac12+i0)\ne 0$ and non-upper-generic if otherwise. It is called lower-generic if $\breve a_1^u(\frac12-i0)\ne 0$ and non-lower-generic otherwise. Here $a_1^u$ is the $a_1$ function defined as in \eqref{def: a_1}, and $\breve a_1^u$ similarly.
\end{definition}

We justify the terminology with the following result.

\begin{theorem}\label{thm: genericity}
    Upper(lower)-generic potentials form a open dense subset of $(L^1_1\cap L^2_1)_\epsilon$. 
\end{theorem}
\begin{proof}
    We only show the proof for the upper-generic case. Recall from \eqref{jost prty: def M} that for $\zeta\in D^+_\delta$,
    $M_1^u = (I-T_L^u)1$.
    Thus 
    \begin{equation}\label{eq: Mu1-Mu2}
        M_1^{u_1}-M_1^{u_2} = (I-T_L^{u_1})^{-1}\left(T_L^{u_1}-T_L^{u_2}\right)\left(I-T_L^{u_2}\right)^{-1}1.
    \end{equation}
    By the proof of Theorem \ref{thm: inv T off disc}, $\|T_L^{u_i}\|_{L^\infty_{-1,+}\to L^\infty_{-1,+}}$, $i=1,2$, are bounded by $\frac12$ for $\zeta\in D^+_\delta$, provided $\|u\|_{L^1_1\cap L^2_1}$ is small enough. We also have from \eqref{eq: est T_L via u}
    \begin{align}
        &~\|T_L^{u_1}-T_L^{u_2}\|_{L^\infty_{-1,+}\to L^\infty_{-1,+}}\notag\\
        \le &~C\|G_L(\cdot,\zeta)\|_{L^2_{-1,+}\cap L^\infty_{-1,+}}\|u_1-u_2\|_{L^1_1\cap L^2_1}.\label{eq: TLu1-TLu2}
    \end{align}
    By Lemma \ref{lem: green uniform split}, $\|G_L(\cdot,\zeta)\|_{L^2_{-1,+}\cap L^\infty_{-1,+}}$ is uniformly bounded for $\zeta\in D^+_\delta$. It follows from \eqref{eq: Mu1-Mu2} and \eqref{eq: TLu1-TLu2} that  $u\mapsto M_1^{u}(\cdot,\frac12+i0)$ is continuous from $(L^1_1\cap L^2_1)_\epsilon$ to $L^\infty_{-1,+}$. Since $L^1_1\cap L^2_1$ is continuously embedded in $L^1_{1,+}$, and $L^\infty_{-1,+}=[L^1_{1,+}]^*$, we get from \eqref{def: a_1} that $u\mapsto  a_1^u(\frac12+i0)$ is continuous on $(L^1_1\cap L^2_1)_\epsilon$. It follows that the set of upper-generic potentials is open in $(L^1_1\cap L^2_1)_\epsilon$.

    To see density, let us assume $u\in (L^1_1\cap L^2_1)_\epsilon$ is non-generic and consider 
    \begin{equation}
        f(t)=a_1^{u+tv}(\tfrac12+i0)=\int_\R (u+tv)M_1^{u+tv}~dx
    \end{equation} for $v\in L^1_1\cap L^2_1$ and $|t|$ small. Note that
    \begin{align}
        \partial_t M_1^{u+tv} \big|_{t=0}&= \partial_t (I-T_L^{u+tv})1\big|_{t=0}\notag\\
        &=-(I-T_L^u)^{-1} (\partial_t T_L^{u+tv})\big|_{t=0} (I-T_L^u)^{-1}1\notag\\
        &=-(I-T_L^u)^{-1}\left(G_L*(vM_1^{u})\right).
    \end{align}
    
    It follows that
    \begin{align}
        f'(0)&=\int_\R \left[vM_1^u-u(I-T_L^u)^{-1}(G_L*(vM_1^u))\right]~dx\notag\\
        &=\int_\R vM_1^u(1-\widetilde {G_L}*w)~dx,
    \end{align}
    where $\widetilde{G_L}(x)=G_L(-x)$ and 
    \begin{equation}
        w=[I-(T_L^u)^*]^{-1}u\in L^1_{1,+}.
    \end{equation}
    Here $(T_L^u)^*=A^u$ is given by \eqref{def: A(zeta)} (see also Lemma \ref{lem: strong cont}).

    We claim that the zero level set of $M_1^u (1-\widetilde{G_L}*w)$ is a null set. It follows that the set on which $M_1^u (1-\widetilde{G_L}*w)\ne 0$ has positive measure. We can then choose $v\in L^1_1\cap L^2_1$ such that $f'(0)>0$. This implies that $a_1^{u+tv}>0$ for $|t|>0$ sufficiently small. Thus generic potentials are dense in $(L^1_1\cap L^2_1)_\epsilon$. 

    To prove the claim, we study the zero level sets  $\{M_1=0\}$ and $\{1-\widetilde{G_L}*w=0\}$ in turns. Note that $uM_1\in L^1_{1,-}$ and $\widehat{uM_1}\in L^\infty$.  By \eqref{jost prty: M1 1/2 ext} and \eqref{eq: R bound}, $M_1(x+iy)\to M_1(x+i0)$ almost everywhere as $y\searrow 0$. Similar to the proof of Lemma \ref{lem: anal cont 1}, we can pick a sequence of $f_n\in C_0^\infty$ such that $f_n\to uM_1$ in $L^1_{1,-}$. Using \eqref{jost prty: M1 1/2 ext} and Lemma \ref{lem: inv Fourier R1}, we get that $G_L(x+iy)*f_n(x)$ converges to $G_L(x+iy)*(uM_1)(x)$ uniformly on compact subsets of the strip $0<y<2$. Thus $M_1(x+iy)$ is holomorphic in the strip. By the Lusin-Privalov theorem, if the zero set of $M_1(x+i0)$ has positive measure, $M_1\equiv 0$. This contracts the left limit $\lim_{x\to-\infty}M_1(x+i0)=1$, which follows from \eqref{jost prty: M1 1/2 ext} in the same way as in the proof of Lemma \ref{lem: anal cont 1/2}.

    We now turn to $\{1-\widetilde{G_L}*w=0\}$. In fact, by reversing the direction, we can equivalently study $\{1-G_L*\widetilde{w}=0\}$, with $\widetilde w \in L^1_{1,-}$. By repeating the above proof on $G_L*(uM_1)$, we conclude that $G_L(x+iy)*\widetilde{w}(x)$ is holomorphic in the strip, with almost everywhere boundary values $G_L(x+i0)*\widetilde{w}(x)$ as $y\searrow 0$. Once again, if $\{1-G_L*\widetilde{w}=0\}$ has positive measure, then 
    \begin{equation}\label{eq: GL*w=1}
        G_L*\widetilde{w}\equiv 1.
    \end{equation} For any $N\in 2^\mathbb Z$, denote by $P_N=\varphi(\frac{D}{N})$ the standard Littlewood-Payley projector with frequency comparable to $N$, so that $P_N=P_NP_{N/2\le \cdot\le 2N}$, where $P_{N/2\le \cdot\le 2N}=P_{N/2}+P_N+P_{2N}$. As $P_N$ is a convolution with a zero average Schwartz class function, we have
    \begin{equation}\label{eq: 7.43}
        G_L*(P_N\widetilde{w})=(P_{N/2\le \cdot\le 2N}G_L)*(P_N\widetilde{w})=P_N 1=0.
    \end{equation}
    By \eqref{def: G} we see easily that for $\zeta=\frac12+i0$
    \begin{equation}
        G_L(x)-\mathcal \mathcal F^{-1}\left[\text{P.V.}\tfrac{1}{\xi-\frac12(1-e^{-2\xi})}\right](x)=\tfrac{i}{3}-\tfrac{x}2.
    \end{equation}
    Thus $\widehat{G_L}-\text{P.V.}\frac{1}{\xi-\frac12(1-e^{-2\xi})}$ is a distribution supported at $\xi=0$. Since the symbol of $P_{N/2\le \cdot\le 2N}$ is supported away from $0$, we get
    \begin{equation}\label{eq: 7.45}
        \mathcal \mathcal F\left[P_{N/2\le \cdot\le 2N}G_L\right] = \frac{\varphi(\frac{2\xi}{N})+\varphi(\frac{\xi}{N})+\varphi(\frac{\xi}{2N})}{\xi-\frac12(1-e^{-2\xi})}.
    \end{equation}
    It follows from \eqref{eq: 7.43} and \eqref{eq: 7.45} that 
    \begin{equation}
        \frac{\varphi(\frac{\xi}{N})\widehat{\widetilde{w}}}{\xi-\frac12(1-e^{-2\xi})}=0,
    \end{equation}
    which gives $P_N\widetilde{w}=0$ for all $N$. Since $\widetilde{w}\in L^1_{1,-}\subset L^1$, we get $\widetilde{w}=0$. This contradicts \eqref{eq: GL*w=1}. Therefore, $\{1-\widetilde{G_L}*w=0\}$ must be a null set. Since $$\{M_1 (1-\widetilde{G_L}*w)=0\}\subset \{M_1=0\}\cup \{1-\widetilde{G_L}*w=0\},$$
    $\{M_1 (1-\widetilde{G_L}*w)=0\}$ is a null set. The proof is complete.
\end{proof}

The two notions of genericity turn out to coincide for stronger decay assumptions on $u$.

\begin{theorem}\label{thm: up low generic}
    There exists an $\epsilon>0$ such that if  $\|u\|_{ L^1_2\cap L^2_1}<\epsilon$, then $u$ is upper-generic if and only if it is lower-generic.
\end{theorem}
\begin{proof}
    Suppose $u$ is upper-generic. By definition $a_1(\frac12+i0)\ne 0$. By the definitions of $a_1$ and $b_1$ (cf. \eqref{def: a_1}, \eqref{def: b1}), it's easy to see that $b_1(\frac12+i0)=a_1(\frac12+i0)$. From the definitions of $a$, $b$ and $\rho$ (cf. \eqref{def: a}, \eqref{def: b}, \eqref{def: rho}), we get 
    \begin{equation}
        \rho(k)=\frac{b(k+i0)}{a(k+i0)}=\frac{1-2k}{1-2k^*}\frac{ib_1(k+i0)}{1-2k+ia_1(k+i0)}.
    \end{equation}
    Thus $\rho$ can be extended to $k=\frac12$ with $\rho(\frac12)=-1$. By Theorem \ref{thm: a b diff k=1/2}, $\rho(k)$ is differentiable on $\R^+$, in particular also at $k=\frac12$. If $u$ is also non-lower-generic, then $\breve a_1(\frac12-i0)=0$. By the definition of $\breve a$ (cf. \eqref{def: breve a}, \eqref{def: breve a_1}), we have
    \begin{align}
        \lim_{k\to\frac12}\breve a(k-i0)&=1-i\lim_{k\to\frac12}\frac{\breve a_1(k-i0)-\breve a_1(\frac12-i0)}{1-2k}\notag\\
        &=1+\frac i 2 \partial_k \breve a_1(\tfrac12-i0).\label{eq: breve a limit form 1}
    \end{align}
    We now use the identity \eqref{eq: a jump}, which will be proven in a later section, and the relation \eqref{eq: a a breve = 1} to get
    \begin{align}
        \breve a(k-i0)&=\frac{1}{a(k+i0)[1-\rho(k)\rho(k^*)]}\notag\\
        &=\frac{1}{1-2k+ia_1(k+i0)}\frac{1-2k}{1-\rho(k)\rho(k^*)}.\label{eq: breve a limit form 2}
    \end{align}
    From \eqref{eq: breve a limit form 1}, \eqref{eq: breve a limit form 2}, we see that there exists a $C>0$ such that 
    \begin{equation}
        \frac{|1-2k|}{|1-\rho(k)\rho(k^*)|}\le C
    \end{equation}
    as $k\to\frac12$. Note that
    \begin{equation}
        1-\rho(\tfrac12)\rho(\tfrac12^*)=1-[\rho(\tfrac{1}{2})]^2=0.
    \end{equation}
    Thus
    \begin{equation}
        \left|\frac{d}{dk}[1-\rho(k)\rho(k^*)]\right|\bigg|_{k=0}=\lim_{k\to \frac12}\frac{|1-\rho(k)\rho(k^*)|}{|k-\frac12|}\ge \frac2{C}>0.
    \end{equation}
    However,
    \begin{align}
        \frac{d}{dk}[1-\rho(k)\rho(k^*)]\bigg|_{k=0}&=-\rho'(\tfrac12)\rho(\tfrac{1}{2})-\rho(\tfrac12)\rho'(\tfrac12)\frac{dk^*}{dk}\bigg|_{k=0}\notag\\
        &=-\rho'(\tfrac12)\rho(\tfrac{1}{2})+\rho(\tfrac12)\rho'(\tfrac12)\notag\\
        &=0.\notag
    \end{align}
    This contradiction shows that $u$ being upper-generic implies that $u$ is also lower-generic. The reverse implication can be proven similarly, using identity \eqref{eq: a breve jump}. 
\end{proof}

By Theorem \ref{thm: up low generic}, we can combine the two notions of genericity for $u$ with stronger decay.
\begin{definition}\label{def: combined generic}
    Let $\epsilon>0$ be small enough so that Theorem \ref{thm: up low generic} holds. A potential $u\in (L^1_2\cap L^2_1)_\epsilon$ is called generic if it is upper-(lower-)generic.
\end{definition}
\begin{remark}
    It is curious whether the two notions of genericity continue to coincide for potentials with weaker decay, for instance, for the more general class $(L^1_1\cap L^2_1)_\epsilon$ used in Definition \ref{def: genericity}. We leave such questions to future studies. On the other hand, the scattering theory studied in this paper can be built for either $M_1$ or $N_1$, and does not care about whether the two types of genericity coincide or not.
    %The two notions of genericity are independent of each other in general. In particular, a potential $u$ can be both upper- and lower-generic, or it can be only upper- but not lower-generic. The former is easy to see, as by Theorem \ref{thm: genericity}, the intersection of upper- and lower-generic potentials is still a dense open set in $(L^1_1\cap L^2_1)_\epsilon$. To see the latter possibility, 
\end{remark}

For real potentials $u$, upper-genericity of $u(x)$ is related to lower-genericity of $u(-x)$.

\begin{theorem}\label{thm: LR genericity symmetry}
    Let $u\in (L^1_1\cap L^2_1)_\epsilon$ for $\epsilon>0$ sufficiently small. Suppose $u$ is real-valued. Then $u(x)$ is upper-generic if and only if $u(-x)$ is lower-generic. $u(x)$ is lower-generic if and only if $u(-x)$ is upper-generic.
\end{theorem}
\begin{proof}
This is a consequence of a more general fact: 
\begin{equation}\label{eq: M_1 conj}
    \overline{M_1^u(x,\zeta)}=N_1^{\tilde u}(-x,\overline\zeta),
\end{equation}
where $\tilde u(x)=u(-x)$. In fact, \eqref{eq: M_1 conj} implies that
\begin{align*}
    \overline{a_1^u(\tfrac12+i0)}&=\overline{\int_\R u(x)M_1^u(x,\tfrac12+i0)~dx}\\
    &=\int_\R u(x)N_1^{\tilde u}(-x,\tfrac12-i0)~dx\\
    &=\int_\R \tilde u(x)N_1^{\tilde u}(x,\tfrac12-i0)~dx\\
    &=\breve a_1^{\tilde u}(\tfrac12-i0).
\end{align*}
To prove \eqref{eq: M_1 conj}, use \eqref{eq: G_L bar} and note that
\begin{align*}
    \overline{[T_L^u(\zeta)m](-x)}&=\overline{\int_\R G_L(-x-y,\zeta)u(y)m(y)~dy}\\
    &= \int_\R G_R(x+y,\overline\zeta)u(y)\overline{m(y)}~dy\\
    &=\int_\R\int_\R G_R(x-y,\overline\zeta)u(-y)\overline{m(-y)}~dy\\
    &=\int_\R\int_\R G_R(x-y,\overline\zeta)\tilde u(y)\overline{\widetilde m(y)}~dy\\
    &=[T_R^{\tilde u}(\overline{\zeta})\overline{\widetilde m}](x).
\end{align*}
Taking the complex conjugate of 
\begin{equation*}
    [(I-T_L^u(\zeta))M_1^u(\cdot,\zeta)](-x)=1
\end{equation*}
to get
\begin{equation}
    [(I-T_R^{\tilde u}(\overline{\zeta}))\overline{M_1^u(-\cdot,\zeta)}](x)=1.
\end{equation}
Thus
\begin{equation}
    \overline{M_1^u(-x,\zeta)}=N_1^{\tilde u}(x,\overline\zeta),
\end{equation}
which is equivalent to \eqref{eq: M_1 conj}.
\end{proof}

% By Theorems \ref{thm: inv T off disc} and \ref{thm: inv T on disc}, $N_1(\cdot,\zeta)$ exists for $\zeta\in \widetilde{\C}\setminus D^+_\delta$ and also for $\zeta\in D^+_\delta\setminus \{\frac12+i0\}$ if $a(\zeta)\ne 0$. If $u$ is upper-generic, then obviously $a(\zeta)\to \infty$ as $\zeta\to \frac12+i0$. Thus $a(\zeta)\ne 0$ in a sufficiently small upper semi-disc around $\frac12+i0$. Furthermore $|a(\zeta)|\ge 1$ for $\zeta\in \R^++i0$ by \eqref{ab.id}. Thus $a(\zeta)\ne 0$ for $\zeta\in (\R^++i0)\cap D^+_\delta\setminus\{\frac12+i0\}$. It follows that the zeros of $a(\zeta)$ must be in a compact subset in the interior of $D^+_\delta$. Since $a(\zeta)$ is analytic there, there can only be finitely many such zeros. In view of Corollary \ref{cor: a a breve duality} and Lemma \ref{lem: M=aN}, we see that $\breve a(\zeta)$ must be meromorphic. The above discussion can be extended to prove the following.

\begin{corollary}
    Let $u\in (L^1_2\cap L^2_1)_\epsilon$ for $\epsilon>0$ sufficiently small. Suppose $u$ is real-valued. Then $u(x)$ is generic if and only if $u(-x)$ is generic.
\end{corollary}
\begin{proof}
    The assertion is obvious from Theorem \ref{thm: LR genericity symmetry} and Definition \ref{def: combined generic}.
\end{proof}

A quick consequence of upper- and lower-genericity is the finiteness of the zeros of $a$ and $\breve a$. The respective Jost functions will also turn out to be continuous on the entire $\widetilde{\C}$.

\begin{theorem}
    Let $\|u\|_{L^1_1\cap L^2_1}<\epsilon$ for $\epsilon$ sufficiently small. If $u(x)$ is upper-generic, then $N_1(\cdot,\zeta)$ is meromorphic for $\zeta$ in the interior of $\widetilde{\C}$ with continuous boundary values. It has at most finitely many poles contained in $D^+_\delta$ coinciding with zeros of $a(\zeta)$.
    If $u(x)$ is lower-generic, then $M_1(\cdot,\zeta)$ is meromorphic for $\zeta$ in the interior of $\widetilde{\C}$ with continuous boundary values. It has at most finitely many poles contained in $D^-_\delta$, which coincide with zeros of $\breve a(\zeta)$.
\end{theorem}

\begin{proof}
    Suppose $u$ is upper-generic. By Theorem \ref{thm: general meromorphy of M N}, it remains to show that $\frac1{a(\zeta)}$ has a finite limit as $\zeta\to\frac12+i0$. But this is obvious as $a(\zeta)=1+\frac{i}{1-2\zeta}a_1(\zeta)$, and $a_1(\frac12+i0)\ne 0$ by Definition \ref{def: genericity}. The proof for the lower-generic case is similar.
\end{proof}

             %%  Scattering Data
\section{The bound states}\label{sec: bound states}

In this section, we study bounds states of the $L_u$ operator in the Lax pair. In the first part, we assume $u\in L^1_1\cap L^2_1\cap L^\infty$ with sufficiently small norm, and show that the bound state eigenvalues correspond to zeros of the scattering function $a(\zeta)$. To establish this connection, we need to show additional exponential decay of the eigenfunctions. Such a decay is easier to prove when the scattering potential $u$ has compact support. The general case can then be established by a limiting process. In the second part, we assume $u\in L^1_2\cap L^\infty$, and show that the discrete spectrum of $L_u$ is finite. Note that we do not require smallness of norm for this particular result, although we assume stronger decay in $L^1$.

\subsection{\texorpdfstring{Bound states and $a(\zeta)$}{Bound states and a(zeta)}}

Throughout this section, we assume $\|u\|_{L^1_1\cap L^2_1}$ is sufficiently small, so that the existence theory of Sections \ref{sec: existence} and \ref{sec: scat data} apply. Furthermore, we assume $\|u\|_{L^\infty}$ is small and $u$ is real-valued, so that the self-adjoint spectral theory of Section \ref{sec: lax pair} applies. By Theorem \ref{thm: general meromorphy of M N}, all zeros of $a$ are in the interior of $D^+_\delta$ and may accumulate only possibly at $\frac12+i0$. Denote by $\zeta_j$, $j=1,2,\dots$ an enumeration of the zeros of $a$ in $D^+_\delta\setminus \{\frac12+i0\}$.  Our goal in this section is to show the following theorem.

\begin{theorem}\label{thm: bound states and a}
    Let $\|u\|_{L^1_1\cap L^2_1}<\epsilon$ for some $\epsilon>0$ sufficiently small, and assume $u$ is real-valued. Let $\zeta_j$, $j=1,2,\dots$ be the zeros of $a$ as above. Then 
    \begin{enumerate}[(i)]
        \item The zeros of $\breve a$ (equivalently, the poles of $a$, cf. Lemma \ref{lem: M=aN}) in $D^-_\delta\setminus\{\frac12-i0\}$ are exactly $\{\zeta_j^*~|~j=1,2,\dots\}$. Furthermore, we have
    \begin{equation}\label{M_1 N_1 relation at pole}
        M_1(x,\zeta_j)e^{-i\zeta_j x}=c_jN_1(x,\zeta_j^*)e^{-i\zeta_j^* x},
    \end{equation}
    where
    \begin{equation}\label{def: c_j}
        c_j=\frac{i}{1-2\zeta_j^*}\int_\R u(x)M_1(x,\zeta_j)e^{-i\lambda_j x}~dx\ne 0,
    \end{equation}
    and $\lambda_j=Z^{-1}(\zeta_j)$ (see Definition \ref{def: Z inv}).
    \item For each $j$, let
    \begin{equation}\label{def: phi_j}
        \varphi_j(x)=N_1(x+i,\zeta_j^*)e^{-i\zeta_j^*x}.
    \end{equation}
    Then $\varphi_j\in L^2$ and 
    \begin{equation}\label{eq: a' identity}
        a'(\zeta_j)=\overline{\breve a'(\zeta_j^*)}=-ic_j e^{-\lambda_j}\int_\R |\varphi_j(x)|^2 ~dx,
    \end{equation}
    where  $c_j$ is given by \eqref{def: c_j}. As a result, each $\zeta_j$ is a simple zero of $a$, and each $\zeta_j^*$ is a simple pole of $a$.

    \item Further assume $u\in L^\infty$. Let $\varphi_j$ be given by \eqref{def: phi_j}.
    Then $\varphi_j\in D(L_u)$ and 
    \begin{equation}
        L_u\varphi_j = \eta_j \varphi_j
    \end{equation}
    for $\eta_j=-\zeta_j e^{-2\zeta_j}\in\R$. In other words, $\eta_j\in \sigma_{\text{disc}}(L_u)$.
    \item Further assume $\|u\|_{L^\infty}<\epsilon$ for a sufficiently small $\epsilon>0$. Let $\eta\in \sigma_{\text{disc}}(L_u)$ and $\varphi\in D(L_u)$ satisfy
    \begin{equation}
        L_u\varphi=\eta\varphi,
    \end{equation}
    then there exists a $j$ such that $\eta=\eta_j=-\zeta_j e^{-2\zeta_j}$, and $\varphi=C\varphi_j$ for some constant $C$, where $\varphi_j$ is defined by \eqref{def: phi_j}.
    \end{enumerate}
\end{theorem}

We will prove this theorem by showing a series of lemmas, starting with a relation between $G_L$ and $G_R$.

\begin{lemma}\label{lem: GL(zeta) GR(zeta*)}
For $\zeta\in D^+_\delta\setminus\{\frac12+i0\}$, we have
\begin{equation}\label{eq: GL(zeta) GR(zeta*) relation}
    G_L(x,\zeta)=e^{i\lambda x}G_R(x,\zeta^*)+\frac{i}{1-2\zeta}+\frac{i}{1-2\zeta^*}e^{i\lambda x}.
\end{equation}
Similarly, for $\zeta\in D^-_\delta\setminus\{\frac12-i0\}$, we have
\begin{equation}\label{eq: GR(zeta) GL(zeta*) relation}
    G_R(x,\zeta)=G_L(x,\zeta^*)e^{i\lambda x}-\frac{i}{1-2\zeta}-\frac{i}{1-2\zeta^*}e^{i\lambda x}.
\end{equation}
\end{lemma}
\begin{proof}
    For $\zeta \in D^+_\delta\setminus\{\frac12+i0\}$, we compute $G_L$ by shifting the contour of integration:
    \begin{align*}
        G_L(x,\zeta)&=\frac1{2\pi}\int_{\Gamma_L}\frac{e^{ix\xi}}{\xi-\zeta(1-e^{-2\xi})}~d\xi\\
        &=\frac1{2\pi}\int_{\Gamma_R+\lambda}\frac{e^{ix\xi}}{\xi-\zeta(1-e^{-2\xi})}~d\xi +\frac{i}{1-2\zeta}+\frac{i}{1-2\zeta^*}e^{i\lambda x}%\label{eq: 231}  %Commented out by PAP 2025-05-20
        \\
        &=\frac{1}{2\pi}\int_{\Gamma_R}\frac{e^{ix(\xi+\lambda)}}{\xi+\lambda-\zeta(1-e^{-2(\xi+\lambda)})}~d\xi+\frac{i}{1-2\zeta}+\frac{i}{1-2\zeta^*}e^{i\lambda x}%\label{eq: 232} % Commented out by PAP 2026-09-18
        \\
        &=\frac{e^{i\lambda x}}{2\pi}\int_{\Gamma_R}\frac{e^{ix\xi}}{\xi-\zeta^*(1-e^{-2\xi})}~d\xi+\frac{i}{1-2\zeta}+\frac{i}{1-2\zeta^*}e^{i\lambda x}\\
        &=e^{i\lambda x}G_R(x,\zeta^*)+\frac{i}{1-2\zeta}+\frac{i}{1-2\zeta^*}e^{i\lambda x}.
    \end{align*}
    The case when $\zeta\in D^-_\delta\setminus\{\frac12-i0\}$ can be proven similarly.
\end{proof}

\begin{lemma}\label{lem: M N relation at zeta_j}
    Let $u\in L^1_1\cap L^2_1$ with small norm. Assume in addition that $u\in C_0^\infty$. Then Theorem \ref{thm: bound states and a} (i) holds.
\end{lemma}
\begin{proof}
Recalling the definitions of $M_1$ and $N_1$ from \eqref{jost exist: def M}, \eqref{jost exist: def N}, and using \eqref{eq: GL(zeta) GR(zeta*) relation}, we get
\begin{align}
M_1(x,\zeta_j) 
    &= 1+ \int_\R G_L(x-y,\zeta_j)u(y)M_1(y,\zeta_j)~dy\notag\\
    &= 1+ e^{i\lambda_j x}\int_\R G_R(x-y,\zeta_j^*)u(y)M_1(y,\zeta_j)e^{-i\lambda_j y}~dy\notag\\
    &\qquad+\frac{i}{1-2\zeta_j}\int_\R u(y)M_1(y,\zeta_j)~dy+\frac{ie^{i\lambda_j x}}{1-2\zeta_j^*}\int_\R u(y)M_1(y,\zeta_j)e^{-i\lambda_j y}~dy\notag\\
    &=e^{i\lambda_j x}\int_\R G_R(x-y,\zeta_j^*)u(y)M_1(y,\zeta_j)e^{-i\lambda_j y}~dy \notag\\
    &\qquad + \frac{ie^{i\lambda_j x}}{1-2\zeta_j^*}\int_\R u(y)M_1(y,\zeta_j)e^{-i\lambda_j y}~dy.
    \label{eq: M1 N1 pole relation 1}
\end{align}
To get the last step, we used
\begin{equation}
a(\zeta_j) = 1+\frac{i}{1-2\zeta_j}\int_\R u(y)M_1(y,\zeta_j)~dy=0.
\end{equation}
Equation \eqref{eq: M1 N1 pole relation 1} can be rewritten as
\begin{equation}\label{eq: M1 N1 pole relation 2}
M_1(x,\zeta_j) e^{-i\lambda_j x} = c_j+\int_\R G_R(x-y,\zeta_j^*)u(y)M_1(y,\zeta_j)e^{-i\lambda_j y}~dy,
\end{equation}
where
\begin{equation}\label{eq: M1 N1 pole relation 3}
c_j=\frac{i}{1-2\zeta_j^*}\int_\R u(y)M_1(y,\zeta_j)e^{-i\lambda_j y}~dy.
\end{equation}
The definition of $N_1$ now implies
\begin{equation}\label{eq: M1 N1 pole relation 4}
M_1(x,\zeta_j)e^{-i\lambda_j x} = c_j N_1(x,\zeta_j^*),
\end{equation}
which is equivalent to \eqref{M_1 N_1 relation at pole} since $\lambda_j=\zeta_j-\zeta_j^*$ (see \eqref{eq: zeta-zeta*}). By \eqref{GL-GR}, we have from \eqref{eq: M1 N1 pole relation 2} that
\begin{equation}
    M_1(x,\zeta_j) e^{-i\lambda_j x} = \int_\R G_L(x-y,\zeta_j^*)u(y)M_1(y,\zeta_j)e^{-i\lambda_j y}~dy.
\end{equation}
Since $c_j$ is obviously nonzero from \eqref{eq: M1 N1 pole relation 4}, we have
\begin{equation}
    N_1(x,\zeta_j^*)=\int_\R G_L(x-y,\zeta_j^*)u(y)N_1(y,\zeta_j^*)~dy.
\end{equation}
By \eqref{a=0 vs trivial kernel}, we get $\breve a(\zeta_j^*)=0$. If we start from $\widetilde\zeta\in D^-_\delta\setminus \{\frac12-i0\}$, a zero of $\breve a$, and reverse the above argument using \eqref{eq: GR(zeta) GL(zeta*) relation} instead of \eqref{eq: GL(zeta) GR(zeta*) relation}, we get $\widetilde\zeta^*=\zeta_j$ for some $j$. Thus $\widetilde\zeta=\zeta_j^*$. In other words, the zeros of $\breve a$ in $D^-_\delta\setminus \{\frac12-i0\}$ are precisely the $\zeta_j^*$'s.
\end{proof}

\begin{lemma}\label{lem: phi N relation}
    Suppose $u$ is real-valued, $u\in L^1_1\cap L^2_1$ with small norm, and $u\in L^\infty$. Then Theorem \ref{thm: bound states and a} (i) implies Theorem \ref{thm: bound states and a} (iii).
\end{lemma}
\begin{proof}
    By Remark \ref{rmk: domain of Lu}, we first need to check that $\varphi_j\in \mathbb H^2(\R+i(-1,1))$.  By  \eqref{M_1 N_1 relation at pole}, \eqref{def: phi_j} and Lemma \ref{lem: anal cont 1}, we have
    \begin{align}
        \varphi_j(x+iy)&=N_1(x+i(y+1),\zeta_j^*)e^{-i\zeta_j^*(x+iy)}\label{eq: phi decay R}\\
        &=\frac1{c_j}M_1(x+i(y+1),\zeta_j)e^{-i\zeta_j(x+iy)}.\label{eq: phi decay L}
    \end{align}   
    Thus by \eqref{eq: phi decay R}, \eqref{eq: phi decay L} and Lemma \ref{lem: anal cont 1}, we have
    \begin{align*}
        &~\|\varphi_j(\cdot+iy)\|_{L^2(-\infty,0)} \notag\\
        \le &~\frac{1}{|c_j|}\|M_1(\cdot+i(y+1),\zeta_j)\|_{L^2+L^\infty}e^{2|\real \zeta_j|}\|e^{\imag \zeta_j x}\|_{(L^\infty\cap L^2)(-\infty,0)} \notag\\
        \le &~C(c_j,\zeta_j,M_1),
    \end{align*}
    \begin{align*}
        &~\|\varphi_j(\cdot+iy)\|_{L^2(0,\infty)} \notag\\
        \le &~\|N_1(\cdot+i(y+1),\zeta_j^*)\|_{L^2+L^\infty}e^{2|\real \zeta_j^*|}\|e^{\imag \zeta_j^* x}\|_{(L^\infty\cap L^2)(0,\infty)} \notag\\
        \le &~C(c_j,\zeta_j,N_1),
    \end{align*}
    for all $y\in (-1,1)$. Note that we used the fact that $\imag \zeta_j>0$, $\imag \zeta_j^*<0$. Thus $\varphi_j\in \mathbb H^2(\R\times (-1,1))$, and 
    \begin{equation}
        e^D \varphi_j(x)=\varphi_j(x-i)=N_1(x,\zeta_j^*)e^{-i\zeta_j^* (x-i)},
    \end{equation}
    \begin{equation}
        e^{-D}\varphi(x)=\varphi_j(x+i)=N_1(x+2i,\zeta_j^*)e^{-i\zeta_j^*(x+i)}.
    \end{equation}
    Since $N_1$ satisfies \eqref{N diff eqn}, we have
    \begin{align}
        De^D \varphi_j(x)&=D(N_1(x,\zeta_j^*)e^{-i\zeta_j^* (x-i)})\notag\\
        &=DN_1(x,\zeta_j^*) \cdot e^{-i\zeta_j^* (x-i)}+N_1(x,\zeta_j^*)De^{-i\zeta_j^* (x-i)}\notag\\
        &=\zeta_j^*[N_1(x,\zeta_j^*)-N_1(x+2i,\zeta_j^*)]e^{-i\zeta_j^* (x-i)}+u(x)N_1(x,\zeta_j^*)e^{-i\zeta_j^* (x-i)}\notag\\
        &\qquad -\zeta_j^* N_1(x,\zeta_j^*)e^{-i\zeta_j^* (x-i)}\notag\\
        &=u(x) e^D \varphi_j(x)-\zeta_j^* e^{-2\zeta_j^*}e^{-D}\varphi_j(x).
    \end{align}
    By Lemma \ref{lem: domain of Lu}, $\varphi_j \in D(L_u)$ and $L_u\varphi_j = \eta_j\varphi_j$ with $\eta_j = -\zeta_j^* e^{-2\zeta_j^*}$. $\eta_j\in\R$ since $L_u$ is self-adjoint. The proof is complete in view of \eqref{eq: zeta e^zeta vs zeta* e^zeta*}.
\end{proof}

Now that we know $\eta_j$ is real, the following simple fact follows:
\begin{lemma}\label{lem: zeta* = bar zeta}
    Let $\zeta_j\in D^+_\delta$ satisfy $\eta_j = -\zeta_je^{-2\zeta_j}\in\mathbb R$. Then
    \begin{equation}\label{eq: zeta* = bar zeta}
        \zeta_j^*=\overline{\zeta_j}.
    \end{equation}
\end{lemma}
\begin{proof}
Since $\eta_j=-\zeta_j e^{-2\zeta_j}\in \R$, we have
\begin{equation}
    -\overline{\zeta_j}e^{-2\overline{\zeta_j}}=-\zeta_j e^{-2\zeta_j}.
\end{equation}
On the other hand, we have from \eqref{eq: zeta e^zeta vs zeta* e^zeta*}
\begin{equation}
    -\zeta_j^* e^{-2\zeta_j^*}=-\zeta_j e^{-2\zeta_j}.
\end{equation}
Note that $\zeta=\frac12$ is a double point of the map $\zeta\mapsto -\zeta e^{-2\zeta}$, and $\zeta_j$, $\zeta_j^*$, $\overline{\zeta_j}$ are all in a small neighborhood of $\frac12$, with $\imag \zeta_j>0$, $\imag\zeta_j^*<0$, $\imag\overline{\zeta_j}<0$. It follows that $\zeta_j^*=\overline{\zeta_j}$.
\end{proof}

\begin{lemma}\label{lem: a' for cpt u}
        Let $u\in L^1_1\cap L^2_1$ with small norm and is real-valued. Assume in addition that $u\in C_0^\infty$. Then Theorem \ref{thm: bound states and a} (ii) holds.
\end{lemma}
\begin{proof}
Let $\zeta\in D^+_\delta\setminus (\R^++i0)$.
Define 
\begin{equation}
    \psi(x+it,\zeta)=M_1(x+it+i,\zeta)e^{-i\zeta(x+it)},
\end{equation}
\begin{equation}
    \varphi(x+it,\zeta)=N_1(x+it+i,\zeta^*)e^{-i\zeta^*(x+it)}
\end{equation}
By Lemma \ref{lem: C^inf regularity jost fcn}, $\psi$ and $\varphi$ are smooth for $x+it\in \R+i [-1,1]$ and $\zeta\in D^+_\delta\setminus (\R^++i0)$. They are also holomorphic for $(x+it,\zeta)$ in the interior of the above set. It follows from \eqref{M diff eqn} and \eqref{N diff eqn} that
\begin{equation}\label{eq: diff eqn psi 1}
    D\psi(x-i,\zeta)-u(x)\psi(x-i,\zeta)=-\zeta e^{-2\zeta} \psi(x+i,\zeta),
\end{equation}
\begin{equation}\label{eq: diff eqn phi 1}
    D\varphi(x-i,\zeta)-u(x)\varphi(x-i,\zeta)=-\zeta^* e^{-2\zeta^*}\varphi(x+i,\zeta).
\end{equation}
Differentiate \eqref{eq: diff eqn psi 1} in $\zeta$ to get
\begin{align}\label{eq: diff eqn psi 2}
    &~D\partial_\zeta \psi(x-i,\zeta)-u(x)\partial_\zeta\psi(x-i,\zeta)\notag\\
    =&~-(1-2\zeta)e^{-2\zeta}\psi(x+i,\zeta)-\zeta e^{-2\zeta}\partial_\zeta\psi(x+i,\zeta).
\end{align}
Evaluating \eqref{eq: diff eqn phi 1} and \eqref{eq: diff eqn psi 2} at $\zeta=\zeta_j$, and noting from \eqref{eq: zeta e^zeta vs zeta* e^zeta*} that $\eta_j=-\zeta_j e^{-2\zeta_j}=-\zeta_j^* e^{-2\zeta_j^*}$, we have
\begin{equation}\label{eq: diff eqn phi 3}
    -D\overline{\varphi(x-i,\zeta_j)}-u(x)\overline{\varphi(x-i,\zeta_j)}=\eta_j \overline{\varphi(x+i,\zeta_j)},
\end{equation}
\begin{align}\label{eq: diff eqn psi 3}
    &~D\partial_\zeta \psi(x-i,\zeta_j)-u(x)\partial_\zeta\psi(x-i,\zeta_j)\notag\\
    =&~-(1-2\zeta_j)e^{-2\zeta_j}\psi(x+i,\zeta_j)+\eta_j\partial_\zeta\psi(x+i,\zeta_j).
\end{align}
Multiply \eqref{eq: diff eqn phi 3} by $-\partial_\zeta\psi(x-i,\zeta_j)$, \eqref{eq: diff eqn psi 3} by $\overline{\varphi(x-i,\zeta_j)}$ and take their sum to get
\begin{align}\label{eq: wroskian 1}
    &D\left[\partial_\zeta \psi(x-i,\zeta_j)\overline{\varphi(x-i,\zeta_j)}\right]\notag\\
    =&-(1-2\zeta_j)e^{-2\zeta_j}\psi(x+i,\zeta_j)\overline{\varphi(x-i,\zeta_j)}\notag\\
    &\quad + \eta_j \left[\partial_\zeta\psi(x+i,\zeta_j)\overline{\varphi(x-i,\zeta_j)}-\partial_\zeta\psi(x-i,\zeta_j)\overline{\varphi(x+i,\zeta_j)}\right].
\end{align}
By analytically continuing \eqref{M_1 N_1 relation at pole}, we get
\begin{equation}
    \psi(x+it,\zeta_j)e^{\zeta_j}=c_j \varphi(x+it,\zeta_j)e^{\zeta_j^*}.
\end{equation} 
By Lemma \ref{lem: M1 improved regularity}, $M_1(x+it+i,\zeta_j)$ and $N_1(x+it+i,\zeta_j^*)$ are bounded for $(x,t)\in \R\times [-1,1]$. It follows that $\psi(x+it,\zeta_j)$ and thus $\varphi(x+it,\zeta_j)$ both have exponential decay as $x\to\pm\infty$. To estimate the terms in the brackets on the right hand side of \eqref{eq: wroskian 1} as $x\to\infty$, we compute
\begin{align}
    &~\partial_\zeta\psi(x+it,\zeta_j)\overline{\varphi(x-it,\zeta_j)}\notag\\
    =&~\left[\partial_\zeta M_1(x+it+i,\zeta_j)-i(x+it)M_1(x+it+i,\zeta_j)\right]\overline{N_1(x-it+i,\zeta_j^*)}.
\end{align}
Here we have used Lemma \ref{lem: zeta* = bar zeta} (reality of $\eta_j$ follows from Lemma \ref{lem: phi N relation}). As $\partial_\zeta M_1(x+it+i,\zeta_j)$ is bounded, $\partial_\zeta\psi(x+it,\zeta_j)\overline{\varphi(x-it,\zeta_j)}$ has exponential decay as $x\to -\infty$. Since $u$ is compactly supported, when $x>0$ is sufficiently large, we have from \eqref{jost prty: M1 ext}
\begin{equation}\label{eq: M_1(x+it+i) 1}
    M_1(x+it+i,\zeta)=a(\zeta)+\mathcal \mathcal F^{-1}\left[R_2(\xi,t+1,\zeta)\mathcal \mathcal F[u(s)M_1(s,\zeta)](\xi)\right](x).
\end{equation}
Take the $\zeta$-derivative of \eqref{eq: M_1(x+it+i) 1} to get
\begin{align}
    \partial_\zeta M_1(x+it+i,\zeta_j)&=a'(\zeta_j)+\mathcal \mathcal F^{-1}\left[\partial_\zeta R_2(\xi,t+1,\zeta_j)\mathcal \mathcal F[u(s)M_1(s,\zeta_j)](\xi)\right](x)\notag\\
    &\qquad +\mathcal \mathcal F^{-1}\left[ R_2(\xi,t+1,\zeta_j)\mathcal \mathcal F[u(s)\partial_\zeta M_1(s,\zeta_j)](\xi)\right](x)
\end{align}
By \eqref{eq: deriv R bound}, \eqref{eq: d_zeta R bound} and smoothness of $u(x)$ and $M_1(x,\zeta_j)$, we obtain that $\partial_\zeta M_1(x+it+i,\zeta_j)$ tends to $a'(\zeta_j)$ as $x\to\infty$ uniformly for $t\in [-1,1]$. Similarly, $(x+it+i)M_1(x+it+i,\zeta_j)$ tend to $0$ uniformly for $t\in [-1,1]$ as $x\to\infty$, while $N_1(x-it+i,\zeta_j^*)\to 1$ uniformly for $t\in[-1,1]$ as $x\to \infty$. It follows that 
\begin{equation}
    \partial_\zeta\psi(x+it,\zeta_j)\overline{\varphi(x-it,\zeta_j)}\to a'(\zeta_j)
\end{equation} 
uniformly for $t\in [-1,1]$ as $x\to\infty$. Similarly,
\begin{equation}
    \partial_\zeta \psi(x-i,\zeta_j)\overline{\varphi(x-i,\zeta_j)}\to a'(\zeta_j)e^{-2\zeta_j}
\end{equation}
as $x\to \infty$ and tends to zero as $x\to-\infty$.
We now integrate \eqref{eq: wroskian 1} on $[-R,R]$ and use the Cauchy integral theorem on the holomorphic functions $\psi(x+it,\zeta_j)\overline{\varphi(x-it,\zeta_j)}$ and $\partial_\zeta\psi(x+it,\zeta_j)\overline{\varphi(x-it,\zeta_j)}$ to get
\begin{align}
    &~\frac1i\partial_\zeta\psi(x-i,\zeta_j)\overline{\varphi(x-i,\zeta_j)}\bigg|_{x=-R}^{x=R}\notag\\
    =&~ -(1-2\zeta_j) e^{-2\zeta_j}\bigg(\int_{-R}^R\psi(x,\zeta_j)\overline{\varphi(x,\zeta_j)}~dx\notag\\
    &\quad +\int_0^1\psi(R+it,\zeta_j)\overline{\varphi(R-it,\zeta_j)}~idt-\int_0^1\psi(-R+it,\zeta_j)\overline{\varphi(-R-it,\zeta_j)}~ idt\bigg)\notag\\
    &\quad +\eta_j \bigg(\int_{-1}^1\partial_\zeta\psi(R+it,\zeta_j)\overline{\varphi(R-it,\zeta_j)}~idt-\int_{-1}^1\partial_\zeta\psi(-R+it,\zeta_j)\overline{\varphi(-R-it,\zeta_j)}~idt\bigg)
\end{align}
Taking the limit as $R\to\infty$ and using the above asymptotic formulas, we get
\begin{equation}
    -ia'(\zeta_j)e^{-2\zeta_j}=-(1-2\zeta_j)e^{-2\zeta_j}e^{\zeta_j^*-\zeta_j}\int_\R c_j|\varphi(x,\zeta_j)|^2~dx+\eta_j a'(\zeta_j)2i,
\end{equation}
which is equivalent to the first part of \eqref{eq: a' identity}. To get the second part involving $\breve a$, we repeat the above calculation using
\begin{equation}
    \psi(x+it,\zeta) = N_1(x+it+i,\zeta)e^{-i\zeta(x+it)},
\end{equation}
\begin{equation}
    \varphi(x+it,\zeta) = M_1(x+it+i,\zeta^*)e^{-i\zeta^*(x+it)},
\end{equation}
for $\zeta\in D^-_\delta\setminus (\R^+-i0)$, and evaluate at $\zeta=\zeta_j^*$. The details are omitted.
\end{proof}

By Lemma \ref{lem: ess spec of Lu}, the discrete eigenvalues of $L_u$ appear  below $-\frac1{2e}$. We will show that they are exactly the $\eta_j$'s obtained in the above lemma and that there are no other discrete eigenvalues. To that end, we need to study the integral kernel of the resolvent of $L_u$. For $\eta<-\frac1{2e}$, let 
\begin{equation}
    K(x,\eta)=\frac1{2\pi}\int_\R \frac{e^{ix\xi}}{\xi-\eta e^{-2\xi}}~d\xi.
\end{equation}
We have the following identity.

\begin{lemma}\label{lem: K GR relation}
    There exists an $\epsilon>0$ sufficiently small, such that if $\eta\in (-\frac{1}{2e}-\epsilon,-\frac{1}{2e})$, then 
    \begin{equation}
        K(x,\eta) =e^{-ix\zeta}G_R(x,\zeta),
    \end{equation}
    where $\zeta=-\frac12W_{-1}(2\eta)$.
\end{lemma}
\begin{proof}
    By the definition of the branches of the Lambert-$W$ function, when $\eta\in \R^-$ is in a small neighborhood below $-\frac1{2e}$, $\zeta=-\frac12W_{-1}(2\eta)$ is in a small semi-disc above $\frac12$, and $\eta=-\zeta e^{-2\zeta}$. We can thus write
    \begin{align}
        e^{ix\zeta}K(x,\eta)&=\frac1{2\pi}\int_\R\frac{e^{ix(\xi+\zeta)}}{\xi+\zeta-\zeta(1-e^{-2(\xi+\zeta)})}~d\xi\notag\\
        &=\frac1{2\pi}\int_{\R+\zeta}\frac{e^{ix\xi}}{\xi-\zeta(1-e^{-2\xi})}~d\xi\notag\\
        &=\frac1{2\pi}\int_{\Gamma_R}\frac{e^{ix\xi}}{\xi-\zeta(1-e^{-2\xi})}~d\xi=G_R(x,\zeta).
    \end{align}
    In the last step above, we shift the contour of integration from $\R+\zeta$ to $\Gamma_R=\R+i0$. To be able to do this, we need to verify that the integrand has no poles between the two contours. Since $\zeta$ is in a small semi-disc above $\frac12$, the region between the two contours is contained in the region $U$ given in Lemma \ref{lem: Z inv domain}. It's easy to see that a pole can only occur at a zero of $Z(\xi)-\zeta$. Thus a pole can occur only at $Z^{-1}(\zeta)$ with $Z^{-1}$ given by Definition \ref{def: Z inv}. Recall that $Z^{-1}(\frac12)=0$ and $(Z^{-1})'(\frac12) = 2$. Thus for $\zeta$ in a sufficiently small semi-disc above $\frac12$, $Z^{-1}(\zeta)$ lies above $\R+\zeta$, and is not in the region between the two contours $\R+\zeta$ and $\R+i0$.
\end{proof}
\begin{remark}
    If we picked the other branches of the Lambert-$W$ function, we would have obtained other values of $\zeta$ such that $\eta=-\zeta e^{-2\zeta}$. However, in general, the shift of contour from $\R+\zeta$ to $\Gamma_L$ or $\Gamma_R$ will pass by poles of the integrand, thus producing extra pole terms. The choice of $\zeta$ in Lemma \ref{lem: K GR relation} is only one that allows a shift of contour without producing extra terms. We will see in the following that this mechanism essentially extracts the correct exponential decay rate of the eigenfunctions.
\end{remark}
\begin{lemma}\label{lem: disc spec a=0 relation, cpt}
    Let $u\in L^1_1\cap L^2_1\cap L^\infty$ with small norm. Assume in addition that $u\in C_0^\infty$. Then Thoerem \ref{thm: bound states and a} (iv) holds.
\end{lemma}
\begin{proof}
    By Lemma \ref{lem: domain of Lu}, $\varphi\in L^2$, $e^{\pm D}\varphi\in L^2$, and
    \begin{equation}\label{eq: e-value}
        De^D\varphi-ue^{D}\varphi = \eta e^{-D}\varphi.
    \end{equation}
    Taking the $L^2$ inner product of \eqref{eq: e-value} with $e^D\varphi$, we get
    \begin{equation}
        \eta \langle \varphi,\varphi\rangle = \langle De^{D}\varphi,e^D\varphi\rangle-\langle u e^{D}\varphi, e^D\varphi\rangle.
    \end{equation}
    We estimate
    \begin{align}
        |\langle u e^{D}\varphi, e^D\varphi\rangle|&\le \|u\|_{L^\infty}\langle e^{D} \varphi, e^D\varphi \rangle\notag\\
        &\le \|u\|_{L^\infty}(\langle De^{D}\varphi, e^D\varphi\rangle +e^2 \langle \varphi,\varphi\rangle).
    \end{align}
    It follows from $\min_{\xi\in \R}\xi e^{2\xi}= -\frac1{2e}$ that
    \begin{align}
        \eta\langle\varphi,\varphi\rangle&\ge (1-\|u\|_{L^\infty})\langle De^{D}\varphi, e^D\varphi\rangle-e^2\|u\|_{L^\infty}\langle\varphi,\varphi\rangle\notag\\
        &\ge \left[-\tfrac1{2e}(1-\|u\|_{L^\infty})-e^2\|u\|_{L^\infty}\right]\langle \varphi,\varphi\rangle\notag\\
        &\ge -\tfrac1{2e}\langle\varphi,\varphi\rangle -e^2 \|u\|_{L^\infty}\langle \varphi,\varphi\rangle.
    \end{align}
    Since $\|u\|_{\infty}$ is small, $\eta$ is in a small neighborhood below $-\tfrac1{2e}$. 
    
    Let $\psi(x) = e^D\varphi(x)=\varphi(x-i)$. We have $\psi\in H^1$, $e^{-2D}\psi = e^{-D}\varphi\in L^2$, and 
    \begin{equation}\label{eq: e-value psi 1}
        D\psi-u\psi=\eta e^{-2D}\psi.
    \end{equation}
    Take the Fourier transform of \eqref{eq: e-value psi 1} to get
    \begin{equation}
        (\xi-\eta e^{-2\xi})\widehat\psi(\xi) = \widehat{u\psi}(\xi)
    \end{equation}
    Since $\xi e^{2\xi}\ge -\frac1{2e}$ and $\eta<-\frac1{2e}$, we have that $\xi-\eta e^{-2\xi}$ is bounded away from zero. Thus
    \begin{equation}
        \widehat\psi(x) = \frac1{\xi-\eta e^{-2\xi}}\widehat{u\psi}(x).
    \end{equation}
    It follows that for $\zeta$ given as in Lemma \ref{lem: K GR relation},
    \begin{align}\label{eq: psi integral eq 1}
        e^{i\zeta x}\psi(x) &=e^{i\zeta x} [K(\cdot,\eta)*(u\psi)](x)\notag\\
        &=\int_\R e^{i\zeta x}K(x-y,\eta)u(y)\psi(y)~dy\notag\\
        &=\int_\R G_R(x-y,\zeta) u(y)e^{i\zeta y}\psi(y)~dy.
    \end{align}
    Since $u$ is compactly supported, $e^{i\zeta y}$ is bounded on the support of $u$. It follows from \eqref{eq: psi integral eq 1} that $e^{i\zeta x}\psi(x)\in L^\infty$, and $e^{i\zeta x}\psi(x)\in \ker (I-T_R(\zeta))$. By Theorem \ref{thm: inv T on disc}, $a(\zeta)=0$. Thus $\zeta=\zeta_j$ for some $j$. Since \eqref{eq: psi integral eq 1} can be rewritten as 
    \begin{equation}
        e^{i\zeta x}\psi(x)=-\frac{i}{1-2\zeta}\int_\R u(y)e^{i\zeta y}\psi(y)~dy+\int_\R G_L(x-y,\zeta) u(y)e^{i\zeta y}\psi(y)~dy,
    \end{equation}
    and $I-T_L(\zeta)$ is invertible on $L^\infty_{-1,+}$, we obtain that $e^{i \zeta_j x}\psi(x) = C M_1(x,\zeta_j)$ for $C=-\frac{i}{1-2\zeta}\int_\R u(y)e^{i\zeta y}\psi(y)~dy$. Via analytic continuation we get
    \begin{equation}
        \varphi(x)=\psi(x+i)=Ce^{-i\zeta_j x}M_1(x+i,\zeta_j)=\widetilde{C}\varphi_j(x)
    \end{equation}
    by \eqref{M_1 N_1 relation at pole} and \eqref{def: phi_j}.
\end{proof}

Next, we work to remove the compact support and regularity conditions on $u$ in the previous lemmas in this section. To that end, we assume in the remaining of this section that $u\in L^1_1\cap L^2_1$ with small norm, and is real-valued. We take a sequence of $u_n\in C_0^\infty(\R)$ such that $u_n\to u$ in $L^1_1\cap L^2_1$. We denote the corresponding convolution operators $T_L$, $T_R$, Jost solutions $M_1$, $N_1$, and scattering functions $a$, $b$, etc. by $T_L^{(n)}$, $T_R^{(n)}$, $M_1^{(n)}$, $N_1^{(n)}$, $a^{(n)}$, $b^{(n)}$, etc.
We start by proving the convergence of the Jost functions and scattering coefficients.

\begin{lemma}\label{lem: M1n convergence}
    $M_1^{(n)}(\cdot,\zeta)\to M_1(\cdot,\zeta)$ in $L^\infty_{-1,+}$ uniformly for $\zeta\in D^+_\delta$, and similarly $N_1^{(n)}(\cdot,\zeta)\to N_1(\cdot,\zeta)$ in $(1+x_-)L^\infty$ uniformly for $\zeta\in D^-_\delta$. $a^{(n)}(\zeta)\to a(\zeta)$ uniformly on compact subsets of $D^+_\delta\setminus \{\frac12+i0\}$, and $\breve a\supn(\zeta)\to \breve a(\zeta)$ uniformly on compact subsets of $D^-_\delta\setminus \{\frac12-i0\}$.
\end{lemma}
\begin{proof}
    The proof is similar to how we proved continuous dependence of $T_L^u$ on $u$ in Theorem \ref{thm: genericity}.
    \begin{comment}
        Recall that 
    \begin{align*}
        M_1&=1+T_L(\zeta)M_1,\\
        M_1\supn &= 1+T_L\supn (\zeta)M_1\supn.
    \end{align*}
    Thus 
    \begin{equation}\label{eq: M1-M1supn}
        M_1-M_1\supn = (I-T_L(\zeta))^{-1}\left(T_L(\zeta)-T_L\supn(\zeta)\right)\left(I-T_L\supn(\zeta)\right)^{-1}1.
    \end{equation}
    By the proof of Theorem \ref{thm: inv T off disc}, $\|T_L(\zeta)\|_{L^\infty_{-1,+}\toL^\infty_{-1,+}}$ and $\|T_L\supn(\zeta)\|_{L^\infty_{-1,+}\toL^\infty_{-1,+}}$ are bounded by $\frac12$ for $\zeta\in D^+_\delta$, provided $\|u\|_{L^1_1\cap L^2_1}$ is small enough. We also have from \eqref{eq: est T_L via u}
    \begin{align}
        &~\|T_L(\zeta)-T_L\supn(\zeta)\|_{L^\infty_{-1,+}\toL^\infty_{-1,+}}\notag\\
        \le &~C\|G_L(\cdot,\zeta)\|_{L^2_{-1,+}+L^\infty_{-1,+}}\|u-u_n\|_{\langle x\rangle ^{-1}(L^1\cap L^2)}.\label{eq: TL-TLsupn}
    \end{align}
    By Lemma \ref{lem: green uniform split}, $\|G_L(\cdot,\zeta)\|_{L^2_{-1,+}+L^\infty_{-1,+}}$ is uniformly bounded for $\zeta\in D^+_\delta$. It follows from \eqref{eq: M1-M1supn} and \eqref{eq: TL-TLsupn} 
    \end{comment}
    We obtain similarly that  $M_1^{(n)}(\cdot,\zeta)\to M_1(\cdot,\zeta)$ in $L^\infty_{-1,+}$ uniformly for $\zeta\in D^+_\delta$. The convergence of $a\supn$ to $a$ now follows again from \eqref{def: a} and the pairing between $L^1_1\cap L^2_1$ and $L^\infty_{-1,+}$. The proof for $N_1$ and $\breve a$ is similar.
\end{proof}

We also need equicontinuity of the Jost functions for flexible limit manipulations.

\begin{lemma}\label{lem: M1n equicont}
    $ \{\zeta\mapsto M_1\supn (\cdot,\zeta)\}$ is equicontinuous as a sequence of maps from any compact subset of $D^+_\delta\setminus \{\frac12+i0\}$ to $L^\infty_{-1,+}$. Similarly, $\{\zeta\mapsto N_1\supn (\cdot,\zeta)\}$ is equicontinuous as a family of maps from any compact subset of $D^-_\delta\setminus \{\frac12-i0\}$ to $(1+x_-)L^\infty$.
\end{lemma}
\begin{proof}
    Using \eqref{eq: T cont in zeta} again we have
    \begin{align}\label{eq: Tsupn cont in zeta}
    &~\|[T_L\supn(\zeta_1)-T_L\supn(\zeta_2)]\|_{L^\infty_{-1,+}\to L^\infty_{-1,+}}\notag\\
    \le&~ C\|G_L(\cdot,\zeta_2)-G_L(\cdot,\zeta_1)\|_{L^2_{-1,+}+L^\infty_{-1,+}} \|u_n\|_{L^1_1\cap L^2_1}.
    \end{align}
    By the estimates on \eqref{eq: G cont in zeta near 1/2} we see that $\|G_L(\cdot,\zeta_2)-G_L(\cdot,\zeta_1)\|_{L^2_{-1,+}+L^\infty_{-1,+}}\to 0$ uniformly as $|\zeta_1-\zeta_2|\to 0$ in any compact subset of $D^+_\delta\setminus \{\frac12+i0\}$. The proof is complete in view of 
    \begin{align}
        &~M_1\supn(\cdot,\zeta_1)-M_1\supn(\cdot,\zeta_2)\notag\\
        =&~\left(I-T_L\supn(\zeta_1)\right)^{-1}\left(T_L\supn(\zeta_1)-T_L\supn(\zeta_2)\right)\left(I-T_L\supn(\zeta_2)\right)^{-1}1.
    \end{align}
\end{proof}

Recall that $\zeta_j$, $j=1,2,\dots$ is an enumeration of the zeros of $a$ in $D^+_\delta\setminus \{\frac12+i0\}$. Denote by $k_j$ the multiplicity of $\zeta_j$. We will eventually show that $k_j=1$, so that all the zeros are simple. But at the moment, they may be bigger than 1. Since the $\zeta_j$'s are isolated, there exists a small circle $\Gamma_j$ centered at each $\zeta_j$ such that $\zeta_j$ is the only zero of $a$ in $\Gamma_j$. We have
\begin{equation}
    k_j=\frac1{2\pi i}\int_{\Gamma_j}\frac{a'(\zeta)}{a(\zeta)}~d\zeta.
\end{equation}
Since $a\supn \to a$, $(a\supn)'\to a'$ uniformly on $\Gamma_j$, there exists $N_j>0$ such that for all $n>N_j$, $a\supn$ has no zero on $\Gamma_j$ and 
\begin{equation}
    k_j=\frac1{2\pi i}\int_{\Gamma_j}\frac{(a\supn)'(\zeta)}{a\supn(\zeta)}~d\zeta.
\end{equation}
In this case, $a\supn$ has exactly $k_j$ zeros inside $\Gamma_j$, all of which are simple by Lemma \ref{lem: a' for cpt u}. Denote by $\zeta\supn_{ji}$, $i=1,\dots, k_j$ an enumeration of these zeros (the index $i$ being different from the complex number $i$ should be clear from context). Note that $\zeta\supn_{ji}$ are defined for all $n>N_j$, $1\le i\le k_j$. We have the following
\begin{lemma}\label{lem: zeta^n_ji convergence}
    For each $j$,
    \begin{equation}
        \max_{1\le i\le k_j} |\zeta\supn_{ji}-\zeta_j|\to 0
    \end{equation}
    as $n\to\infty.$
\end{lemma}
\begin{proof}
    Let $\Gamma_{j,\epsilon}$ be any small circle centered at $\zeta_j$ contained inside $\Gamma_j$. By repeating the above calculation with $\Gamma_j$ replaced by $\Gamma_{j,\epsilon}$, we see that there are exactly $k_j$ zeros of $a\supn$ inside $\Gamma_{j,\epsilon}$ for $n$ sufficiently large. Of course they are the enumerated $\zeta\supn_{ji}$'s.
\end{proof}

\begin{lemma}\label{lem: M N relation at zeta_j general}
    Theorem \ref{thm: bound states and a} (i) holds.
\end{lemma}
\begin{proof}
    By Lemma \ref{lem: M N relation at zeta_j}, we have
    \begin{equation}\label{eq: M1n N1n relation at zeta^n_ji}
        M_1\supn\left(x,\zeta\supn_{ji}\right)e^{-i\zeta\supn_{ji} x}=c\supn_{ji}N_1\supn\left(x,\left(\zeta\supn_{ji}\right)^*\right)e^{-i\left(\zeta\supn_{ji}\right)^* x},
    \end{equation}
    where
    \begin{equation}\label{eq: c^n_ji}
        c\supn_{ji}=\frac{i}{1-2\left(\zeta\supn_{ji}\right)^*}\int_\R u_n(x)M_1\supn\left(x,\zeta\supn_{ji}\right)e^{-i\lambda\supn_{ji} x}~dx,
    \end{equation}
    and $\lambda\supn_{ji}=Z^{-1}\left(\zeta\supn_{ji}\right)$.
    By Lemmas \ref{lem: M1n convergence}, \ref{lem: M1n equicont}, \ref{lem: zeta^n_ji convergence}, for almost every $x\in\mathbb R$, we have
    \begin{align}
        M_1\supn\left(x,\zeta\supn_{ji}\right)e^{-i\zeta\supn_{ji} x}&\to M_1(x,\zeta_j)e^{-i\zeta_j x},\label{eq: M1^n limit}\\
        N_1\supn\left(x,\left(\zeta\supn_{ji}\right)^*\right)e^{-i\left(\zeta\supn_{ji}\right)^* x}&\to N_1(x,\zeta_j^*)e^{-i\zeta_j^*x},\label{eq: N1^n limit}
    \end{align}
    as $n\to\infty$. It is tempting to take a limit as $n\to\infty$ of \eqref{eq: M1n N1n relation at zeta^n_ji} and \eqref{eq: c^n_ji}. However, since $e^{-i\lambda\supn_{ji}x}$ has exponential growth as $x\to\infty$, the $L^\infty_{-1,+}$ convergence of $M\supn_1$ is not strong enough to directly ensure a limit of \eqref{eq: c^n_ji}. Instead, we use the following argument. Since $N_1(x,\zeta_j^*)\to 1$ as $x\to\infty$, there exists $x_0\in \mathbb R$ such that $N_1(x_0,\zeta_j^*)\ne 0$ and \eqref{eq: M1^n limit} and \eqref{eq: N1^n limit} hold at $x=x_0$. Using \eqref{eq: M1n N1n relation at zeta^n_ji} again, we get 
    \begin{equation}
        \lim_{n\to\infty}c\supn_{ji}=\frac{M_1(x_0,\zeta_j)e^{-i\zeta_j x_0}}{N_1(x_0,\zeta_j^*)e^{-i\zeta_j^*x_0}}
    \end{equation}
    exists and is independent of $i$. Denote this limit by $c_j$. We then have by taking a limit of \eqref{eq: M1n N1n relation at zeta^n_ji} that
    \begin{equation}\label{eq: M1 N1 relation repeated}
        M_1(x,\zeta_j)e^{-i\zeta_j x }=c_jN_1(x,\zeta_j^*)e^{-i\zeta_j^* x}.
    \end{equation}
    Recall \eqref{eq: zeta-zeta*} that $\zeta\supn_{ji}-\left(\zeta\supn_{ji}\right)^*=\lambda\supn_{ji}$. It follows from \eqref{eq: M1n N1n relation at zeta^n_ji} and \eqref{eq: c^n_ji} that
    \begin{equation}\label{eq: c^n_ji relation 1}
        c\supn_{ji}=\frac{ic\supn_{ji}}{1-2\left(\zeta\supn_{ji}\right)^*}\int_\R u_n(x)N_1\supn\left(x,\left(\zeta\supn_{ji}\right)^*\right)~dx,
    \end{equation}
    Taking a limit of \eqref{eq: c^n_ji relation 1} as $n\to\infty$ we obtain
    \begin{equation}\label{eq: c_j expression 2}
        c_j=\frac{ic_j}{1-2\zeta_j^*}\int_\R u(x)N_1(x,\zeta_j^*)~dx.
    \end{equation}
    Using \eqref{eq: M1 N1 relation repeated} again, we obtain \eqref{def: c_j}. After canceling $c_j$ from \eqref{eq: c_j expression 2} and recalling \eqref{def: breve a}, we see $\breve a(\zeta_j^*)=0$. Thus $\zeta_j^*$ is a zero of $\breve a$. We can show $\breve a$ has no other zeros than the $\zeta_j^*$ in the same way as in the proof of Lemma \ref{lem: M N relation at zeta_j}.
\end{proof}

Note that the proof of Lemma \ref{lem: phi N relation} does not assume the stronger condition $u\in C_0^\infty$. So Thoerem \ref{thm: bound states and a} (iii) has also been proven.

\begin{lemma}\label{lem: a' general}
    Theorem \ref{thm: bound states and a} (ii) holds. In particular, $k_j=1$ for all $j$.
\end{lemma}
\begin{proof}
    Applying Lemmas \ref{lem: M N relation at zeta_j}, \ref{lem: phi N relation} and \ref{lem: a' for cpt u} on $u_n$, we have
    \begin{equation}\label{eq: an' identity 1}
        \left(a\supn\right)'\left(\zeta\supn _{ji}\right)= \overline{ \left(\breve{a}\supn\right)'\left(\left(\zeta\supn _{ji}\right)^*\right)}=-2c\supn_{ji} e^{-\lambda\supn_{ji}}\int_\R\left|\varphi\supn_{ji}\right|^2~dx,
    \end{equation}
    where
    \begin{align}
        \varphi\supn_{ji}(x)&=N_1\supn\left(x+i,\left(\zeta\supn_{ji}\right)^*\right)e^{-i\left(\zeta\supn_{ji}\right)^* x}\label{eq: phin N1 relation 1}\\
        &= \frac1{c\supn_{ji}}M_1\supn\left(x+i,\zeta\supn_{ji}\right)e^{-i\zeta\supn_{ji}(x+i)}e^{-\left(\zeta\supn_{ji}\right)^*}.\label{eq: phin N1 relation 2}
    \end{align}
    We may rewrite $M_1\supn\left(x+i,\zeta\supn_{ji}\right)$ in terms of $M_1\supn\left(x,\zeta\supn_{ji}\right)$ as in \eqref{jost prty: M1 ext}. Since $u_n\to u$ in $L^1_1\cap L^2_1$, $M_1\supn\left(x,\zeta\supn_{ji}\right)\to M_1\left(x,\zeta_j\right)$ in $L^\infty_{-1,+}$ by Lemmas \ref{lem: M1n convergence}, \ref{lem: M1n equicont} and \ref{lem: zeta^n_ji convergence}. $R_2\left(\xi,i,\zeta\supn_{ji}\right)\to R_2(\xi,i,\zeta_j)$ in $L^2$ by \eqref{eq: R bound} and \eqref{eq: d_zeta R bound} (note that the constants are uniform for $\zeta$ in the range we consider). It follows that $M_1\supn\left(x+i,\zeta\supn_{ji}\right)\to M_1(x+i,\zeta_j)$ in $L^\infty$. Similarly, $N_1\supn\left(x+i,\left(\zeta\supn_{ji}\right)^*\right)\to N_1(x+i,\zeta_j)$ in $L^\infty.$ It follows from \eqref{eq: phin N1 relation 1}, \eqref{eq: phin N1 relation 2} and the definition of $\varphi_j$ that $\varphi\supn_{ji}\to \varphi_j$ in $L^2$. Now take the limit of \eqref{eq: an' identity 1} as $n\to\infty$ to get \eqref{eq: a' identity}, and the proof is complete.
     
\end{proof}

Since each $k_j=1$ by the above lemma, we may simplify notation and denote $\zeta\supn_{ji}$ by $\zeta\supn_j$.

\begin{lemma}
    Theorem \ref{thm: bound states and a} (iii) holds.
\end{lemma}
\begin{proof}
    Without loss of generality, assume the approximating sequence $u_n$ of $u$ satisfy $\|u_n\|_{L^\infty}\le \|u\|_{L^\infty}$. By Lemma \ref{lem: ess spec of Lu} and \eqref{eq: semi-bdd}, $\sigma_{\text{disc}}(L_{u})\subset (-\infty,-\frac1{2e})$ and is close to $-\frac1{2e}$ as $\|u\|_{L^\infty}$ is small. Note that $\zeta\mapsto -\zeta e^{-2\zeta}$ maps $D_\delta^+\cap \{\imag \zeta>0\}$ biholomorphically to an open set around $-\frac1{2e}$ containing $\sigma_{\text{disc}}(L_u)$. Let $\eta\in \sigma_{\text{disc}}(L_u)$ with eigenspace $E_\eta$. Denote the dimension of $E_\eta$ by $k$. By Lemma \ref{lem: Lu norm resolv conv}, $L_{u_n}\to L_u$ in the norm resolvent sense. Let $(a,b)$ be any small interval containing only the eigenvalue $\eta$. By Theorem VIII.23 of \cite{reed1980methods}, $a,b\notin \sigma(L_{u_n})$ for $n$ sufficiently large, and $P_{(a,b)}(L_{u_n})-P_{(a,b)}(L_u)\to 0$ in operator norm. It follows that the dimension of the range of $P_{(a,b)}(L_{u_n})$ equals $k$ for $n$ sufficiently large. Let $\zeta$ be the unique preimage in $D_\delta^+\cap \{\imag \zeta>0\}$ such that $-\zeta e^{-2\zeta}=\eta$. If $\zeta\ne \zeta_j$ for any $j$, there exists a small disc $U$ around $\zeta$ such that $a$ has no zero in $\overline U$. Since $a\supn \to a$ uniformly on $\overline U$, the same is true for $a\supn$ with $n$ sufficiently large. By Lemma \ref{lem: disc spec a=0 relation, cpt}, $L_{u_n}$ has no eigenvalues in some fixed open neibhorhood of $\eta$ for all $n$ sufficiently large.  This contradicts the above assertion that the projection $P_{(a,b)}(L_{u_n})$ has dimension $k$. Thus $\eta=\eta_j=-\zeta_j e^{-2\zeta_j}$ for some $j$, and $\varphi_j\in E_\eta=E_{\eta_j}$. Take $U$ small enough so that $\zeta_j$ is the only zero of $a$ in $\overline U$. The same is true for $a\supn$ with $n$ sufficiently large. Again by Lemma \ref{lem: disc spec a=0 relation, cpt}, there is only one simple eigenvalue of $L_{u_n}$ in some fixed open neighborhood of $\eta$ for all $n$ sufficiently large. In other words, the projection $P_{(a,b)}(L_{u_n})$ is one dimensional for some $(a,b)$. It follows that $k=1$, i.e. $E_\eta$ is one dimensional. Thus any eigenvector must be a multiple of $\varphi_j$.
\end{proof}

\subsection{Finiteness of the discrete spectrum}
\label{sec: finite discrete spec}
    In  this section, we show that the discrete spectrum of $L_u$ is finite by assuming $u\in L^1_2\cap L^\infty$ and is real-valued, but without a smallness assumption.
    We write $\eta = -\frac1{2e} -\mu^2$ for $\mu>0$. The operator $L_0 -\eta$ is a positive operator and has a positive square root for any $\mu>0$:
\begin{equation}
	\label{A}
		A=A(\mu) = (L_0-\eta)^{\frac12}=(L_0 +\tfrac1{2e} + \mu^2)^{\frac12}=\left( De^{2D} + \tfrac{1}{2e} + \mu^2 \right)^\frac12.
\end{equation}
$A(\mu)$ has a bounded inverse
$$ A^{-1}=A^{-1}(\mu) = \left(D e^{2D} + \tfrac{1}{2e} + \mu^2 \right)^{-\frac12}. $$
We now factor
$$ L_u-\eta = A\left( 1 - (A^{-1} e^D) \tilde u (e^D A^{-1}) \right) A$$
and obtain a corresponding factorization of the resolvent
$$ (L_u-\eta)^{-1} = 
		A^{-1} 
			\left( 1 - C\right)^{-1}
		A^{-1}
$$
where
\begin{equation} 
C=C(\mu) = \left(A^{-1}(\mu) e^D\right) \tilde u\left( e^D A^{-1}(\mu)\right) 
\end{equation}
is a self-adjoint, compact operator on $L^2(\R)$ if $u\in L^p$ for some $p\in (1,\infty)$. Compactness follows from the factorization
\begin{equation}\label{eq: C decomp}
    C(\mu) = C_1(\mu)C_2(\mu)=\left((A^{-1}e^D)\tilde u^{\frac12}\right)\left(|\tilde u|^{\frac12}(e^D A^{-1})\right),
\end{equation}
where $u^{\frac12}=\sgn(u)|u|^{\frac12}$,
and Theorem 4.1 in Chapter 4 of \cite{Simon:2005}. In fact, the symbol of $e^D A^{-1}$ belongs to $L^{2p}$ for any $p>1$, and so does $|u|^{\frac12}$ provided $u \in L^{p}$. As a result $C_2(\mu)$ belongs to the trace ideal $\mathcal I_{2p}$. $C_1(\mu)$ can be analyzed similarly by duality.

We now state the following Birman-Schwinger type principle for discrete eigenvalues of $L_u$. 

\begin{theorem}[Birman-Schwinger Principle]
\label{thm:BS}
Suppose that $u \in L^p\cap L^\infty$ for some $p \in [1,\infty)$. A number $\eta=-\tfrac{1}{2e}-\mu^2$ is an $L^2$-eigenvalue of $L_u$ if and only if $1$ is an eigenvalue of the operator $C(\mu)$. In this case, the dimensions of the eigenspaces are the same. If $u(x)\le 0$, then $L_u$ has no eigenvalues below $\tfrac{1}{2e}$. 
\end{theorem}

\begin{proof} Suppose that $L_u \varphi=\eta\varphi$ for $\eta$ of the above form. Since $\varphi \in D(L_u)$, we have $|D|^\frac12 e^D u \in L^2(\R)$ and hence $A(\mu) \varphi \in L^2(\R)$. We may compute
$$
	A^{-1}(L_u-\eta)\varphi =  \left( 1 - (A^{-1} e^D) \tilde u (e^D A^{-1}) \right) A \varphi = 0
$$	
so that
$$ A \varphi = (A^{-1} e^D) \tilde u (e^D A^{-1})A\varphi$$
or
$$ \psi = C(\mu) \psi, $$
where $\psi = A(\mu) \varphi$. 

On the other hand, if $\psi \in L^2(\R)$ with $\psi = C(\mu) \psi$, let $\varphi=A^{-1}\psi$. Then $\varphi \in L^2(\R)$ and $A(\mu)\varphi \in L^2(\R)$ so that $\varphi\in Q(q_u)$. From the equation for $\psi$, we compute
	\begin{equation*}
		\varphi = (A^{-2} e^D)\tilde u e^D \varphi
	\end{equation*}
	so that for any $\chi \in Q(q_u)$, 
	\begin{align*}
		\langle |D|^\frac12 e^D\chi,&|D|^\frac12 \sgn (D) e^D \varphi \rangle
			- \eta \langle \chi, \varphi \rangle \\ 
			&= \langle |D|^\frac12 e^D\chi, |D|^\frac12 \sgn(D) (A^{-2} e^D) u e^D \varphi \rangle - \eta\langle \chi,(A^{-2} e^D) u e^D \varphi \rangle\\
            &=\langle e^D\chi, \left((L_0-\eta)A^{-2}\right)ue^D\varphi\rangle\\
			&=\langle e^D \chi, u e^D \varphi  \rangle
	\end{align*}
	from which it follows that
	$$ q_u(\chi,\varphi) = \eta \langle \chi,\varphi \rangle $$
	or $L_u\varphi = \eta \varphi$. 
	
Finally, if $u(x)\le 0$ we have
$$ q_u(\varphi,\varphi) = \langle |D|^\frac12 e^D \varphi,|D|^\frac12 \sgn(D) e^D \varphi \rangle - \langle e^D \varphi, u e^D \varphi \rangle \geq -\frac{1}{2e} \langle \varphi,\varphi \rangle
$$
so that $L_u$ has no eigenvalue in $(-\infty,-\tfrac{1}{2e})$. 
\end{proof}

The following standard result reduces estimates on the number of discrete eigenvalues to the case for $u$ non-negative.
\begin{lemma}\label{lem: u positive}
    Let $u\in L^p\cap L^\infty$ for some $p\in[1,\infty)$ and be real-valued, and let $\mu>0$. Then the number of eigenvalues of $L_u$ below $-\frac1{2e}-\mu^2$ is no more than the number of of eigenvalues of $L_{u_+}$ below $-\frac1{2e}-\mu^2$. Here $u_+=\max(u,0)$ is the positive part of $u$.
\end{lemma}
\begin{proof}
    Recall from Lemma \ref{lem: ess spec of Lu} that $-\frac1{2e}$ is the bottom of the essential spectrum. By the quadratic form version of the min-max principle (Theorem XIII.2 in \cite{RSIV:1978}),
    %% Slight change in code to avoid warning about "\atop" and also add space between sup and inf
    \begin{equation}
    \label{eq: min max}
        \lambda_n(L_u) = \sup_{V\in S_n} \,\,
                        \inf_{\substack{\varphi\perp V, \varphi\in Q(L_u)\\ \|\varphi\|=1} }
                        q_u(\varphi,\varphi)
    \end{equation}
    is the $n$-th smallest eigenvalue of $L_u$ below $-\frac1{2e}$ (counting multiplicity), or is $-\frac1{2e}$ if the number of eigenvalues below $-\frac{1}{2e}$ is less than $n$. In \eqref{eq: min max}, $S_n$ is the set of all $n$ dimesnional subspaces of $L^2$, and $Q(L_u)$ is the form domain of $L_u$.
    Note from the proof of Lemma \ref{lem: domain of Lu} that $L_u$ and $L_{u_+}$ have the same form domain. Since 
    \begin{equation}
        q_u(\varphi,\varphi)-q_{u_+}(\varphi,\varphi)=\langle (u_+-u)e^D\varphi, e^D\varphi\rangle \ge 0,
    \end{equation}
    we see that $\lambda_n(L_u)\ge \lambda_n(L_{u_+})$, and the result follows.
\end{proof}

By the above lemma, we may assume without loss of generality that $u\ge 0$ when estimating the number of discrete eigenvalues of $L_u$.
Recall that $C=C_1C_2$ by \eqref{eq: C decomp}. Since $C_1C_2$ has the same nonzero eigenvalues (counting multiplicities) as $C_2C_1$, we turn our problem to estimating the number of eigenvalues $\ge 1$ of 
\begin{equation}
    M(\mu)=C_2(\mu)C_1(\mu)=|\tilde u|^{\frac12} \left(e^{2D}A^{-2}\right)\tilde u^{\frac12}.
\end{equation}
We have 
\begin{lemma}\label{lem: e-val >=1}
    Assume $u\in L^p\cap L^\infty$ for some $p\in[1,\infty)$, and $u\ge 0$. Given $\mu>0$, the number of eigenvalues of $L_u$ (counting multiplicities)  $\le -\tfrac{1}{2e}-\mu^2$ is no more than the number of eigenvalues of $M(\mu)$ (counting multiplicities) $\ge 1$. 
\end{lemma}
\begin{proof}
    Let $\lambda_n(\mu)$ be the $n$-th largest eigenvalue of $M(\mu)$, counting multiplicity. Note that $M(\mu)$ is positive semi-definite, and we have the following min-max principle:
    \begin{equation}
        \lambda_n(\mu)=\max_{V\in S_n}\min_{\varphi\in V, \|\varphi\|_{L^2}=1}\langle M(\mu)\varphi,\varphi\rangle.
    \end{equation}
    Here $S_n$ is the set of $n$ dimensional subspaces of $L^2$. Recalling \eqref{A}, we have 
    \begin{equation}
        \frac{d}{d\mu}\langle M(\mu)\varphi,\varphi\rangle=-2\mu\langle e^{2D}A^{-4}\sqrt u \varphi,\sqrt u\varphi\rangle\le 0.
    \end{equation}
    It follows that $\lambda_n(\mu)$ is a decreasing function of $\mu$. By Theorem \ref{thm:BS}, if $\mu_1\ge \mu$ and $-\tfrac1{2e}-\mu_1^2$ is an eigenvalue of $L_u$ of multiplicity $k$, then $1$ is an eigenvalue of $C(\mu_1)=C_1(\mu_1)C_2(\mu_1)$ of multiplicity $k$. Thus $1$ is an eigenvalue of $M(\mu_1)=C_2(\mu_1)C_1(\mu_1)$ of multiplicity $k$, and there exists a unique $n$ such that $\lambda_n(\mu_1)=\dots=\lambda_{n+k}(\mu_1)=1$. It follows that $\lambda_m(\mu)\ge \lambda_m(\mu_1)=1$ for all $m=n,\dots n+k$. Thus $M(\mu)$ has $k$ eigenvalues $\ge 1$. It remains to show that distinct eigenvalues of $L_u$ below $-\frac1{2e}-\mu^2$ give rise to different values of $n$ in $\lambda_n(\mu)$ by the above procedure. In fact, if $\mu_2>\mu_1\ge \mu$ and $-\tfrac{1}{2e}-\mu_1^2$ and $-\tfrac{1}{2e}-\mu_2^2$ are both eigenvalues of $L_u$, and if for the same $n$ one has $\lambda_n(\mu_1)=\lambda_n(\mu_2)=1$, then $\lambda_n(\mu_3)=1$ for any $\mu_3\in (\mu_1,\mu_2)$ due to the monotonicity of $\lambda_n$. Reversing the above arguments, this means that $-\tfrac{1}{2e}-\mu_3^2$ is an eigenvalue of $L_u$ for all $\mu_3\in (\mu_1,\mu_2)$. This contradicts the fact that $L_u$ only has discrete eigenvalues below $-\tfrac{1}{2e}$.
\end{proof}

Next we decompose $M(\mu)$ into a singular and a regular part.
\begin{proposition}
    Let $u\in L^1_2\cap L^\infty$ and $u\ge 0$. Then for every $\mu>0$, $M(\mu)\in \mathcal I_2$, and it has a decomposition
    \begin{equation}
        M(\mu)=\frac{L(\mu)}\mu+ R(\mu^2),
    \end{equation}
    where $R(\kappa)$ is an $\mathcal I_2$-valued holomorphic function at $\kappa=0$, and $L(\mu)$ satisfies 
    \begin{equation}\label{eq: L O(mu)}
        \|L(\mu)-L(0)\|_{\mathcal I_2}\lesssim \mu
    \end{equation} 
    for $\mu$ near 0, where $L(0)$ is the integral operator with kernel
    \begin{equation}
        K_{L(0)}=\frac{\sqrt{u(x)u(y)}}{2\sqrt e}e^{-\frac12 i(x-y)}.
    \end{equation}
\end{proposition}
\begin{proof}
    $e^{2D}A^{-2}(\mu)$ is a convolution operator with an $L^2$ kernel 
    \begin{equation}\label{eq: h def}
        h(x,\mu)=\frac{1}{2\pi}\int_\R \frac{e^{2\xi}e^{ix\xi}}{\xi e^{2\xi}+\frac1{2e}+\mu^2}~d\xi.
    \end{equation}
    One obviously has 
    \begin{equation}
        \|M(\mu)\|_{\mathcal I_2}^2\le \|h\|_{L^2}^2\|u\|_{L^1}\|u\|_{L^\infty}. 
    \end{equation}
    Thus $M(\mu)\in \mathcal I_2$. We now shift the contour of integration in \eqref{eq: h def} from $\R$ to $\R\pm i\pi$ when $\pm x>0$. Note that $\xi e^{2\xi}+\tfrac1{2e}+\mu^2$ has exactly two zeros in the strip $-\pi\le \imag \xi\le\pi$ that are complex conjugates of each other. For $\mu>0$ small, they are given by 
    \begin{equation}\label{eq: z pm}
        \zeta_\pm = -\frac12\pm i\sqrt{e}\mu + \mathcal O(\mu^2).
    \end{equation}
    After shifting the contour and accounting for the pole, we get for $\pm x>0$
    \begin{align}
        h(x,\mu)&=\pm \frac{ie^{i\zeta_\pm x}}{1+2\zeta_\pm}+\frac{e^{\mp \pi x}}{2\pi}\int_\R \frac{e^{ix\xi}e^{2\xi}}{(\xi\pm i\pi) e^{2\xi}+\frac1{2e}+\mu^2}~d\xi\notag\\
        &:=\frac{l(x,\mu)}{\mu}+r(x,\mu^2).
    \end{align}
    Note that for $\pm x>0$,
    \begin{equation}
        l(x,\mu)=\pm \frac{i\mu}{1+2\zeta_\pm}e^{i\zeta_\pm x},
    \end{equation}
    and
    \begin{equation}
        l(x,0)=\frac{1}{2\sqrt e}e^{-\frac12i x}
    \end{equation}
    by \eqref{eq: z pm}. We also have
    \begin{equation}
        |l(x,\mu)-l(x,0)|\lesssim \mu \langle x\rangle 
    \end{equation}
    by the elementary estimate $|e^{-\mu|x|}-1|\le \mu|x|$. Letting $L(\mu)$ be the integral operator with kernel $l(x-y,\mu)\sqrt{u(x)u(y)}$, we get 
    \begin{align}
        \|L(\mu)-L(0)\|_{\mathcal I_2}^2&= \int_\R\int_\R |l(x-y,\mu)-l(x-y,0)|^2u(x)u(y)~dx~dy\notag\\
        &\lesssim \mu^2\int_\R\int_\R  (1+x^2+y^2)u(x)u(y)~dx~dy\notag\\
        &\lesssim \mu^2 \|u\|_{L^1_2}^2,
    \end{align}
    from which \eqref{eq: L O(mu)} follows. Similarly, we let $R(\kappa)$ be the integral operator with kernel $r(x-y,\kappa)\sqrt{u(x)u(y)}$. It is easy to see that $r(x,\kappa)$ has uniformly bounded $L^2$ norm for $\kappa\in \C $ near $0$. As a result, $R(\kappa)$ is bounded uniformly in $\mathcal I_2$ for $\kappa$ near 0. A standard argument involving the Cauchy integral formula on difference quotients shows that the $\mathcal I_2$-valued analyticity of $R(\kappa)$ follows from analyticity of the complex-valued functions $\langle \psi, R(\kappa)\varphi\rangle$ for any $L^2$ functions $\psi,\varphi$, which then follows by differentiation under the integral sign.
\end{proof}
Note that $L(0)$ is a rank 1 operator with eigenvalue $\lambda_0=\text{Tr}(L(0))=\frac{\|u\|_{L^1}}{2\sqrt e}$.
We use the decomposition in the above lemma to perturb $\lambda_0$ to an eigenvalue of $\mu M(\mu)=L(\mu)+\mu R(\mu^2)$. In the following, to remove clutter in the notation, we denote $L(0)$ by $L_0$ and $R(0)$ by $R_0$.

\begin{lemma}\label{lem: M e-val}
    Assume $u\in L^p\cap L^\infty$ for some $p\in[1,\infty)$, and $u\ge 0$. For every $\mu>0$ sufficiently close to $0$, there exists a unique eigenvalue $\lambda$ of $\mu M(\mu)$ in a small neighborhood of $\lambda_0$. Furthermore, $\lambda$ has the following expansion:
    \begin{equation}
        \lambda = \lambda_0+\frac{1}{\lambda_0}\text{Tr}\left[L_0(L-L_0+\mu R_0)\right]+\mathcal O(\mu^2).
    \end{equation}
\end{lemma}
\begin{proof}
    We use the regularized Fredholm determinant for Hilbert-Schmidt operators. $\lambda$ is an eigenvalue of $\mu M(\mu)$ if and only if $\text{det}_2\left(1-\frac{\mu M(\mu)}{\lambda}\right)=0$. We write
    \begin{align}
        1-\tfrac{\mu M(\mu)}{\lambda}&=1-\tfrac{L-L_0+\mu R}{\lambda}-\tfrac{L_0}{\lambda}\notag\\
        &=\left(1-\tfrac{L-L_0+\mu R}{\lambda}\right)\left(1-\left(1-\tfrac{L-L_0+\mu R}{\lambda}\right)^{-1}\tfrac{L_0}{\lambda}\right)
    \end{align}
    By \eqref{eq: L O(mu)}, $L-L_0+\mu R$ has small operator norm when $\mu$ is small. So $1-\tfrac{L-L_0+\mu R}{\lambda}$ is invertible when $\mu$ is small and $\lambda$ is close to $\lambda_0$. Thus we have
    \begin{equation}\label{eq: det = 0}
        \text{det}_2\left(1-\left(1-\tfrac{L-L_0+\mu R}{\lambda}\right)^{-1}\tfrac{L_0}{\lambda}\right)=0.
    \end{equation}
    Note that $L_0$ is a rank 1 operator, and so is $\left(1-\tfrac{L-L_0+\mu R}{\lambda}\right)^{-1}\tfrac{L_0}{\lambda}$. Thus \eqref{eq: det = 0} is reduced to
    \begin{equation}
        1=\text{Tr}\left(\left(1-\tfrac{L-L_0+\mu R}{\lambda}\right)^{-1}\tfrac{L_0}{\lambda}\right),
    \end{equation}
    or 
    \begin{equation}\label{eq: e-val of M}
        \lambda = \text{Tr}\left(L_0\left(1-\tfrac{L-L_0+\mu R}{\lambda}\right)^{-1}\right).
    \end{equation}
    The existence and uniqueness of $\lambda$ near $\lambda_0$ solving \eqref{eq: e-val of M} now follows easily from a contraction mapping argument for small $\mu$. To get the small $\mu$ expansion of $\lambda$, we expand $\left(1-\tfrac{L-L_0+\mu R}{\lambda}\right)^{-1}$ in a Neumann series:
    \begin{align}
        \lambda&=\text{Tr}(L_0)+\text{Tr}\left(L_0\tfrac{L-L_0+\mu R}{\lambda}\right)+\mathcal O(\mu^2)\notag\\
        &=\lambda_0+\tfrac1{\lambda_0} \text{Tr}[L_0(L-L_0+\mu R_0)]+\mathcal O(\mu^2).
    \end{align}
\end{proof}

We are now ready to show the main result
\begin{theorem}
    Let $u\in L^1_2\cap L^\infty$ and be real-valued, then the discrete spectrum of $L_u$ is finite.
\end{theorem}
\begin{proof}
    By Lemma \ref{lem: u positive}, we can assume $u\ge 0$. By Lemma \ref{lem: e-val >=1}, we only need to estimate the number of eigenvalues of $M(\mu)$ that are $\ge 1$ as $\mu\searrow 0$. We compute the Hilbert-Schmidt norm of $M(\mu)$ as 
    \begin{align}
        \|M(\mu)\|_{\mathcal I_2}^2&=\text{Tr}(M^2)\notag\\
        &=\tfrac{1}{\mu^2}\text{Tr}\left[(L_0+L-L_0+\mu R)^2\right]\notag\\
        &=\tfrac{1}{\mu^2}\left(\text{Tr}(L_0^2)+2\text{Tr}[L_0(L-L_0+\mu R_0)]+\mathcal O(\mu^2)\right)\notag\\
        &=\tfrac{\lambda_0^2}{\mu^2}+\tfrac{2}{\mu^2}\text{Tr}[L_0(L-L_0+\mu R_0)]+\mathcal O(1).
    \end{align}
    In the last step, we used $\text{Tr}(L_0^2)=(\text{Tr} ~L_0)^2$, as $L_0$ is of rank 1. By Lemma \ref{lem: M e-val}, $M(\mu)$ has an eigenvalue $\tfrac{\lambda}{\mu}$, with small $\mu$ expansion:
    \begin{equation}
        \left(\tfrac{\lambda}{\mu}\right)^2=\tfrac{\lambda_0^2}{\mu^2}+\tfrac{2}{\mu^2}\text{Tr}[L_0(L-L_0+\mu R_0)]+\mathcal O(1).
    \end{equation}
    The number of eigenvalues of $M(\mu)$ that are $\ge 1$ is obviously bounded by 
    \begin{equation}
        \|M(\mu)\|_{\mathcal I_2}^2-\left(\tfrac{\lambda}{\mu}\right)^2=\mathcal O(1)
    \end{equation}
    as $\mu\searrow 0$.  The proof is complete.
\end{proof}                    %%  Bound States
%%%%%%%%%%%%%%%%%%%%
%
%		rhp.tex	-	Riemann-Hilbert
%						Problem
%
%%%%%%%%%%%%%%%%%%%%

\section{Riemann-Hilbert problems}\label{sec: RH}

In this section, we show that the scattering coefficients and Jost solutions satisfy certain jump relations across $\R^+$, which make them solutions of certain nonlocal Riemann-Hilbert problems in the spectral variable $\zeta$. These problems will set the basis for future studies of the inverse scattering transform.

Once again, throughout this section, we assume $u\in L^1_1\cap L^2_1$ is real-valued with sufficiently small norm, so that the existence and analyticity theories of Sections \ref{sec: prop of jost}, \ref{sec: existence}, \ref{sec: scat data} apply.

\begin{lemma}
    For $k\in \mathbb R^+$, we have 
    \begin{equation}\label{eq: GL(zeta) GL(zeta*)}
        G_L(x,k+i0)=G_L(x,k^*+i0)e^{i\lambda x},
    \end{equation}
    \begin{equation}\label{eq: GR(zeta) GR(zeta*)}
        G_R(x,k-i0)=G_R(x,k^*-i0)e^{i\lambda x},
    \end{equation}
    where $\lambda=Z^{-1}(k)$ (see Definition \ref{def: Z inv}).
\end{lemma}
\begin{proof}
    This lemma is essentially the same as Lemma \ref{lem: GL(zeta) GR(zeta*)}, although here we don't try to connect $G_L$ to $G_R$.    Make a change of variable $\xi\to \xi+\lambda$ in the integral defining $G_L$ (see \eqref{def: G}) to get
    \begin{align}
        G_L(x,k+i0)&=\frac1{2\pi}\int_{\Gamma_L}\frac{e^{ix\xi}}{\xi-(k+i0)(1-e^{-2\xi})}~d\xi\notag\\
        &=\frac1{2\pi}\int_{\Gamma_L-\lambda}\frac{e^{ix(\xi+\lambda)}}{\xi+\lambda-(k+i0)(1-e^{-2(\xi+\lambda)})}~d\xi\notag\\
        &=\frac{e^{i\lambda x}}{2\pi}\int_{\Gamma_L}\frac{e^{ix\xi}}{\xi-(k^*+i0)(1-e^{-2\xi})}~d\xi\notag\\
        &=G_L(x,k^*+i0)e^{i\lambda x}.\notag
    \end{align}
    In the above calculations, we used the relations \eqref{eq: zeta-zeta*} and $k e^{-2\lambda}=k^*$. This proves \eqref{eq: GL(zeta) GL(zeta*)}. The corresponding \eqref{eq: GR(zeta) GR(zeta*)} can be proven similarly.
\end{proof}

\begin{lemma}
    For $k\in \R$, we have
    \begin{equation}\label{eq: Me(zeta) M1(zeta*)}
        M_e(x,k+i0)= M_1(x,k^*+i0)e^{i\lambda x},
    \end{equation}
    \begin{equation}\label{eq: Ne(zeta) N1(zeta*)}
        N_e(x,k-i0)= N_1(x,k^*-i0)e^{i\lambda x}.
    \end{equation}
\end{lemma}
\begin{proof}
    Recalling from \eqref{jost prty: def M_e} and  the integral equations of $M_e $, and using \eqref{eq: GL(zeta) GL(zeta*)}, we get
    \begin{align*}
        M_e(x,k+i0)e^{-i\lambda x}&= 1+e^{-i\lambda x}\int_\R G_L(x-y,k+i0)u(y)M_e(y,k+i0)~dy\\
        &=1+\int_\R G_L(x-y,k^*+i0)e^{-i\lambda y}M_e(y,k+i0)~dy.
    \end{align*}
    Comparing this with \eqref{jost prty: def M}, we get \eqref{eq: Me(zeta) M1(zeta*)}. The proof of \eqref{eq: Ne(zeta) N1(zeta*)} is similar.
\end{proof}

\begin{lemma}
    For $k\in \R^+$, $k\ne \frac12$, we have
    \begin{equation}\label{eq: jump N1}
        N_1(x,k+i0)-N_1(x,k-i0)=\rho(k) N_e(x,k-i0),
    \end{equation}
    \begin{equation}\label{eq: jump M1}
        M_1(x,k+i0)-M_1(x,k-i0)=\breve\rho(k) M_e(x,k+i0).
    \end{equation}
    Here $\rho$ and $\breve \rho$ are given in \eqref{def: rho}. 
\end{lemma}
\begin{proof}

In view of the definition 
    \eqref{def: rho}, and relation \eqref{eq: M1 = a N1},
we only need to prove 
\begin{align}
    \label{M1.ab}
M_1(x,k+i0)	
	&=	a(k+i0)N_1(x,k-i0) + b(k+i0)N_e(x,k-i0),\\
\label{N1.abbb}
N_1(x,k-i0)	
	&=	\breve{a}(k-i0) M_1(x,k+i0) +  \breve{b}(k-i0) M_e(x,k+i0),
\end{align}
where $a$, $b$, $\breve{a}$ and $\breve{b}$ are given by \eqref{def: a}, \eqref{def: breve a}, \eqref{def: b}, \eqref{def: breve b}. 
In the following, we only show proof of \eqref{M1.ab}.
Let $k\in \R^+\setminus\{\frac12\}$. Using 
\eqref{eq: G up-down} and \eqref{GL-GR} we get
\begin{equation}\label{eq: T_L up-down}
    [T_L(k+i0)-T_L(k-i0)]m=e^{i\lambda x} \frac{i}{1-2k^*}\int_\R u(y)m(y)e^{-i\lambda y}~dy,
\end{equation}
\begin{equation}\label{eq: T_L-T_R}
    [T_L(k-i0)-T_R(k-i0)]m=\frac{i}{1-2k}\int_\R u(y)m(y)~dy.
\end{equation}
Recall by the definition of $M_1$ that
\begin{equation}
    [I-T_L(k+i0)]M_1(\cdot,k+i0)=1.
\end{equation}
Using \eqref{eq: T_L up-down} and \eqref{eq: T_L-T_R} and the definitions of $a(\zeta)$ and $b(k+i0)$ are given by \eqref{def: a} and \eqref{def: b}, we have
\begin{align}
    &~[I-T_R(k-i0)]M_1(x,k+i0)\notag\\
    =&~1+[T_L(k+i0)-T_L(k-i0)]M_1(x,k+i0)\notag\\
    &\quad +[T_L(k-i0)-T_R(k-i0)]M_1(x,k+i0)\notag\\
    =&~a(k+i0)+b(k+i0)e^{i\lambda x}.
\end{align}
%\begin{equation}
%\label{M1.ie.GR}
%M_1(x,\zeta)=
%	a(\zeta) + b(\zeta) e_\lambda(x) + \int G_R(x-x',\zeta) u(x') M_1(x') \, dx'
%\end{equation}
Since $I-T_R(k-i0)$ is invertible by Theorem \ref{thm: inv T off disc}, we have by the definitions of $N_1$ and $N_e$ that
\begin{align}
    M_1(x,k+i0)
    &=a(k+i0)[I-T_R(k-i0)]^{-1}1+b(k+i0)[I-T_R(k-i0)]^{-1}e^{i\lambda x}\notag\\
    &= a(k+i0)N_1(x,k-i0)+b(k+i0)N_e(x,k).
\end{align}

%\eqref{eq: jump N1} and \eqref{eq: jump M1} are simply restatements of \eqref{M1.ab} and \eqref{N1.abbb} 
    
%     By \eqref{jost prty: def N}
%     \begin{align}
%         N_1(\cdot, \zeta+i0)-N_1(\cdot,\zeta-i0)=&~[G_R(\cdot,\zeta+i0)-G_R(\cdot,\zeta-i0)]*(uN_1(\cdot,\zeta+i0)) \notag\\
% &~+ G_R(\cdot,\zeta-i0)*(u(N_1(\cdot,\zeta+i0)-N_1(\cdot,\zeta-i0)))\notag\\
% =&~ e^{i\lambda x}\frac{i}{1-2\zeta^*}\int_\R u(y)N_1(y,\zeta+i0)e^{-i\lambda y}~dy\notag\\
% &~+ G_R(\cdot,\zeta-i0)*(u(N_1(\cdot,\zeta+i0)-N_1(\cdot,\zeta-i0))).
%     \end{align}
%     By \eqref{jost prty: def N_e}, it remains to verify that 
%     \begin{equation}\label{eq: rho 2nd form}
%         \rho(\zeta)=\frac{i}{1-2\zeta^*}\int_\R u(y)N_1(y,\zeta+i0)e^{-i\lambda y}~dy.
%     \end{equation}
%     \eqref{eq: rho 2nd form} is simply a consequence of the definitions \eqref{beta.zeta}, \eqref{def: rho} and the relation \eqref{eq: M1 = a N1 }. \eqref{eq: jump M1} can be proven similarly.
    
\end{proof}

\begin{lemma}
    For $k\in \R^+$, $k\ne \frac12$, we have
    \begin{equation}\label{eq: a jump}
        a(k+i0)[1-\rho(k)\rho(k^*)]=a(k-i0),
    \end{equation}
    \begin{equation}\label{eq: a breve jump}
        \breve a(k+i0)=[1-\breve \rho(k)\breve \rho(k^*)]\breve a(k-i0),
    \end{equation}
\end{lemma}
\begin{proof}
    By \eqref{eq: jump M1}, \eqref{def: a}, \eqref{def: rho} and \eqref{eq: Me(zeta) M1(zeta*)} we have for $k\in \R^+$, $k\ne \frac12$
    \begin{align}
        a(k+i0)-a(k-i0) &= \breve \rho(k)\frac{i}{1-2k}\int_\R u(x)M_e(x,k+i0)~dx\notag\\
        &= \breve \rho(k) \frac{i}{1-2k}\int_\R u(x) M_1(x,k^*+i0) e^{i\lambda x}~dx\notag\\
        &=\breve \rho(k)a(k^*+i0)\frac{i}{1-2k}\int_\R u(x) N_1(x,k^*+i0) e^{i\lambda x}~dx\notag\\
        &=\breve \rho(k)\rho(k^*) a(k^*+i0). \label{eq: a jump 1}
    \end{align}
    By \eqref{def: a}, \eqref{def: b}, \eqref{def: rho} and \eqref{eq: jump M1}, we have
    \begin{align}
        \rho(k) a(k+i0)&=\frac{i}{1-2k^*}\int_\R u(x)M_1(x,k+i0) e^{-i\lambda x}~dx\notag\\
        &=\frac{i}{1-2k^*}\int_\R u(x)[M_1(x,k-i0)e^{-i\lambda x}+\breve \rho(k)M_1(x,k^*+i0)]~dx\notag\\
        &=\breve \rho(k)+\breve\rho(k)[a(k^*+i0)-1]\notag\\
        &=\breve\rho(k) a(k^*+i0).\label{eq: rho rho breve rel 1}
    \end{align}
    Eliminating $\breve \rho(k)$ from \eqref{eq: a jump 1} and \eqref{eq: rho rho breve rel 1}, we get \eqref{eq: a jump}. A similar calculation gives
    \begin{equation}
        \breve a(k+i0)-\breve a(k-i0)=-\rho(k)\breve \rho(k^*)\breve a(k^*-i0), 
    \end{equation}
    \begin{equation}\label{eq: rho rho breve rel 2}
        \breve\rho(k) \breve a(k-i0)=\rho(k) \breve a(k^*-i0),
    \end{equation}
    from which \eqref{eq: a breve jump} follows.
\end{proof}
\begin{remark}
    From \eqref{eq: a jump}, we get
    \begin{equation}
        \frac{a(k-i0)}{a(k+i0)}=\frac{a(k^*-i0)}{a(k^*+i0)},
    \end{equation}
    as the jump $1-\rho(k)\rho(k^*)$ is invariant under $k\to k^*$.
\end{remark}

\begin{lemma}\label{lem: rho rho* unit circle}
    For $k\in \R^+$, $k\ne \frac12$, we have $\rho(k)\rho(k^*)=\breve \rho(k)\breve \rho(k^*)$, with $|\rho(k)\rho(k^*)|=|\breve \rho(k)\breve \rho(k^*)|<1$. 
    
\end{lemma}
\begin{remark}
    In view of this lemma and \eqref{eq: a a breve = 1}, we see that \eqref{eq: a jump} and \eqref{eq: a breve jump} are essentially the same equation.
\end{remark}
\begin{proof}
    Replacing $k$ by $k^*$ in \eqref{eq: rho rho breve rel 1}, we have
    \begin{equation}\label{eq: rho rho breve rel 3}
        \rho(k^*)a(k^*+i0)=\breve \rho(k^*)a(k+i0).
    \end{equation}
    Multiplying \eqref{eq: rho rho breve rel 1} with \eqref{eq: rho rho breve rel 3}, and recalling $|a(k+i0)|\ge 1$ from Lemma \ref{lem: a b identity}, we obtain
    \begin{equation}
        \rho(k)\rho(k^*)=\breve\rho(k)\breve\rho(k^*)\frac{a(k^*+i0)}{a(k+i0)}\frac{a(k+i0)}{a(k^*+i0)}=\breve\rho(k)\breve\rho(k^*).
    \end{equation}
    By \eqref{eq: tau rho uni} we have
    \begin{equation}\label{eq: rho tau uni sub 1}
        \frac{2k^*-1}{1-2k}|\rho(k)|^2=1-|\tau(k)|^2\in [0,1).
    \end{equation}
    Replace $k$ by $k^*$ to get
    \begin{equation}\label{eq: rho tau uni sub 2}
        \frac{2k-1}{1-2k^*}|\rho(k^*)|^2=1-|\tau(k^*)|^2\in [0,1).
    \end{equation}
    Multiplying \eqref{eq: rho tau uni sub 1} with \eqref{eq: rho tau uni sub 2}, we get
    \begin{equation}
        |\rho(k)\rho(k^*)|^2<1.
    \end{equation}
\end{proof}

\begin{lemma}\label{lem: M N limit infty}
    As $\zeta\to\infty$, $\zeta\in \widetilde{\mathbb C} $,
    \begin{equation}\label{eq: jost sol limit at inf}
        \|M_1(\cdot,\zeta)-1\|_{L^\infty}\to 0,\quad \|N_1(\cdot,\zeta)-1\|_{L^\infty}\to 0
    \end{equation}
    
\end{lemma}
\begin{proof}
    We only show the proof for $M_1$ as that for $N_1$ is similar. Recall from \eqref{jost exist: def M} that $M_1(\cdot,\zeta)=(I-T_L(\zeta))^{-1}1$, where $T_L(\zeta) f=G_L(\cdot,\zeta)*(uf)$. By \eqref{eq: G Lp est 3}, \eqref{eq: G decay 4}, if $|\arg\zeta|\ge\alpha$ for some small $\alpha>0$ and $|\zeta|$ is sufficiently large, 
    \begin{equation}\label{eq: TL norm decay}
        \|T_L(\zeta)\|_{L^\infty\to L^\infty}\le \frac{C}{|\zeta|}\|u\|_{L^1}+\frac{C}{\sqrt{|\zeta|}}\|u\|_{L^2}.
    \end{equation}
    Thus \eqref{eq: jost sol limit at inf} follows. If $|\arg\zeta|<\alpha$, we show the proof for the case $\imag \zeta\ge 0$ only, and the case $\imag \zeta\le 0$ will be analogous. By \eqref{eq: G Lp est 1}, we may write 
    \begin{equation}\label{def: R_L 1}
        T_L(\zeta)=R_L(\zeta)+S_L(\zeta),
    \end{equation}
    where 
    \begin{equation}\label{def: SL(zeta)}
        S_L(\zeta)f = \left(\frac{i}{1-2\zeta}\chi_{\R^+}+H(\cdot,\zeta)\right)*(uf)
    \end{equation}
    and
    \begin{equation}\label{def: R_L 2}
        [R_L(\zeta)f](x) = \frac{i}{1-2\zeta^*}\int_{-\infty}^ x e^{i\lambda (x-y)}u(y)f(y)~dy.
    \end{equation}
    Note that for $\zeta$ in the range we consider, $\lambda$ is defined and has non-negative imaginary part.
    By \eqref{def: SL(zeta)}, \eqref{eq: G decay 1}, we again have
    \begin{equation}\label{eq: S_L decay}
        \|S_L(\zeta)\|_{L^\infty\to L^\infty}\le \frac{C}{|\zeta|}\|u\|_{L^1}+\frac{C}{\sqrt{|\zeta|}}\|u\|_{L^2}.
    \end{equation}
    It is easy to verify by solving an ODE that
    \begin{equation}\label{eq: (I-R) inv}
        [(I-R_L(\zeta))^{-1}g](x)=g(x)+\frac{i}{1-2\zeta^*}\int_{-\infty}^x e^{\frac{i}{1-2\zeta^*}\int_y^x u(t)~dt}e^{i\lambda(x-y)}u(y)g(y)~dy.
    \end{equation}
    We get
    \begin{align}
        M_1(\cdot,\zeta)&=(I-R_L-S_L)^{-1}1\notag\\
        &=\left(I-(I-R_L)^{-1}S_L)^{-1}(I-R_L\right)^{-1}1\notag\\
        &=(I-R_L)^{-1}1+\sum_{n=1}^\infty \left((I-R_L)^{-1}S_L\right)^n1.\label{eq: M1 remainder}
    \end{align}
    By \eqref{eq: (I-R) inv}, $\|\left(I-R_L(\zeta)\right)^{-1}\|_{L^\infty\to L^\infty}$ is uniformly bounded as $\zeta\to\infty$. It follows from \eqref{eq: S_L decay} that the $L^\infty$ norm of the series in \eqref{eq: M1 remainder} tends to 0 as $\zeta\to\infty$. It remains to show that $(I-R_L)^{-1}1$, or
    \begin{equation}\label{eq: def U}
        U(x,\zeta)=\int_{-\infty}^x e^{\frac{i}{1-2\zeta^*}\int_y^x u(t)~dt}e^{i\lambda(x-y)}u(y)~dy
    \end{equation}
    tends to 0 in $L^\infty$ as $\zeta\to\infty$. We approximate $u$ in $L^1$ by an $L^1$-bounded sequence of $u_n\in C_0^\infty$, and denote by $U_n$ the  function corresponding to \eqref{eq: def U} with $u$ replaced by $u_n$. Note that as $\zeta\to\infty$, $\lambda\to\infty$ and $\zeta^*\to0$. Thus for $\zeta$ of large modulus we have
    \begin{equation}
        \|U(\cdot,\zeta)-U_n(\cdot,\zeta)\|_{L^\infty}\le e^{C\|u\|_{L^1}}\left(C\|u-u_n\|_{L^1} \|u_n\|_{L^1}+ \|u-u_n\|_{L^1}\right).
    \end{equation}
    Thus we only need to show $\|U_n(\cdot,\zeta)\|_{L^\infty}\to 0$ as $\zeta\to\infty$ for all fixed $n$. Since $u_n\in C_0^\infty$, we can write $e^{i\lambda(x-y)}=\frac{\partial_y e^{i\lambda(x-y)}}{-i\lambda}$ and integrate by parts:
    \begin{align*}
        &~|U_n(x,\zeta)|\notag\\
        =&~\left|-\frac{u_n(x)}{i\lambda}+\frac{1}{i\lambda}\int_{-\infty}^x e^{\frac{i}{1-2\zeta^*}\int_y^x u_n(t)~dt}e^{i\lambda(x-y)}\left(-\frac{i}{1-2\zeta^*}(u_n(y))^2+u_n'(y)\right)~dy\right|\notag\\
        \le &~\frac1{|\lambda|}\left(\|u_n\|_{L^\infty}+Ce^{C\|u_n\|_{L^1}}(\|u_n\|_{L^2}^2+\|u_n'\|_{L^1})\right)\to 0
    \end{align*}
    as $u_n\in C_0^\infty$ is fixed.
\end{proof}

\begin{corollary}\label{eq: a limit infty}
    As $\zeta\to\infty$, $\zeta\in \widetilde{\mathbb C} $,
    \begin{align}
        &|\zeta|\left\|a(\zeta)-1+\frac{i\int_\R u}{2\zeta}\right\|_{L^\infty}\to 0, \\
        &|\zeta|\left\|\breve a (\zeta)-1-\frac{i\int_\R u}{2\zeta}\right\|_{L^\infty}\to 0.
    \end{align}
\end{corollary}
\begin{proof}
    This is obvious by Lemma \ref{lem: M N limit infty} and \eqref{def: a}, \eqref{def: breve a}.
\end{proof}

In order to write $a$ as a solution to an RH problem with jump relation \eqref{eq: a jump}, it is useful to control the growth and decay of $1-\rho(k)\rho(k^*)$ globally on $\R^+$. Lemma \ref{lem: rho rho* unit circle} shows that $1-\rho(k)\rho(k^*)$ is in the unit disc centered at $1$. This disc falls in $\mathbb C\setminus (-\infty,0]$, on which the principal branch of $\log$ is defined such that $\log 1=0$. Using the principal branch logarithm, we have
\begin{lemma}\label{lem: log(1-rho rho*) bound}
    $\log(1-\rho(k)\rho(k^*))\in L^2(\R^+)$, and $\partial_k\log(1-\rho(k)\rho(k^*))\in L^\infty(1,\infty)$.
\end{lemma}
\begin{proof}
    By Corollary \ref{cor: tau rho cont} and Lemma \ref{lem: rho rho* unit circle}, we only need to control $\log(1-\rho(k)\rho(k^*))$ as $k\searrow 0$, $k\to\frac12$, and $k\nearrow \infty$ in order to show that it's in $L^2(\R^+)$.
    It's easy to see from \eqref{def: G}, \eqref{jost exist: def M}, and \eqref{def: a} that  $G_L(x,0)=i\chi_{\R^+}(x)$, $M_1(x,0)=e^{i\int_{-\infty}^x u(s)~ds}$, and $a(0)=e^{i\int_\R u(s)~ds}$. By Remark \ref{rmk: cont at 0}, we have $a(k+i0)\to a(0)\ne 0$ as $k\searrow 0$. We now consider the limit of $b(k+i0)$ as $k\searrow 0$. In fact, $b_1(k+i0)$ is obviously bounded by \eqref{def: b1}, while $k^*\to \infty$. Thus $b(k+i0)\to 0$ by \eqref{def: b}, and $\rho(k)\to 0$ as $k\searrow 0$. 
    
    As $k\nearrow\infty$, we recall from Corollary \ref{eq: a limit infty} that $a(k)\to 1$. We want to show that $b(k+i0)$, and as a consequence $\rho(k)$, is in $L^2(1,\infty)$. From \eqref{eq: M1 remainder} and \eqref{eq: S_L decay}, we have that
    \begin{align}
        &-i(1-2k^*)b(k+i0)\notag\\
        =&\int_\R e^{-i\lambda x} u(x) \{[(1-R_L)^{-1}1](x)+[(I-R_L)^{-1}S_L1(x)]\}~dx+
        \mathcal O(\tfrac{1}{\lambda}).
    \end{align}
    Note that $u\in L^1\cap L^2$. By \eqref{eq: (I-R) inv} and Plancherel's theorem,
    \begin{align}
        &\int_\R e^{-i\lambda x} u(x) [(1-R_L)^{-1}1](x)~dx\notag\\
        =& 2\pi \hat u(\lambda)+\int_\R \left(e^{\frac{i}{1-2k^*}\int_y^\infty u(t)~dt}-1\right)u(y)e^{-i\lambda y}~dy\notag\\
        =&\int_\R e^{\frac{i}{1-2k^*}\int_y^\infty u(t)~dt}u(y)e^{-i\lambda y}~dy\label{eq: rho(tau) tail first}
    \end{align}
    is in $L^2((1,\infty),d\lambda)=L^2((1,\infty),dk)$ as $k=Z(\lambda)$ and $Z'(\lambda)\approx 1$ for $\lambda$ large. We also have that \eqref{eq: rho(tau) tail first} tends to 0 as $\lambda\nearrow\infty$ by the Riemann-Lebesgue Lemma. Similarly, we have
    \begin{align}
        &\int_\R e^{-i\lambda x} u(x) [(1-R_L)^{-1}S_L1](x)~dx\notag\\
        =&\int_\R e^{\frac{i}{1-2k^*}\int_y^\infty u(t)~dt}u(y)(S_L1)(y)e^{-i\lambda y}~dy.\label{eq: rho(tau) tail}
    \end{align}
    We want to show that \eqref{eq: rho(tau) tail} is bounded by $\frac{C}{k}$ for some $C>0$.
    By \eqref{def: SL(zeta)},
   \begin{equation}\label{eq: SL1}
       [S_L 1](x)=\frac i{1-2k}\int_{-\infty}^x u(y)~dy+[H(\cdot,k+i0)*u](x).
   \end{equation}
   The first term of \eqref{eq: SL1} is obviously bounded by $\frac{C\|u\|_{L^1}}{k}$. By \eqref{eq: rho(tau) tail}, it remains to show that $\|H(\cdot,k+i0)*u\|_{L^2}\le \frac{C}{k}$. Recall from \eqref{def: H} that $\widehat{H}$ is the left hand side of \eqref{eq: G sym decay 3}, which is bounded in $L^\infty$ by $\frac{C}{k}$. Thus $\|H*u\|_{L^2}=\|\widehat{H}\hat u\|_{L^2}\le \frac{C\|u\|_{L^2}}{k}$. In summary, we have shown that $\rho(k)\in L^2(1,\infty)$ and $\rho(k)\to 0$ as $k\to\infty$. It follows that $\rho(k)\rho(k^*)\in L^2(1,\infty)$ and $\rho(k)\rho(k^*)\to 0$ as $k\searrow 0$ and as $k\nearrow\infty$. Thus the same is true for $\log(1-\rho(k)\rho(k^*))$.
   
   It remains to control $\log(1-\rho(k)\rho(k^*))$ as $k\to\frac12$.
    Note that 
    \begin{equation}
        |\log(1-\rho(k)\rho(k^*))|\lesssim 1+\left|\log |1-\rho(k)\rho(k^*)|\right|,
    \end{equation}
    and $|1-\rho(k)\rho(k^*)|\le 2$ by Lemma \ref{lem: rho rho* unit circle}. We only need to show $|1-\rho(k)\rho(k^*)|$ has a lower bound as $k\to\frac12$. Observe that
   \begin{align}
       |1-\rho(k)\rho(k^*)|\ge 1-|\rho(k)\rho(k^*)|&\ge \tfrac12(1+|\rho(k)\rho(k^*)|)(1-|\rho(k)\rho(k^*)|)\notag\\
       &=\tfrac12(1-|\rho(k)\rho(k^*)|^2).
   \end{align}
   By \eqref{eq: rho tau uni sub 1}
   \begin{align}
       1-|\rho(k)\rho(k^*)|^2&=1-(1-|\tau(k)|^2)(1-|\tau(k^*)|^2)\notag\\
       &=|\tau(k)|^2+|\tau(k^*)|^2(1-|\tau(k)|^2)\notag\\
       &\ge |\tau(k)|^2.
   \end{align}
   By \eqref{def: tau} and \eqref{def: a}, 
   \begin{equation}
       \tau(k)=\frac{1-2k}{1-2k+ia_1(k+i0)}.
   \end{equation}
   By Theorem \ref{thm: a b cont}, $a_1(k+i0)$ is continuous at $\frac12$, and so $|1-2k+ia_1(k+i0)|\lesssim 1$ as $k\to\frac12$. It follows that  
   \begin{align}
       |1-\rho(k)\rho(k^*)|\ge |\tau(k)|^2 \gtrsim |1-2k|^2
   \end{align}
   as $k\to \frac12$. Thus
   \begin{equation}
       |\log(1-\rho(k)\rho(k^*))|\lesssim 1+|\log|1-2k||
   \end{equation}
   as $k\to \frac12$. This completes the proof that $\log(1-\rho(k)\rho(k^*))\in L^2(\R^+)$.

   Note that
   \begin{equation}
       \partial_k\log(1-\rho(k)\rho(k^*)) = \frac{-\rho'(k)\rho(k^*)-\rho(k)\rho'(k^*)\frac{dk^*}{dk}}{1-\rho(k)\rho(k^*)}.
   \end{equation}
   Recall that as $k\nearrow \infty$, $\rho(k)\rho(k^*)\to 0$. The $L^\infty$ bounds then follows from Lemma \ref{lem: D rho uniform est} and \eqref{eq: k* bounds k large}. In fact, we have $\rho(k), \rho'(k)\in L^\infty(1,\infty)$, $|\rho(k^*)|\lesssim \frac1{k^*}$, $|\rho'(k^*)|\lesssim \frac{1}{k^*|\log k^*|}$, $|\frac{dk^*}{dk}|\lesssim {k^*}$ for $k>1$. We actually also have $\partial_k \log(1-\rho(k)\rho(k^*))\to 0$ as $k\nearrow \infty$.
\end{proof}

The above lemmas imply that $a(\zeta)$ can be recovered from the function $\rho(k)$, $k\in \R^+\setminus \{\frac12\}$ by solving an RH problem. More precisely, 

\begin{theorem}\label{thm: a from rho}
    Let $u\in L^1_1\cap L^2_1$ be real-valued with small norm, and assume $a(\zeta)$ has finitely many zeros $\zeta_j$, $j=1,\dots,N$ in $D_\delta^+$. Then
    \begin{equation}\label{eq: a from rho}
        a(\zeta)=\exp\left(-\frac1{2\pi i}\int_0^\infty\frac{\log(1-\rho(k)\rho(k^*))}{k-\zeta}~dk\right)\prod_{j=1}^N\frac{\zeta-\zeta_j}{\zeta-\zeta_j^*}.
    \end{equation}
\end{theorem}
\begin{proof}
    Denote the right hand side of \eqref{eq: a from rho} by $\tilde a(\zeta)$. By the Plemelj-Privalov theorem, $\tilde a(\zeta)$ is holomorphic in $\mathbb C\setminus [0,\infty)$ with boundary values on $\R^+\pm i0$ satisfying \eqref{eq: a jump} (with $a$ replaced by $\tilde a$). $\tilde a(\zeta)$ matches the zeros and poles of $a(\zeta)$ by Theorem \ref{thm: bound states and a} (i), (ii). Since all zeros of $a$ are simple, we conclude that $\frac{\tilde a(\zeta)}{a(\zeta)}$ has no poles and is in fact entire by a standard argument using Morera's theorem near $\R^+$. Note that $a(\zeta)\to 1$ as $\zeta\to\infty$ in $\widetilde{\C}$ by Corollary \ref{eq: a limit infty}. Also, by the standard $L^2$ and H\"older theory of the Cauchy integral, Lemma \ref{lem: log(1-rho rho*) bound} implies that $\tilde a(\zeta)$ is bounded for $\zeta\in \widetilde{\C}$, $|\zeta|\ge 10$. It follows that $\frac{\tilde a(\zeta)}{a(\zeta)}$ is a bounded entire function, thus must be a constant. Note that $\tilde a(\zeta)\to 1$ as $\zeta\to\infty$ with $|\arg \zeta|>\alpha$ for any $\alpha>0$, Thus $\frac{\tilde a(\zeta)}{a(\zeta)}=1$.
\end{proof}
%\begin{remark}
 %   Theorem \ref{thm: a from rho} accidentally implies that the poles $\zeta_j^*$ of $a$ are also simple. This can be established more directly, without assuming $a$ has finitely many zeros, by proving an analogous identity to \eqref{eq: a' identity} for $\breve a'(\zeta_j^*)$. In fact, one could establish a dual theory for $\breve a$ and $M_1$ in parallel to $a$ and $N_1$. We omit the details for simplicity.
%\end{remark}

To extract the next order of the asymptotic limit of $M_1$ and $N_1$, we impose the extra assumptions that $Du\in L^1$ and $D^2 u \in L^2$. A convenient space that meets all the requirements is $H^{2,2}$. Recall that $H^{2,2}$ is the space of functions $f$ such that $\|f\|_{H^{2,2}}=\sum_{j=0}^2\|\langle x \rangle ^{2-j}D^jf\|_{L^2}<\infty$. We first prove a simple lemma for convenience of later estimates.

\begin{lemma}\label{lem: H^2,2 Fourier decay}
    There exists a $C>0$ such that for all $u\in H^{2,2}$, $x\in\mathbb R$ and $\lambda\in \mathbb C$ with $\imag \lambda\ge 0$,
    \begin{equation}
        \left|\int_{-\infty}^x e^{i\lambda (x-y)}Du(y)~dy\right|\le \frac{C}{\sqrt{|\lambda|}}\|u\|_{H^{2,2}}.
    \end{equation}
\end{lemma}
\begin{proof}
    We decompose the integral into the two regions: $I_1$ is the integral on $y\le x, |y|\le |\lambda|$; $I_2$ is the integral on $y\le x, |y|\ge |\lambda|$. We integrate $I_1$ by part to get 
    \begin{equation}\label{eq: I1 first}
        I_1=\frac{1}{\lambda}\int_{y\le x, |y|\le|\lambda|}e^{i\lambda(x-y)}D^2u(y)~dy + I_3,
    \end{equation}
    where the boundary terms $I_3$ can be estimated by
    \begin{equation}
        |I_3|\lesssim \frac{\|Du\|_{L^\infty}}{|\lambda|}\lesssim \frac{\|u\|_{H^2}}{|\lambda|}.
    \end{equation}
    The integral in \eqref{eq: I1 first} can be estiamted as
    \begin{equation}
        \int_{|y|\le |\lambda|} |D^2 u|^2~dy\le \left(\int_{|y|\le|\lambda|}dy\right)^{\frac12}\|D^2 u\|_{L^2}\le \sqrt{|\lambda|}\|u\|_{H^2}.
    \end{equation}
    Together we get
    $$|I_1|\lesssim\frac{\|u\|_{H^2}}{\sqrt{|\lambda|}}.$$
    Finally
    \begin{equation*}
        |I_2|\le \int_{|y|\ge |\lambda|}|Du(y)|~dy\le \left(\int_{|y|\ge |\lambda|}\langle y\rangle^{-2}dy\right)^{\frac12}\|\langle y\rangle Du(y)\|_{L^2}\lesssim\frac1{\sqrt{|\lambda|}}\|u\|_{H^{2,2}}.
    \end{equation*}
\end{proof}

\begin{lemma}\label{lem: M N first order remainder}
    Assume $u\in H^{2,2}$ with sufficiently small norm. Define the reminders
    \begin{align}
        R_M(\cdot,\zeta)=M_1(\cdot,\zeta)-1-\frac{i}{2\zeta}\left(T_+u-\frac12\int_\R u ~dx\right),\\
        R_N(\cdot,\zeta)=N_1(\cdot,\zeta)-1-\frac{i}{2\zeta}\left(T_+u+\frac12\int_\R u ~dx\right).
    \end{align}
    Then there exists a $C=C(u)>0$ such that for sufficiently large $ \zeta \in\widetilde{\mathbb C}\cup (\R\pm i0) $,
    \begin{equation}
        \|R_M(\cdot,\zeta)\|_{L^\infty} \le \frac{C}{|\zeta|^{3/2}},
    \quad 
        \|R_N(\cdot,\zeta)\|_{L^\infty} \le \frac{C}{|\zeta|^{3/2}},
    \end{equation}
    and for every sufficiently small $\alpha>0$, there exists a $C=C(u,\alpha)>0$, such that for $ \zeta \in \widetilde{\mathbb C}\cup (\R\pm i0) $ sufficiently large, $|\arg\zeta|>\alpha$, we have
    \begin{equation}\label{eq: M1 first order remainder}
        \|R_M(\cdot,\zeta)\|_{L^\infty} \le \frac{C}{|\zeta|^2},\quad
        \|R_N(\cdot,\zeta)\|_{L^\infty} \le \frac{C}{|\zeta|^2}.
    \end{equation}
\end{lemma}
\begin{remark}
    The small norm assumption on $u$ was posed only for the existence theory of Section \ref{sec: existence} to apply so that $M_1$, $N_1$ are defined for almost all $\zeta$. However, since $G_L$, $G_R$ decay in $\zeta$ as $\zeta\to\infty$, the smallness assumption can be dropped if we only define $M_1$ and $N_1$ for $\zeta$ large.
\end{remark}
\begin{proof}
    We only show the proof for $R_M$. The proof for $R_N$ is similar, if only we notice \eqref{GL-GR}.
    We first consider the case $|\arg \zeta|>\alpha$. Take the derivative of \eqref{jost prty: def M} to get
    \begin{equation}\label{eq: DM regularity}
        DM_1=(I-T_L)^{-1}(G_L*(M_1 Du)).
    \end{equation}
    Recall that $\|(I-T_L(\zeta))^{-1}\|_{L^\infty\to L^\infty}$ is uniformly bounded by \eqref{eq: TL norm decay}, and $\|M_1(\cdot,\zeta)\|_{L^\infty}$ is uniformly bounded. We get uniform boundedness of $\|DM_1(\cdot,\zeta)\|_{L^\infty}$ since $Du\in L^1\cap L^2$ and $G_L$ satisfies \eqref{eq: G Lp est 3}, \eqref{eq: G decay 4}. By \eqref{eq: G decay 5}, 
    \begin{align}\label{eq: M1 remainder sub 1}
        &~\left\|M_1(\cdot,\zeta)-1-\frac{i}{2\zeta-1}\left(T_+(uM_1(\cdot,\zeta))-\frac12\int_\R uM_1(x,\zeta) ~dx\right)\right\|_{L^\infty}\notag\\
        \lesssim&~ \frac{1}{|\zeta|}\left\|\frac{\langle \xi\rangle }{|\zeta|+|\xi|}\widehat{uM_1}(\xi)\right\|_{L^1}.
    \end{align}
    Note that $\|uM_1(\cdot,\zeta)\|_{H^1\cap L^1}$ is uniformly bounded, and $T_+$ is bounded from $H^1\cap L^1$ to $H^1+L^\infty\subset L^\infty$ (see \eqref{eq: T bounded}). \eqref{eq: M1 remainder sub 1} implies 
    \begin{equation}\label{eq: M1-1 L infty}
        \|M_1-1\|_{L^\infty}\lesssim \frac{1}{|\zeta|}
    \end{equation}
     for $|\zeta|$ sufficiently large. 
     Rewrite \eqref{eq: DM regularity} as
     \begin{equation}\label{eq: DM regularity 1}
         DM_1=(I-T_L)^{-1}[G_L*((M_1-1)Du)+G_L*Du].
     \end{equation}
     We estimate by using \eqref{eq: G Lp est 3} and \eqref{eq: G decay 4} again that
     \begin{align}
         \|G_L*((M_1-1)Du)\|_{L^\infty}&\le \|G_L\|_{L^\infty+L^2}\|(M_1-1)Du\|_{L^1\cap L^2}\notag\\
         &\lesssim \frac{1}{|\zeta|^{1/2}}\|M_1-1\|_{L^\infty}\|Du\|_{L^1\cap L^2}\notag\\
         &\lesssim \frac{1}{|\zeta|^{3/2}},\label{eq: DM reg split 1}
         \end{align}
     and by \eqref{eq: G decay 5} that
     \begin{equation}\label{eq: DM reg split 2}
         \|G_L*(Du)\|_{L^\infty}\lesssim \frac{1}{|\zeta|}(\|Du\|_{H^1\cap L^1}+\|\langle \xi\rangle \hat u\|_{L^1})\lesssim \frac{1}{|\zeta|}.
     \end{equation}
     It follows from \eqref{eq: DM regularity 1}, \eqref{eq: DM reg split 1}, \eqref{eq: DM reg split 2} that $\|DM_1\|_{L^\infty}\lesssim \frac{1}{|\zeta|}$. Thus 
     \begin{equation}\label{eq: u(M-1) sub 1}
         \|u(M_1-1)\|_{H^1}\lesssim \frac{1}{|\zeta|}.
     \end{equation}This implies that
    \begin{equation}\label{eq: uM1 approx u}
        \left\|T_+(u(M_1-1))-\frac12\int_\R u(M_1-1)~dx\right\|_{L^\infty}\lesssim \frac{1}{|\zeta|}.
    \end{equation}
    Thus we may replace the terms in the big parentheses on the left hand side of  \eqref{eq: M1 remainder sub 1} by $T_+u-\frac12\int_\R u~dx$ with a tolerable error. Also, $\frac1{2\zeta-1}$ can be replaced by $\frac1{2\zeta}$ with a tolerable error. It remains to show that the right hand side of \eqref{eq: M1 remainder sub 1} is of order $\frac1{|\zeta|^2}$. Indeed, 
    \begin{align}\label{eq: M1 remainder last term}
        \left\|\frac{\langle \xi\rangle }{|\zeta|+|\xi|}\widehat{uM_1}(\xi)\right\|_{L^1}&\lesssim \left\|\frac{1}{|\zeta|+|\xi|}\right\|_{L^2}\|u(M_1-1)\|_{H^1}+\left\|\frac{\langle\xi\rangle \hat u(\xi)}{|\zeta|+|\xi|}\right\|_{L^1}\notag\\
        &\lesssim\frac{1}{|\zeta|^{3/2}}+\left\|\frac{1}{(|\zeta|+|\xi|)\langle\xi\rangle }\right\|_{L^2}\|u\|_{H^2} \lesssim\frac{1}{|\zeta|}.
    \end{align}

    We now consider the case $|\arg \zeta|<\alpha$. We will only treat the case when $\imag \zeta\ge 0$, as the other case is similar. Working similarly as above, we get that $\|M_1(\cdot,\zeta)\|_{L^{\infty}}$, $\|DM_1(\cdot,\zeta)\|_{L^\infty}$, $\|uM_1(\cdot,\zeta)\|_{H^1}$ are uniformly bounded. By \eqref{eq: G decay 2}, we get 
    \begin{align}\label{eq: M1 remainder sub 2}
        \bigg\|M_1(x,\zeta)-1&-\frac{i}{1-2\zeta^*}\int_{-\infty}^ x e^{i\lambda(x-y)}u(y)M_1(y,\zeta)~dy\notag\\
        &-\frac{i}{2\zeta-1}\left(T_-(uM_1(\cdot,\zeta))(x)-\frac12\int_\R u(y)M_1(y,\zeta) ~dy\right)\bigg\|_{L^\infty}\notag\\
         \lesssim&~ \frac{1}{|\zeta|}\left\|\frac{\langle \xi\rangle }{|\zeta|+|\xi|}\widehat{uM_1}(\xi)\right\|_{L^1}.
    \end{align}
    We rewrite the first integral term above by integration by parts
    \begin{align}
        &\int_{-\infty}^x e^{i\lambda(x-y)}u(y)M_1(y,\zeta)~dy\notag\\
        =&\frac{1}{i\lambda }\left(-u(x)M_1(x,\zeta)+\int_{-\infty}^x e^{i\lambda(x-y)}\partial_y(u(y)M_1(y,\zeta))~dy\right).
    \end{align}
    Thus we have $\|M_1-1\|\lesssim\frac{1}{|\zeta|}$ once again. We repeat the splitting \eqref{eq: DM regularity 1} and estimate like \eqref{eq: DM reg split 1}, using $\|G_L\|_{L^\infty+L^2}\lesssim 1$ this time, to get
    \begin{equation}
        \|G_L*((M_1-1)Du)\|_{L^\infty}\lesssim\frac{1}{|\zeta|}.
    \end{equation}
    By \eqref{eq: G decay 2} again and Lemma \ref{lem: H^2,2 Fourier decay}, we get
    \begin{align}
        \|G_L*(Du)\|_{L^\infty}&\lesssim \frac{1}{|\zeta|}(\|Du\|_{H^1\cap L^1}+\|u\|_{H^2})+\left\|\int_{-\infty}^x e^{i\lambda(x-y)}Du(y)~dy\right\|_{L^\infty}\notag\\
        &\lesssim \frac{1}{\sqrt{|\zeta|}}.
    \end{align}
    We obtain in a similar way as above that
    $\|DM_1\|_{L^\infty}\lesssim\frac{1}{\sqrt{|\zeta|}}$, and 
    \begin{equation}\label{eq: u(M-1) sub 2}
        \|u(M_1-1)\|_{H^1}\lesssim\frac{1}{\sqrt{|\lambda|}}.
    \end{equation} 
    It follows that
\begin{equation}
        \left\|T_+(u(M_1-1))-\frac12\int_\R u(M_1-1)~dx\right\|_{L^\infty}\lesssim \frac{1}{\sqrt{|\zeta|}}.
    \end{equation}
    Thus we may replace the terms in the big parentheses on the left hand side of \eqref{eq: M1 remainder sub 2} by $T_-u-\frac12\int_\R u$ with a tolerable error. As before, we replace $\frac1{2\zeta-1}$ by $\frac{1}{2\zeta}$ and $\frac{1}{1-2\zeta^*}$ by $1$ with a tolerable error. We obtain from \eqref{eq: M1 remainder sub 2} that 
    \begin{align}\label{eq: M1 remainder sub 3}
        &~\|R_M(x,\zeta)\|_{L^\infty}\notag\\
        \lesssim &~\frac{1}{|\zeta|}\left(\left\|\int_{-\infty}^x e^{i\lambda(x-y)}D(uM_1)(y)~dy\right\|_{L^\infty}+\left\|\frac{\langle \xi\rangle }{|\zeta|+|\xi|}\widehat{uM_1}(\xi)\right\|_{L^1}\right).
    \end{align}
    Note that 
    \begin{equation}
        \|D(u(M_1-1))\|_{L^1}\le \|Du\|_{L^1}\|M_1-1\|_{L^\infty}+\|u\|_{L^1}\|DM_1\|_{L^\infty}\lesssim\frac{1}{\sqrt{|\zeta|}}.
    \end{equation}
    We conclude by using Lemma \ref{lem: H^2,2 Fourier decay} again that
    \begin{align}
        &~\left\|\int_{-\infty}^x e^{i\lambda(x-y)}D(uM_1)(y)~dy\right\|_{L^\infty}\notag\\
        \le&~ \|D(u(M_1-1))\|_{L^1}+\left\|\int_{-\infty}^x e^{i\lambda(x-y)}Du(y)~dy\right\|_{L^\infty}\lesssim\frac{1}{\sqrt{|\zeta|}}.
    \end{align}
    The last term on the right hand side of \eqref{eq: M1 remainder sub 3} can be estimated in the same way as \eqref{eq: M1 remainder last term}, with \eqref{eq: u(M-1) sub 1} replaced by \eqref{eq: u(M-1) sub 2}. We obtain 
    \begin{equation}
        \left\|\frac{\langle \xi\rangle }{|\zeta|+|\xi|}\widehat{uM_1}(\xi)\right\|_{L^1}\lesssim\frac{1}{|\zeta|}.
    \end{equation}
    The proof is complete.
\end{proof}

With the knowledge of the direct scattering problem we have obtained so far, we define for $u\in L^1_1\cap L^2_1$ with sufficiently small norm the  scattering data 
\begin{equation}\label{def: Sigma}
    \Sigma = \{ \rho,\{\zeta_j\}_{j=1}^N, \{C_j\}_{j=1}^N\},
\end{equation}
where
%\begin{equation}
%    I=\int_\R u(x)~dx
%\end{equation}
%is the total mass of $u$; 
$\rho$ is the reflection coefficient given in \eqref{def: rho}; $\{\zeta_j\}_{j=1}^N$ (in general $N$ may be finite or $\infty$) are the poles of the transmission coefficient $\tau(\zeta)$; and finally $\{C_j\}_{j=1}^N$ are the norming constants given by
\begin{equation}\label{def: Cj}
    C_j = \left(\int_\R |\varphi_j(x)|^2~dx\right)^{-1},
\end{equation}
where $\varphi_j$ are the normalized eigenfunctions given in \eqref{def: phi_j}.

We observe that $N_1(x,\zeta)$ as a function of $\zeta$ solves the following non-local RH problem with data depending only on $\Sigma$:

\begin{proposition}\label{prop: RH}
    Let $u\in L^1_1\cap L^2_1$ be real-valued with small norm. Let $W(\zeta)=N_1(x,\zeta)$ be the Jost function constructed by Theorem \ref{thm: inv T off disc} and Theorem \ref{thm: inv T on disc}, and let $\Sigma$ be the scattering data given as above. Then $W(\zeta)$ is holomorphic for $\zeta\in \mathbb C\setminus[0,\infty)\setminus \{\zeta_j~|~j=1,\dots, N\}$ such that
    \begin{enumerate}
        \item $W$ has non-tangential boundary values on $\R^+\pm i 0$, such that
        \begin{equation}\label{eq: W jump}
            W(k+i0)-W(k-i0)=\rho(k)e^{ix\lambda}W(k^*-i0),
        \end{equation}
        where $\lambda=Z^{-1}(k)$ (see Definition \ref{def: Z inv}).
        \item $W$ has simple poles only at each $\zeta_j$ with residue
        \begin{equation}\label{eq: W residue}
            \underset{\zeta=\zeta_j}{\text{Res}}~ W = iC_j e^{\lambda_j (1+ix) }W(\zeta_j^*),
        \end{equation}
        where $\lambda_j = Z^{-1}(\zeta_j)$.
        \item $W(\zeta)\to 1$ as $\zeta\to\infty$ in $\mathbb C\setminus \R^+$.
    \end{enumerate}
        
\end{proposition}
\begin{proof}
    The holomorphy of $W$ and existence of its boundary values are proven in Theorem \ref{thm: general meromorphy of M N}. \eqref{eq: W jump} follows from \eqref{eq: jump N1}, \eqref{eq: Ne(zeta) N1(zeta*)}. \eqref{eq: W residue} follows from Theorem \ref{thm: bound states and a} (i), (ii), and Lemma \ref{lem: M=aN}. The limit of $W$ at infinity follows from Lemma \ref{lem: M N limit infty}.
\end{proof}
\begin{remark}
    Of course, we can state a dual RH problem for $M_1(x,\zeta)$. The theory will be similar.
\end{remark}
Note that the non-locality of the jump condition in the above RH problem is contained entirely in the inner composition of $W$ with $k\mapsto k^*$.
If one can solve the above function $W$ from only the given scattering data $\Sigma$, then $u$ can be recovered from $\Sigma$, as is shown by the following 
\begin{proposition}\label{prop: rec u}
    Suppose $u\in H^{2,2}$ is real-valued with sufficiently small norm. Let $W(\zeta)=N_1(x,\zeta)$ be given as in Proposition \ref{prop: RH}. Then 
    \begin{equation}
        u(x) = -2 \real\lim_{\zeta\to\infty} \zeta(W(\zeta)-1).
    \end{equation}
\end{proposition}
\begin{proof}
    This is a direct consequence of Lemma \ref{lem: M N first order remainder}, recalling that $Tu$ and $\int_\R u$ are real.
\end{proof}	                            %%	Riemann-Hilbert problem
\appendix
\section{Time evolution of the scattering data}\label{app: A}

We  now derive formally the time evolution of the scattering data $\Sigma$ defined in section \ref{sec: RH}. In particular, we assume $u$ is a classical solution with sufficiently rapid decay and regularity in space. We regard all Jost solutions as depending on $t$ now since $u$ also does.
We prefer to work with the original eigenfunctions here. In view of \eqref{eq: M1 vs eik}, \eqref{eq: N1 vs eik}, we define for $\zeta\in \widetilde{\C}$ not equal to any of the poles of $a(\zeta,t)$ or $\tfrac12-i0$, 
\begin{equation}
    e^{-i\zeta x}_{-\infty} = e^{-i\zeta x} e^{-D}M_1(x,\zeta,t)=e^{-i\zeta x}M_1(x+i,\zeta,t),
\end{equation}
and for $\zeta\in \widetilde{\C}$ not equal to and zeros of $a(\zeta,t)$ or $\tfrac12+i0$
\begin{equation}\label{def: e i zeta x+}
    e^{-i\zeta x}_{+\infty} = e^{-i\zeta x}e^{-D}N_1 (x,\zeta,t)= e^{-i\zeta x}N_1(x+i,\zeta,t).
\end{equation}
It follows from Lemmas \ref{lem: M1 improved regularity}, \ref{lem: improved regularity k=1/2}, \ref{lem: equiv int diff}, \ref{lem: equiv int diff on R+}, \ref{lem: equiv diff int eq 1/2} (see also the discussion at the beginning of Section \ref{sec: lax pair} for the definition of $e^{-D}L_u$) that 
\begin{equation}\label{eq: e val general zeta}
    (e^{-D} L_u) e^{-i\zeta x}_{\pm\infty}=-\zeta e^{-2\zeta}e^{-D} e^{-i\zeta x}_{\pm\infty},
\end{equation}
and for $\zeta\notin \R^+-i0$,
\begin{equation}
    e^{-i\zeta (x+iy)}_{-\infty}\sim e^{-i\zeta (x+iy)} \quad \text{ as }x\to-\infty \text{ for all }y\in [-1,1],
\end{equation}
for $\zeta\notin \R^++i0$,
\begin{equation}
    e^{-i\zeta (x+iy)}_{+\infty}\sim e^{-i\zeta (x+iy)} \quad \text{ as }x\to+\infty \text{ for all }y\in [-1,1].
\end{equation}
for all $y\in [-1,1]$. Note that $e^{-i\zeta x}_{\pm\infty}$  can have exponential growth as $x\to\pm\infty$ unless $\zeta$ is real. The $e^{-ikx}_{\pm\infty}$ given by \eqref{eq: eigen original}, \eqref{eq: asym original} coincide now with $e^{-i(k-i0)x}_{+\infty}$ and $e^{-i(k+i0)x}_{-\infty}$.  

By taking the $t$-derivative of \eqref{jost prty: def M}, \eqref{jost prty: def N}, and \eqref{jost prty: M1 ext}, we obtain that $\partial_tM_1$, $\partial_t N_1$, like $M_1$, $N_1$, extend to functions holomorphic  in $z=x+iy$ on the strip $S$ and continuous on $\overline S$, and depend meromorphically on $\zeta\in \C\setminus \R^+$. We can modify the proof of \eqref{jost prty: M1 limit left} to \eqref{jost prty: N1 limit left} accordingly to obtain 
\begin{equation}
    \lim_{x\to-\infty}\partial_t M_1 = 0,~\lim_{x\to+\infty} \partial_t M_1=\partial_t a,
\end{equation}
\begin{equation}
    \lim_{x\to-\infty}\partial_t N_1 = \partial_t\left(\tfrac{1}{a}\right),~\lim_{x\to+\infty} \partial_t N_1=0,
\end{equation}
uniformly for $y\in [0,2]$, for $\zeta\in \C\setminus \R^+$ not equal to the poles or zeros of $a(\zeta,t)$.
In other words, the $t$-derivative commutes with the asymptotic limits of $M_1,N_1$ at $x=\pm \infty$. Similar arguments apply to $x$-derivatives of  $M_1$, $N_1$, and when $\zeta\in\mathbb R^+\pm i0$. By the above discussion of $e^{-ikx}_{\pm\infty}$, we also obtain that for all $k\in \mathbb R$, the $x$ and $t$-derivatives of $e^{-ikx}_{\pm\infty}$ commute with their asymptotics at $x=\pm\infty$. %Note that $e^{-ikx}_{\pm\infty}$ depends on $t$ implicitly through $u$. 

\begin{lemma}
For all $\zeta\in \widetilde{\C}$ for which $e^{-i\zeta x}_{\pm\infty}$ is defined, 
\begin{equation}\label{eq: time evo eig}
    \partial_t e^{-i\zeta x}_{\pm\infty} + B_ue^{-i\zeta  x}_{\pm\infty} = -i(\zeta -\zeta ^2)e^{-i\zeta x}_{\pm\infty}.
\end{equation}
\end{lemma}
\begin{proof}
    Since both sides depend meromorphically on $\zeta$ with continuous boundary values on $\mathbb R^+\pm i0$, we only need to show \eqref{eq: time evo eig} for $\zeta=k\in \mathbb R^-$. Take the $t$-derivative of \eqref{eq: e val general zeta} and use Lemma \ref{lem: Lax pair} to get
\begin{equation}
    e^{-D}[L_u,B_u]e^{-ikx}_{\pm\infty} +(e^{-D}L_u) \partial_t e^{-ikx}_{\pm\infty}=-ke^{-2k}e^{-D}\partial_t e^{-ikx}, 
\end{equation}
or
\begin{equation}
    (e^{-D}L_u)(\partial_t +B_u)e^{-ikx}_{\pm\infty}=-ke^{-2k}e^{-D}(\partial_t + B_u)e^{-ikx}_{\pm\infty}.
\end{equation}
Thus $(\partial_t+B_u)e^{-ikx}_{\pm\infty}$ satisfies the same equation \eqref{eq: e val general zeta} as $e^{-ikx}_{\pm\infty}$. By uniqueness of the Jost solutions with the given asymptotics (see Lemma \ref{lem: equiv int diff} (b)),  it remains to show 
\begin{equation}
    (\partial_t+B_u)e^{-ik\cdot}_{\pm\infty}(x+iy)\sim -i(k-k^2)e^{-ik(x+iy)} \quad \text{ as }x\to\pm\infty,
\end{equation}
for all $y\in [-1,1]$. This is obvious by \eqref{Lax pair: B}, the commutativity of the asymptotics of $e^{-ikx}_{\pm\infty}$ with the $t$ and $x$-derivatives, and the decay of $T_iu_x$ at infinity due to the decay of $u_x$.
\end{proof}

We can now compute evolution of the scattering coefficients $a(\zeta,t)$ and $b(k,t)$.

\begin{lemma}
    For $\zeta\in \widetilde{\C}\setminus \{\tfrac12\pm i0\}$ which is not a pole of $a(\zeta,0)$, we have
    \begin{equation}\label{eq: a evo}
        a(\zeta,t)=a(\zeta,0).
    \end{equation}
    For $k\in \mathbb R^+\setminus \{\tfrac12\}$, we have
    \begin{equation}\label{eq: b evo}
        b(k,t) = b(k,0)e^{i(k^*-k)(1-k-k^*)t}.
    \end{equation}
 \end{lemma}
 \begin{proof}
     Since $a(\zeta,t)$ is meromorphic in $\zeta$ with continuous boundary values, we only need to show \eqref{eq: a evo} for $\zeta=k\in \R^-$. For such $k$, we have by \eqref{eq: M1 = a N1}
     \begin{equation}\label{eq: e_-inf = a e_+inf}
         e^{-i k x}_{-\infty}=a(k,t)e^{-ikx}_{+\infty}.
     \end{equation}
     Taking the $t$-derivative and using \eqref{eq: time evo eig} we get
     \begin{equation}
         -(B_u+i(k-k^2))e^{-ikx}_{-\infty} = [\partial_t a - a(B_u+i(k-k^2))]e^{-ikx}_{+\infty}.
     \end{equation}
     Using \eqref{eq: e_-inf = a e_+inf} again we get
     \begin{equation}
         (\partial_t a )e^{-ikx}_{+\infty}=0,
     \end{equation}
     which implies \eqref{eq: a evo}.
     Now for $k>0$, $k\ne \tfrac12$, we have by \eqref{M1.ab}
     \begin{equation}\label{eq: e_-inf = ae+be}
         e^{-ikx}_{-\infty} = a(k,t)e^{-ikx}_{+\infty}+b(k,t)e^{-ik^*x}_{+\infty}.
     \end{equation}
     Taking the $t$-derivative and using \eqref{eq: time evo eig}, \eqref{eq: a evo}, we get
     \begin{align}
         -(B_u+i(k-k^2))e^{-ikx}_{-\infty} &= -a(B_u+i(k-k^2))e^{-ikx}_{+\infty}\notag\\
         &\qquad [\partial_t b -b(B_u+i(k^*-k^{*2}))]e^{-ik^*x}_{+\infty}.
     \end{align}
     Using \eqref{eq: e_-inf = ae+be} again we get
     \begin{equation}
         [\partial_t b -ib(k^*-k^{*2}-k+k^2)]e^{-ik^*x}_{+\infty}=0,
     \end{equation}
     which implies \eqref{eq: b evo}.
 \end{proof}

 The above lemma gives us the evolution of $\rho$ and $\{\zeta_j\}_{j=1}^N$ in \eqref{def: Sigma} as 
 \begin{align}
     \rho(k,t) &= \rho(k,0)e^{i(k^*-k)(1-k-k^*)t},\\
     \zeta_j(t) &= \zeta_j(0),
 \end{align}
 as $\rho$ is defined by \eqref{def: rho}, and $\zeta_j$ are the zeros of $a$ as defined at the beginning of Section \ref{sec: bound states}.

 Finally, the evolution of $C_j$ is given by
 \begin{lemma}
     For each $j=1,\dots, N$,
     \begin{equation}\label{eq: evo Cj}
         C_j(t) = C_j(0)e^{2\imag [\zeta_j(1-\zeta_j)]t}.
     \end{equation}
 \end{lemma}
 \begin{proof}
     By \eqref{def: phi_j}, \eqref{def: e i zeta x+}, we see that $\varphi_j = e^{-i\zeta_j^* x}_{+\infty}$.
    By \eqref{eq: time evo eig},
    \begin{equation}
        \partial_t\varphi_j + B_u\varphi_j = -i(\zeta_j^*-\zeta_j^{*2})\varphi_j.
    \end{equation}
    By the skew-symmetry of $B_u$, we have
    \begin{align}
        \partial_t\langle\varphi_j,\varphi_j\rangle &= 2\real \langle \partial_t\varphi_j,\varphi_j\rangle\notag\\
        &=2\real \langle -i(\zeta_j^*-\zeta_j^{*2})\varphi_j,\varphi_j\rangle\notag\\
        &=2\imag(\zeta_j^*-\zeta_j^{*2})\langle \varphi_j,\varphi_j\rangle\notag\\
        &=2\imag(\zeta_j^2-\zeta_j)\langle \varphi_j,\varphi_j\rangle.
    \end{align}
    Here we used \eqref{eq: zeta* = bar zeta} in the last step. \eqref{eq: evo Cj} now follows from \eqref{def: Cj}.
 \end{proof}                             %%  Time evo of scat data 
\section{Convergence of ILW Lax pair to those of KdV and BO}\label{app: B}
In this appendix, we provide some  calculations that show convergence of the ILW Lax pair to those of the Kortweg-deVries (KdV) equation and the Benjamin-Ono (BO) equation. %One may make these formal limits rigorous by proving uniform estimates on the solution as $\delta$ approaches the limiting values, but that is not the main purpose of the current paper. 
As it turns out, to get the usual KdV and BO equations in the limits, we need to take slightly different scaling normalizations of ILW. 

We first consider the limit to KdV. This corresponds to taking $\delta\searrow 0$ in the following version of ILW:
\begin{equation}\label{eq: ILW delta}
u_t+\tfrac1{\delta^2}u_x+2uu_x+\tfrac1\delta T^\delta u_{xx}=0.
\end{equation}
Here 
\begin{equation}
T^\delta f(x) = \text{P.V.}\frac{1}{2\delta}\int_{-\infty}^\infty \coth\frac{\pi(y-x)}{2\delta}f(y)~dy = i\coth(\delta D)f.
\end{equation}
Let $u$ be real-valued and in a bounded subset of $L^\infty_tH^k_x$ for some large $k$. We have as $\delta\searrow 0$:
\begin{equation}
T^\delta u_{xx} = -\tfrac1\delta u_x-\tfrac\delta3 u_{xxx}+\mathcal O(\delta^3),
\end{equation}
where the error term is measured in $L^\infty_t H^{k-5}_x$.
Thus \eqref{eq: ILW delta} can be written as
\begin{equation}
u_t+2uu_x-\tfrac13u_{xxx}=\mathcal O(\delta^2).
\end{equation}
Thus a solution to \eqref{eq: ILW delta} can be regarded as an approximate solution to the KdV equation
\begin{equation}\label{eq: KdV}
u_t+2uu_x-\tfrac13u_{xxx}=0
\end{equation}
for $\delta$ small. We now consider a normalized Lax pair for \eqref{eq: ILW delta}. Let 
\begin{equation}\label{Lax pair: L delta}
L_u^\delta=e^{\delta D}\left[\tfrac{1}{\delta}\left(D-\tfrac1{2\delta}\right)-\tilde u\right]e^{\delta D}+\tfrac1{2\delta^2},
\end{equation}
\begin{equation}\label{Lax pair: B delta}
B_u^\delta=e^{\delta D}(\tfrac i \delta D^2+i\widetilde{T^\delta_-u_x})e^{-\delta D}-\tfrac i\delta L_u^\delta.
\end{equation}
Here $T^\delta_\pm=T^\delta \pm i I$. One can find the connection from the normalized Lax pair to the standard one \eqref{Lax pair: L}, \eqref{Lax pair: B} by inserting $\delta$ in front of $D$ and normalize to get a finite limit. Note that in addition to the usual normalization of adding constants depending on $\delta$, we also need some unusual normalizations such as shifting $D$ to $D-\frac1{2\delta}$, and the addition of an entire copy of $L_u^\delta$ in $B_u^\delta$, which has no effect in the commutator $[L_u^\delta,B_u^\delta]$, but plays an important role in taming the limiting behavior of the $B_u^\delta$ operator. Let us first check \eqref{Lax pair: L delta}, \eqref{Lax pair: B delta} are indeed a Lax pair for \eqref{eq: ILW delta}. In fact, by a calculation similar to Lemma \ref{lem: Lax pair} we get 
\begin{equation}\label{eq: B.lax}
-\partial_t e^{-\delta D}L_u^\delta+e^{-\delta D}[L_u^\delta,B_u^\delta]=\left(\widetilde{u_t}+\tfrac1{\delta^2}\widetilde{u_x}+\tfrac1\delta \widetilde{T^\delta u_{xx}}+2\widetilde{uu_x}\right)e^{\delta D}=0,
\end{equation}
so the assertion holds. Instead of looking at limits of $L^\delta_u$ and $B^\delta_u$, we renormalize and look at 
\begin{align}
    \overline L^\delta_u &= e^{-\delta D}L^\delta_u e^{\delta D}=\left[\tfrac{1}{\delta}\left(D-\tfrac1{2\delta}\right)-\tilde u\right]e^{2\delta D}+\tfrac1{2\delta^2},  \\
    \overline B_u^\delta &= e^{-\delta D}B_u e^{\delta D}=\tfrac i \delta D^2+i\widetilde{T^\delta_-u_x}-\tfrac i\delta \overline L_u^\delta.
\end{align}
\eqref{eq: B.lax} can be written as 
\begin{equation}
    -\partial_t \overline L^\delta_u+[\overline L^\delta_u,\overline B^\delta_u]=\left(\widetilde{u_t}+\tfrac1{\delta^2}\widetilde{u_x}+\tfrac1\delta \widetilde{T^\delta u_{xx}}+2\widetilde{uu_x}\right)e^{2\delta D}=0.
\end{equation}
Let $\epsilon>0$ be some small constant. Regard $\overline L _u^\delta$ and $\overline B_u^\delta$ as operators from $(\langle D\rangle^{4}+e^{\epsilon D})^{-1}L^2$ to $L^2$, where $(\langle D\rangle^{4}+e^{\epsilon D})^{-1}L^2=\{f\in L^2~|~(\langle \xi\rangle^4+e^{\epsilon\xi})\hat f\in L^2\}$. We then have
\begin{align}
\overline L_u^\delta
&=D^2-\tilde u +\delta\left(\tfrac43 D^3-2\tilde u D\right)+\mathcal O(\delta^2),\\
\overline B_u^\delta&=
\tfrac i\delta (D^2-\tilde {u})+i\widetilde{Du}-\tfrac i\delta \overline L_u^\delta+\mathcal O(\delta)\notag\\
&=-i\left(\tfrac43D^3-2\tilde u D-\widetilde{Du}\right)+\mathcal O(\delta).
\end{align}
Here the error is measured in the $(\langle D\rangle^{4}+e^{\epsilon D})^{-1}L^2\to L^2$ operator norm. 
It follows that
\begin{align}
\overline L_u^0&=D^2-\tilde u,\\
\overline B_u^0&=-i\left(\tfrac43D^3-\tilde u D-D\widetilde{u}\right).
\end{align}
Note that these are the Lax pair for the KdV equation \eqref{eq: KdV}.

The limit to the BO equation turns out to be less straightforward. We first take a different $\delta$-scaling of the ILW equation:
\begin{equation}\label{eq: ILW delta BO}
u_t+\tfrac1\delta u_x+2uu_x+T^\delta u_{xx}=0.
\end{equation}
Assuming $u$ as before, we have as $\delta\to\infty$:
\begin{equation}
T^\delta u_{xx} = Hu_{xx}+\mathcal O(\delta^{-2}),
\end{equation}
where the error is measured in $L^\infty_tH^{k-2}_x$.
Here $H=i\text{sgn}(D)$ is the Hilbert transform.
Thus the left hand side of \eqref{eq: ILW delta BO} is
\begin{equation}
u_t+2uu_x+Hu_{xx}+\mathcal O(\delta^{-2}).
\end{equation}
Thus a solution to \eqref{eq: ILW delta BO} can be regarded as an approximate solution to the BO equation
\begin{equation}
u_t+2uu_x+Hu_{xx}=0
\end{equation}
for $\delta$ large. We now consider a different normalization of the Lax pair:
\begin{align}
L &= e^{\delta D}\left(D-\tilde u\right)e^{\delta D},\\
B &= e^{-\delta D}\left(iD^2+\tfrac{iD}{\delta}+i\widetilde{T^\delta_+u_x}\right)e^{\delta D}\\
&=e^{\delta D}\left(iD^2+\tfrac{iD}{\delta}+i\widetilde{T^\delta_-u_x}\right)e^{-\delta D}.\notag
\end{align}
For simplicity of notation, we suppress the superscript $\delta$ and subscript $u$ in $L$ and $B$.
A calculation as before shows 
\begin{equation}\label{eq: Lax pre BO 1}
-\partial_t e^{-\delta D}L +e^{-\delta D}[L,B]=\left(\widetilde{u_t}+\tfrac1\delta\widetilde{u_x}+2\widetilde{uu_x}+\widetilde{T^\delta u_{xx}}\right)e^{\delta D}=0.
\end{equation}
However, these operators don't obviously converge to the Lax pair of the BO equation. To clarify the convergence, we make a series of changes to the Lax pair.  
For $k\in \mathbb R$, and $\eta=-ke^{-2\delta k}$, let
\begin{align}
    L_1&=\widetilde{e^{ikx}}e^{-\delta D}(L-\eta) e^{-\delta D}\widetilde{e^{-ikx}}\\
    &=D-k(1-e^{-2\delta D})-\tilde u,\notag\\
    B_1&=\widetilde{e^{ikx}}e^{\delta D}Be^{-\delta D}\widetilde{e^{-ikx}}-2iDL_1\\
    &=iD^2+\tfrac{iD}{\delta}+i\widetilde{T^\delta_+u_x}-2iD(D+ke^{-2\delta D}-\tilde u)+ik^2-\tfrac{ik}{\delta}.\notag\\
    &=-iD^2-2ikDe^{-2\delta D}+\tfrac{iD}{\delta}+i\widetilde{T^\delta_+u_x}+2iD\tilde u+ik^2-\tfrac{ik}{\delta}.
\end{align}
Note that
\begin{align}
    -\partial_t L_1+[L_1,B_1]&=\widetilde{e^{ikx}}e^{-\delta D}\{\partial_t L+[L,B]+(B-e^{2\delta D}Be^{-2\delta D})(L-\eta)\} e^{-\delta D}\widetilde{e^{-ikx}}\notag\\
    &\quad -2i[L_1,DL_1]\notag\\
    &=\widetilde{e^{ikx}}(e^{-\delta D}Be^{\delta D}-e^{\delta D}Be^{-\delta D})\widetilde{e^{-ikx}}L_1+2i[D,L_1]L_1\notag\\
    &=2i\widetilde{Du}L_1+2i[D,-\tilde u]L_1\notag\\
    &=0.
\end{align}
Thus $L_1$, $B_1$ is also a Lax pair.
Take $k=1$ and let 
\begin{align}
    L_2&=L_1+k\notag\\
    &= D+e^{-2\delta D}-\tilde u,\\
    B_2&=B_1-ik^2+\tfrac{ik}{\delta}\notag\\
    &=-iD(D+2e^{-2\delta D})+\tfrac{iD}{\delta}+i\widetilde{T^\delta u_x}+i(D\tilde u+\tilde u D).
\end{align}
It's easy to see that $L_2$, $B_2$ is still a Lax pair, and that $L_2$ is self-adjoint and bounded from below, whereas $B_2$ is skew-adjoint. Thus for $\kappa>0$ sufficiently large, $L_2+\kappa$ and $B_2+\kappa$ are both invertible. Let 
\begin{equation}
    L_3=L_2+\kappa,~B_3=B_2+\kappa.
\end{equation}
We have
\begin{equation}
    -\partial_t L_3+[L_3,B_3]=0,
\end{equation}
which implies
\begin{equation}\label{eq: inverse lax}
    -B_3^{-1}(\partial_t L_3^{-1})B_3^{-1}+[B_3^{-1},L_3^{-1}]=0.
\end{equation}
Denote $h(\xi,\delta) = \xi+e^{-2\delta \xi}+\kappa$. Note that $h\ge \frac{1+\log(2\delta)}{2\delta}+\kappa$. Now
\begin{align}
    L_3^{-1}&=[h(D,\delta)-\tilde u]^{-1}\notag\\
    &=[h(D,\delta)]^{-1/2}\{1-[h(D,\delta)]^{-1/2}\tilde u[h(D,\delta)]^{-1/2}\}^{-1}[h(D,\delta)]^{-1/2}
\end{align}
Recall that $\kappa>0$ is sufficiently large. As $\delta\nearrow \infty$,  $\|[h(D,\delta)]^{-1/2}\|_{L^2\to L^2}$ is uniformly small and $[h(D,\delta)]^{-1/2}$ converges strongly to $(D+\kappa)^{-1/2}C_+$, where $C_+=\chi_{\R^+}(D)$ is the Cauchy (Szeg\"o) projection onto positive frequencies.  It follows that $L_3^{-1}$ converges strongly to 
\begin{align}
    L_3^{-1,\infty} =~&(D+\kappa)^{-1/2}C_+\notag\\&\quad\{1-(D+\kappa)^{-1/2}C_+
     \tilde u(D+\kappa)^{-1/2}C_+\}^{-1}(D+\kappa)^{-1/2}C_+.
\end{align}
$L^{-1,\infty}_3$ is obviously not invertible on $L^2$. However, if we restrict it to $L^{2,+}=C_+L^2$, it becomes invertible. If fact $L_3^{-1,\infty}\big|_{L^2,+}=(L_4^\infty)^{-1}$, where 
\begin{equation}
    L_4^\infty=C_+(D+\kappa-\tilde u)C_+=D-C_+\tilde u+\kappa.
\end{equation}
We can take the strong limit of $B_3^{-1}$ similarly to get 
\begin{align}
    B_3^{-1,\infty}=~&[-iD^2+\kappa]^{-1/2}C_+\notag\\
    &\quad \{1+[-iD^2+\kappa]^{-1/2}C_+[i\widetilde{Hu_x}+i(D\tilde u +\tilde u D)][-iD^2+\kappa]^{-1/2}C_+\}^{-1}\notag\\
    &\qquad [-iD^2+\kappa]^{-1/2}C_+.
\end{align}
To provide the necessary bounds on the operators, one only need the elementary uniform estimates
\begin{equation}
    \left|\frac{1}{(-i\xi^2-2i\xi e^{-2\delta \xi}+\kappa)^{1/2}}\right|\lesssim \frac1{\sqrt \kappa},
\end{equation}
\begin{equation}
    \left|\frac{\xi}{(-i\xi^2-2i\xi e^{-2\delta \xi}+\kappa)^{1/2}}\right|\lesssim 1.
\end{equation}
We can again write $B_3^{-1,\infty}\big|_{L^{2,+}}=(B_4^\infty)^{-1}$, where
\begin{equation}
    B_4^\infty = -iD^2+ C_+[i\widetilde{Hu_x}+i(D\tilde u+\tilde uD)]+ \kappa.
\end{equation}
The limit of \eqref{eq: inverse lax} gives 
\begin{equation}
    -(B_4^\infty)^{-1}\left(\partial_t (L_4^\infty)^{-1}\right)(B_4^\infty)^{-1}+[{(B_4^\infty)^{-1},(L_4^\infty)^{-1}}]=0,
\end{equation}
which then implies
\begin{equation}
    -\partial_t L_4^\infty+[L_4^\infty, B_4^\infty] = 0.
\end{equation}
Finally letting $L_5^\infty = L_4^\infty-\kappa$, $B_5^\infty = B_4^\infty-\kappa$, we recover the Lax pair for the BO equation on $L^{2,+}$:
\begin{align}
    L_5^\infty&=D-C_+\tilde u,\\
    B_5^\infty &= -iD^2+ C_+[i\widetilde{Hu_x}+i(D\tilde u+\tilde uD)].
\end{align}                            %%  Limits to KdV/BO
%% Appendix

\bibliographystyle{amsplain}
\bibliography{ILW}

\end{document}